\documentclass[11pt]{amsart}
\usepackage[titletoc]{appendix}
\usepackage[noprefix]{nomencl} %run ` makeindex hre-revised.nlo -s nomencl.ist -o hre-revised.nls ' in terminal to update nomenclature 

\renewcommand{\nompreamble}{The following list describes several symbols that will be later used within the body of the article, as well as their first page of occurrence.}

\calclayout

\def\bfomega{\boldsymbol{\omega}}
\def\bfeta{\boldsymbol{\eta}}
\def\bfb{\boldsymbol{b}}
\def\bfp{\boldsymbol{p}}

\usepackage{amssymb,amsmath,amsthm}
\usepackage{enumerate}
\usepackage{mathtools}

\usepackage[utf8]{inputenc}
\usepackage{color}
\usepackage[colorlinks,linkcolor=black,citecolor=black]{hyperref}

\usepackage{tikz-cd}
\tikzcdset{
arrow style=tikz,
diagrams={>={Straight Barb[scale=0.8]}}
}

\def\AA{\mathord{\mathbf{A}}}
\def\CC{\mathord{\mathbf{C}}}
\def\RR{\mathord{\mathbf{R}}}
\def\QQ{\mathord{\mathbf{Q}}}

\def\ZZ{\mathord{\mathbf{Z}}}

\def\NN{\mathord{\mathbf{N}}}
\def\PP{\mathord{\mathbf{P}}}
\def\GG{\mathord{\mathbf{G}}}

\def\transp{^{\mathord{\textsf{T}}}}

\def\opp{{\rm op}}

\def\Sets{\textsf{Set}}

\def\Sch{\textsf{Sch}}

\def\Man{\textsf{Man}}
\def\SmVar{\textsf{SmVar}}

\def\minus{\setminus}
\def\Sym{\mathop{\rm Sym}}
\def\Sp{\mathop{\rm Sp}}
\def\GSp{\mathop{\rm GSp}}
\def\Hom{\mathop{\rm Hom}}
\def\End{\mathop{\rm End}}
\def\Aut{\mathop{\rm Aut}}

\def\Der{\mathop{\rm Der}}
\def\Gal{\mathop{\rm Gal}}
\def\GL{\mathop{\rm GL}}

\def\Res{\mathop{\rm Res}}
\def\rank{\mathop{\rm rank}}
\def\Im{\mathop{\rm Im}}
\def\Re{\mathop{\rm Re}}
\def\Lie{\mathop{\rm Lie}}
\def\MT{\mathop{\rm MT}}
\def\Tr{\mathop{\rm Tr}}

\def\Spec{\mathop{\rm Spec}}

\def\Proj{\mathop{\rm Proj}}

\def\Frac{\mathop{\rm Frac}}

\def\Pic{\mathop{\rm Pic }}

\def\SL{\mathop{\rm SL}}
\def\dlog{\mathop{\rm dlog}}
\def\Et{\mathop{\text{Ét}}}

\def\trdeg{\text{\rm trdeg}}

\def\id{\text{\rm id}}
\def\an{\text{\rm an}}
\def\dR{\text{\rm dR}}

\def\et{\text{\rm ét}}

\def\et{\text{\rm ét}}

\def\std{\text{\rm std}}
\def\comp{\text{\rm comp}}

\def\defeq{\coloneqq}
\def\eqdef{=\vcentcolon}
\def\tensor{\otimes}

\def\to{\longrightarrow}
\def\mapsto{\longmapsto}

\newtheorem{theorem}{Theorem}[section]
\newtheorem{prop}[theorem]{Proposition}
\newtheorem{conj}[theorem]{Conjecture}
\newtheorem{lemma}[theorem]{Lemma}

\newtheorem{coro}[theorem]{Corollary}

\theoremstyle{definition}
\newtheorem{ex}[theorem]{Example}
\newtheorem{defi}[theorem]{Definition}

\newtheorem{obs}[theorem]{Remark}

\numberwithin{equation}{section}

\makenomenclature

\title[Higher Ramanujan equations]{Higher Ramanujan equations \\and periods of abelian varieties}
\author{Tiago J. Fonseca}

\address{Mathematical Institute, University of Oxford, Oxford, United Kingdom}
\email{tiago.jardimdafonseca@maths.ox.ac.uk}

\subjclass[2010]{14D23, 14K99, 37F75, 11J81, 11G18}
\keywords{Moduli of abelian varieties; algebraic differential equations; integrality; values of derivatives of modular functions; Grothendieck's period conjecture; functional transcendence}

\date{}

\begin{document}
\maketitle

\begin{abstract}
  We describe higher dimensional generalizations of Ramanujan's classical differential relations satisfied by the Eisenstein series $E_2$, $E_4$, $E_6$. Such ``higher Ramanujan equations'' are given geometrically in terms of vector fields living on certain moduli stacks classifying abelian schemes equipped with suitable frames of their first de Rham cohomology. These vector fields are canonically constructed by means of the Gauss-Manin connection and the Kodaira-Spencer isomorphism. Using Mumford's theory of degenerating families of abelian varieties, we construct remarkable solutions of these differential equations generalizing $(E_2,E_4,E_6)$, which are also shown to be defined over $\ZZ$.

This geometric framework taking account of integrality issues is mainly motivated by questions in Transcendental Number Theory regarding an extension of Nesterenko's celebrated theorem on the algebraic independence of values of Eisenstein series. In this direction, we discuss the precise relation between periods of abelian varieties and the values of the above referred solutions of the higher Ramanujan equations, thereby linking the study of such differential equations to Grothendieck's Period Conjecture. Working in the complex analytic category, we prove ``functional'' transcendence results, such as the Zariski-density of every leaf of the holomorphic foliation induced by the higher Ramanujan equations.
\end{abstract}

\tableofcontents

\section{Introduction}

\subsection{Motivation}

The \emph{higher Ramanujan equations} are higher dimensional generalizations of the classical Ramanujan differential relations between the Eisenstein series
\begin{align*}
E_{2}(q) = 1 - 24 \sum_{n=1}^{\infty}\frac{n q^n}{1-q^n}\text{, }\ \ E_{4}(q) = 1+240 \sum_{n=1}^{\infty}\frac{n^3 q^n}{1-q^n}\text{, } \ \ 
E_6(q) = 1-504 \sum_{n=1}^{\infty}\frac{n^5 q^n}{1-q^n}\text{.}
\end{align*}
In 1916 \cite{ramanujan16} Ramanujan proved that these formal series satisfy the system of algebraic differential equations
\begin{align} \label{rameq} \tag{R}
\theta E_2 = \frac{E_2^2 - E_4}{12}\text{, }\ \ \theta E_4 = \frac{E_2E_4 - E_6}{3}\text{, } \ \ \theta E_6 = \frac{E_2E_6 - E_4^2}{2}
\text{,}
\end{align}
where $\theta \defeq q\frac{d}{dq}$. The study of equivalent forms of such differential equations actually predates Ramanujan. To the best of our knowledge, Jacobi \cite{jacobi48} was the first to prove in 1848  that his \emph{Thetanullwerte} satisfy a third order algebraic differential equation. Equivalent differential equations were also introduced by Darboux in 1878 and subsequently studied by Halphen and Brioschi; see the introduction of \cite{guillot07} and the references therein.

Further, in 1911,  Chazy \cite{chazy11} considered a differential equation\footnote{In Chazy's original notation (cf. \cite{chazy11} (4)) the equation he considered is written as $y''' = 2yy'' -3(y')^2$. If derivatives in this equation are with respect to a variable $t$, equation (\ref{chazyeq}) is obtained from this one by the change of variables $q=e^{2t}$.} satisfied by the Eisenstein series $E_2$ which plays an important role in his classification of differential equations of third order:
\begin{align*} \label{chazyeq} \tag{C}
\theta^3E_2 = E_2\theta^2E_2 - \frac{3}{2}(\theta E_2)^2\text{.}
\end{align*}
We refer to \cite{ohyama95} for a thorough study of Jacobi's, Halphen's, and Chazy's equations, and the relations between them. Note that Ramanujan's and Chazy's equations concern level 1 (quasi)modular forms, whereas the equations of Jacobi and Halphen involve level 2 (quasi)modular forms.\footnote{The reader might also be familiar with the fact that the $j$-invariant $j = 1728\frac{E_4^3}{E_4^3-E_6^2}$ (as any other elliptic modular function) satisfies an algebraic differential equation of the third order; this follows immediately from the Ramanujan equations, which show that the ring of quasimodular forms $\QQ[E_2,E_4,E_6]$ is closed under $\theta$ (cf. \cite{zagier08}).}

 A higher dimensional generalization of Jacobi's equation concerning \emph{Thetanullwerte} of complex abelian varieties of dimension $2$ was first given by Ohyama \cite{ohyama96} in 1996, and for any dimension by Zudilin \cite{zudilin00} in 2000; see also Bertrand-Zudilin \cite{BZ03}. In another direction, differential equations related to Hilbert modular forms were studied by Resnikoff \cite{resnikoff72} in 1972, and by Pellarin \cite{pellarin05} in 2005.

This paper grew out from our attempt to obtain a more conceptual understanding of the Ramanujan equations and of their higher dimensional extensions, aiming to shed some light on their arithmetic and geometric properties. A key motivation for this program is the crucial role played by the original Ramanujan equations (\ref{rameq}) and by the integrality properties of the series $E_2$, $E_4$, $E_6$ in Nesterenko's celebrated result on the transcendence of their values, when regarded as holomorphic functions on the complex unit disc $\mathbf{D} = \{q \in \CC \mid |q|<1\}$:

\begin{theorem}[Nesterenko \cite{nesterenko96}, 1996]\label{intro-thmnesterenko}
For every $q \in \mathbf{D}\setminus\{0\}$,
\begin{align*}
\trdeg_{\QQ}\QQ(q,E_2(q),E_4(q),E_6(q)) \ge 3\text{.}
\end{align*}
\end{theorem}

Note that Zudilin's work on \emph{Thetanullwerte} \cite{zudilin00} and Pellarin's study of the differential properties of Hilbert modular forms \cite{pellarin05} were also motivated by this same algebraic independence result.

In contrast with the concrete methods of Ohyama, Resnikoff, Zudilin, Bertrand, and Pellarin, relying on modular functions and their derivatives, we follow a geometric approach initially based on Movasati's reinterpretation of the Ramanujan equations as a vector field living on a suitable moduli space of elliptic curves (see \cite{movasati08}, \cite{movasati12}).\footnote{One may argue that this point of view is already contained, although not explicitly in the form of a vector field on a moduli space, in the concept of Serre derivative of modular forms (\cite{serre72} 1.4) and in its geometric interpretation in terms of the Gauss-Manin connection given by Deligne (\cite{katz73} A1.4).} Namely, we construct by purely algebraic methods some higher dimensional avatars of the system (\ref{rameq}), involving suitable moduli spaces of abelian varieties enjoying  remarkable smoothness properties over $\ZZ$. The definition of such moduli spaces presupposes the choice of a PEL moduli problem of abelian varieties, and we work out this theory in the Siegel and the Hilbert-Blumenthal cases. 

Another distinguishing feature of our approach lies in our emphasis on integrality phenomena. Accordingly, it is imperative to work in ``level 1'', although it should be clear that we can also include higher level structures in the picture. This introduces certain representability issues, and naturally leads to the use of (Deligne-Mumford) algebraic stacks. As we shall explain below, the appearance of stacks is not a serious problem, since it is possible to recover a purely scheme-theoretic situation (preserving integrality) if needed.

Besides the construction of the higher Ramanujan equations and the study of some of their geometric properties, we take Nesterenko's theorem as a guiding example to explore the deep connections between such differential equations and the vast landscape of problems in the theory of transcendental numbers pertaining to Grothendieck's Period Conjecture, specially in relation with \emph{periods of abelian varieties}. We also discuss future directions, and speculate on possible applications of our constructions to transcendental number theory, such as the algebraic independence of $\pi$, $\Gamma(1/5)$, and $\Gamma(2/5)$.

\subsection{Higher Ramanujan equations over $\ZZ$; Siegel case}\label{subsec-introhresiegel}

We now explain our main results regarding the construction of the higher Ramanujan equations attached to a Siegel moduli problem. This suffices for the purposes of this introduction, since their Hilbert-Blumenthal counterparts are obtained through a similar yoga. 

Fix an integer $g\ge 1$. Let $k$ be a field, and $(X,\lambda)$ be a principally polarized abelian variety over $k$ of dimension $g$ (here, $\lambda$ denotes a suitable isomorphism from $X$ onto the dual abelian variety $X^t$). The first algebraic de Rham cohomology $H^1_{\dR}(X/k)$ is a $k$-vector space of dimension $2g$ endowed with a canonical subspace $F^1(X/k)\cong H^0(X,\Omega^1_{X/k})$ of dimension $g$ --- the Hodge filtration --- and a non-degenerate alternating $k$-bilinear form 
\begin{align*}
\langle \ , \ \rangle_{\lambda} : H^1_{\dR}(X/k) \times H^1_{\dR}(X/k) \to k
\end{align*}
induced by the principal polarization $\lambda$. By a \emph{symplectic-Hodge basis} of $(X,\lambda)$, we mean a basis $b=(\omega_1,\ldots,\omega_g,\eta_1,\ldots,\eta_g)$ of the $k$-vector space $H^1_{\dR}(X/k)$, such that
\begin{enumerate}
   \item each $\omega_i$ is in $F^1(X/k)$, and
   \item $b$ is symplectic with respect to $\langle \ , \ \rangle_{\lambda}$, that is, $\langle \omega_i,\omega_j \rangle_{\lambda} = \langle \eta_i,\eta_j \rangle_{\lambda} =0$ and $\langle \omega_i, \eta_j \rangle_{\lambda}=\delta_{ij}$ for every $1\le i, j\le g$. 
\end{enumerate} 
The above notions generalize to abelian schemes over arbitrary base schemes (see Paragraph \ref{shb}). We may thus consider a moduli stack $\mathcal{B}_g$ over $\Spec \ZZ$ classifying principally polarized abelian varieties of dimension $g$ equipped with a symplectic-Hodge basis.

Let $\mathcal{A}_g$ denote the moduli stack of $g$-dimensional principally polarized abelian varieties, and $P_g$ denote the Siegel parabolic subgroup of $\Sp_{2g}$.  Then, the stack $\mathcal{B}_g$ can be regarded as a ``principal $P_g$-bundle'' over $\mathcal{A}_g$ via the canonical forgetful map $\mathcal{B}_g \to \mathcal{A}_g$. We shall deduce from this that $\mathcal{B}_g$ is a smooth Deligne-Mumford stack over $\Spec \ZZ$ of relative dimension $2g^2 + g$ (Theorem \ref{smoothdmstack}).

The Deligne-Mumford stack $\mathcal{B}_g$ is not representable by a scheme, or even an algebraic space. Nevertheless, we have the following representability theorem.

\begin{theorem}[see Theorem \ref{repr}]
The Deligne-Mumford stack $\mathcal{B}_g\tensor \ZZ[1/2]$ is representable by a smooth quasi-affine scheme $B_g$ over $\ZZ[1/2]$ of relative dimension $2g^2+g$.
\end{theorem}

This also answers a question of Movasati (see Paragraph \ref{subsec-movasati} below). The representability of $\mathcal{B}_g\tensor \ZZ[1/2]$ by a scheme relies essentially on a theorem of Oda (\cite{oda69} Corollary 5.11) relating $H^1_{\dR}(X/k)$ to the Dieudonné module associated to the $p$-torsion subscheme $X[p]$ when $k$ is a perfect field of characteristic $p$. 

Next, we study the tangent bundle $T_{\mathcal{B}_g/\ZZ}$. We show that the \emph{Gauss-Manin connection} induces a canonical horizontal structure on $T_{\mathcal{B}_g/\ZZ}$ with respect to $\mathcal{B}_g \to \mathcal{A}_g$. Namely, if $\nabla$ denotes the Gauss-Manin connection on the de Rham cohomology of the universal abelian scheme over $\mathcal{B}_g$, and $b=(\omega_1,\ldots,\omega_g,\eta_1,\ldots,\eta_g)$ denotes the universal symplectic-Hodge basis over $\mathcal{B}_g$, then we have the following result.

\begin{theorem}[see Theorem \ref{thm-vertbg} and Definition \ref{defi-ramubbundsiegel}]
  Let $\mathcal{R}_g$ be the subsheaf of $T_{\mathcal{B}_g/\ZZ}$ given by the vector fields $v$ such that $\nabla_v\eta_j=0$ for every $1\le j \le g$. Then $\mathcal{R}_g$ is an integrable subbundle of $T_{\mathcal{B}_g/\ZZ}$ such that
  $$
  T_{\mathcal{B}_g/\mathcal{A}_g} \oplus \mathcal{R}_g = T_{\mathcal{B}_g/\ZZ}\text{.}
  $$ 
\end{theorem}

We then explain how the deformation theory of abelian schemes canonically yields a global trivialization $(v_{ij})_{1\le i \le j \le g}$ of $\mathcal{R}_g$; these are the \emph{higher Ramanujan vector fields} (see Section \ref{ramvecfields} for precise statements). Alternatively, these vector fields may be characterized by the following formulas.

\begin{prop}[see Proposition \ref{caracchamps} and Remark \ref{rem-rephrase}]
For every $1\le i \le j \le g$ we have
\begin{enumerate}
   \item $\nabla_{v_{ij}}\omega_i = \eta_j$, $\nabla_{v_{ij}}\omega_j=\eta_i$, and $\nabla_{v_{ij}}\omega_k = 0$ for every  $k\not\in \{i,j\}$, 
  \item $\nabla_{v_{ij}}\eta_k = 0$ for every $1\le k\le g$,
\end{enumerate}
and these equations completely determine $v_{ij}$.
\end{prop}

Next, we explain in Section \ref{sec-intsol} how to construct a particular \emph{integral solution of the higher Ramanujan equations}. Namely, for $1\le i\le j\le g$, let $q_{ij}$ be a formal variable, and consider the ring
$$
\ZZ(\!(q_{ij})\!) \defeq \ZZ[\![q_{11},\ldots,q_{gg}]\!][(q_{11}\cdots q_{gg})^{-1}]\text{.}
$$
We obtain from Mumford's classical construction of degenerating families of abelian varieties \cite{mumford72}, a principally polarized abelian scheme $(\hat{X}_g,\hat{\lambda}_g)$ over $\ZZ(\!(q_{ij})\!)$ which can be formally represented by the quotient
$$
\hat{X}_g = \mathbf{G}_m^g/\langle (q_{1j},\ldots,q_{gj}) \mid 1\le j \le g \rangle \text{,}
$$
and admits a canonical trivialization of $F^1(\hat{X}_g/\ZZ(\!(q_{ij})\!))\cong H^0(\hat{X}_g,\Omega^1_{\hat{X}_g/\ZZ(\!(q_{ij})\!)})$ given by
$$
\hat{\omega}_j=\frac{dt_j}{t_j}\text{, } \ \ \ 1\le j \le g\text{,}
$$
where $t_1,\ldots,t_g$ denote the coordinates on $\mathbf{G}_m^g$.

\begin{theorem}[see Theorem \ref{thm-intsolsiegel}]\label{intro-thm-instsolsiegel}
  Let $\nabla$ be the Gauss-Manin connection on $H^1_{\dR}(\hat{X}_g/\ZZ(\!(q_{ij})\!))$ and, for $1\le k \le g$, define
$$
\hat{\eta}_k = \nabla_{q_{kk}\frac{\partial }{\partial q_{kk}}}\hat{\omega}_k\text{.}
$$
Then:
\begin{enumerate}
 \item the $2g$-uple $\hat{b}_g=(\hat{\omega}_1,\ldots,\hat{\omega}_g,\hat{\eta}_1,\ldots,\hat{\eta}_g)$ is a symplectic-Hodge basis of $(\hat{X}_g,\hat{\lambda}_g)$, and
 \item the morphism
 $$
 \hat{\varphi}_g: \Spec \ZZ(\!(q_{ij})\!)\to \mathcal{B}_g\text{,}
 $$
 associated to $\hat{b}_g$ by the universal property of $\mathcal{B}_g$, satisfies the differential equations
 $$
q_{ij}\frac{\partial \hat{\varphi}_{g}}{\partial q_{ij}} = v_{ij}\circ\hat{\varphi}_g
$$
for every $1\le i \le j \le g$.
\end{enumerate}
\end{theorem}

In spite of the above result being purely algebraic, we shall actually prove it via analytic methods in Section \ref{sec-analhre}.

At this point, let us briefly remark that it is possible to pass to a scheme-theoretic picture by considering the ring of global sections $\Gamma(\mathcal{B}_g,\mathcal{O}_{\mathcal{B}_g})$. Namely, the higher Ramanujan vector fields ``extend'' to derivations of $\Gamma(\mathcal{B}_g,\mathcal{O}_{\mathcal{B}_g})$, so that the composition of $\hat{\varphi}_g$ with the canonical map $\mathcal{B}_g \to \Spec \Gamma(\mathcal{B}_g,\mathcal{O}_{\mathcal{B}_g})$ still satisfies the higher Ramanujan equations. Since $\mathcal{B}_g\tensor \ZZ[1/2]$ is representable by a quasi-affine scheme, little information is lost when replacing $\mathcal{B}_g$ by $\Spec \Gamma(\mathcal{B}_g,\mathcal{O}_{\mathcal{B}_g})$.

When $g=1$, we shall recall how $B_1$ may be identified, by means of the classical theory of elliptic curves, with an open subscheme of $\AA^3_{\ZZ[1/2]} = \Spec \ZZ[1/2,b_2,b_4,b_6]$. Under this isomorphism, the vector field $v_{11}$ gets identified with
\begin{align*}
2b_4\frac{\partial}{\partial b_2} + 3b_6\frac{\partial}{\partial b_4} + (b_2b_6-b_4^2)\frac{\partial}{\partial b_6}
\end{align*} 
(which is, up to scaling, the vector field associated to Chazy's equation (\ref{chazyeq})), and
$$
\hat{\varphi}_1 = (E_2,\frac{1}{2}\theta E_2, \frac{1}{6}\theta^2E_2)\text{.}
$$
We also show that $B_1\tensor \ZZ[1/6]$ may be identified with the open subscheme $\Spec \ZZ[1/6,e_2,e_4,e_6,(e_4^3-e_6^2)^{-1}]$ of $\AA^3_{\ZZ[1/6]}$, and that, under this isomorphism, the vector field $v_{11}$ gets identified with the ``original'' vector field associated to the Ramanujan equations (\ref{rameq}):
\begin{align} \label{eq-introramvectorfield}
 v=\frac{e_2^2-e_4}{12}\frac{\partial}{\partial e_2} + \frac{e_2e_4-e_6}{3}\frac{\partial}{\partial e_4} + \frac{e_2e_6-e_4^2}{2}\frac{\partial}{\partial e_6}\text{.}
\end{align}
Naturally, under this identification, we have
$$
\hat{\varphi}_1 = (E_2,E_4,E_6)\text{.}
$$

\begin{obs}\label{rem-conditionatinfinity}
  One might remark that our theory in $g=1$ yields a curve $\hat{\varphi}_1$ with coefficients in $\ZZ(\!(q)\!)$, while Eisenstein series are actually regular at $q=0$, i.e., $E_{2k} \in \ZZ[\![q]\!]$. To remedy this (with $g$ arbitrary), one must work more generally with semi-abelian schemes, with logarithmic de Rham cohomology, and with smooth toroidal compactifications of $\mathcal{A}_g$, as developed in \cite{FC90}. In this paper, we shall not elaborate further on this point.
\end{obs}

\subsection{Interlude: Grothendieck's Period Conjecture}

As explained above, questions in Transcendental Number Theory constitute our main source of motivation for the study of these higher dimensional analogs of Ramanujan's equations. In order to fully motivate the precise statements of our next results, we now digress into a discussion of periods of abelian varieties and Grothendieck's conjecture on the algebraic relations between them.

Let $X$ be an abelian variety defined over a subfield $k\subset \CC$. By a \emph{period} of $X$ over $k$, we mean any complex number of the form
\begin{align*}
\int_{\gamma}\alpha
\end{align*}
where $\alpha$ is an element of the first algebraic de Rham cohomology $H^1_{\dR}(X/k)$ and $\gamma \in H_1(X(\CC),\ZZ)$ is the class of a singular 1-cycle. We define the \emph{field of periods} $\mathcal{P}(X/k)$ as the smallest subfield of $\CC$ containing $k$ and all the periods of $X$ over $k$. Equivalently, $\mathcal{P}(X/k)$ may be regarded as the field of rationality of the comparison isomorphism
\begin{align*}
H_{\dR}^1(X/k)\tensor_k \CC \stackrel{\sim}{\to}H^1(X(\CC), \CC) = \Hom (H_1(X(\CC),\ZZ),\CC) \text{.} 
\end{align*}
A central problem in the theory of transcendental numbers is to determine, or simply to estimate, the transcendence degree over $\mathbf{Q}$ of the field of periods $\mathcal{P}(X/k)$.

In a first approach, one might observe that any algebraic cycle in some power $X^n=X\times_k\cdots \times_k X$ of $X$ induces an algebraic relation between its periods (cf. \cite{DMOS82} Proposition I.1.6). Broadly speaking, Grothendieck conjectured that \emph{every} algebraic relation between periods of an abelian variety can be ``explained'' through algebraic cycles on its powers.

A convenient way of giving a precise formulation for Grothendieck's conjecture for abelian varieties is by means of Mumford-Tate groups. Let $X$ be a complex abelian variety, and denote by $H$ the $\QQ$-Hodge structure of weight 1 with underlying $\QQ$-vector space given by $H^1(X(\CC),\QQ)$, and Hodge filtration $F^1H$ given by $H^0(X,\Omega^1_{X/\CC})\subset H^1_{\dR}(X/\CC) \cong H^1(X(\CC),\QQ)\tensor_{\QQ}\CC$. The decomposition $H_{\CC} = F^1H \oplus \overline{F^1H}$ corresponds to the morphism of real algebraic groups
\begin{align*}
h : \CC^{\times} \to \GL(H_{\RR})\text{,}
\end{align*}
where $h(z)$ acts on $F^1H$ by a homothety of ratio $z^{-1}$, and on $\overline{F^1H}$ by a homothety of ratio $\bar{z}^{-1}$. The \emph{Mumford-Tate group} $\MT(X)$ of $X$ is defined as the smallest $\QQ$-algebraic subgroup of $\GL(H)$ such that $h$ factors through $\MT(X)_{\RR}$. It can also be interpreted as the smallest $\QQ$-algebraic subgroup of $\GL(H)\times \GG_{m,\QQ}$ fixing all Hodge classes in twisted mixed tensor powers of the $\QQ$-Hodge structure $H$ (cf. \cite{DMOS82} I.3).

The following formulation of \emph{Grothendieck's Period Conjecture} (GPC) for abelian varieties is a specialization of the ``Generalized Period Conjecture'' proposed by André (\cite{andre04} 23.4.1; see also \cite{lang66} Historical Note pp. 40-44 and \cite{grothendieck66} footnote 10).

\begin{conj}[Grothendieck-André]
For any abelian variety $X$ over a subfield $k\subset \CC$, we have
\begin{align*}
\trdeg_{\QQ}\mathcal{P}(X/k) \stackrel{\text{?}}{\ge} \dim \MT(X_{\CC})\text{.}
\end{align*} 
\end{conj}

It follows from Deligne \cite{deligne80} (cf. \cite{DMOS82} Corollary I.6.4) that we always have the upper bound
\begin{align*}
\trdeg_{\QQ}\mathcal{P}(X/k) \le \dim \MT(X_{\CC}) + \trdeg_{\QQ}k\text{.}
\end{align*}
In particular, if $k$ is contained in the field of algebraic numbers $\overline{\QQ}\subset \CC$ --- the case originally considered by Grothendieck --- the above conjectural inequality becomes the conjectural equality
\begin{align*}
\trdeg_{\QQ}\mathcal{P}(X/k) \stackrel{\text{?}}{=} \dim \MT(X_{\CC})\text{.}
\end{align*} 
In the case $\dim X = 1$, the Mumford-Tate group of a complex elliptic curve may be easily computed. Its dimension only depends on the existence or not of complex multiplication, and GPC predicts  that
\begin{align*}
\trdeg_{\QQ}\mathcal{P}(X/k) \stackrel{\text{?}}{\ge} \left\{\begin{array}{cl}
                                     2 & \text{if }X_{\CC}\text{ has complex multiplication}\\
4& \text{otherwise}\text{.}
                                    \end{array}
                              \right.
\end{align*}
Even in this minimal case, GPC is not yet established in full generality --- only the complex multiplication case is understood; see below. Nevertheless, an approach that has been proved fruitful for obtaining non-trivial lower bounds in the direction of GPC relies on a \emph{modular description} of the fields of periods of elliptic curves, which we now recall.

Let $E$ be a complex elliptic curve and let $j\in \CC$ be its $j$-invariant. Then $E$ admits a model
\begin{align*}
E: \ \ y^2 = 4x^3 -g_2x-g_3
\end{align*}
with $g_2,g_3 \in \QQ(j)$, and we can consider the algebraic differential forms defined over $\QQ(j)$
\begin{align*}
\omega \defeq \frac{dx}{y}\text{, } \ \ \ \eta \defeq x\frac{dx}{y}\text{.}
\end{align*} 
They form a (symplectic-Hodge) basis of the first algebraic de Rham cohomology $H^1_{\dR}(E/\QQ(j))$. If $(\gamma,\delta)$ is any basis of the first singular homology group $H_1(E(\CC),\ZZ)$, we may consider the periods
\begin{align*}
\omega_1 = \int_\gamma \omega\text{, }\ \ \omega_2 = \int_{\delta}\omega\text{, }\ \ \eta_1= \int_{\gamma}\eta\text{, }\ \ \eta_2=\int_{\delta}\eta\text{.}
\end{align*}
We may assume moreover that the basis $(\gamma,\delta)$ is oriented, in the sense that their topological intersection product $\gamma \cdot \delta =1$.

The field of periods of $E$ is given by
\begin{align*}
\mathcal{P}(E/\QQ(j)) = \QQ(j,\omega_1,\omega_2,\eta_1,\eta_2)\text{.}
\end{align*}
Now, observe that $\omega_1\neq 0$ and let
\begin{align*}
\tau \defeq \frac{\omega_2}{\omega_1}\text{.}
\end{align*}
As the basis $(\gamma,\delta)$ of $H_1(E(\CC),\ZZ)$ is oriented, the complex number $\tau$ is in the Poincaré upper half-plane $\mathbf{H}$. By the classical theory of modular forms, we have
\begin{align*}
E_2(\tau) = 12 \left(\frac{\omega_1}{2\pi i} \right)\left(\frac{\eta_1}{2\pi i} \right)\text{, }\ \ E_4(\tau) = 12g_2\left(\frac{\omega_1}{2\pi i} \right)^4\text{, }\ \ E_6(\tau)= -216g_3\left(\frac{\omega_1}{2\pi i} \right)^6\text{.}
\end{align*}
Here, we see the Eisenstein series $E_{2k}$ as analytic functions on $\mathbf{H}$ via the change of variables $q=e^{2\pi i \tau}$. 

Finally, Legendre's period relation and the definition of $j$ show that $\mathcal{P}(E/\QQ(j))$ is a finite extension of the field $\QQ(2\pi i, \tau, E_2(\tau),E_4(\tau),E_6(\tau))$, and we obtain in particular
\begin{align} \label{modular}
\trdeg_{\QQ}\mathcal{P}(E/\QQ(j)) = \trdeg_{\QQ}\QQ(2\pi i, \tau, E_2(\tau),E_4(\tau),E_6(\tau))\text{.}
\end{align}

In this way, the problem of estimating the transcendence degree of fields of periods of elliptic curves translates into the problem of estimating the transcendence degree of values of some analytic functions. Accordingly, the theorem of Nesterenko stated above asserts that, for any $\tau \in \mathbf{H}$,
\begin{align*}
\trdeg_{\QQ}\QQ(e^{2\pi i\tau}, E_2(\tau),E_4(\tau),E_6(\tau)) \ge 3\text{.}
\end{align*}
As an immediate consequence, we obtain
\begin{align*}
\trdeg_{\QQ}\QQ(2\pi i, \tau, E_2(\tau),E_4(\tau),E_6(\tau)) \ge \trdeg_{\QQ} \QQ(E_2(\tau),E_4(\tau),E_6(\tau))\ge 2
\end{align*}
for any $\tau \in \mathbf{H}$. Equivalently, by equation (\ref{modular}), for any complex elliptic curve $E$, we obtain the uniform bound
\begin{align*}
\trdeg_{\QQ}\mathcal{P}(E/\QQ(j)) \ge 2\text{,}
\end{align*}
which is sharp when $E$ has complex multiplication. This last result had already been previously established by Chudnovsky (cf. \cite{chudnovsky80}) via elliptic methods.\footnote{We should also point out that the modular parameter $e^{2\pi i \tau}$, ignored in our discussion, can also be seen as a period. Namely, it is a period of a certain $1$-motive naturally attached to $E$. We refer to \cite{bertolin02} (cf. \cite{andre04} 23.4.3) for further discussion on these matters.}

\subsection{Analytic higher Ramanujan equations, periods of abelian varieties, and transcendence}

In this paper, we also generalize the modular description (\ref{modular}). For this, we consider a complex analytic avatar of $\hat{\varphi}_g$: an analytic map
$$
\varphi_g: \mathbf{H}_g \to B_g(\CC)\text{,}
$$
parametrized in the Siegel upper half-space
\begin{align*}
    \mathbf{H}_g \defeq\{\tau=(\tau_{kl})_{1\le k,l\le g} \in M_{g\times g}(\CC) \mid \tau\transp = \tau\text{, }\Im \tau > 0\}\text{,}
\end{align*}
which, loosely speaking, coincides with $\hat{\varphi}_g$ through the change of variables $q_{kl} = e^{2\pi i \tau_{kl}}$. For instance, under the above identification of $B_1\tensor \ZZ[1/6]$ with an open subscheme of $\AA^3_{\ZZ[1/6]}$, the analytic map $\varphi_1: \mathbf{H}_1=\mathbf{H} \to B_1(\CC)$ is given by
  $$
  \tau \mapsto (E_2(\tau),E_4(\tau),E_6(\tau))\text{.}
  $$
  In other words, $\hat{\varphi}_g$ should be regarded as the ``$q$-expansion'' of $\varphi_g$.
  
  Now, for any $\tau \in \mathbf{H}_g$, let $X_{\tau}$ be the complex abelian variety given by the (polarizable) complex torus $\CC^g/(\ZZ^g + \tau \ZZ^g)$. It admits a canonical principal polarization $\lambda_{\tau}$ induced by the Riemann form
  \begin{align*}
    \CC^g \times \CC^g &\to \RR\\
    (v,w)&\mapsto \Im(\overline{v}\transp (\Im \tau)^{-1}w)\text{.}  
  \end{align*}
 Let $k_{\tau}$ be the field of definition of $(X_{\tau},\lambda_{\tau})$; formally, $k_\tau$ is the residue field of the point in the (coarse) moduli space of principally polarized abelian varieties $A_g$ given by the isomorphism class of $(X_{\tau},\lambda_{\tau})$.

\begin{theorem}[see Theorem \ref{trdeg}]
For any $\tau \in \mathbf{H}_g$, the field of periods $\mathcal{P}(X_{\tau}/k_{\tau})$ is a finite extension of $\QQ(2\pi i , \tau, \varphi_g(\tau))$.
\end{theorem}

Here, $\QQ(2\pi i , \tau, \varphi_g(\tau))$ is defined as the residue field in $\AA_{\QQ}^1\times_{\QQ} \Sym_{g,\QQ}\times_{\QQ} B_{g,\QQ}$ of the complex point $(2\pi i, \tau,\varphi_g(\tau))$,  where $\Sym_g$ denotes the group scheme of symmetric matrices of order $g\times g$.

It follows from the above theorem that 
\begin{align*}
\trdeg_{\QQ}\mathcal{P}(X_{\tau}/k_{\tau}) = \trdeg_{\QQ}\QQ(2\pi i , \tau, \varphi_g(\tau))\text{.}
\end{align*}
This generalized modular description raises the question of whether it is possible to adapt Nesterenko's methods to this higher dimensional setting; see Paragraph \ref{subsec-introconjecture} below. This problem leads us to the study of the \emph{higher Ramanujan foliation}, namely, the holomorphic foliation on $B_g(\CC)$ generated by the higher Ramanujan vector fields. We prove the following result.

\begin{theorem}[see Theorem \ref{densite}]\label{intro-thmzdensity}
Every leaf of the higher Ramanujan foliation on $B_g(\CC)$ is Zariski-dense in $B_{g,\CC}$.
\end{theorem}

This property of a foliation plays an important role, at least in the case in which leaves are one dimensional (where it implies Nesterenko's $D$-property), in the ``multiplicity estimates'' appearing in applications of differential equations to transcendental number theory (cf. \cite{binyamini14}, \cite{nesterenko89}, \cite{nesterenko96}).

The Zariski-density of the image of $\varphi_g : \mathbf{H}_g \to B_g(\CC)$ in $B_{g,\CC}$ also implies the \emph{a priori} stronger result that its graph
\begin{align*}
\{(\tau, \varphi_g(\tau)) \in {\Sym}_g(\CC)\times B_g(\CC) \mid \tau \in \mathbf{H}_g\}
\end{align*}
is Zariski-dense in $\Sym_{g,\CC} \times_{\CC} B_{g,\CC}$.  This can be interpreted as a ``functional version'' of GPC: roughly, it says that there is no algebraic relation simultaneously satisfied  by the periods of every (principally polarized) abelian variety other than the relations given by the polarization data.\footnote{Such ``functional version'' is an example of a statement that must hold if GPC is true. This follows from the existence of $\tau \in \mathbf{H}^g\cap \Sym_g(\overline{\QQ})$ such that $\dim \text{MT}(X_{\tau}) = 2g^2+g+1$ (or, equivalently, $\text{MT}(X_{\tau})=\text{GSp}_{2g,\QQ}$); cf. \cite{sawin18}.}

We shall also use our Zariski-density result to establish a relation between our work and that of Bertrand and Zudilin \cite{BZ03} concerning derivatives of Siegel modular functions.

\begin{prop}[see Paragraph \ref{subsec-bertrandzudilin}]
The field of functions $\QQ(B_{g,\QQ})$, identified with a field of meromorphic functions on $\mathbf{H}_g$ via $\varphi_g$, is a finite extension of the differential field generated by the Siegel modular functions defined over $\QQ$.
\end{prop}

In particular, the generalization of Mahler's result \cite{mahler69} on the algebraic independence of the holomorphic functions $\tau$, $e^{2\pi i \tau}$, $E_2(\tau)$, $E_4(\tau)$, and $E_6(\tau)$, of $\tau \in \mathbf{H}$, obtained by Bertrand and Zudilin \cite{BZ01} in the context of Siegel modular functions, also holds in our context: the set
\begin{align*}
\{(\tau, q(\tau), \varphi_g(\tau)) \in {\Sym}_g(\CC)\times{\Sym}_g(\CC)\times B_g(\CC) \mid \tau \in \mathbf{H}_g\}
\end{align*}
is Zariski-dense in $\Sym_{g,\CC} \times_{\CC}\Sym_{g,\CC}\times_{\CC} B_{g,\CC}$, where $q(\tau) \defeq (e^{2\pi i \tau_{kl}})_{1\le k,l\le g}$.

Our proof of Theorem \ref{intro-thmzdensity} will rely on a characterization of the leaves of the higher Ramanujan foliation in terms of an action by $\Sp_{2g}(\CC)$. In fact, from the complex analytic viewpoint, the complex manifold $B_g(\CC)$ and the higher Ramanujan vector fields admit a simple description in terms of Lie groups.

Namely, we shall explain in Section \ref{gpinterpret} how to realize $B_g(\CC)$ as a domain (in the analytic topology) of the quotient manifold $\Sp_{2g}(\ZZ)\backslash \Sp_{2g}(\CC)$. 

\begin{theorem}[see Theorem \ref{unifhrvf}]
  Under this identification:
  \begin{enumerate}
    \item The vector field $v_{kl}$ is induced by the left invariant holomorphic vector field on $\Sp_{2g}(\CC)$ associated to
\begin{align*}
\frac{1}{2\pi i} \left(\begin{array}{cc} 0 & \mathbf{E}^{kl} \\ 0 & 0  \end{array}\right)\in \Lie {\Sp}_{2g}(\CC)\text{.}
\end{align*}
  \item The map $\varphi_g : \mathbf{H}_g \to B_g(\CC)$ is given by
\begin{align*}
  \tau \mapsto \left[\left(\begin{array}{cc} \mathbf{1}_g & \tau \\
                                             0 & \mathbf{1}_g
                     \end{array}\right)\right] \in {\Sp}_{2g}(\ZZ)\backslash {\Sp}_{2g}(\CC)\text{.}
\end{align*}
  \end{enumerate}
\end{theorem}

In the above statement, $\mathbf{E}^{kl}$ is the symmetric matrix of order $g\times g$ whose entry in the $k$th row and $l$th column (resp. $l$th row and $k$th column) is 1, and whose all other entries are 0, and $\mathbf{1}_g$ denotes the identity matrix of order $g\times g$.

 This result enables us to obtain every leaf of the higher Ramanujan foliation as the image of a holomorphic map $\varphi_{\delta}: U_{\delta} \to B_g(\CC)$ defined on some explicitly defined open subset $U_{\delta}\subset \mathbf{H}_g$ obtained from $\varphi_g$ via a ``twist'' by some element $\delta\in \Sp_{2g}(\CC)$.

In the case $g=1$, the above twisting procedure may be illustrated as follows. Let 
\begin{align*}
\delta = \left(\begin{array}{cc}a & b \\ c& d\end{array} \right) \in {\SL}_2(\CC)\text{,}
\end{align*}
let $U_{\delta} = \{\tau \in \mathbf{H} \mid c\tau +d \neq 0\}$, and define a holomorphic map $\varphi_{\delta}:U_{\delta} \to B_1(\CC)\subset \CC^3$ by
\begin{align*}
\varphi_{\delta}(\tau) = \left((c\tau+d)^2E_2(\tau) + \frac{12c}{2\pi i}(c\tau+d), (c\tau +d)^4E_4(\tau), (c\tau+d)^6E_6(\tau)  \right)
\end{align*}
 Then one may easily check that $\varphi_{\delta}$ satisfy the differential equation
\begin{align*}
\frac{1}{2\pi i}\frac{d\varphi_{\delta}}{d\tau} = (c\tau+d)^{-2}v\circ \varphi_{\delta}
\end{align*}
where $v$ is the classical Ramanujan vector field defined by (\ref{eq-introramvectorfield}).

\subsection{The Hilbert-Blumenthal case and an algebraic independence conjecture} \label{subsec-introconjecture}

Parallel to the above geometric generalization of the Ramanujan equations in terms of a Siegel moduli problem, we may develop similar theories concerning polarized abelian varieties with extra endomorphism structure, which has the effect of producing moduli spaces with fewer dimensions. This might be advantageous for applications to transcendental numbers, which should necessarily take ``special subvarieties'' into account, as we shall explain below. 

To illustrate this point, we consider abelian varieties with \emph{real multiplication}. Namely, let $F$ be a totally real number field of degree $g\ge 1$, and denote by $R$ its ring of integers. Then, an $R$-multiplication (with Rapoport's condition) on a principally polarized abelian variety $(X,\lambda)$ is a morphism of rings $m:R \to \End_kX$ invariant by the Rosatti involution defined by $\lambda$, and for which $F^1(X/k)$ becomes a free $k\tensor_{\ZZ}R$-module of rank 1. The moduli problem of principally polarized abelian varieties endowed with an $R$-multiplication is an example of a Hilbert-Blumenthal moduli problem.

Accordingly, we shall also consider a smooth Deligne-Mumford moduli stack $\mathcal{B}_F$ over $\Spec \ZZ$ of relative dimension $3g$, classifying principally polarized abelian varieties with an $R$-multiplication and a symplectic-Hodge basis ``compatible'' with it. Here, we also have that $\mathcal{B}_F\tensor \ZZ[1/2]$ is representable by a quasi-affine smooth scheme $B_F$ over $\ZZ[1/2]$.

As in the Siegel case, we shall also construct a family of higher Ramanujan vector fields on $\mathcal{B}_F$, and a canonical analytic solution
$$
\varphi_F : \mathbf{H}^g \to B_F(\CC)
$$
with integral ``$q$-expansion'' $\hat{\varphi}_F$ (see Paragraphs \ref{subsec-ramsubbundrm}, \ref{subsec-hrvfhb}, \ref{subsec-integralsolhb}, \ref{subsec-ahrerm}, and \ref{subsec-compatibilityhb} for precise statements). Moreover, we shall also establish a precise relation between the values of $\varphi_F$ with fields of periods of principally polarized abelian varieties with $R$-multiplication (Theorem \ref{trdegrm}). 

\begin{obs}
The Siegel and Hilbert-Blumenthal higher Ramanujan equations are constructed by a similar procedure, and satisfy various natural compatibilities (see Remarks \ref{rem-relationbfbg}, \ref{rem-compatsieglhbhre}, and \ref{rem-relationphianal}). This observation hints to the existence of an underlying theory of higher Ramanujan equations attached to more general Shimura varieties (cf. Section \ref{gpinterpret}). We refer to Movasati \cite{movasati13} for a Hodge-theoretic approach to these questions, which also allows to consider examples unrelated with abelian varieties (cf. Scholium \ref{subsec-movasati} below). 
\end{obs}

In the case of abelian surfaces, we formulate the following algebraic independence conjecture.

\begin{conj}\label{conj-Fquadratic0}
 Let $F$ be a real quadratic number field. Then, for every $\tau\in \mathbf{H}^2\minus {\rm HZ}_F$, we have
 $$
 \trdeg_{\QQ}\QQ(\varphi_F(\tau))\stackrel{?}{\ge} 3\text{.}
 $$
\end{conj}

Here, ${\rm HZ}_{F}$ is a countable union of certain special divisors of $\mathbf{H}^2$, first introduced and studied by Hirzebruch and Zagier (see Paragraph \ref{subsec-hirzebruchzagier}), classifying abelian surfaces with quaternionic multiplication.

The above statement is a higher dimensional analog of the uniform bound
$$
\trdeg_{\QQ} \QQ(E_2(\tau),E_4(\tau),E_6(\tau))\ge 2
$$
for $\tau\in \mathbf{H}$, which can be obtained, as explained above, as a corollary of Nesterenko's theorem. Correspondingly, we shall prove that Conjecture \ref{conj-Fquadratic0} implies Grothendieck's Period Conjecture for complex multiplication abelian surfaces; for instance, by considering the Jacobian of the curve $y^2=1-x^5$, we see that such conjecture for $F=\QQ(\sqrt{5})$ contains the classical conjecture on the algebraic independence of $\pi$, $\Gamma(1/5)$, and $\Gamma(2/5)$ (see Paragraph \ref{subsec-gpccm}).

A natural strategy to attack Conjecture \ref{conj-Fquadratic0} would consist in adapting Nesterenko's method to prove Theorem \ref{intro-thmnesterenko} to our geometric context, and in generalizing it in ``two variables''. A first step in this program was taken in \cite{fonseca18}, where we show that Nesterenko's method, still in one variable, can be cast in purely geometric terms, not relying on the Taylor expansion of explicitly defined analytic functions.

\subsection{Scholia}

\subsubsection{}\label{subsec-movasati}

As acknowledged above, our definition of the moduli stack $\mathcal{B}_g$ was inspired by Movasati's point of view on the Ramanujan vector field in terms of the Gauss-Manin connection on the de Rham cohomology of the universal elliptic curve (cf. \cite{movasati12} 4.2), which corresponds to the case $g=1$ of our construction.

After I completed a first version this article, H. Movasati has kindly indicated to me that a number of our results and constructions has some overlap with his article \cite{movasati13}. In this work, he considers complex analytic spaces $U$ classifying lattices in maximal totally real subspaces of some given complex vector space $V_0$ (i.e., subgroups of $V_0$ generated by a $\CC$-basis of $V_0$) satisfying suitable compatibility conditions with a fixed Hodge filtration $F^{\bullet}_0$ on $V_0$, and a fixed polarization $\psi_0$; these spaces come equipped with a natural analytic right action of the complex algebraic group
$$
G_0 = \{g\in \GL(V_0) \mid gF_0^i=F^i_0 \text{ for every }i\text{, and }g^*\psi_0=\psi_0\}\text{.}
$$
For the particular case where $V_0 = \CC^{2g}$,
$$
F_0^{\bullet} = (F^0_0=V_0\supset F^1_0=\CC^{g}\times\{0\}\supset F^2_0=0)\text{,}
$$
and $\psi_0$ is the standard (complex) symplectic form (\cite{movasati13} 5.1), the space $U$ becomes the analytic moduli space $B_g(\CC)$, investigated in the present article. Of course, the algebraic group $G_0$ coincides with our $P_g$, and the action of $G_0$ on $U$ gets identified with the action of $P_g$ on $B_g(\CC)$ under $U\cong B_g(\CC)$. 

In \cite{movasati13} 3.2, Movasati also describes $U$ as a quotient $\Gamma_{\ZZ}\backslash P$, where $P$ is the space of ``period matrices'' and $\Gamma_{\ZZ}$ is some explicitly defined discrete group. In our particular case, $P$ may be identified with our $\mathbf{B}_g$ (cf. Proposition \ref{prop1}) and $\Gamma_{\ZZ}=\Sp_{2g}(\ZZ)$. Moreover, the map $\mathbf{H}_g \to P$ defined in \cite{movasati13} p. 584 coincides with our $\varphi_g: \mathbf{H}_g \to B_g(\CC)$ constructed via the universal property of $B_g(\CC)$.

In his article, Movasati explicitly states the problem of algebraizing $U$ --- i.e., of finding the algebraic variety $T$ over $\overline{\QQ}$, in his notation --- and the action of $G_0$. This is solved ``by definition'' in our construction, where $T$ is here called $B_{g,\overline{\QQ}}$. Note that our methods also yield that $B_{g,\overline{\QQ}}$ is quasi-affine, which was previously conjectured by Movasati. On his web page\footnote{See ``What is a Siegel quasi-modular form?'' in \url{http://w3.impa.br/~hossein/WikiHossein/WikiHossein.html}.}, Movasati also indicates a construction
of what we call ``higher Ramanujan vector fields'' with slightly different normalizations.

\subsubsection{} \label{subsubsec-intronearlyhol}

The moduli stacks $\mathcal{B}_g$, or variants of it, have also appeared elsewhere in the literature in different contexts, most notably in relation with sheaf theoretic reformulations of Shimura's theory of \emph{nearly holomorphic modular forms}, as in Urban \cite{urban14} and Liu \cite{liu18}.

For instance, in \cite{liu18}, Paragraph 2.1, the parabolic subgroup $\mathbf{Q}$ of $\GSp_{2g}$, and the $\mathbf{Q}$-torsor $T_{\mathcal{H}}^{\times}$, used in the definition of \emph{automorphic sheaves} are ``up to similitude'' versions  of our $P_g$ and $\mathcal{B}_g$. Moreover, the definition of the polynomial $q$-expansions in \cite{liu18}, Paragraph 2.6, involves the construction of $(\omega_{\rm can},\delta_{\rm can})$, which coincides with our $\hat{b}_g$ (see Theorem \ref{intro-thm-instsolsiegel} above). In \cite{liu18}, it is stated that $(\omega_{\rm can},\delta_{\rm can})$ belongs to $T_{\mathcal{H}}^{\times}$, and that this can be checked analytically; this is proved in details in Section \ref{sec-analhre} below.  

The connections between the present work and the theory of nearly holomorphic modular forms should come as no surprise. Indeed, in the case $g=1$, recall that the differential ring of quasimodular forms is isomorphic to the differential ring of nearly holomorphic modular forms endowed with the Maass-Shimura differential operator (cf. \cite{zagier08} 5). Using the results of \cite{urban14}, this can be explained geometrically as follows.

To fix ideas, we ignore the ``condition at infinity'', i.e., we work with ``weakly holomorphic forms'', although \cite{urban14} does consider it; otherwise, see Remark \ref{rem-conditionatinfinity} above. Let $\mathcal{H}$ be the first de Rham cohomology of the universal elliptic curve over $\mathcal{A}_{1,\CC}$, and let $\mathcal{F}$ be its Hodge subbundle. It is shown in \cite{urban14} that the ring of nearly holomorphic modular forms is isomorphic to $H^0(\mathcal{A}_{1,\CC},\Sym \mathcal{H})$, and that the Maass-Shimura operator corresponds to the $\CC$-derivation $\partial$ on this ring induced by Gauss-Manin connection on $\mathcal{H}$ together with the Kodaira-Spencer isomorphism $\Omega^1_{\mathcal{A}_{1,\CC}/\CC} \cong \Sym^2 \mathcal{F}$. On the other hand, $H^0(\mathcal{A}_{1,\CC},\Sym \mathcal{H})$ can be shown to be isomorphic to $H^0(B_{1,\CC},\mathcal{O}_{B_{1,\CC}})$, with $\partial$ being induced by the Ramanujan vector field $v_{11}$ on $B_{1,\CC}$ (see also \cite{movasati12} Sections 6 and 7).

\subsection{Acknowledgments}

This work started as part of my PhD thesis under the supervision of Jean-Benoît Bost, at Université Paris-Sud, Orsay, and was supported by a public grant as part of the FMJH project. It was completed during a postdoctoral stay at the Max-Planck-Institut für Mathematik, Bonn.

I am grateful to Jean-Benoît Bost for introducing me to Nesterenko's theorem and its related open problems, for his encouragement, and for his crucial comments and suggestions on this paper. I thank Hossein Movasati for his kind remarks on the historical development of this subject, and for making me better acquainted with his work. I am greatly indebted to Daniel Bertrand for his interest and for clarifying some aspects related to derivatives of modular functions. It is a pleasure to acknowledge that I have also benefited from remarks of Yves André, Emmanuel Ullmo, and Javier Fresán.

\section*{Terminology and conventions} \label{terminology}

\subsection{} \label{vectorbundles} By a \emph{vector bundle} over a scheme $U$ we mean a locally free sheaf $\mathcal{E}$ over $U$ of finite rank. A \emph{line bundle} is a vector bundle of rank $1$. A \emph{subbundle} of $\mathcal{E}$ is a subsheaf $\mathcal{F}$ of $\mathcal{E}$ such that $\mathcal{F}$ and $\mathcal{E}/\mathcal{F}$ are also vector bundles, that is, $\mathcal{F}$ is locally a direct factor of $\mathcal{E}$. If $\mathcal{E}$ has constant rank $r$, by a \emph{basis} of $\mathcal{E}$ over $U$ we mean an ordered family of $r$ global sections of $\mathcal{E}$ that generate this sheaf as an $\mathcal{O}_U$-module. The \emph{dual} of a vector bundle $\mathcal{E}$ is the vector bundle $\mathcal{E}^{\vee} \defeq \mathcal{H}om_{\mathcal{O}_U}(\mathcal{E},\mathcal{O}_U)$.

\subsection{} Let $U$ be a scheme. By an \emph{abelian scheme} over $U$, we mean a proper and smooth group scheme $p:X \to U$ over $U$ with geometrically connected fibers. The group law of $X$ over $U$ is commutative (cf. \cite{GIT94} Corollary 6.5) and will be denoted additively. A \emph{morphism of abelian schemes} over $U$ is a morphism of $U$-group schemes.

 When $p$ is projective, the relative Picard functor $\Pic_{X/U}$ is representable by a group scheme over $U$ (\cite{BLR90} Chapter 8). Then, the open group subscheme $X^t$ of $\Pic_{X/U}$, whose geometric points correspond to line bundles some power of which are algebraically equivalent to zero, is a projective abelian scheme over $U$, called the \emph{dual abelian scheme}; we denote its structural morphism by $p^t:X^t \to U$. There is a canonical biduality isomorphism $X \stackrel{\sim}{\to} X^{tt}$ (cf. \cite{BLR90} 8.4 Theorem 5). The formation of both the dual abelian scheme and the biduality isomorphism is compatible with every base change in $U$. The universal line bundle over $X\times_U X^t$, the so-called \emph{Poincaré line bundle}, will be denoted by $\mathcal{P}_{X/U}$.

A \emph{principal polarization} on a projective abelian scheme $X$ over $U$ is an isomorphism of $U$-group schemes $\lambda : X \to X^t$ satisfying the equivalent conditions (cf. \cite{GIT94} 6.2 and \cite{DP94} 1.4)
\begin{enumerate}
    \item $\lambda$ is symmetric (i.e. $\lambda=\lambda^t$ under the biduality isomorphism $X\cong X^{tt}$) and  $(\id_X,\lambda)^*\mathcal{P}_{X/U}$ is relatively ample over $U$.
   \item Étale locally over $U$, $\lambda$ is \emph{induced by a line bundle on $X$} (cf. \cite{GIT94} Definition 6.2) relatively ample over $U$. 
\end{enumerate}
A \emph{principally polarized abelian scheme} over $U$ is a couple $(X,\lambda)$, where $X$ is a projective abelian scheme over $U$ and $\lambda$ is a principal polarization on $X$.\label{symb:ppas}

\subsection{} If $X\to S$ is a smooth morphism of schemes, the dual $\mathcal{O}_X$-module of the sheaf of relative differentials $\Omega^1_{X/S}$ (i.e. the sheaf of $\mathcal{O}_S$-derivations of $\mathcal{O}_X$) is denoted by $T_{X/S}$.  It is a vector bundle over $X$ whose rank is given by the relative dimension of $X \to S$. If $S=\Spec R$ is affine, we denote $T_{X/S}=T_{X/R}$.

 The \emph{Lie bracket} $[ \ , \ ] : T_{X/S} \times T_{X/S} \to T_{X/S}$ is defined on derivations by $[\theta_1,\theta_2] = \theta_1\circ \theta_2 - \theta_2\circ \theta_1$. 

If $S$ is a scheme, and $f: X \to Y$ is a morphism of smooth $S$-schemes, then there is a canonical morphism of $\mathcal{O}_X$-modules $f^*\Omega^1_{Y/S} \to \Omega^1_{X/S}$. Further, as $Y\to S$ is smooth, the canonical morphism of $\mathcal{O}_X$-modules  $f^*T_{Y/S} \to (f^*\Omega^1_{Y/S})^{\vee}$ is an isomorphism. We denote by
\begin{align*}
Df: T_{X/S} \to f^*T_{Y/S}
\end{align*}
the dual $\mathcal{O}_X$-morphism of $f^*\Omega^1_{Y/S} \to \Omega^1_{X/S}$ after the identification $(f^*\Omega^1_{Y/S})^{\vee}\cong f^*T_{Y/S}$. If $f$ is smooth, we have an exact sequence of vector bundles over $X$
\begin{align*}
0 \to T_{X/Y} \to T_{X/S} \stackrel{Df}{\to} f^*T_{Y/S} \to 0\text{.}
\end{align*}

\subsection{} If $U$ is any scheme, the category of $U$-schemes (resp. $U$-group schemes) is denoted by $\Sch_{/U}$ (resp. $\textsf{GpSch}_{/U}$). The category of sets is denoted by $\Sets$. If $\textsf{C}$ is any category, its opposite category is denoted by $\textsf{C}^{\opp}$.

\subsection{} \label{stacks}

We shall use the language of \emph{categories fibered in groupoids} and the elements of the theory of \emph{Deligne-Mumford stacks} (\cite{DM69} Paragraph 4). We follow the same conventions and terminology of \cite{olsson16}. In particular, if $S$ is a scheme, whenever we talk about a \emph{stack} over the category of $S$-schemes $\Sch_{/S}$ (cf. \cite{olsson16} Definition 4.6.1), or simply a stack over $S$ (or an $S$-stack), we shall always assume that $\Sch_{/S}$ is endowed with the \emph{étale topology}.

In view of \cite{olsson16} Corollary 8.3.5, by an \emph{algebraic space} over a scheme $S$ we mean a Deligne-Mumford stack $\mathcal{X}$ over $S$ such that for any $S$-scheme $U$ the fiber category $\mathcal{X}(U)$ is discrete (i.e. any automorphism is the identity).

The \emph{étale site} of a Deligne-Mumford stack $\mathcal{X}$ is denoted by $\text{Ét}(\mathcal{X})$ (cf. \cite{olsson16} Paragraph 9.1). We recall that the objects of the underlying category of $\Et(\mathcal{X})$ are \emph{étale schemes over $\mathcal{X}$}, that is, pairs $(U,u)$ where $U$ is an $S$-scheme and $u: U \to \mathcal{X}$ is an étale $S$-morphism; morphisms are given by couples $(f,f^b): (U',u')\to (U,u)$, where $f:U'\to U$ is an $S$-morphism and $f^b :u' \to u\circ f$ is an isomorphism of functors $U' \to \mathcal{X}$. Coverings in $\Et(\mathcal{X})$ are given by families of morphisms $\{(f_i,f_i^{b}):(U_i,u_i) \to (U,u)\}_{i\in I}$ such that $\{f_i:U_i \to U\}_{i\in I}$ is an étale covering of $U$.

The structural sheaf on $\Et(\mathcal{X})$, which to any $(U,u)$ associates the ring $\Gamma(U,\mathcal{O}_U)$, is denoted by $\mathcal{O}_{\mathcal{X}_{\et}}$. We recall that an $\mathcal{O}_{\mathcal{X}_{\et}}$-module $\mathcal{F}$ is said to be \emph{quasi-coherent} if $u^*\mathcal{F}$ is a quasi-coherent $\mathcal{O}_U$-module for any object $(U,u)$ of $\Et(\mathcal{X})$.

By a \emph{vector bundle} over a Deligne-Mumford stack $\mathcal{X}$, we mean a locally free $\mathcal{O}_{\mathcal{X}_{\et}}$-module of finite rank. We define subbundles, bases, and duals as in \ref{vectorbundles}.

\subsection{} \label{tangentstacks}  Sheaves of differentials and tangent sheaves can also be defined for Deligne-Mumford stacks. If $\mathcal{X}$ is a Deligne-Mumford stack over $S$, we define a presheaf of $\mathcal{O}_{\mathcal{X}_{\et}}$-modules $\Omega^1_{\mathcal{X}/S}$ on $\Et(\mathcal{X})$ by 
\begin{align*}
\Gamma((U,u),\Omega^1_{\mathcal{X}/S}) \defeq \Gamma(U,\Omega^1_{U/S})
\end{align*}
for any étale scheme $(U,u)$ over $\mathcal{X}$; restriction maps are defined in the obvious way. Since, for any étale morphism of $S$-schemes $f:U' \to U$, the induced morphism $f^*\Omega^1_{U/S} \to \Omega^1_{U'/S}$ is an isomorphism of $\mathcal{O}_{U'}$-modules, and for any $S$-scheme $U$ the sheaf $\Omega^1_{U/S}$ is a quasi-coherent $\mathcal{O}_U$-module, we see that $\Omega^1_{\mathcal{X}/S}$ is in fact a quasi-coherent sheaf over $\mathcal{X}$ (cf. \cite{olsson16} Lemma 4.3.3). Note that $u^*\Omega^1_{\mathcal{X}/S} = \Omega^1_{U/S}$ for any étale scheme $(U,u)$ over $\mathcal{X}$.

Let $\varphi: \mathcal{X} \to \mathcal{Y}$ be a morphism of Deligne-Mumford stacks over $S$. If $\varphi$ is representable by schemes, then there exists a unique morphism of $\mathcal{O}_{\mathcal{Y}}$-modules $\Omega^1_{\mathcal{Y}/S} \to \varphi_*\Omega^1_{\mathcal{X}/S}$ inducing, for any étale scheme $(V,v)$ over $\mathcal{Y}$, the canonical morphism $\Omega^1_{V/S} \to \varphi'_*\Omega^1_{U/S}$, where $(U,u)$ (resp. $\varphi' :U \to V$) denotes the étale scheme over $\mathcal{X}$ (resp. the morphism of $S$-schemes) obtained from $(V,v)$ (resp. $\varphi$) by base change. If, moreover, $\varphi$ is quasi-compact and quasi-separated, by adjointness (cf. \cite{olsson16} Proposition 9.3.6), we obtain a morphism of $\mathcal{O}_{\mathcal{X}_{\et}}$-modules
\begin{align}\label{pullbackmorphism}
\varphi^*\Omega^1_{\mathcal{Y}/S} \to \Omega^1_{\mathcal{X}/S}\text{.}
\end{align}
We then define a quasi-coherent $\mathcal{O}_{\mathcal{X}_{\et}}$-module
\begin{align*}
\Omega^1_{\mathcal{X}/\mathcal{Y}} \defeq \text{coker} (\varphi^*\Omega^1_{\mathcal{Y}/S} \to \Omega^1_{\mathcal{X}/S})\text{.}
\end{align*}

Recall that a Deligne-Mumford stack $\mathcal{X}$ over $S$ is \emph{smooth} if there exists a surjective étale $S$-morphism $u:U \to \mathcal{X}$ such that $U$ is smooth over $S$ (see \cite{DM69} page 100). In this case,  $\Omega^1_{\mathcal{X}/S}$ is a vector bundle over $\mathcal{X}$. We define $T_{\mathcal{X}/S}$ as the dual $\mathcal{O}_{\mathcal{X}_{\et}}$-module of $\Omega^1_{\mathcal{X}/S}$. If $\varphi: \mathcal{X} \to \mathcal{Y}$ is a morphism of smooth Deligne-Mumford stacks over $S$ representable by smooth schemes, then  $\Omega^1_{\mathcal{X}/\mathcal{Y}}$ is a vector bundle over $\mathcal{X}$, and its dual is denoted by $T_{\mathcal{X}/\mathcal{Y}}$. Moreover, in this case, the morphism in (\ref{pullbackmorphism}) is injective and induces a surjective morphism of $\mathcal{O}_{\mathcal{X}_{\et}}$-modules $D\varphi : T_{\mathcal{X}/S} \to \varphi^*T_{\mathcal{Y}/S}$. We thus obtain an exact sequence of quasi-coherent $\mathcal{O}_{\mathcal{X}_{\et}}$-modules
\begin{align*}
0 \to T_{\mathcal{X}/\mathcal{Y}} \to T_{\mathcal{X}/S} \stackrel{D\varphi}{\to}\varphi^*T_{\mathcal{Y}/S} \to 0\text{.} 
\end{align*}

\subsection{} Let $M$ be a complex manifold. Every holomorphic vector bundle $\pi: V \to M$ may be seen as a (commutative) relative complex Lie group over $M$. We shall occasionally identify $V$ with its corresponding locally free sheaf of $\mathcal{O}_M$-modules of holomorphic sections of $\pi$.

\subsection{} If $R$ is any ring, we denote the constant sheaf with values in $R$ over some complex manifold $M$ by $R_M$. A \emph{local system} of $R$-modules over $M$ is a locally constant sheaf $L$ of $R$-modules over $M$. The \emph{dual} of $L$ is denoted by  $L^{\vee} \defeq \mathcal{H}om_{R}(L,R_M)$.

The \emph{étalé space} of a local system of $R$-modules $L$ over $M$ will be denoted by $E(L)$; this is a topological covering space over $M$ whose fiber at each $p\in M$ is naturally identified to $L_p$. 

\subsection{} \label{notationmatrices}

Let $m,n\ge 1$ be integers. The set of matrices of order $m\times n$ over a ring  $R$ is denoted by $M_{m\times n }(R)$. We shall frequently adopt a block notation for elements in $M_{2n \times 2n}(R)$:
$$
\left(\begin{array}{cc}
A & B \\
C & D
\end{array}\right) = ( A \ B \ ; \ C \ D)\text{,}
$$
where $A,B,C,D \in M_{n\times n}(R)$.

The transpose of a matrix $M \in M_{m\times n}(R)$ is denoted by $M\transp \in M_{n\times m}(R)$. For $1\le i \le n$, $\textbf{e}_i \in M_{n\times 1}(R)$ denotes for the column vector whose entry in the $i$th line is 1, and all the others are 0. The identity matrix in $M_{n\times n}(R)$ is denoted by $\mathbf{1}_n$. For every $1\le i \le j \le n$, we denote by $\textbf{E}^{ij}$ the unique symmetric matrix $(\textbf{E}^{ij}_{kl})_{1\le k,l\le n} \in M_{n\times n}(R)$ such that
\begin{align*}
\textbf{E}^{ij}_{kl} = \begin{cases}
          1 & \text{if }  (k,l)=(i,j) \text{ or } (k,l)=(j,i)\\
          0 & \text{otherwise}\text{.}
          \end{cases}
\end{align*}

\label{symb:symgp}
The \emph{symmetric group} $\Sym_n$ is the subgroup scheme of $M_{n\times n}$ consisting of symmetric matrices. The \emph{symplectic group} $\Sp_{2n}$ is defined as the subgroup scheme of $\GL_{2n}$ such that for every affine scheme $V=\Spec R$ \label{symb:symplectic}
\begin{align*}
{\Sp}_{2g}(V)  =  \{M \in {\GL}_{2n}(R) \mid MJM\transp =J\}
\end{align*}
where
\begin{align*}
J\defeq \left(\begin{array}{cc}
                  0 & \mathbf{1}_n \\
                  -\mathbf{1}_n& 0
                  \end{array}\right)\text{.}
\end{align*}

\begin{obs} \label{eqsympl}
As $J^2=-\mathbf{1}_{2n}$, the condition $MJM\transp = J$ is equivalent to $M^{-1} = -JM\transp J$; thus $MJM\transp = J$ if and only if $M\transp J M = J$. In particular, if we write
\begin{align*}
M = \left(\begin{array}{cc}
                  A & B \\
                  C & D
                  \end{array}\right) \in M_{2n\times 2n}(R)
\end{align*}
for some $A,B,C,D\in M_{n\times n}(R)$, then $M$ is in $\Sp_{2n}(R)$ if and only if one of the following two conditions is satisfied
\begin{enumerate}
   \item $AB\transp = BA\transp$, $CD\transp= DC\transp$, and $AD\transp -BC\transp = \mathbf{1}_n$.
   \item $A\transp C = C\transp A\text{, } B\transp D = D\transp B\text{, and } A\transp D - C\transp B = \mathbf{1}_n$.
\end{enumerate}
\end{obs}
Finally, the \emph{Siegel parabolic subgroup} $P_n$ of $\Sp_{2n}$ consists of matrices $(A \ B \ ; \ C \ D )$ in $\Sp_{2n}$ such that $C=0$. \label{symb:parabolicsiegel}

\subsection{} \label{notationresidue}

 Let $K$ be a subfield of $\CC$ and $X$ be an algebraic variety over $K$ (i.e. a reduced separated scheme of finite type over $K$). For any complex point $\overline{x}: \Spec \CC \to X$, if $x \in X$ denotes the point in the image of $\overline{x}$, and $k(x)$ denotes its residue field, we put
\begin{align*}
K(\overline{x}) \defeq k(x)\text{,}
\end{align*}
and we call it the \emph{field of definition} of $\overline{x}$ in $X$. Let us remark that
\begin{align*}
\trdeg_{K}K(\overline{x}) = \min \{\dim Y \mid Y \text{ is an integral closed $K$-subscheme of }X \text{ such that }\overline{x}\in Y(\CC)\}\text{.}
\end{align*}

\clearpage

\mbox{}

\nomenclature[0]{$\langle \ , \ \rangle_{\lambda}$}{symplectic form on $H^1_{\dR}(X/U)$ induced by a principal polarization $\lambda:X \to X^t$ of an abelian scheme $X$ over $U$, page \pageref{symb:sympl}}

% \nomenclature{$(X,\lambda)_{/U}$}{principally polarized abelian scheme over $U$, page \pageref{symb:ppas}}

% \nomenclature{$(X,\lambda,m)_{/U}$}{principally polarized abelian scheme with $R$-multiplication over $U$, page \pageref{symb:ppasrm}} 

\nomenclature[X]{$(\hat{X}_g,\hat{\lambda}_g)$}{principally polarized abelian scheme of relative dimension $g$ over $\Spec \ZZ(\kern-.2em(q_{ij})\kern-.2em)$ given by Mumford's construction, page \pageref{symb:hatXg}}

\nomenclature[X]{$(\hat{X}_F,\hat{\lambda}_F,\hat{m}_F)$}{principally polarized abelian scheme with $R$-multiplication over $\Spec \ZZ(\kern-.2em(q^{r_i})\kern-.2em)$ given by Mumford's construction, page \pageref{symb:hatXF}}

\nomenclature[X]{$(\mathbf{X}_g,E_g)$}{``universal'' principally polarized complex torus of relative dimension $g$ over $\mathbf{H}_g$, page \pageref{symb:Xg}}

\nomenclature[X]{$(\mathbf{X}_F,E_F,m_F)$}{``universal'' principally polarized complex torus with $R$-multiplication over $\mathbf{H}^g$, page \pageref{symb:XFEFmF}} 

\nomenclature[A]{$\mathcal{A}_F$}{moduli stack over $\Spec \ZZ$ of principally polarized abelian schemes with $R$-multiplication, page \pageref{symb:AgAF}}

\nomenclature[A]{$\mathcal{A}_g$}{moduli stack over $\Spec \ZZ$ of principally polarized abelian schemes of relative dimension $g$, page \pageref{symb:AgAF}}

\nomenclature[A]{$A_F$}{coarse moduli scheme over $\Spec \ZZ$ of principally polarized abelian varieties with $R$-multiplication, page \pageref{symb:coarseAF}}

\nomenclature[A]{$A_g$}{coarse moduli scheme over $\Spec \ZZ$ of principally polarized abelian varieties of dimension $g$, page \pageref{symb:coarseAg}} 

\nomenclature[B]{$b_F$}{universal symplectic-Hodge basis over $\mathcal{B}_F$, page \pageref{symb:universalbF}}

\nomenclature[B]{$b_g$}{universal symplectic-Hodge basis over $\mathcal{B}_g$, page \pageref{symb:universalbg}}

\nomenclature[B]{$\mathcal{B}_F$}{moduli stack over $\Spec \ZZ$ of principally polarized abelian schemes with $R$-multiplication endowed with a symplectic-Hodge basis (see Definition \ref{defi-shbasisrm}), page \pageref{symb:BF}}

\nomenclature[B]{$\mathcal{B}_g$}{moduli stack over $\Spec \ZZ$ of principally polarized abelian schemes of relative dimension $g$ endowed with a symplectic-Hodge basis (see Definition \ref{defi-shb}), page \pageref{symb:Bg}}

\nomenclature[B]{$B_F$}{smooth quasi-affine scheme over $\Spec \ZZ[1/2]$ representing $\mathcal{B}_F\tensor \ZZ[1/2]$, page \pageref{symb:straightBgBF}}

\nomenclature[B]{$B_g$}{smooth quasi-affine scheme over $\Spec \ZZ[1/2]$ representing $\mathcal{B}_g\tensor \ZZ[1/2]$, page \pageref{symb:straightBgBF}}

\nomenclature[C]{$\comp$}{comparison isomorphism between de Rham and Betti cohomology, page \pageref{symb:comp}}

\nomenclature[D]{$D$}{different ideal of a totally real number field $F$ of degree $g$ over $\QQ$, page \pageref{symb:totrealfield}}

\nomenclature[F]{$F$}{totally real number field of degree $g$ over $\QQ$, page \pageref{symb:totrealfield}}

\nomenclature[F]{$F^1(X/U)$}{Hodge subbundle of $H^1_{\dR}(X/U)$ for an abelian scheme $X$ over $U$, page \pageref{symb:hodgesubbundle}}

\nomenclature[F]{$\mathcal{F}_F$}{Hodge subbundle of $\mathcal{H}_F$, page \pageref{symb:FF}}

\nomenclature[F]{$\mathcal{F}_g$}{Hodge subbundle of $\mathcal{H}_g$, page \pageref{symb:F_g}}

\nomenclature[FI]{$\varphi_F$}{analytic solution of the higher Ramanujan equations over $B_F(\CC)$ defined on $\mathbf{H}^g$, page \pageref{symb:phiF}} 

\nomenclature[FI]{$\varphi_g$}{analytic solution of the higher Ramanujan equations over $B_g(\CC)$ defined on $\mathbf{H}_g$, page \pageref{symb:phig}}

\nomenclature[FI]{$\hat{\varphi}_F$}{solution of the higher Ramanujan equations over $\mathcal{B}_F$ defined on $\Spec \ZZ(\kern-.2em(q^{r_i})\kern-.2em)$, page \pageref{symb:hatphiF}}

\nomenclature[FI]{$\hat{\varphi}_g$}{solution of the higher Ramanujan equations over $\mathcal{B}_g$ defined on $\Spec \ZZ(\kern-.2em(q_{ij})\kern-.2em)$, page \pageref{symb:hatphig}} 

\nomenclature[H]{$H^i_{\dR}(X/U)$}{$i$th algebraic de Rham cohomology sheaf of an abelian scheme $X$ over $U$, page \pageref{symb:derhamcohom}}

\nomenclature[H]{$\mathcal{H}^i_{\dR}(X/M)$}{$i$th analytic de Rham cohomology sheaf of a complex torus $X$ over $M$, page \pageref{symb:analyticderham}} 

\nomenclature[H]{$\mathcal{H}_F$}{vector bundle over $\mathcal{A}_F$ given by the first de Rham cohomology of the ``universal abelian scheme'' over $\mathcal{A}_F$, page \pageref{symb:HF}}

\nomenclature[H]{$\mathcal{H}_g$}{vector bundle over $\mathcal{A}_g$ given by the first de Rham cohomology of the ``universal abelian scheme'' over $\mathcal{A}_g$, page \pageref{symb:Hg}}

\nomenclature[H]{$\mathbf{H}_g$}{Siegel upper half-space, page \pageref{symb:siegelspace}} 

\nomenclature[H]{$\mathbf{H}^g$}{$g$th Cartesian power of the Poincaré upper half-plane $\mathbf{H}$, page \pageref{symb:powerH}} 

\nomenclature[J]{$j_F$}{``uniformization map'' from $\mathbf{H}^g$ to $A_F(\CC)$, page \pageref{symb:jF}}

\nomenclature[J]{$j_g$}{``uniformization map'' from $\mathbf{H}_g$ to $A_g(\CC)$, page \pageref{symb:jg}} 

\nomenclature[P]{$P_F$}{parabolic subgroup scheme of $\Res_{R/\ZZ}\Aut_{(M,\Psi)}$ fixing the Lagrangian $R\oplus 0 \subset M$ (see Paragraph \ref{subsec-shbrm}), page \pageref{symb:PF}}

\nomenclature[P]{$P_g$}{parabolic Siegel subgroup of the symplectic group $\Sp_{2g}$, page \pageref{symb:parabolicsiegel}}

\nomenclature[P]{$\mathcal{P}(X/k)$}{field of periods of an abelian variety $X$ over $k\subset \CC$, page \pageref{symb:PXk}} 

\nomenclature[PI]{$\pi_F$}{forgetful functor $\mathcal{B}_F \to \mathcal{A}_F$, page \pageref{symb:piF}}

\nomenclature[PI]{$\pi_g$}{forgetful functor $\mathcal{B}_g \to \mathcal{A}_g$, page \pageref{symb:pig}}

\nomenclature[PS]{$\Psi_{\lambda}$}{$\mathcal{O}_U\tensor R$-bilinear form on $H^1_{\dR}(X/U)$ with values in $\mathcal{O}_U\tensor D^{-1}$ satisfying $\Tr \Psi_{\lambda} = \langle \ , \ \rangle_{\lambda}$, page \pageref{symb:psilambda}}

\nomenclature[R]{$\mathcal{R}_F$}{Ramanujan subbundle of $T_{\mathcal{B}_F/\ZZ}$, page \pageref{symb:RF}}

\nomenclature[R]{$\mathcal{R}_g$}{Ramanujan subbundle of $T_{\mathcal{B}_g/\ZZ}$, page \pageref{symb:Rg}}

\nomenclature[R]{$R_1\pi_*\ZZ_X$}{dual of the local system of abelian groups $R^1\pi_*\ZZ_X$ over $M$, where $\pi:X \to M$ is a complex torus over a complex manifold $M$, page \pageref{symb:R1ZX}}
 
\nomenclature[R]{$R$}{ring of integers of a totally real number field $F$ of degree $g$ over $\QQ$, page \pageref{symb:totrealfield}}

\nomenclature[S]{$\Sp_{2g}$}{symplectic group scheme of order $2g$ over $\Spec \ZZ$, page \pageref{symb:symplectic}}

\nomenclature[S]{$\Sym_g$}{additive group scheme over $\Spec \ZZ$ of symmetric matrices of orger $g$, page \pageref{symb:symgp}}

\nomenclature[T]{$\Tr$}{trace map $\Tr_{F/\QQ}:F \to \QQ$, page \pageref{symb:trace}}

\nomenclature[TH]{$\theta_{ij}$}{either the derivation $q_{ij}\frac{\partial}{\partial q_{ij}}$ of the ring $\ZZ(\kern-.2em(q_{ij})\kern-.2em)$ or the holomorphic vector field $\frac{1}{2\pi i}\frac{\partial}{\partial \tau_{ij}}$ over $\mathbf{H}_g$, pages \pageref{symb:thetaijalg} and \pageref{symb:thetaijanal}} 

\nomenclature[TH]{$\theta^{r_i}$}{either the derivation $q^{r_i}\frac{\partial}{\partial q^{r_i}}$ of the ring $\ZZ(\kern-.2em(q^{r_i})\kern-.2em)$ or the holomorphic vector field $\frac{1}{2\pi i}\sum_{j=1}^g\sigma_j(x_i)\frac{\partial}{\partial \tau_j}$ over $\mathbf{H}^g$, pages \pageref{symb:thetarialg} and \pageref{symb:thetarianal}}

\nomenclature[V]{$v_F$}{higher Ramanujan vector field over $\mathcal{B}_F$, page \pageref{symb:vF}}

\nomenclature[V]{$v^{r_i}$}{vector field $v_F(1\tensor x_i)$ over $\mathcal{B}_F$, page \pageref{symb:vri}} 

\nomenclature[V]{$v_{ij}$}{higher Ramanujan vector field over $\mathcal{B}_g$, page \pageref{symb:vg}} 

\nomenclature[Z]{$\ZZ(\kern-.2em(q_{ij})\kern-.2em)$}{ring of formal Laurent power series over $\ZZ$ in the variables $q_{ij}$, for $1\le i\le j \le g$, page \pageref{symb:Zqij}}

\nomenclature[Z]{$\ZZ(\kern-.2em(q^{r_i})\kern-.2em)$}{ring of formal Laurent power series over $\ZZ$ in the variables $q^{r_1},\ldots, q^{r_g}$, page \pageref{symb:Zqri}} 

\printnomenclature[0.95in] 

\newpage

\part{The arithmetic theory of the higher Ramanujan equations}

\section{Symplectic vector bundles over schemes} \label{sec-sympl}

In this section we develop (or recall) some preliminary general material on vector bundles over schemes endowed with a symplectic bilinear form with values in some line bundle.

We fix once and for all a scheme $U$, and a line bundle $\mathcal{L}$ over $U$.

\subsection{Symplectic vector bundles}

 Let $\mathcal{E}$ be a vector bundle over $U$. An $\mathcal{O}_U$-bilinear form with values in $\mathcal{L}$
\begin{align*}
\langle \ , \ \rangle : \mathcal{E} \tensor_{\mathcal{O}_{U}} \mathcal{E} \longrightarrow \mathcal{L}
\end{align*}
is said to be
\begin{enumerate}
    \item \emph{perfect} if the $\mathcal{O}_U$-morphism $e \mapsto \langle \ \ , e \rangle$ from $\mathcal{E}$ to $\mathcal{L}\tensor_{\mathcal{O}_U}\mathcal{E}^{\vee}$ is an isomorphism,
    \item \emph{alternating} if $\langle \ , \ \rangle$ factors through $\mathcal{E}\tensor_{\mathcal{O}_U}\mathcal{E} \to \bigwedge^2 \mathcal{E}$, i.e., if $\langle e, e \rangle =0$ for every section $e$ of $\mathcal{E}$.
\end{enumerate} 

\begin{defi} \label{defisymplbundle}
An \emph{$\mathcal{L}$-valued symplectic form} over $\mathcal{E}$ is a perfect alternating $\mathcal{O}_U$-bilinear form over $\mathcal{E}$ with values in $\mathcal{L}$. An \emph{$\mathcal{L}$-symplectic vector bundle} over $U$ is a couple $(\mathcal{E},\langle \ , \ \rangle)$, where $\mathcal{E}$ is a vector bundle over $U$ and $\langle \ , \ \rangle$ is an $\mathcal{L}$-valued symplectic form over $\mathcal{E}$.
\end{defi}

When $\mathcal{L}=\mathcal{O}_U$, we write simply \emph{symplectic form} and \emph{symplectic vector bundle}.

By considering $\mathcal{O}_U$-linear morphisms preserving the $\mathcal{L}$-valued symplectic forms, we obtain a category of $\mathcal{L}$-symplectic vector bundles over $U$.

% If $U=\Spec R$ is affine, we prefer to consider the module of global sections $E=\Gamma(U,\mathcal{E})$ in place of $\mathcal{E}$, and we shall also say that $E$ is a vector bundle over $R$.

\subsection{Lagrangian subbundles}

 Let $(\mathcal{E}, \langle \ , \ \rangle)$ be an $\mathcal{L}$-valued symplectic vector bundle over $U$ and $\mathcal{F}$ be a subbundle of $\mathcal{E}$. We denote by $\mathcal{F}^{\perp}$ the subsheaf of $\mathcal{E}$ consisting of those sections $e$ of $\mathcal{E}$ such that $\langle f , e \rangle = 0$ for every section $f$ of $\mathcal{F}$.

\begin{lemma} \label{exactseq}
We have an exact sequence of $\mathcal{O}_U$-modules
\begin{align*}
0 \longrightarrow  \mathcal{F}^{\perp} \longrightarrow \mathcal{E} &\longrightarrow \mathcal{L}\tensor_{\mathcal{O}_U} \mathcal{F}^{\vee} \longrightarrow 0\\
e &\mapsto\langle \ \ , e \rangle|_{\mathcal{F}}
\end{align*}
In particular, $\mathcal{F}^{\perp}$ is a subbundle of $\mathcal{E}$ of rank ${\rm rank}( \mathcal{E}) - {\rm rank}(\mathcal{F})$.
\end{lemma}

\begin{proof}
The sequence $0 \longrightarrow  \mathcal{F}^{\perp} \longrightarrow \mathcal{E} \longrightarrow \mathcal{L}\tensor_{\mathcal{O}_U} \mathcal{F}^{\vee}$ is exact by definition. To see that $\mathcal{E} \to \mathcal{L}\tensor_{\mathcal{O}_U} \mathcal{F}^{\vee}$ defined above is surjective, one may work locally. In this case, $\mathcal{F}$ is a direct factor of $\mathcal{E}$, and thus any $\mathcal{O}_U$-linear map $\mathcal{F}\to \mathcal{L}$ can be extended to $\mathcal{E}$; we conclude by using that $\langle \ , \ \rangle$ is perfect.
\end{proof}

\begin{defi} \label{isolagr}
A subbundle $\mathcal{F}$ of $\mathcal{E}$ is said to be \emph{isotropic} with respect to $\langle \ , \ \rangle$ if $\mathcal{F} \subset \mathcal{F}^{\perp}$. An isotropic subbundle of $\mathcal{E}$ such that $\mathcal{F}=\mathcal{F}^{\perp}$ is said to be a \emph{Lagrangian subbundle}.
\end{defi}

The next result easily follows from Lemma \ref{exactseq}.

\begin{coro} \label{corosympl}
Let $\mathcal{F}$ be an isotropic subbundle of $\mathcal{E}$. Then $2 \rank( \mathcal{F}) \le \rank( \mathcal{E})$. Moreover, $\mathcal{F}$ is Lagrangian if and only if $2 \rank(\mathcal{F})=\rank(\mathcal{E})$. \hfill $\blacksquare$
\end{coro}

The next lemma shows that Lagrangian subbundles exist locally for the Zariski topology over $U$. This implies in particular that the rank of every symplectic vector bundle is even.

\begin{lemma} \label{locallagrangian}
Let $(\mathcal{E},\langle \ , \ \rangle)$ be an $\mathcal{L}$-valued symplectic vector bundle over $U$, and assume that $U$ is the spectrum of a local ring. Then there exists a Lagrangian subbundle of $\mathcal{E}$.
\end{lemma}

\begin{proof}
Let $S$ be the set of isotropic subbundles of $\mathcal{E}$ ordered by inclusion. It is sufficient to prove that every maximal element in $S$ is Lagrangian (maximal elements always exist: consider the rank).

We proceed by contraposition. Let $\mathcal{F}$ be an element of $S$ that is not Lagrangian. As $U$ is local, and both $\mathcal{F}$ and $\mathcal{F}^{\perp}$ are subbundles $\mathcal{E}$ (cf. Lemma \ref{exactseq}), there exists an integer $k\ge 1$ and global sections $e_1,\ldots,e_k$ of $\mathcal{F}^{\perp}$ such that
\begin{align*}
\mathcal{F}^{\perp} = \mathcal{F} \oplus \mathcal{O}_Ue_1 \oplus \cdots \oplus \mathcal{O}_Ue_k\text{.}
\end{align*}
 In particular, $\mathcal{F} \oplus \mathcal{O}_Ue_1$ is an element of $S$ strictly containing $\mathcal{F}$; thus, $\mathcal{F}$ is not maximal.
\end{proof}

\begin{obs}
  The same statement (and the same proof) holds for every scheme $U$ over which any vector bundle is trivializable, e.g., $U$ the spectrum of a principal ideal domain or of a polynomial ring over a field.
\end{obs}

\subsection{Symplectic bases} 

In what follows, we take $\mathcal{L}=\mathcal{O}_U$. Let $(\mathcal{E}, \langle \ , \ \rangle)$ be a symplectic vector bundle of constant rank $2n$ over $U$.

\begin{defi} \label{defsymplbasis}
A \emph{symplectic basis} of $(\mathcal{E},\langle \ , \ \rangle)$ over $U$ is a basis of $\mathcal{E}$ over $U$ of the form $(e_1,\ldots,e_n,f_1,\ldots,f_n)$ with $\langle e_i, e_j \rangle = \langle f_i, f_j \rangle =0$ and $\langle e_i,f_j\rangle = \delta_{ij}$ for all $1\le i,j\le n$.
\end{defi}

\begin{obs}
 Equivalently, if the trivial vector bundle $\mathcal{O}^{2n}_U$ is given the \emph{standard} symplectic form
$$
\langle v, w \rangle_{\std} \defeq v\transp \left(\begin{array}{cc}
                                                    0 & \mathbf{1}_g \\
                                                    \mathbf{1}_g & 0
                                                   \end{array}
 \right) w\text{,}
$$
then a symplectic basis of $(\mathcal{E},\langle \ , \ \rangle)$ can be regarded as an isomorphism of symplectic vector bundles $(\mathcal{O}^{2n}_U,\langle \ , \ \rangle_{\std}) \stackrel{\sim}{\to} (\mathcal{E},\langle \ , \ \rangle)$. This point of view turns out to be useful when dealing with symplectic vector bundles with real multiplication; see Section \ref{sec-abrm} below.

\end{obs}

 As Lagrangian subbundles exist locally by Lemma \ref{locallagrangian}, the next proposition implies in particular that symplectic bases also exist locally.

\begin{prop} \label{exisunic}
Let $U$ be an affine scheme, $(\mathcal{E},\langle \ , \ \rangle)$ be a symplectic vector bundle over $U$, and $\mathcal{E}_0$ be a Lagrangian subbundle of $\mathcal{E}$. Then
\begin{enumerate}
    \item Every basis $(e_1,\ldots,e_n)$ of $\mathcal{E}_0$ over $U$ can be completed to a symplectic basis $(e_1,\ldots,e_n,f_1,\ldots,f_n)$ of $\mathcal{E}$ over $U$.
    \item If $\mathcal{F}$ is a Lagrangian subbundle of $\mathcal{E}$ such that $\mathcal{E}_0\oplus \mathcal{F}=\mathcal{E}$, and $(f_1,\ldots,f_n)$ is a basis of $\mathcal{F}$ over $U$, then there exists a unique basis $(e_1,\ldots,e_n)$ of $\mathcal{E}_0$ over $U$ such that $(e_1,\ldots,e_n,f_1,\ldots,f_n)$ is a symplectic basis of $\mathcal{E}$ over $U$.
\end{enumerate}
\end{prop}

\begin{proof}
Consider the surjective morphism of $\mathcal{O}_U$-modules (cf. Lemma \ref{exactseq})
\begin{align*}
\mathcal{E} &\to \mathcal{E}_0^{\vee}\\
         e &\mapsto \langle \ \ , e\rangle|_{\mathcal{E}_0}\text{.}
\end{align*}
Since $U$ is affine, there exists a sequence $(f_1',\ldots,f_n')$ of global sections of $\mathcal{E}$ lifting the dual basis of $(e_1,\ldots,e_n)$ in $\mathcal{E}_0^{\vee}$, so that $\langle e_i,f_j'\rangle = \delta_{ij}$ for every $1\le i,j\le n$. As $\mathcal{E}_0$ is an isotropic subbundle of $\mathcal{E}$, to prove (1) it is sufficient to show the existence of global sections $g_j$ of $\mathcal{E}_0$ such that 
\begin{align*}
f_j \defeq f_j' + g_j
\end{align*}
satisfy $\langle f_i,f_j\rangle = 0$ for every $1\le i,j\le n$.

Since the bilinear form $\langle \ , \ \rangle$ is alternating, $A \defeq (\langle f_i',f_j'\rangle)_{1\le i,j\le n}$ is an antisymmetric matrix in $M_{n\times n}(\mathcal{O}_U(U))$. Thus, there exists a matrix $B = (b_{ij})_{1\le i,j \le g}$ in $M_{n\times n}(\mathcal{O}_U(U))$ such that $A = B - B\transp$. We put
\begin{align*}
g_i \defeq \sum_{j=1}^n b_{ij}e_j\text{,}
\end{align*}
hence
\begin{align*}
\langle f_i,f_j \rangle  = \langle f_i',f_j'\rangle + \langle g_i,f_j' \rangle - \langle g_j,f_i' \rangle = \langle f_i',f_j'\rangle + b_{ij} - b_{ji} = 0\text{.} 
\end{align*}

We now proceed to the proof of (2). As $\mathcal{F}$ is an isotropic subbundle of $\mathcal{E}$ satisfying $\mathcal{E}_0 \oplus \mathcal{F} = \mathcal{E}$, and since $\langle \ , \ \rangle$ is perfect, the morphism of $\mathcal{O}_U$-modules
\begin{align*}
\mathcal{F} &\to \mathcal{E}_0^{\vee}\\
        f &\mapsto \langle \ \ , f\rangle|_{\mathcal{E}_0}  
\end{align*}
is injective, thus an isomorphism since $\mathcal{F}$ and $\mathcal{E}_0^{\vee}$ have equal rank. The existence and unicity of $(e_1,\ldots,e_n)$ follows from remarking that $(e_1,\ldots,e_n, f_1,\ldots,f_n)$ is a symplectic basis of $\mathcal{E}$ over $U$ if and only if $(e_1,\ldots,e_n)$ is the basis of $\mathcal{E}_0$ over $U$ dual to the basis $(\langle \ \ , f_1\rangle|_{\mathcal{E}_0}, \ldots, \langle \ \ , f_n\rangle|_{\mathcal{E}_0})$ of $\mathcal{E}_0^{\vee}$.
\end{proof}

\section{Symplectic-Hodge bases of principally polarized abelian schemes} \label{shb}

We start this section by recalling the definition of the de Rham cohomology of an abelian scheme and its main properties. We next recall how to associate to a principal polarization on an abelian scheme a symplectic form, as defined in Section \ref{sec-sympl}, on its first de Rham cohomology. This leads us to the definition of \emph{symplectic-Hodge bases}. 

\subsection{De Rham cohomology of abelian schemes}

Let $p:X \to U$ be an abelian scheme of relative dimension $g$.

 Recall that, for any integer $i\ge0$, the \emph{$i$-th de Rham cohomology} sheaf of $\mathcal{O}_U$-modules associated to $p$ is defined as the $i$-th left hyperderived functor of $p_*$ applied to the complex of relative differential forms $\Omega^{\bullet}_{X/U}$: \label{symb:derhamcohom}
\begin{align*}
H^i_{\dR}(X/U) \defeq \mathbf{R}^ip_*\Omega^{\bullet}_{X/U}\text{.}
\end{align*}
If $\varphi:X \to Y$ is a morphism of abelian schemes over $U$, we denote by $\varphi^* : H^i_{\dR}(Y/U) \to H^i_{\dR}(X/U)$ the induced $\mathcal{O}_U$-morphism on de Rham cohomology.

One can prove that there is a canonical isomorphism given by cup product
\begin{align*}
{\bigwedge}^i H^1_{\dR}(X/U) \stackrel{\sim}{\longrightarrow} H^i_{\dR}(X/U)\text{,} 
\end{align*}
and that $H^1_{\dR}(X/U)$ is a vector bundle over $U$ of rank $2g$. Moreover, the canonical $\mathcal{O}_U$-morphism $p_*\Omega^1_{X/U} \to H^1_{\dR}(X/U)$ induces an isomorphism of $p_*\Omega^1_{X/U}$ with a rank $g$ subbundle of $H^1_{\dR}(X/U)$, its \emph{Hodge subbundle} $F^1(X/U)$. It fits into a canonical exact sequence of $\mathcal{O}_U$-modules:\label{symb:hodgesubbundle}
\begin{align} \label{hodgefiltr}
%\tag{H}
0 \longrightarrow F^1(X/U) \longrightarrow H^1_{\dR}(X/U) \longrightarrow R^1p_*\mathcal{O}_X \longrightarrow 0\text{.}
\end{align}
The formation of $H^1_{\dR}(X/U)$, $F^1(X/U)$, $R^1p_*\mathcal{O}_X$, and the above exact sequence is compatible with every base change in $U$.

 For a proof of all these facts, the reader may consult  \cite{BBM82} 2.5.

\subsection{Symplectic form associated to a principal polarization} \label{shb2}

 Let $p:X \to U$ be a projective abelian scheme of relative dimension $g$, and $\lambda :X \to X^t$ be a principal polarization.  In this paragraph, we recall how to associate to $\lambda$ a canonical \emph{symplectic} $\mathcal{O}_U$-bilinear form \label{symb:sympl} 
\begin{align*}
\langle \ , \ \rangle_{\lambda} : H^1_{\dR}(X/U) \tensor_{\mathcal{O}_U} H^1_{\dR}(X/U) \longrightarrow \mathcal{O}_U\text{.}
\end{align*}

Recall that to any line bundle $\mathcal{L}$ on $X$ we can associate its \emph{first Chern class in de Rham cohomology} $c_{1,\dR}(\mathcal{L})$, namely the global section of $H^2_{\dR}(X/U)$ given by the image of the class of the line bundle $\mathcal{L}$ under the morphism of $\mathcal{O}_U$-modules
\begin{align*}
R^1p_*\mathcal{O}_X^{\times} \to \mathbf{R}^1p_*\Omega^{\bullet}_{X/U}[1]\cong H^2_{\dR}(X/U)
\end{align*}
induced by $\dlog : \mathcal{O}_{X}^{\times} \to \Omega^{\bullet}_{X/U}[1]$.\footnote{We adopt the same sign conventions of \cite{BBM82} 0.3 for the differentials of the shifted complex $\Omega^{\bullet}_{X/U}[1]$ and for the isomorphism $\mathbf{R}^1p_*\Omega^{\bullet}_{X/U}[1]\cong H^2_{\dR}(X/U)$.}

We apply the above construction to the Poincaré line bundle $\mathcal{P}_{X/U}$ on the projective abelian scheme $X\times_U X^t$ over $U$. Let 
\begin{align*}
\phi_{X/U} : H^1_{\dR}(X/U)^{\vee} \longrightarrow H^1_{\dR}(X^t/U)
\end{align*}
be the morphism of  $\mathcal{O}_U$-modules given by the image of $c_{1,\dR}(\mathcal{P}_{X/U})$ in the Künneth component  $H^1_{\dR}(X/U) \tensor_{\mathcal{O}_U}H^1_{\dR}(X^t/U)$ of $H^2_{\dR}(X/U)$. By \cite{BBM82} 5.1.3.1, $\phi_{X/U}$ is in fact an isomorphism.

\begin{obs}[cf. \cite{BBM82} (5.1.3.3)]
  The isomorphisms $\phi_{X/U}$ are natural in the following sense. If $\varphi:X\to Y$ is a morphism of projective abelian schemes over $U$, then the diagram of $\mathcal{O}_U$-modules
  $$
  \begin{tikzcd}
    H^1_{\dR}(X/U)^{\vee} \arrow{d}[swap]{(\varphi^*)^{\vee}}\arrow{r}{\phi_{X/U}} & H^1_{\dR}(X^t/U)\arrow{d}{(\varphi^t)^*}\\
    H^1_{\dR}(Y/U)^{\vee} \arrow{r}[swap]{\phi_{Y/U}} & H^1_{\dR}(Y^t/U)
  \end{tikzcd}
  $$
  commutes.
\end{obs}

Consider the  isomorphism of $\mathcal{O}_U$-modules 
\begin{align*}
\lambda^*: H^1_{\dR}(X^t/U) \to H^1_{\dR}(X/U)
\end{align*}
induced by the principal polarization $\lambda : X \to X^t$. For any sections $\gamma$ and $\delta$ of $H^1_{\dR}(X/U)^{\vee}$, we set
\begin{align*}
E_{\lambda}^{\dR}(\gamma,\delta) \defeq \delta\circ\lambda^*\circ \phi_{X/U}(\gamma)\text{.} 
\end{align*}
%alors E_{\lambda}^{\dR}(\gamma, \ ) = \lambda^*\circ \phi_{X/U}(\gamma)

It is clear that $E_{\lambda}^{\dR}$ defines an $\mathcal{O}_U$-bilinear form over $H^1_{\dR}(X/U)^{\vee}$. Since $\phi_{X/U}$ is an isomorphism, $E_{\lambda}^{\dR}$ is perfect. By duality, we can thus define a perfect bilinear form $\langle \ , \ \rangle_{\lambda}$ over $H^1_{\dR}(X/U)$ via
\begin{align*}
\langle E_{\lambda}^{\dR}( \gamma , \ ), E_{\lambda}^{\dR}( \delta , \ )\rangle_{\lambda} \defeq E_{\lambda}^{\dR}(\gamma,\delta)\text{,}
\end{align*}
where we identified $H^1_{\dR}(X/U)^{\vee \vee}$ with $H^1_{\dR}(X/U)$. 
%alors \langle \ , \beta \rangle_{\lambda} = (\lambda^*\circ \phi_{X/U})^{-1}\beta

\begin{lemma} \label{alternating}
The perfect bilinear form $\langle \ , \ \rangle_{\lambda}$ is alternating, thus symplectic.
\end{lemma}

\begin{proof}
It suffices to prove that $E_{\lambda}^{\dR}$ is alternating. Since $\lambda$ is a polarization, it is étale locally over $U$ induced by a line bundle $\mathcal{L}$ over $X$ relatively ample over $U$. We consider the first Chern class $c_{1,\dR}(\mathcal{L})$ in $H_{\dR}^2(X/U) \cong \bigwedge^2 H_{\dR}^1(X/U)$. Then, one can verify that $E_{\lambda}^{\dR}$ defined above coincides with the alternating form 
\begin{align*}
(\gamma, \delta) \mapsto \gamma \wedge \delta ( c_{1,\dR}(\mathcal{L}))\text{.}
\end{align*}
We refer to \cite{DP94}, Section 1, for further details. 
\end{proof}

Thus we obtain a symplectic vector bundle $(H^1_{\dR}(X/U),\langle \ , \ \rangle_{\lambda})$ over $U$ in the sense of Definition \ref{defisymplbundle}.

\begin{lemma} \label{f1lagrangian}
$F^1(X/U)$ is a Lagrangian subbundle of $H^1_{\dR}(X/U)$ with respect to the symplectic form $\langle \ , \ \rangle_{\lambda}$.
\end{lemma}

\begin{proof}
Since the rank of $H^1_{\dR}(X/U)$ is $2g$, and $F^1(X/U)$ is a rank $g$ subbundle of $H^1_{\dR}(X/U)$, it suffices to prove that $F^1(X/U)$ is isotropic with respect to $\langle \ , \ \rangle_{\lambda}$ (cf. Corollary \ref{corosympl}). This follows immediately from the compatibility of $\phi_{X/U}$ with the exact sequence (\ref{hodgefiltr}), that is, from the existence of canonical morphisms $\phi_{X/U}^0$ and $\phi_{X/U}^1$ making the diagram
%$$
%  \raisebox{-0.5\height}{\includegraphics{rameq1-d1.pdf}}
%$$
 $$
 \begin{tikzcd}
 0 \arrow{r} & (R^1p_*\mathcal{O}_X)^{\vee} \arrow{r} \arrow{d}{\phi_{X/U}^0} & H^1_{\dR}(X/U)^{\vee} \arrow{r} \arrow{d}{\phi_{X/U}} & F^1(X/U)^{\vee} \arrow{r} \arrow{d}{\phi_{X/U}^1} & 0 \\
 0 \arrow{r} & F^1(X^t/U) \arrow{r} & H^1_{\dR}(X^t/U) \arrow{r} & R^1p^t_*\mathcal{O}_{X^t} \arrow{r} & 0   
 \end{tikzcd}
 $$
commute (\cite{BBM82} Lemme 5.1.4; the morphisms $\phi_{X/U}^0$ and $\phi_{X/U}^1$ are uniquely determined by this commutative diagram, and are isomorphisms).
\end{proof}

\begin{obs} \label{remarkbasechange}
It is clear from the above construction that the formation of the symplectic form $\langle \ , \ \rangle_{\lambda}$ is compatible with base change. Namely, if $f:U' \to U$ is a morphism of schemes, and $(X',\lambda')$ denotes the principally polarized abelian scheme over $U'$ obtained by base change via $f$, then $f^*\langle \ , \ \rangle_{\lambda}$ coincides with $\langle \ , \ \rangle_{\lambda'}$ under the base change isomorphism $f^*H^1_{\dR}(X/U) \stackrel{\sim}{\to} H^1_{\dR}(X'/U')$. 
\end{obs}

\subsection{Symplectic-Hodge bases of $H^1_{\dR}(X/U)$} 

Let $U$ be a scheme and $(X,\lambda)$ be a principally polarized abelian scheme over $U$ of relative dimension $g$.

\begin{defi} \label{defi-shb}
 A \emph{symplectic-Hodge basis} of $(X,\lambda)_{/U}$ is a $2g$-uple $b=(\omega_1,\ldots,\omega_g,\eta_1,\ldots,\eta_g)$ global sections of $H^1_{\dR}(X/U)$ such that:
\begin{enumerate}
   \item $\omega_1,\ldots,\omega_g$ are sections of $F^1(X/U)$, and
   \item $b$ is a symplectic basis of $(H^1_{\dR}(X/U),\langle \ , \ \rangle_{\lambda})$ (Definition \ref{defsymplbasis}).
\end{enumerate} 
\end{defi}

Note that symplectic-Hodge bases may not exist globally, but such bases always exist locally for the Zariski topology over $U$ by Proposition \ref{exisunic}.

\section{Abelian schemes with real multiplication}\label{sec-abrm}

In this section, we introduce notation and analogs of the above basic notions for principally polarized abelian schemes with \emph{real multiplication}.

From now on, we fix a totally real number field $F$ of degree $g$ over $\QQ$, and we denote its ring of integers by $R$. Recall that the inverse different ideal $D^{-1}\subset F$ is a fractional ideal of $F$ which can be identified with the $\ZZ$-dual of $R$ via the trace form.\label{symb:totrealfield}

Tensor products without subscripts are taken over $\ZZ$.

\subsection{Symplectic vector bundles with real multiplication}\label{ssec-svbrm}

Let $U$ be a scheme and  $\mathcal{M}$ be a quasi-coherent $\mathcal{O}_U$-module. An \emph{$R$-multiplication} on $\mathcal{M}$ is a ring morphism $R \to \End_{\mathcal{O}_U}(\mathcal{M})$; giving such a ring morphism amounts to giving $\mathcal{M}$ the structure of an $\mathcal{O}_U\tensor R$-module compatible with its structure of $\mathcal{O}_U$-module via $\mathcal{O}_U\to \mathcal{O}_U\tensor R$.

\begin{obs}\label{rem-equiv-r-mult}
 Consider the natural projection $f: U_R \defeq U\tensor_{\ZZ}R \to U$. Observe that ${f}_*\mathcal{O}_{U_R} = \mathcal{O}_U\tensor R$. Since $f$ is finite, thus affine, the functor
 $$
 \mathcal{F}\mapsto {f}_*\mathcal{F}
 $$
 induces an equivalence between the category of quasi-coherent $\mathcal{O}_{U_R}$-modules and the category of quasi-coherent $\mathcal{O}_U$-modules with $R$-multiplication (i.e., quasi-coherent $\mathcal{O}_U\tensor R$-modules; cf. \cite{EGAII} Proposition 1.4.3). 
\end{obs}

Following \cite{rapoport78} and \cite{DP94}, we denote the $\mathcal{O}_U\tensor R$-dual of a quasi-coherent $\mathcal{O}_U$-module with $R$-multiplication $\mathcal{M}$ by
$$
\mathcal{M}^* \defeq \mathcal{H}om_{\mathcal{O}_U\tensor R}(\mathcal{M},\mathcal{O}_U\tensor R)\text{.}
$$
The trace map $\Tr \defeq \Tr_{F/\QQ}: F \to \QQ$ induces an isomorphism of quasi-coherent $\mathcal{O}_U$-modules with $R$-multiplication\label{symb:trace}
\begin{align}\label{eq-duality}
\Tr: \mathcal{M}^{*}\tensor_RD^{-1} \stackrel{\sim}{\to} \mathcal{M}^{\vee}\text{.}
\end{align}

\begin{obs}\label{rem-dualityrelation}
  The above duality relation comes from the following general fact (cf. \cite{DP94} 2.11).  Let $A$ be a commutative ring, $M$ be an $A\tensor R$-module, and $N$ be an $A$-module. Then the trace map induces an $A\tensor R$-isomorphism 
  $$
\Tr : {\Hom}_{A\tensor R}(M,N\tensor D^{-1}) \stackrel{\sim}{\to} {\Hom}_A(M,N)\text{.}
$$
\end{obs}

\begin{obs}
 In the light of Remark \ref{rem-equiv-r-mult}, we may interpret (\ref{eq-duality}) as a version of the Serre-Grothendieck duality for the \emph{finite} morphism $f$. For a quasi-coherent $\mathcal{O}_U$-module $\mathcal{G}$, we define a quasi-coherent $\mathcal{O}_{U_R}$-module $f^!\mathcal{G}$ by $f_*f^!\mathcal{G} = \mathcal{H}om_{\mathcal{O}_U}(\mathcal{O}_{U}\tensor R,\mathcal{G}) = \mathcal{G}\tensor \Hom_{\ZZ}(R,\ZZ)$. We then have natural isomorphisms $f_*\mathcal{H}om_{\mathcal{O}_{U_R}}(\mathcal{F},f^!\mathcal{G}) \stackrel{\sim}{\to} \mathcal{H}om_{\mathcal{O}_U}(f_*\mathcal{F},\mathcal{G})$ for any \emph{quasi}-coherent sheaves $\mathcal{F}$ on $U_R$ and $\mathcal{G}$ on $U$.
\end{obs}

By a \emph{vector bundle with $R$-multiplication} over $U$, we mean a quasi-coherent sheaf with $R$-multiplication $\mathcal{E}$ over $U$ which is, locally over $U$, a free $\mathcal{O}_U\tensor R$-module of finite rank. Equivalently, under the notation of Remark \ref{rem-equiv-r-mult}, $\mathcal{E}$ is given by the direct image of a vector bundle over $U_R$. Clearly, $\mathcal{E}$ is also a vector bundle over $U$ and we have
$$
{\rank}_{\mathcal{O}_U}\mathcal{E} = g\cdot {\rank}_{\mathcal{O}_U\tensor R}\mathcal{E}\text{.}
$$
By the \emph{rank} of a symplectic vector bundle with $R$-multiplication, we mean its rank as a locally free $\mathcal{O}_U\tensor R$-module.

We say that an $\mathcal{O}_U$-bilinear form $\langle \ , \ \rangle$ on the vector bundle with $R$-multiplication $\mathcal{E}$ over $U$ is \emph{compatible with the $R$-multiplication} if it factors through
$$
\mathcal{E}\tensor_{\mathcal{O}_U\tensor R}\mathcal{E} \to \mathcal{O}_U\text{.}
$$
In this case, it follows from (\ref{eq-duality}) that there exists a unique $\mathcal{O}_U\tensor R$-bilinear form
$$
\Psi : \mathcal{E}\tensor_{\mathcal{O}_U\tensor R}\mathcal{E} \to \mathcal{O}_U\tensor D^{-1}
$$
such that
$$
\langle \ ,\ \rangle = \Tr \Psi\text{.}
$$
If, moreover, $\langle \ , \ \rangle$ is symplectic, then $\Psi$ is perfect and alternating --- that is, if $\mathcal{E}=f_*\mathcal{F}$ under the notation of Remark \ref{rem-equiv-r-mult}, then $\Psi$ is given by the direct image of a $\mathcal{O}_{U_R}\tensor_RD^{-1}$-valued symplectic form on $\mathcal{F}$. The couple $(\mathcal{E}, \Psi)$ is then said to be a \emph{symplectic vector bundle with $R$-multiplication} over $U$.

\subsection{Principally polarized abelian schemes with real multiplication}

Let $(X,\lambda)$ be a principally polarized abelian scheme over some scheme $U$. Then $\lambda$ defines a Rosatti involution $\varphi \mapsto \lambda^{-1}\circ \varphi^{t}\circ \lambda$ on the ring of abelian scheme endomorphisms $\End_U(X)$. We denote by $\End_U(X)^{\lambda}$ the subset of $\End_U(X)$ of elements fixed by the Rosatti involution.\label{symb:ppasrm}

\begin{defi}\label{defi-ppasrm}
  A \emph{principally polarized abelian scheme with $R$-multiplication} over $U$ is a triple $(X,\lambda,m)$, where $(X,\lambda)$ is a principally polarized abelian scheme over $U$, and $m: R \to \End_U(X)$ is a ring morphism such that:
  \begin{enumerate}
  \item $m(R) \subset \End_U(X)^{\lambda}$, and
    \item $m$ gives $F^1(X/U)$ the structure of a vector bundle with $R$-multiplication of rank 1 over $U$.
    \end{enumerate}
    A morphism of principally polarized abelian schemes with $R$-multiplication is a morphism of principally polarized abelian schemes commuting with the $R$-multiplications.
\end{defi}

The condition (2) above, which implies in particular that $X$ is of relative dimension $g$ over $U$, is due to Rapoport (cf. \cite{rapoport78} Definition 1.1); it is automatically satisfied whenever the discriminant of $R$ is invertible in $U$ (\cite{DP94} Corollaire 2.9).

\begin{obs}\label{rem-isogeny}
 For any non-zero $r\in R$, the endomorphism $m(r):X\to X$ is an isogeny over $U$, i.e., surjective and quasi-finite --- which, in this case, is equivalent to finite and locally free. Indeed, if $N(r)\in \ZZ\setminus\{0\}$ denotes the norm of $r\in R$, then there exists $s\in R$ such that $rs=N(r)$. Thus, that  $m(r)$ is an isogeny follows easily from the fact that and the composition $m(r)\circ m(s) : X \to X$ is the multiplication by $N(r)$, which is an isogeny itself. In particular, $m$ is always injective.
\end{obs}

For a principally polarized abelian scheme with $R$-multiplication $(X,\lambda, m)$ over $U$, it follows from \cite{rapoport78}, Lemme 1.3, that $H^1_{\dR}(X/U)$ is a rank 2 vector bundle with $R$-multiplication over $U$. Since the image of $m$ lies in $\End_U(X)^{\lambda}\subset \End_U(X)$, we may check using the explicit construction given in Paragraph \ref{shb2} that the symplectic form $\langle \ , \ \rangle_{\lambda}$ is compatible with the $R$-multiplication. We denote by \label{symb:psilambda}
$$
\Psi_{\lambda} : H^1_{\dR}(X/U)\tensor_{\mathcal{O}_U\tensor R}H^1_{\dR}(X/U)\to \mathcal{O}_U \tensor D^{-1}
$$
the unique (perfect alternating) $\mathcal{O}_U\tensor R$-bilinear form for which
$$
\Tr \Psi_{\lambda} = \langle \ , \ \rangle_{\lambda}\text{,}
$$
so that $(H^1_{\dR}(X/U),\Psi_{\lambda})$ is a rank 2  symplectic vector bundle with $R$-multiplication over $U$.

Note that any rank 1 subbundle with $R$-multiplication of $H^1_{\dR}(X/U)$ is isotropic for $\Psi_{\lambda}$; this applies in particular to $F^{1}(X/U)$.

\subsection{Symplectic-Hodge bases}\label{subsec-shbrm}

Consider the rank 2 projective $R$-module $M \defeq R \oplus D^{-1}$ endowed with the standard $D^{-1}$-valued symplectic form
\begin{align*}
\Psi : M \times M &\to D^{-1}\\
       ((r,x),(r',x'))&\mapsto rx'-r'x\text{.}
\end{align*}
For any scheme $U$, we obtain a rank 2 symplectic vector bundle with $R$-multiplication 
$$
(\mathcal{O}_U\tensor M , 1\tensor \Psi)
$$
over $U$.

\begin{defi}\label{defi-shbasisrm}
  Let $U$ be a scheme and $(X,\lambda,m)$ be a principally polarized abelian scheme with $R$-multiplication over $U$. A \emph{symplectic-Hodge basis} of $(X,\lambda,m)_{/U}$ is an isomorphism of symplectic vector bundles with $R$-multiplication over $U$
  $$
   b: (\mathcal{O}_U\tensor M,1\tensor \Psi) \stackrel{\sim}{\to} (H^1_{\dR}(X/U),\Psi_{\lambda})
   $$
 sending $\mathcal{O}_U\tensor (R\oplus 0) \subset \mathcal{O}_U\tensor M$ to $F^1(X/U)\subset H^1_{\dR}(X/U)$.
\end{defi}

Note that $1\tensor \Psi$ induces an $\mathcal{O}_U\tensor R$-isomorphism 
$$
{\bigwedge}^2_{\mathcal{O}_U\tensor R}\mathcal{O}_U\tensor M\stackrel{\sim}{\to} \mathcal{O}_U\tensor D^{-1}
$$
trivializing the $\mathcal{O}_U\tensor R$-module of alternating $\mathcal{O}_U\tensor R$-bilinear forms over $\mathcal{O}_U\tensor M$ with values in $\mathcal{O}_U\tensor D^{-1}$:
\begin{align}\tag{$*$}
\mathcal{O}_U\tensor R \stackrel{\sim}{\to} {\Hom}_{\mathcal{O}_U\tensor R}\left({\bigwedge}_{\mathcal{O}_U\tensor R}^2\mathcal{O}_U\tensor M,\mathcal{O}_U\tensor D^{-1}\right).
\end{align}
A symplectic-Hodge basis $b$ of $(X,\lambda,m)_{/U}$ may be seen as an $\mathcal{O}_U\tensor R$-isomorphism 
$$
b=(\omega,\eta) : \mathcal{O}_U\tensor M\cong (\mathcal{O}_U\tensor R)\oplus (\mathcal{O}_U\tensor D^{-1}) \stackrel{\sim}{\to} H^1_{\dR}(X/U)
$$
such that
\begin{enumerate}
 \item $\omega:\mathcal{O}_U\tensor R\to H^1_{\dR}(X/U)$ factors through $F^1(X/U)\subset H^1_{\dR}(X/U)$, and
 \item $\Psi_{\lambda}(\omega,\eta)=1$.
\end{enumerate}
Here, $\Psi_{\lambda}(\omega,\eta)$ is regarded as the element of $\mathcal{O}_U\tensor R$ mapping to $b^*\Psi_{\lambda}$ via $(*)$.

Equivalently, if we regard $\Psi_{\lambda}$ as an alternating $\mathcal{O}_U\tensor R$-bilinear form
$$
\Psi_{\lambda}: H^1_{\dR}(X/U)\tensor_{\mathcal{O}_U\tensor R} H^1_{\dR}(X/U)\tensor_R D \to \mathcal{O}_U\tensor R\text{,}
$$
then a symplectic-Hodge basis of $(X,\lambda,m)_{/U}$ is a couple $b=(\omega,\eta)$, where $\omega$ is a global section of $F^1(X/U)\subset H^1_{\dR}(X/U)$ generating it as an $\mathcal{O}_U\tensor R$-module, $\eta$ is a global section of $H^1_{\dR}(X/U)\tensor_R D$ whose image in $(H^1_{\dR}(X/U)/F^1(X/U))\tensor_R D$ generates it as an $\mathcal{O}_U\tensor R$-module, and $\Psi_{\lambda}(\omega,\eta)=1$.

\begin{obs}\label{remshb}
 Since $\Psi_{\lambda}$ is perfect, if $\omega$ is an $\mathcal{O}_U\tensor R$-trivialization of $F^1(X/U)$, and $\eta$ is any global section of $H^1_{\dR}(X/U)\tensor_R D$ satisfying $\Psi_{\lambda}(\omega,\eta)=1$, then $b=(\omega,\eta)$ is a symplectic-Hodge basis.
\end{obs}

\begin{obs}\label{remshb2}
 If $\eta$ is a global section of $H^1_{\dR}(X/U)\tensor_R D$ whose image in $(H^1_{\dR}(X/U)/F^1(X/U))\tensor_R D$ generates it as an $\mathcal{O}_U\tensor R$-module, then there exists a unique $\mathcal{O}_U\tensor R$-trivialization $\omega$ of $F^1(X/U)$ such that $(\omega,\eta)$ is a symplectic-Hodge basis.
\end{obs}

\section{The moduli stacks $\mathcal{B}_g$ and $\mathcal{B}_F$} \label{modulistack}

In this section we define for every integer $g\ge 1$ (resp. for every totally real number field $F$) a category $\mathcal{B}_g$ (resp. $\mathcal{B}_F$) fibered in groupoids over the category of schemes $\Sch_{/\ZZ}$ classifying principally polarized abelian schemes of relative dimension $g$ (resp. principally polarized abelian schemes with $R$-multiplication) endowed with a symplectic-Hodge basis.

Using classical results on moduli stacks of abelian schemes, we then prove that $\mathcal{B}_g \to \Spec \ZZ$ (resp. $\mathcal{B}_F\to \Spec \ZZ$) is a smooth Deligne-Mumford stack over $\Spec \ZZ$ of relative dimension $2g^2+g$ (resp. $3g$).

%The main point in proving this result will be to remark that for any principally polarized abelian scheme $(X,\lambda)$ of relative dimension $g$ over an affine scheme $U=\Spec R$, there is a natural free and transitive right action of the Siegel parabolic subgroup $P_g(R)$ of $\Sp_{2g}(R)$, consisting of ``upper triangular matrices'', on the set of symplectic-Hodge bases of $(X,\lambda)_{/U}$.   

\subsection{The moduli stacks $\mathcal{A}_g$ and $\mathcal{A}_F$} 

Let $g\ge 1$ be an integer (resp. $F$ be a totally real number field of degree $g$ with ring of integers $R$). To fix ideas and notation we recall the definition of the moduli stack of principally polarized abelian schemes of relative dimension $g$ (resp. principally polarized abelian schemes with $R$-multiplication). 

For any scheme $S$, we define a category fibered in groupoids $\mathcal{A}_{g,S} \to \Sch_{/S}$ (resp. $\mathcal{A}_{F,S}\to \Sch_{/S}$) as follows.\label{symb:AgAF}
\begin{enumerate}[(i)]
    \item An object of $\mathcal{A}_{g,S}$ (resp. $\mathcal{A}_{F,S}$) is given by an $S$-scheme $U$ and a principally polarized abelian scheme $(X,\lambda)$ of relative dimension $g$ (resp. a principally polarized abelian scheme with $R$-multiplication $(X,\lambda,m)$) over $U$; when $U$ is not clear in the context, we shall incorporate it in the notation by writing $(X,\lambda)_{/U}$. A morphism $(X,\lambda)_{/U} \to (Y,\mu)_{/V}$ (resp. $(X,\lambda,m)_{/U} \to (Y,\mu,n)_{/V}$) in $\mathcal{A}_{g,S}$ (resp. $\mathcal{A}_{F,S}$), denoted $\varphi_{/f}$, is given by a Cartesian diagram of $S$-schemes
%$$
%  \raisebox{-0.5\height}{\includegraphics{rameq1-d2.pdf}}
%$$
  $$
 \begin{tikzcd}
   X \arrow{r}{\varphi}\arrow{d} & Y \arrow{d}\\
   U \arrow{r}[swap]{f} & V \arrow[draw=none]{ul}[description]{\square}
 \end{tikzcd}
 $$
preserving the identity sections of the abelian schemes and identifying $\lambda$ with the pullback of $\mu$ by $f:U \to V$ (and satisfying $n(r)\circ \varphi = \varphi \circ m(r)$ for every $r\in R$, in the case of $R$-multiplication). We shall occasionally denote $\varphi_{/f}$ simply by $\varphi$ when there will be no danger of confusion. We may also denote $(X,\lambda) = (Y,\mu)\times_UV$ (resp. $(X,\lambda,m) = (Y,\mu,n)\times_UV$). 
 \item The structural functor $\mathcal{A}_{g,S} \to \Sch_{/S}$ (resp. $\mathcal{A}_{F,S}\to \Sch_{/S}$) is given by sending an object $(X,\lambda)_{/U}$ of $\mathcal{A}_{g,S}$ (resp. $(X,\lambda,m)_{/U}$ of $\mathcal{A}_{F,S}$) to the $S$-scheme $U$, and a morphism $\varphi_{/f}$ to $f$.
\end{enumerate}
 
If $S=\Spec \Lambda$ is affine, then we denote $\mathcal{A}_{g,S}\eqdef\mathcal{A}_{g,\Lambda}$ (resp. $\mathcal{A}_{F,S}\eqdef\mathcal{A}_{F,\Lambda}$). When $\Lambda=\ZZ$, we simply drop it from notation.

Recall that the category of $S$-schemes can be seen as a subcategory of the 2-category of categories fibered in groupoids over $\Sch_{/S}$ by sending each $S$-scheme $U$ to the category $\Sch_{/U}$ endowed with its natural functor $\Sch_{/U} \to \Sch_{/S}$. In the sequel, we shall adopt the standard convention of denoting $\Sch_{/U}$ simply by $U$ when working in the context of categories fibered in groupoids. Then $\mathcal{A}_{g,S}$ (resp. $\mathcal{A}_{F,S}$) is canonically equivalent to $\mathcal{A}_g\times_{\ZZ}S$ (resp. $\mathcal{A}_F\times_{\ZZ} S$) as categories fibered in groupoids over $S$.

% We summarize the main properties of $\mathcal{A}_{g,S}$ (resp. $\mathcal{A}_{K,S}$) we are going to use in the form of the next theorem.

\begin{theorem} \label{DMstack}
For any scheme $S$, $\mathcal{A}_{g,S}$ (resp. $\mathcal{A}_{F,S}$) is a smooth Deligne-Mumford stack over $S$ of relative dimension $g(g+1)/2$ (resp. $g$). %Moreover, $\mathcal{A}_g$ admits a quasi-projective coarse moduli space $A_g$ over $\ZZ$.
\end{theorem}

A proof that $\mathcal{A}_{g,S}$ is a Deligne-Mumford stack over $S$ is essentially contained in \cite{GIT94} Theorem 7.9 (cf. \cite{olsson12} Theorem 2.1.11). Smoothness and relative dimension are obtained by a theorem of Grothendieck (cf. \cite{oort71} Proposition 2.4.1). The case of real multiplication is treated in \cite{rapoport78} Théorème 1.20; in Rapoport's notation, our $\mathcal{A}_F$ corresponds to $\mathcal{M}^{L}$ with $L=R$. %By a theorem of Keel and Mori (cf. \cite{olsson16} Theorem 11.1.2), $\mathcal{A}_g$ admits a coarse moduli space $A_g$ over $\ZZ$ in the category of algebraic spaces; it follows from \cite{moret-bailly85}, VII, Théorème 4.2, that $A_g$ is actually a quasi-projective scheme over $\ZZ$.

\begin{obs}
The stack $\mathcal{A}_g$ is often called a \emph{Siegel moduli stack}, whereas $\mathcal{A}_F$  is known as a \emph{Hilbert-Blumenthal moduli stack}.
\end{obs}

\begin{obs}
Beware that there is a fundamental difference between the moduli stack $\mathcal{A}_g$ and the \emph{coarse moduli scheme} $A_g$ (see page \pageref{symb:coarseAg}), often referred in the literature simply as ``the moduli space of principally polarized abelian varieties of dimension $g$'' (and similarly for the case of real multiplication). Even over $\CC$, the moduli stack $\mathcal{A}_{g,\CC}$ is \emph{not} representable by a scheme (or an algebraic space). Let us also remark that, while $\mathcal{A}_{g,\CC}$ is smooth over $\Spec \CC$ in the sense of Deligne-Mumford stacks for every $g\ge 1$ ,the coarse moduli scheme $A_{g,\CC}$ is not a smooth scheme over $\Spec \CC$ for $g\ge 3$ (see \cite{oort76}).
\end{obs}

\subsection{Definition of the moduli stacks $\mathcal{B}_g$ and $\mathcal{B}_F$} \label{defbg}

We first treat the Siegel case. Let $\varphi_{/f}: (X,\lambda)_{/U} \to (Y,\mu)_{/V}$ be a morphism in $\mathcal{A}_{g}$. By the compatibility with base change of the symplectic forms induced by principal polarizations (Remark \ref{remarkbasechange}), the pullback $\varphi^*b$ of every symplectic-Hodge basis $b$ of $(Y,\mu)_{/V}$ is a symplectic-Hodge basis of $(X,\lambda)_{/U}$. We can thus define a functor
\begin{align*}
 \underline{B}_g : \mathcal{A}_{g}^{\opp} \longrightarrow \Sets 
\end{align*}
that sends every object $(X,\lambda)_{/U}$ of $\mathcal{A}_{g}$  to the set of symplectic-Hodge bases of $(X,\lambda)_{/U}$, and whose action on morphisms is given by pullbacks as above. 

From the functor $\underline{B}_g$, we form a category fibered in groupoids \label{symb:Bg}
\begin{align*}
\mathcal{B}_g \to \Spec \ZZ 
\end{align*}
 as follows.
\begin{enumerate}[(i)]
      \item  An object of $\mathcal{B}_g$ is a ``triple" $(X,\lambda,b)_{/U}$  where $(X,\lambda)_{/U}$  is an object of $\mathcal{A}_{g}$ and $b \in \underline{B}_g(X,\lambda)$. An arrow $(X,\lambda,b)_{/U} \to (Y,\mu,c)_{/V}$ is given by a morphism $\varphi_{/f}: (X,\lambda)_{/U} \to (Y,\mu)_{/V}$ in $\mathcal{A}_g$ such that $b=\varphi^*c$. We denote by \label{symb:pig}
\begin{align*}
\pi_g: \mathcal{B}_g \to \mathcal{A}_{g}
\end{align*}
     the forgetful functor $(X,\lambda,b)_{/U}\mapsto (X,\lambda)_{/U}$.
     \item The structural functor $\mathcal{B}_g \to \Spec \ZZ$ is defined as the composition of $\pi_g$ with the structural functor $\mathcal{A}_{g} \to \Spec \ZZ$.
\end{enumerate}

Analogously, in the Hilbert-Blumenthal case, we consider a functor \label{symb:BF}
$$
\underline{B}_F : \mathcal{A}_{F}^{\opp} \longrightarrow \Sets 
$$
sending a principally polarized abelian scheme with $R$-multiplication to the set of its symplectic-Hodge basis (Definition \ref{defi-shbasisrm}), and we derive from it a category fibered in groupoids
$$
\mathcal{B}_F\to \Spec \ZZ
$$
whose objects over a scheme $U$ are given by ``quadruples" $(X,\lambda,m,b)_{/U}$. We denote by\label{symb:piF}
$$
\pi_F: \mathcal{B}_F\to \mathcal{A}_F
$$
the natural forgetful functor.

\begin{obs}[Relating $\mathcal{B}_F$ with $\mathcal{B}_g$]\label{rem-relationbfbg}
  Consider the canonical morphism of stacks $f:\mathcal{A}_F \to \mathcal{A}_g$ given by the forgetful functor.  Let $(x_1,\ldots,x_g)$ be a $\ZZ$-basis of $D^{-1}$, and $(r_1,\ldots,r_g)$ be the corresponding dual $\ZZ$-basis of $R$, so that
  $$
  t\defeq (r_1,\ldots,r_g,x_1,\ldots,x_g) : (\ZZ^{2g}, \langle \ , \ \rangle_{\text{std}}) \stackrel{\sim}{\to} (M=R\oplus D^{-1}, \Tr \Psi)
  $$
  is an isomorphism of symplectic $\ZZ$-modules (notation as in Paragraph \ref{subsec-shbrm}). Then it is easy to check that $t$ induces a morphism of stacks
  \begin{align*}
  f_{t}:\mathcal{B}_F &\to \mathcal{B}_g\\
         (X,\lambda,m,b)_{/U}&\mapsto (X,\lambda,b\circ \theta_U)_{/U}
  \end{align*}
  making the diagram
 \begin{equation}\label{eq-cdft}
  \begin{tikzcd}
    \mathcal{B}_F \arrow{r}{f_{t}}\arrow{d}[swap]{\pi_g} & \mathcal{B}_g \arrow{d}{\pi_F} \\
    \mathcal{A}_F \arrow{r}[swap]{f} & \mathcal{A}_g
  \end{tikzcd}
 \end{equation}
  commute.
\end{obs}

The rest of this section is devoted to the proof of the next theorem.

\begin{theorem} \label{smoothdmstack}
The category fibered in groupoids $\mathcal{B}_g\to \Spec \ZZ$ (resp. $\mathcal{B}_F\to \Spec \ZZ$) is a smooth Deligne-Mumford stack over $\Spec \ZZ$ of relative dimension $2g^2+g$ (resp. $3g$).
\end{theorem}

\subsection{Siegel parabolic subgroup and proof of Theorem \ref{smoothdmstack} for $\mathcal{B}_g$} \label{uppertriang}

Fix a scheme $U$ and an object $(X,\lambda)$ of $\mathcal{A}_{g}$ lying over $U$. Then we can define a functor
\begin{align*}
\underline{B}_{(X,\lambda)} : \Sch_{/U}^{\opp} \longrightarrow \Sets
\end{align*}  
that sends a $U$-scheme $U'$ to the set $\underline{B}_g((X,\lambda)\times_UU')$. It is clear that this functor defines a sheaf for the Zariski topology over $\Sch_{/U}$.

Let us now consider the \emph{symplectic group} $\Sp_{2g}$, namely the smooth affine group scheme over $\Spec \ZZ$ of relative dimension $2g^2+g$ such that for every affine scheme $V=\Spec \Lambda$
\begin{align*}
{\Sp}_{2g}(V) = \left.\left\{\left(\begin{array}{cc}
                           A & B \\
                           C & D
                          \end{array} \right) \in M_{2g\times 2g}(\Lambda) \ \right| \begin{aligned} &\ \ \ \ \ \ \ \ \ \ \ \   A,B,C,D \in M_{g\times g}(\Lambda) \text{ satisfy  } \\ &AB\transp = BA\transp\text{, }  CD\transp= DC\transp\text{, and } AD\transp -BC\transp = \mathbf{1}_g \end{aligned}\right\}\text{.}
 \end{align*}

The \emph{Siegel parabolic subgroup} $P_g$ of $\Sp_{2g}$ is defined as the subgroup scheme of $\Sp_{2g}$ such that, for every affine scheme $V=\Spec \Lambda$,
\begin{align*} P_g(V) = 
 \left.\left\{\left(\begin{array}{cc}
                           A & B \\
                           0 & (A\transp)^{-1}
                          \end{array} \right) \in M_{2g\times 2g}(\Lambda) \ \right|\  A \in {\GL}_g(\Lambda)\text{ and } B\in M_{g\times g}(\Lambda)\text{ satisfy }AB\transp=BA\transp \right\}\text{.}
\end{align*}
Note that $P_g$ is a smooth affine group scheme over $\Spec \ZZ$ of relative dimension $g(3g+1)/2$.

Let $(X,\lambda,b)$ be an object of $\mathcal{B}_g$ lying over $V=\Spec \Lambda$ and consider $b = (\  \omega \ \  \eta\  )$ as a row vector of order $2g$ with coefficients in the $R$-module $H^1_{\dR}(X/V)$. For any
\begin{align*}
p=\left(\begin{array}{cc}
                           A & B \\
                           0 & (A\transp)^{-1}
                          \end{array} \right) \in P_g(V)
\end{align*}
it easy to check that
\begin{align*}
b\cdot p \defeq (\begin{array}{cc}\omega A &  \omega B + \eta (A\transp)^{-1}\end{array})
\end{align*}
is a symplectic-Hodge basis of $(X,\lambda)_{/V}$. This defines a right action of $P_g(V)$ on $\underline{B}_g(X,\lambda)$:
\begin{align*}
\underline{B}_g(X,\lambda) \times P_g(V) \longrightarrow \underline{B}_g(X,\lambda)\text{.}
\end{align*}
 Moreover, it is clear that if $V'\subset V$ is an affine open subscheme of $V$, then the natural diagram
%$$
%\raisebox{-0.5\height}{\includegraphics{rameq1-d3.pdf}}
%$$
 $$
 \begin{tikzcd}
 \underline{B}_g(X,\lambda) \times P_g(V) \arrow{r}\arrow{d}& \underline{B}_g(X,\lambda)\arrow{d}\\
 \underline{B}_g(X',\lambda') \times P_g(V') \arrow{r}& \underline{B}_g(X',\lambda')
 \end{tikzcd}
 $$
commutes, where $(X',\lambda')=(X,\lambda)\times_VV'$.

Thus, for any scheme $U$, and any object $(X,\lambda)$ of $\mathcal{A}_{g}$ lying over $U$, we obtain a right action of the $U$-group scheme $P_{g,U} = P_g\times_{\ZZ}U$ on $\underline{B}_{(X,\lambda)}$.

\begin{lemma} \label{torsor}
The Zariski sheaf $\underline{B}_{(X,\lambda)}$ over $\mathsf{Sch}_{/U}$ is a right Zariski $P_{g,U}$-torsor for the above action.
\end{lemma}

\begin{proof}
If $V$ is any affine scheme over $U$ such that $\underline{B}_{(X,\lambda)}(V)$ is non-empty, a routine computation shows that the action of $P_g(V)$ on $\underline{B}_{(X,\lambda)}(V)$ is free and transitive. Moreover, it was already remarked above that symplectic-Hodge bases exist locally for the Zariski topology.
\end{proof}

Since $P_{g,U}$ is affine, smooth, and of relative dimension $g(3g+1)/2$ over $U$, Lemma \ref{torsor} immediately implies the following.

\begin{coro} \label{relrepr0}
For every scheme $U$, and every object $(X,\lambda)$ of $\mathcal{A}_{g}$ lying over $U$, the functor $\underline{B}_{(X,\lambda)}$ is representable by a smooth affine $U$-scheme $B(X,\lambda)$ of relative dimension $g(3g+1)/2$. \hfill $\blacksquare$
\end{coro}

\begin{obs} \label{relrepr1}
With the notation of the above corollary, recall that the principally polarized abelian scheme $(X,\lambda)$ over $U$ corresponds to a morphism $U \to \mathcal{A}_g$, so that $B(X,\lambda)$ is a scheme representing $\mathcal{B}_g\times_{\mathcal{A}_g}U$.
\end{obs} 

\begin{proof}[Proof of Theorem \ref{smoothdmstack} for $\mathcal{B}_g$]
Recall that for any scheme $U$ and any abelian scheme $X$ over $U$, $H^1_{\dR}(X/U)$ is a quasi-coherent sheaf over $U$, and that any quasi-coherent sheaf over $U$ induces a sheaf over $\Sch_{/U}$ endowed with the fppf topology (\cite{olsson16} Lemma 4.3.3). Since the étale topology is coarser than the fppf topology, this shows in particular that $H^1_{\dR}(X/U)$ induces a sheaf over $\Sch_{/U}$ endowed with the étale topology; this immediately implies that $\mathcal{B}_g\to \Spec \ZZ$ is a stack over $\Spec \ZZ$. 

It follows in particular from Corollary \ref{relrepr0} that the morphism $\pi_g : \mathcal{B}_g \to \mathcal{A}_{g}$ is representable by smooth schemes (Remark \ref{relrepr1}). Hence, as $\mathcal{A}_{g}\to \Spec \ZZ$ is a Deligne-Mumford stack over $\Spec \ZZ$, the same holds for $\mathcal{B}_g\to \Spec \ZZ$ (\cite{olsson16} Proposition 10.2.2). The smoothness of $\mathcal{B}_g \to \Spec \ZZ$ follows by composition from that of $\mathcal{A}_{g} \to \Spec \ZZ$ and that of $\pi_g$. Finally, we can compute the relative dimension of $\mathcal{B}_g \to \Spec \ZZ$ as the sum of that of  $\mathcal{A}_{g} \to \Spec \ZZ$ and that of $\pi_g$:
\begin{align*}
\frac{g(g+1)}{2} + \frac{g(3g +1)}{2} = 2g^2+g\text{.}
\end{align*}  
\end{proof}

\subsection{Proof of Theorem \ref{smoothdmstack} for $\mathcal{B}_F$}\label{subsec-proofsmothdmstackrm}

% We merely explain the main modifications in the above argument to obtain Theorem \ref{smoothdmstack} for $\mathcal{B}_K$. 

Let $M$ and $\Psi$ be as in Paragraph \ref{subsec-shbrm}, and consider the affine group scheme $\Aut_{(M,\Psi)}$ over $\Spec R$ of $R$-automorphisms of $M$ preserving $\Psi$. It contains a (Borel) subgroup scheme $\Aut_{(M,\Psi,R\oplus 0)}$ of those automorphisms fixing the Lagrangian $R\oplus 0 \subset M=R\oplus D^{-1}$. We set \label{symb:PF}
$$
P_F \defeq {\Res}_{R/\ZZ} {\Aut}_{(M,\Psi,R\oplus 0)}\text{.}
$$
This is a smooth affine group scheme of relative dimension $2g$ over $\Spec \ZZ$. If $V=\Spec \Lambda$ is an affine scheme, then 
$$
P_F(V) = \left\{\left.\left(\begin{array}{cc}
                       a & b \\
                       0  & a^{-1}
                      \end{array}\right)\in {\SL}_2(\Lambda \tensor F) \right| a\in (\Lambda\tensor R)^{\times}\text{, }b\in \Lambda\tensor D
 \right\}
$$
where $D\subset R$ denotes the different ideal.

Arguing as above, for an object $(X,\lambda,m)_{/U}$ of $\mathcal{B}_F$, we see that the Zariski sheaf
$$
\underline{B}_{(X,\lambda,m)}: \Sch_{/U}\to \Sets
$$
sending an $U$-scheme $U'$ to $\underline{B}_F((X,\lambda,m)\times_UU')$ is a right Zariski $P_{F,U}$-torsor. This implies that $\pi_F: \mathcal{B}_{F}\to\mathcal{A}_F$ is relatively representable by smooth affine schemes of relative dimension $2g$. We conclude, as in the proof for $\mathcal{B}_g$, with an application of Theorem \ref{DMstack}. \hfill $\blacksquare$

\section{The tangent bundles of $\mathcal{B}_g$ and $\mathcal{B}_F$; higher Ramanujan vector fields} \label{ramvecfields}

This section is devoted the study of the tangent bundles $T_{\mathcal{B}_g/\ZZ}$ and $T_{\mathcal{B}_F/\ZZ}$.

We shall first explain how the Gauss-Manin connection on the first de Rham cohomology of abelian schemes induces a canonical decomposition
$$
T_{\mathcal{B}_g/\ZZ}=T_{\mathcal{B}_g/\mathcal{A}_g} \oplus \mathcal{R}_g \text{ (resp. }T_{\mathcal{B}_F/\ZZ}=T_{\mathcal{B}_F/\mathcal{A}_F} \oplus \mathcal{R}_F\text{);} 
$$
$\mathcal{R}_g\subset T_{\mathcal{B}_g/\ZZ}$ and $\mathcal{R}_F\subset T_{\mathcal{B}_F/\ZZ}$ are called \emph{Ramanujan subbundles}.

Then, we show that the deformation theory of abelian varieties, in the guise of the Kodaira-Spencer morphism, allows us to canonically trivialize the Ramanujan subbundles. These trivializations are the \emph{higher Ramanujan vector fields}.

\subsection{Horizontal subbundles and linear connections}

We briefly review Ehresmann's point of view on connections over vector bundles. In the context of differential geometry, this is standard material; for a more general discussion in the algebraic setting, we refer to \cite{bost13} 6.1.

Let $S$ be a scheme, $X$ be a smooth $S$-scheme, and $\pi:E \to X$ be a smooth scheme over $X$. 

\begin{defi}
 A subbundle $\mathcal{F}$ of $T_{E/S}$ is said to be \emph{horizontal} (with respect to $\pi:E \to X$) if $T_{E/S} = T_{E/X}\oplus \mathcal{F}$.
\end{defi}

As $T_{E/X}=\ker (T\pi:T_{E/S}\to \pi^*T_{X/S})$, a horizontal subbundle is a splitting of the exact sequence
\begin{align*}
0 \longrightarrow T_{E/X} \longrightarrow T_{E/S} \stackrel{T\pi}{\longrightarrow}  \pi^*T_{X/S} \to 0\text{.}
\end{align*}
In particular, $T\pi$ restricts to an isomorphism $\mathcal{F}\stackrel{\sim}{\to}\pi^*T_{X/S}$.

Assume now that $\mathcal{E}$ is a vector bundle over $X$, and that $\pi:E=\mathbf{V}(\mathcal{E}^{\vee})\to X$ is its associated space over $X$. Then, to any  $\mathcal{O}_S$-linear connection on $\mathcal{E}$
$$
\nabla : \mathcal{E} \to \mathcal{E}\tensor_{\mathcal{O}_X}\Omega^1_{X/S}
$$
there is attached a canonical horizontal subbundle of $T_{E/S}$.

Indeed, observe first that there is a canonical identification
\begin{align}\label{eq-isomcantan}
T_{E/X}\stackrel{\sim}{\to} \pi^*\mathcal{E} 
\end{align}
given by the dual of $\pi^*\mathcal{E}^{\vee} \stackrel{\sim}{\to} \Omega^1_{E/X}$, defined locally (on $X$) by $1\tensor f \mapsto df$ (cf. \cite{EGAIV4} Corollaire 16.4.9).

\begin{lemma}\label{lemma-projconnection}
 Let $e\in \Gamma (E,\pi^*\mathcal{E})$ be the ``universal section" of $\pi^*\mathcal{E}$, and $\pi^*\nabla$ be the pullback of $\nabla$ to $\pi^*\mathcal{E}$. The $\mathcal{O}_E$-morphism
 \begin{align*}
  P_{\nabla}:T_{E/S} &\to \pi^*\mathcal{E}\\
  \theta &\mapsto (\pi^*\nabla)_{\theta}e
 \end{align*}
 restricts to the isomorphism (\ref{eq-isomcantan}) on $T_{E/X}\subset T_{E/S}$.\hfill $\blacksquare$
\end{lemma}

It follows that the subbundle $\ker P_{\nabla}\subset T_{E/S}$ is horizontal: under the identification (\ref{eq-isomcantan}), $P_{\nabla}$ becomes a projection of $T_{E/S}$ onto the subbundle $T_{E/X}$. This is the horizontal subbundle attached to $\nabla$.

\begin{obs}\label{rem-integrablehorizontal}
  If $\nabla$ is integrable, then $\ker P_{\nabla}$ is an integrable subbundle of $T_{E/S}$.
  %and $T\pi : \ker P_{\nabla} \to \pi^*T_{X/S}$ is an isomorphism of Lie algebras over $\mathcal{O}_E$.
\end{obs}

It is not difficult to transpose the above considerations to the case of smooth Deligne-Mumford stacks (cf. \ref{tangentstacks}).

\subsection{The Ramanujan subbundle $\mathcal{R}_g\subset T_{\mathcal{B}_g/\ZZ}$}

\subsubsection{} \label{sssec-univgm}

Fix a base scheme $S$ and let $p:X \to U$ be a projective abelian scheme, with $U$ a \emph{smooth} $S$-scheme. Then there is defined an integrable $S$-connection over the de Rham cohomology sheaves (\cite{KO68}; see also \cite{katz70}), the \emph{Gauss-Manin connection}
\begin{align} \label{GMcon}
\nabla : H^i_{\dR}(X/U) \longrightarrow  H^i_{\dR}(X/U)\tensor_{\mathcal{O}_U} \Omega^1_{U/S}\text{,}
\end{align}
whose formation is compatible with every base change $U'\to U$, where $U'$ is a smooth $S$-scheme.

\label{symb:Hg}

We next construct a ``universal" version of Gauss-Manin connection over $\mathcal{A}_g$. Consider the presheaf $\mathcal{H}_g$ of $\mathcal{O}_{\mathcal{A}_{g,\et}}$-modules on $\Et(\mathcal{A}_g)$ defined as follows. Let $(U,u)$ be an étale scheme over $\mathcal{A}_g$, and $(X,\lambda)$ be the principally polarized abelian scheme over $U$  corresponding to $u:U \to \mathcal{A}_g$. We put
\begin{align*}
\Gamma((U,u),\mathcal{H}_g)\defeq \Gamma(U,H^1_{\dR}(X/U))
\end{align*}
If $(f,f^b):(U',u') \to (U,u)$ is a morphism in $\Et(\mathcal{A}_g)$, the restriction map is given by the base change morphism $f^*H^1_{\dR}(X/U) \to H^1_{\dR}(X'/U')$, where $(X',\lambda') = (X,\lambda)\times_UU'$. As the base change morphism is actually an isomorphism (i.e., the formation of $H^1_{\dR}(X/U)$ is compatible with base change), and $H^1_{\dR}(X/U)$ is quasi-coherent, $\mathcal{H}_g$ is a quasi-coherent sheaf over $\mathcal{A}_g$ (cf. \ref{stacks} and \cite{olsson16} Lemma 4.3.3). We finally remark that $\mathcal{H}_g$ is actually a vector bundle of rank $2g$ over $\mathcal{A}_g$.

\begin{obs}
The sheaf $\mathcal{H}_g$ should be thought as the first de Rham cohomology of the ``universal abelian scheme'' over $\mathcal{A}_g$. % Indeed, if $X_g$ denotes the universal abelian scheme over $B_g$ (cf. Theorem \ref{repr}), then, under the isomorphism $\mathcal{B}_{g,\ZZ[1/2]} \cong B_g$, the sheaf $\mathcal{H}_g$ gets identified with $H^1_{\dR}(X_g/B_g)$, and $\mathcal{F}_g$ with $F^1(X_g/B_g)$.
\end{obs}

For any scheme $S$, let $\mathcal{H}_{g,S}$ be the vector bundle over $\mathcal{A}_{g,S}$ obtained from $\mathcal{H}_g$ by the base change $\mathcal{A}_{g,S}\to \mathcal{A}_g$. Since the formation of the Gauss-Manin connection is compatible with base change, we have an $S$-connection on $\mathcal{H}_{g,S}$
\begin{align*}
\nabla : \mathcal{H}_{g,S} \to \mathcal{H}_{g,S}\tensor_{\mathcal{O}_{\mathcal{A}_{g,S,\et}}}\Omega^1_{\mathcal{A}_{g,S}/S}
\end{align*}
defined by (\ref{GMcon}) over every étale $S$-scheme $(U,u)$ over $\mathcal{A}_{g,s}$ as above.

\subsubsection{}

Consider the morphism of coherent $\mathcal{O}_{\mathcal{B}_{g,\et}}$-modules
\begin{align}\label{eq-morpairing}
\pi_g^*\mathcal{H}_g^{\oplus g} \longrightarrow M_{g\times g}(\mathcal{O}_{\mathcal{B}_{g,\et}})
\end{align}
given on an étale scheme $(U,u)$ over $\mathcal{B}_g$ corresponding to $(X,\lambda,b)_{/U}$, $b=(\omega_1,\ldots,\omega_g,\eta_1,\ldots,\eta_g)$, by
\begin{align*}
H^1_{\dR}(X/U)^{\oplus g} &\to M_{g\times g}(\mathcal{O}_U)\\
(\alpha_1,\ldots,\alpha_g)&\mapsto (\langle \alpha_i,\eta_j\rangle_{\lambda})_{1\le i,j\le g}\text{,} 
\end{align*}
and let $\mathcal{S}_g$ be the subbundle of $\pi_g^*\mathcal{H}_g^{\oplus g}$ defined as the inverse image of the subbundle of symmetric matrices $\Sym_g(\mathcal{O}_{\mathcal{B}_{g,\et}})\subset M_{g\times g}(\mathcal{O}_{\mathcal{B}_{g,\et}})$ by (\ref{eq-morpairing}).

\begin{obs} \label{rem-ranksg}
Note that (\ref{eq-morpairing}) is surjective: for a given matrix $(a_{ij})_{1\le i,j\le g}$ in $M_{g\times g}(\mathcal{O}_U)$, take $\alpha_i \defeq \sum_{j=1}^ga_{ij}\omega_j$. In particular, $\mathcal{S}_g$ is a subbundle of $\pi_g^*\mathcal{H}_g^{\oplus g}$ of rank $g^2 + g(g+1)/2 = g(3g+1)/2$.
\end{obs}

\begin{theorem}\label{thm-vertbg}
 Consider the morphism of quasi-coherent $\mathcal{O}_{\mathcal{B}_{g,\et}}$-modules
\begin{align*}
P:T_{\mathcal{B}_g/\ZZ} \longrightarrow \pi_g^*\mathcal{H}_g^{\oplus g}
\end{align*}
defined by
\begin{align*}
T_{U/\ZZ} &\to H^1_{\dR}(X/U)^{\oplus g}\\
\theta&\mapsto (\nabla_{\theta}\eta_1, \ldots,\nabla_{\theta}\eta_g)
\end{align*}
for every étale scheme $(U,u)$ over $\mathcal{B}_g$ corresponding to the object $(X,\lambda,b)_{/U}$ of $\mathcal{B}_g(U)$, where $b=(\omega_1,\ldots,\omega_g,\eta_1,\ldots,\eta_g)$, and $\nabla$ denotes the Gauss-Manin connection on $H^1_{\dR}(X/U)$. Then the morphism $P$
\begin{enumerate}
 \item factors through $\mathcal{S}_g\subset \pi_g^*\mathcal{H}_g^{\oplus g}$, and
 \item restricts to an isomorphism $P:T_{\mathcal{B}_g/\mathcal{A}_g}\stackrel{\sim}{\to} \mathcal{S}_g$. 
\end{enumerate}
\end{theorem}

%The notation $P_{\nabla}$ will be justified in the proof of the theorem.

\label{symb:Rg}
\begin{defi}\label{defi-ramubbundsiegel}
With the above notation, the \emph{Ramanujan subbundle} of $T_{\mathcal{B}_g/\ZZ}$ is the horizontal subbundle with respect to $\pi_g:\mathcal{B}_g \to \mathcal{A}_g$ defined by $\mathcal{R}_g\defeq \ker P$.
\end{defi}

We now proceed to the proof of Theorem \ref{thm-vertbg}.

\subsubsection{}

Consider the associated space of the vector bundle $\mathcal{H}_g^{\oplus g}$ (cf. \cite{olsson16} 10.2)
\begin{align*}
\mathcal{V}_g \defeq \mathbf{V}((\mathcal{H}_g^{\oplus g})^{\vee})=\underline{\Spec}_{\mathcal{A}_g}\mathcal{S}ym((\mathcal{H}_g^{\oplus g})^{\vee})\text{.}
\end{align*}
This is a Deligne-Mumford stack over $\Spec \ZZ$ whose objects lying over a scheme $U$ are given by ``$(g+2)$-uples''
\begin{align*}
(X, \lambda, \alpha_1,\ldots,\alpha_g)_{/U}\text{,}
\end{align*}
where $(X,\lambda)_{/U}$ is an object of $\mathcal{A}_{g}(U)$, and $\alpha_i$ is a global section of $H^1_{\dR}(X/U)$ for every $1\le i\le g$. Note that the forgetful functor
\begin{align*}
\tilde{\pi}_g : \mathcal{V}_g \to \mathcal{A}_{g}
\end{align*}
defines a morphism of stacks representable by smooth affine schemes. 

We define a morphism of stacks
\begin{align*}
i_g : \mathcal{B}_g \to \mathcal{V}_g
\end{align*}
 as follows. Let $(X,\lambda,b)_{/U}$ be an object of $\mathcal{B}_g$ and denote $b=(\omega_1,\ldots,\omega_g,\eta_1,\ldots,\eta_g)$. Then $i_g$ sends $(X,\lambda,b)_{/U}$ to the object
\begin{align*}
(X,\lambda,\eta_1,\ldots,\eta_g)_{/U}
\end{align*}
of $\mathcal{V}_g$. The action of $i_g$ on morphisms is evident. Note that the diagram of morphisms of stacks
%$$
%\raisebox{-0.5\height}{\includegraphics{rameq1-d6.pdf}}
%$$
 $$
 \begin{tikzcd}[column sep=tiny]
 \mathcal{B}_g\arrow{rd}[swap]{\pi_g} \arrow{rr}{i_g} & & \mathcal{V}_g \arrow{dl}{\tilde{\pi}_g}\\
 & \mathcal{A}_g
 \end{tikzcd}
 $$
is (strictly) commutative.

\begin{lemma} \label{immersionlemma}
The morphism $i_g: \mathcal{B}_g\to \mathcal{V}_g$ is an immersion of stacks.
\end{lemma}

\begin{proof}
Let $U$ be a scheme and $U \to \mathcal{V}_g$ be a morphism corresponding to the object $(X,\lambda, \alpha_1,\ldots,\alpha_g)_{/U}$ of $\mathcal{V}_g(U)$. Then the fiber product $\mathcal{B}_g \times_{\mathcal{V}_g} U$ can be naturally identified with the locally closed subscheme of $U$ defined by the equations
\begin{align*}
\overline{\alpha}_1\wedge \cdots \wedge \overline{\alpha}_g &\neq 0\\
\langle \alpha_i,\alpha_j \rangle_{\lambda} &= 0 \text{, }\ \ \forall i,j
\end{align*}
where $\overline{\alpha}_i$ denotes the image of $\alpha_i$ in $H^1_{\dR}(X/U)/F^1(X/U)$ (cf. Proposition \ref{exisunic} (2)).
\end{proof}

\begin{proof}[Proof of Theorem \ref{thm-vertbg}]

To prove (1), let $(U,u)$ be an étale scheme over $\mathcal{B}_g$ corresponding to the object $(X,\lambda,b)_{/U}$ of $\mathcal{B}_g(U)$, with $b=(\omega_1,\ldots,\omega_g,\eta_1,\ldots,\eta_g)$, and let $\theta$ be a section of $T_{U/\ZZ}$. As $\langle \eta_i,\eta_j\rangle_{\lambda}=0$, we obtain
\begin{align*}
0 = \nabla_{\theta}\langle \eta_i,\eta_j\rangle_{\lambda} = \langle \nabla_{\theta}\eta_i,\eta_j \rangle_{\lambda} + \langle \eta_i,\nabla_{\theta}\eta_j\rangle_{\lambda} = \langle \nabla_{\theta}\eta_i,\eta_j \rangle_{\lambda} - \langle \nabla_{\theta}\eta_j,\eta_i \rangle_{\lambda}\text{.}
\end{align*}

We now prove (2). Observe that $\mathcal{H}^{\oplus g}_g$ is endowed with an integrable connection $\nabla$ given by the sum of the ``universal" Gauss-Manin connection on each factor. As $\tilde{\pi}_g:\mathcal{V}_g\to\mathcal{A}_g$ is the space associated to $\mathcal{H}^{\oplus g}_g$, we obtain from Lemma \ref{lemma-projconnection} a morphism of $\mathcal{O}_{\mathcal{V}_{g,\et}}$-modules
 $$
 P_{\nabla}: T_{\mathcal{V}_g/\ZZ} \to \tilde{\pi}_g^*\mathcal{H}_g^{\oplus g}
 $$
 inducing an isomorphism
 $$
 T_{\mathcal{V}_g/\mathcal{A}_g} \stackrel{\sim}{\to} \tilde{\pi}_g^*\mathcal{H}_g^{\oplus g}\text{.}
 $$
 
 The morphism $P$ is simply the restriction of $P_{\nabla}$ to $T_{\mathcal{B}_g/\ZZ}$ via the immersion $i_g:\mathcal{B}_g\to \mathcal{V}_g$. In particular, as $Ti_g$ identifies $T_{\mathcal{B}_g/\mathcal{A}_g}$ with a subbundle of $i_g^*T_{\mathcal{V}_g/\mathcal{A}_g}$, the induced the morphism
 $$
 P_{\nabla}=P: T_{\mathcal{B}_g/\mathcal{A}_g}\to \mathcal{S}_g
 $$
 is injective; since both vector bundles have the same rank (cf. Remark \ref{rem-ranksg}), this must be an isomorphism.
\end{proof}

\begin{obs}\label{rem-rgint}
It follows from the above proof and from Remark \ref{rem-integrablehorizontal} that the Ramanujan subbundle $\mathcal{R}_g\subset T_{\mathcal{B}_g/\ZZ}$ is integrable.
\end{obs}

\subsection{The Ramanujan subbundle $\mathcal{R}_F\subset T_{\mathcal{B}_F/\ZZ}$}\label{subsec-ramsubbundrm}

Let $S$ be a scheme, $U$ be a smooth $S$-scheme, and $(X,\lambda,m)$ be a principally polarized abelian scheme with $R$-multiplication over $U$.

Since, for every $r\in R$, the endomorphism $m(r):X\to X$ is an isogeny (Remark \ref{rem-isogeny}), the action of $R$ on $H^1_{\dR}(X/U)$ induced by $m$ is horizontal for the Gauss-Manin connection $\nabla$ (cf. \cite{mori11} Proposition 2.2). In particular, by linearity, $\nabla$ induces a connection on $H^1_{\dR}(X/U)\tensor_RD$; by abuse, we denote it by the same symbol:
$$
\nabla : H^1_{\dR}(X/U)\tensor_R D \to (H^1_{\dR}(X/U)\tensor_R D) \tensor_{\mathcal{O}_U}\Omega^1_{U/S}\text{.}
$$

\label{symb:HF}
By the same reasoning of (\ref{sssec-univgm}), we define a universal first de Rham cohomology $\mathcal{H}_F$ over $\mathcal{A}_F$. For any scheme $S$, we denote by $\mathcal{H}_{F,S}$ the vector bundle over $\mathcal{A}_{F,S}$ obtained from $\mathcal{A}_F$ by base change. We also have a universal Gauss-Manin connection
$$
\nabla : \mathcal{H}_{F,S} \to \mathcal{H}_{F,S}\tensor_{\mathcal{O}_{\mathcal{A}_{F,S,\et}}}\Omega^1_{\mathcal{A}_{F,S}/S}\text{.}
$$
Note that the vector bundle $\mathcal{H}_{F,S}$ over $\mathcal{A}_{F,S}$ is endowed with a canonical $R$-multiplication which is horizontal for the universal Gauss-Manin connection above. In particular, we also have a connection
$$
\nabla : \mathcal{H}_{F,S}\tensor_R D \to (\mathcal{H}_{F,S}\tensor_RD)\tensor_{\mathcal{O}_{\mathcal{A}_{F,S,\et}}}\Omega^1_{\mathcal{A}_{F,S}/S}\text{.}
$$

\begin{theorem}\label{thm-isomvertbk}
 Consider the morphism of quasi-coherent $\mathcal{O}_{\mathcal{B}_{F,\et}}$-modules
\begin{align*}
P:T_{\mathcal{B}_F/\ZZ} \longrightarrow \pi_F^*\mathcal{H}_F\tensor_RD
\end{align*}
defined by
\begin{align*}
T_{U/\ZZ} &\to H^1_{\dR}(X/U)\tensor_RD\\
\theta&\mapsto \nabla_{\theta}\eta
\end{align*}
for every étale scheme $(U,u)$ over $\mathcal{B}_F$ corresponding to the object $(X,\lambda,m,b)_{/U}$ of $\mathcal{B}_F(U)$, where $b=(\omega,\eta)$, and $\nabla$ denotes the Gauss-Manin connection on $H^1_{\dR}(X/U)\tensor_RD$. Then the morphism $P$ restricts to an isomorphism
\begin{align}\label{eq-isomvertbk}
 P:T_{\mathcal{B}_F/\mathcal{A}_F}\stackrel{\sim}{\to} \pi_F^*\mathcal{H}_F\tensor_RD\text{.}
\end{align}
\end{theorem}

The proof below is analogous to the case $g=1$ of Theorem \ref{thm-vertbg}.

\begin{proof}
 Consider the stack
 $$
 \mathcal{V}_F \defeq \mathbf{V}((\mathcal{H}_F\tensor_RD)^{\vee})\text{,}
 $$
 and denote by $\tilde{\pi}_F : \mathcal{V}_F \to \mathcal{A}_F$  the natural projection. Let $\nabla$ be the universal Gauss-Manin connection on $\mathcal{H}_{F}\tensor_R D$, and let
 $$
 P_{\nabla} : T_{\mathcal{V}_F/\ZZ} \longrightarrow \tilde{\pi}_F^*\mathcal{H}_F\tensor_RD
 $$
 be defined as in Lemma \ref{lemma-projconnection}, so that it induces an isomorphism
 $$
 T_{\mathcal{V}_F/\mathcal{A}_F} \stackrel{\sim}{\longrightarrow} \tilde{\pi}_F^*\mathcal{H}_F\tensor_RD\text{.}
 $$
 
 It follows from Remark \ref{remshb2} that the morphism
 $$
 i_F : \mathcal{B}_F \to \mathcal{V}_F
 $$
 over $\mathcal{A}_F$ given by $(X,\lambda,m,b=(\omega,\eta))_{/U}\mapsto (X,\lambda,m,\eta)_{/U}$ is an \emph{open immersion} of stacks. We conclude by remarking that the morphism $P$ is simply the restriction of the above $P_{\nabla}$ to $T_{\mathcal{B}_F/\ZZ}$ via $i_F$.
\end{proof}

\label{symb:RF}
\begin{defi}
 With the above notation, the \emph{Ramanujan subbundle} of $T_{\mathcal{B}_F/\ZZ}$ is the horizontal subbundle with respect to $\pi_F: \mathcal{B}_F \to \mathcal{A}_F$ defined by $\mathcal{R}_F\defeq \ker P$.
\end{defi}

Observe that the Ramanujan subbundle $\mathcal{R}_F\subset T_{\mathcal{B}_F/\ZZ}$ is integrable by Remark \ref{rem-integrablehorizontal}.

\begin{obs}
  The morphism $f_t: \mathcal{B}_F \to \mathcal{B}_g$ defined in Remark \ref{rem-relationbfbg} preserves the decomposition of the tangent bundles of $\mathcal{B}_F$ and $\mathcal{B}_g$ induced by the Ramanujan subbundles. Observe first that the commutativity of the diagram (\ref{eq-cdft}) implies that $Tf_t: T_{\mathcal{B}_F/\ZZ} \to f_t^*T_{\mathcal{B}_g/\ZZ}$ preserves the vertical subbundles:
  $$
   Tf_t(T_{\mathcal{B}_F/\mathcal{A}_F}) \subset f_t^*T_{\mathcal{B}_g/\mathcal{A}_g}\text{.}
   $$
   Now, it follows from the definition of the Ramanujan subbundles that $\mathcal{R}_g$ (resp. $\mathcal{R}_F$) is given by the equations $\nabla_{v}\eta_i=0$ (resp. $\nabla_v\eta_F(1\tensor x_i)=0$), for $1\le i \le g$, where $(\omega_1,\ldots,\omega_g,\eta_1,\ldots,\eta_g)$ (resp. $(\omega_F,\eta_F)$) denotes the ``universal'' symplectic Hodge basis over $\mathcal{B}_g$ (resp. $\mathcal{B}_F$). Since, by definition of $f_t$, we have $f_t^*\eta_i = \eta_F(1\tensor x_i)$, and since the formation of the Gauss-Manin connection commutes with base change, we deduce that
   $$
Tf_t(\mathcal{R}_F)\subset f_t^*\mathcal{R}_g\text{.}
   $$
\end{obs}

\subsection{Recollections on the Kodaira-Spencer morphism} \label{kodairaspencer}

\subsubsection{}

 Fix a base scheme $S$ and let $p:X \to U$ be a projective abelian scheme, with $U$ a smooth $S$-scheme. The Gauss-Manin connection on $H^1_{\dR}(X/U)$ induces a morphism
\begin{align*}
T_{U/S} &\longrightarrow \mathcal{H}om_{\mathcal{O}_S}(H_{\dR}^1(X/U),H^1_{\dR}(X/U))\\
    \theta &\longmapsto \nabla_{\theta}( \ \ )\text{.}
\end{align*}
Restricting to $F^1(X/U)$ and passing to the quotient (cf. exact sequence (\ref{hodgefiltr})), we obtain an $\mathcal{O}_{U}$-morphism 
\begin{align*}
T_{U/S} \longrightarrow &\mathcal{H}om_{\mathcal{O}_U}(F^1(X/U),R^1p_*\mathcal{O}_X)\cong F^1(X/U)^{\vee}\tensor_{\mathcal{O}_U}R^1p_*\mathcal{O}_X\text{.}
\end{align*}
Applying the inverse of the canonical isomorphism $\phi_{X^t/U}^1: F^1(X^t/U)^{\vee} \stackrel{\sim}{\to} R^1p_*\mathcal{O}_{X}$ (cf. proof of Lemma \ref{f1lagrangian}, where we identified $X$ with $X^{tt}$ via the canonical biduality isomorphism), we obtain an $\mathcal{O}_U$-morphism
\begin{align*}
\delta : T_{U/S} \longrightarrow F^1(X/U)^{\vee}\tensor_{\mathcal{O}_U}F^1(X^t/U)^{\vee}\text{.}
\end{align*}
This is, possibly up to a sign, the dual of $\rho$ defined in \cite{FC90} III.9.\footnote{With notation as in the proof of Lemma \ref{f1lagrangian}, there are two natural ways of identifying $R^1p_*\mathcal{O}_X$ with $F^1(X^t/U)^{\vee}$: one by $(\phi_{X/U}^0)^{\vee}$, and another by $\phi_{X^t/U}^1$. These produce the same isomorphisms up to a sign. In \cite{FC90} this choice is not specified.}

\subsubsection{}

 With the same notation as above, let $\lambda :X \to X^t$ be a principal polarization. The Gauss-Manin connection $\nabla$ on $H^1_{\dR}(X/U)$ is compatible with the symplectic form $\langle \ , \ \rangle_{\lambda}$ in the following sense. For every sections $\theta$ of $T_{U/S}$, and $\alpha$ and $\beta$ of $H^1_{\dR}(X/U)$, we have
\begin{align} \label{C}
\theta \langle \alpha, \beta\rangle_{\lambda} = \langle \nabla_{\theta}\alpha,\beta \rangle_{\lambda} + \langle \alpha, \nabla_{\theta}\beta \rangle_{\lambda}\text{.}
\end{align}
This can be deduced from the fact that the first Chern class in $H^2_{\dR}(X\times_UX^t/U)$ of the Poincaré line bundle $\mathcal{P}_{X/U}$ is horizontal for the Gauss-Manin connection, since it actually comes from a class in $H^2_{\dR}(X\times_UX^t/S)$.

 By composing $\delta$ with $((\lambda^*)^{\vee})^{-1} : F^1(X^t/U)^{\vee} \stackrel{\sim}{\to}F^1(X/U)^{\vee}$, we obtain a morphism 
\begin{align} \label{KSmor}
\kappa : T_{U/S} \longrightarrow F^1(X/U)^{\vee}\tensor_{\mathcal{O}_U}F^1(X/U)^{\vee}\text{.}
\end{align}
This is the \emph{Kodaira-Spencer morphism} associated to $(X,\lambda)_{/U}$ over $S$. It follows from the compatibility (\ref{C}) that $\kappa$ factors through the second divided power $\Gamma^2(F^1(X/U)^{\vee})$, i.e., the submodule of symmetric tensors in $F^1(X/U)^{\vee}\tensor_{\mathcal{O}_U}F^1(X/U)^{\vee}$.

\begin{obs} \label{explicite}
As $\phi_{X^t/U}^{\vee} = -\phi_{X/U}$ under the canonical biduality isomorphism $X\cong X^{tt}$ (cf. \cite{BBM82} Lemme 5.1.5), one may verify that the composition
\begin{align*}
R^1p_*\mathcal{O}_X \stackrel{(\phi^1_{X^t/U})^{-1}}{\to} F^1(X^t/U)^{\vee} \stackrel{((\lambda^*)^{\vee})^{-1}}{\to} F^1(X/U)^{\vee}
\end{align*}
considered above is given by the isomorphism of vector bundles $H^1_{\dR}(X/U)/F^1(X/U) \stackrel{\sim}{\to} F^1(X/U)^{\vee}$ induced by (cf. Lemma \ref{exactseq})
\begin{align*}
H^1_{\dR}(X/U) &\to H^1_{\dR}(X/U)^{\vee}\\
         \alpha & \mapsto \langle \  \ , \alpha\rangle_{\lambda}\text{.}
\end{align*}
Thus, if $(\omega_1,\ldots,\omega_g)$ is a trivialization of $F^1(X/U)$, $\kappa$ admits the following explicit description:
% \begin{align*}
%   \kappa (\theta) = \sum_{i=1}^g \langle \ \ , \eta_i \rangle_{\lambda}\tensor\langle \omega_i,\nabla_{\theta}( \ \ )\rangle_{\lambda} \text{.}
% \end{align*}
\begin{align*}
  \kappa (\theta) = \sum_{i=1}^g \omega_i^{\vee}\tensor\langle \ \ ,\nabla_{\theta}\omega_i\rangle_{\lambda} \text{.}
\end{align*}
\end{obs}

 Finally, we observe that the Kodaira-Spencer morphism is natural in the following sense. Let $U'$ be a smooth scheme over $S$ and let $F_{/f}:(X',\lambda')_{/U'} \to (X,\lambda)_{/U'}$ be a morphism in $\mathcal{A}_{g,S}$. Denote by $\kappa$ (resp. $\kappa'$) the Kodaira-Spencer morphism associated to $(X,\lambda)_{/U}$ (resp.  $(X',\lambda')_{/U'}$) over $S$. Then the diagram
%$$
%\raisebox{-0.5\height}{\includegraphics{rameq1-d5.pdf}}
%$$
 $$
 \begin{tikzcd}[column sep = huge]
 T_{U'/S} \arrow{r}{Df}\arrow{d}[swap]{\kappa'} &f^*T_{U/S}\arrow{d}{f^*\kappa} \\
 \Gamma^2(F^1(X'/U')^{\vee}) \arrow{r}[swap]{(f^*)^{\vee}\tensor(f^*)^{\vee}} & \Gamma^2(f^*F^1(X/U)^{\vee})
 \end{tikzcd}
 $$
commutes. 

\subsubsection{} We keep the above notation and we further assume that $(X,\lambda)_{/U}$ is endowed with an $R$-multiplication $m:R \to \End_U(X)^{\lambda}$. 

Since the action of $R$ on $H^1_{\dR}(X/U)$ is horizontal for the Gauss-Manin connection, we obtain an $\mathcal{O}_U$-morphism
\begin{align*}
 T_{U/S} &\to \mathcal{H}om_{\mathcal{O}_U\tensor R}(F^1(X/U),R^1p_*\mathcal{O}_X)\\
      \theta&\mapsto \nabla_{\theta}( \ \ ) \mod F^1(X/U)\text{.}
\end{align*}
By combining this with the $\mathcal{O}_U\tensor R$-isomorphism induced by $\Psi_{\lambda}$
\begin{align*}
 R^1p_*\mathcal{O}_X = H^1_{\dR}(X/U)/F^1(X/U)&\stackrel{\sim}{\to} F^1(X/U)^*\tensor_RD^{-1}\\
                 \alpha \mod F^1(X/U) &\mapsto \Psi_{\lambda} ( \ \ , \alpha)
\end{align*}
we obtain a Kodaira-Spencer morphism (of $\mathcal{O}_U$-modules)
$$
\kappa : T_{U/S} \to \Gamma^2_{\mathcal{O}_U\tensor R}(F^1(X/U)^*)\tensor_R D^{-1}
$$
associated to $(X,\lambda,m)_{/U}$ over $S$.

% By Remark \ref{explicite}, the composition of the above Kodaira-Spencer morphism with the natural map given by the duality formula (\ref{eq-duality})
% $$
% \Gamma^2_{\mathcal{O}_U\tensor R}(F^1(X/U)^*)\tensor_R D^{-1}\to \Gamma^2_{\mathcal{O}_U}(F^1(X/U)^{\vee})
% $$
% is the previously defined Kodaira-Spencer morphism (\ref{KSmor}). 

\begin{obs}\label{rem-formulaksrm}
 If $\omega$ is an $\mathcal{O}_U\tensor R$-trivialization of $F^1(X/U)$, then 
 $$
 \kappa(\theta)=  \omega^*\tensor \Psi_{\lambda}(\ , \nabla_{\theta}\omega) = \Psi_{\lambda}(\omega,\nabla_{\theta}\omega)\, \omega^*\tensor \omega^*\text{.}
 $$
\end{obs}

\begin{obs}\label{rem-formulaksrmdual}
  By the natural duality between second divided powers $\Gamma^2$ and second symmetric powers $S^2$, we get the following canonical isomorphisms (cf. Remark \ref{rem-dualityrelation})
  \begin{align*}
\Gamma^2_{\mathcal{O}_U\tensor R}(F^1(X/U)^*)\tensor_R D^{-1}\cong S_{\mathcal{O}_U\tensor R}^2(F^1(X/U))^*\tensor_R D^{-1} \stackrel{\Tr}{\cong} S^2_{\mathcal{O}_U\tensor R}(F^1(X/U))^{\vee}\text{.}
  \end{align*}
  Under these identifications, the $\mathcal{O}_U$-dual of $\kappa$ is given explicitly by
 \begin{align*}
\kappa^{\vee}: S^2_{\mathcal{O}_U\tensor R}(F^1(X/U))&\to \Omega^1_{U/S} \\
                                               \omega\tensor \omega &\mapsto \langle \omega,\nabla \omega\rangle_{\lambda}\text{.}
\end{align*}
\end{obs}

\subsection{The Kodaira-Spencer isomorphism for $\mathcal{A}_g$ and $\mathcal{A}_F$}\label{ksiso}

\subsubsection{}

Just like we defined a universal first de Rham cohomology $\mathcal{H}_g$ over $\mathcal{A}_g$, we may define a universal Hodge subbundle $\mathcal{F}_g$: for any étale scheme $(U,u)$ over $\mathcal{A}_g$ corresponding to the object $(X,\lambda)_{/U}$ of $\mathcal{A}_g(U)$ we have $u^*\mathcal{F}_g = F^1(X/U)$. \label{symb:F_g}

Let $S$ be a scheme, and denote by $\mathcal{F}_{g,S}$ the rank $g$ vector bundle over $\mathcal{A}_{g,S}$ obtained from $\mathcal{F}_g$ by base change. The naturality of the Kodaira-Spencer morphism permits us to construct a ``universal'' Kodaira-Spencer morphism
\begin{align*}
\kappa : T_{\mathcal{A}_{g,S}} \longrightarrow \Gamma^2(\mathcal{F}_{g,S}^{\vee})\text{.}
\end{align*}
We remark that $\kappa$ is actually an \emph{isomorphism} of $\mathcal{O}_{\mathcal{A}_{g,S,\et}}$-modules by \cite{FC90} Theorem 5.7.(3) (cf. \cite{lan13} 2.3.5).

Let $\mathcal{U}$ be a smooth Deligne-Mumford stack over $S$ and $u : \mathcal{U} \to \mathcal{A}_{g,S}$ be a quasi-compact and quasi-separated morphism of $S$-stacks representable by schemes. Then, the Gauss-Manin connection over $(\mathcal{U},u)$, or simply over $\mathcal{U}$ if $u$ is implicit, 
\begin{align*}
\nabla : u^*\mathcal{H}_{g,S} \to u^*\mathcal{H}_{g,S}\tensor_{\mathcal{O}_{\mathcal{U}_{\et}}} \Omega^1_{\mathcal{U}/S}
\end{align*} 
is defined by pulling back the universal Gauss-Manin connection on $\mathcal{A}_{g,S}$. Further, we may define a Kodaira-Spencer morphism over $(\mathcal{U},u)$ as the composition
\begin{align*}
\kappa_{u} : T_{\mathcal{U}/S} \stackrel{T u }{\longrightarrow} u^*T_{\mathcal{A}_{g,S}/S} \stackrel{u^*\kappa}{\longrightarrow} \Gamma^2(u^*\mathcal{F}_{g,S}^{\vee})\text{.}
\end{align*}
 
\subsubsection{}

\label{symb:FF}
Analogously, we define a Hodge subbundle $\mathcal{F}_F\subset \mathcal{H}_F$ endowed with a canonical $R$-multiplication. For any scheme $S$, we also have a ``universal" Kodaira-Spencer \emph{isomorphism} (cf. \cite{rapoport78} 1.5 and \cite{lan13} 2.3.5)
 \begin{align*}
\kappa : T_{\mathcal{A}_{F,S}} \stackrel{\sim}{\longrightarrow} \Gamma_{\mathcal{O}_{\mathcal{A}_{F,S,\et}}\tensor R}^2(\mathcal{F}_{F,S}^{*})\tensor_R D^{-1}\text{.}
\end{align*}

For a smooth Deligne-Mumford stack  $\mathcal{U}$ over  $S$  endowed with a quasi-compact and quasi-separated morphism of $S$-stacks representable by schemes $u : \mathcal{U} \to \mathcal{A}_{F,S}$, we can also associate a Gauss-Manin connection 
\begin{align*}
\nabla : u^*\mathcal{H}_{F,S} \to u^*\mathcal{H}_{F,S}\tensor_{\mathcal{O}_{\mathcal{U}_{\et}}} \Omega^1_{\mathcal{U}/S}
\end{align*} 
and a Kodaira-Spencer morphism
\begin{align*}
\kappa_{u} : T_{\mathcal{U}/S} \to \Gamma_{\mathcal{O}_{\mathcal{U}_{\et}}\tensor R}^2(u^*\mathcal{F}_{F,S}^{*})\tensor_R D^{-1}\text{.}
\end{align*}

\subsection{The higher Ramanujan vector fields on $\mathcal{B}_g$} \label{ssec-hrvf}

Recall that the Ramanujan subbundle $\mathcal{R}_g\subset T_{\mathcal{B}_g/\ZZ}$ is a horizontal subbundle with respect to $\pi_g: \mathcal{B}_g\to \mathcal{A}_g$. In particular, the tangent map
$$
T\pi_g: \mathcal{R}_g \to \pi_g^*T_{\mathcal{A}_g/\ZZ}
$$
is an isomorphism. By composing it with (the pullback by $\pi_g$ of) the Kodaira-Spencer isomorphism for $\mathcal{A}_g$, we obtain an isomorphism
\begin{align}\label{eq-ksisomramanujan}
  \kappa_{\pi_g}:\mathcal{R}_g \stackrel{\sim}{\to} \Gamma^2(\pi_g^*\mathcal{F}_g^{\vee})\text{.}
\end{align}

Consider the ``universal'' symplectic-Hodge basis over $\mathcal{B}_g$ \label{symb:universalbg}
\begin{align*}
b_g = (\omega_1,\ldots,\omega_g,\eta_1,\ldots,\eta_g)\text{;}
\end{align*}
that is, the basis of the vector bundle $\pi_g^*\mathcal{H}_g$ such that for every étale scheme $(U,u)$ over $\mathcal{B}_g$ corresponding to the object $(X,\lambda,b)_{/U}$ of $\mathcal{B}_g(U)$ we have $u^*b_g=b$. In particular, $(\omega_1,\ldots,\omega_g)$ trivializes $\pi_g^*\mathcal{F}_g$, and its dual basis induces an isomorphism
\begin{align*}
 \Gamma^2(\pi_g^*\mathcal{F}_g^{\vee}) \stackrel{\sim}{\to} \Gamma^2(\mathcal{O}_{\mathcal{B}_{g,\et}}^{\oplus g}) = \mathcal{O}_{\mathcal{B}_{g,\et}}\tensor \Gamma^2(\ZZ^g)\text{.}
\end{align*}
By composing the above isomorphism with (\ref{eq-ksisomramanujan}), we obtain
\begin{align}\label{eq-isomrg}
 \mathcal{R}_g \stackrel{\sim}{\to} \Gamma^2(\mathcal{O}_{\mathcal{B}_{g,\et}}^{\oplus g}) = \mathcal{O}_{\mathcal{B}_{g,\et}}\tensor \Gamma^2(\ZZ^g)\text{.}
\end{align}

\label{symb:vg}
\begin{defi} \label{deframvf}
For every $1\le i \le j \le g$, we define the \emph{higher Ramanujan vector field} $v_{ij}$ as being the unique global section of $\mathcal{R}_g\subset T_{\mathcal{B}_g/\ZZ}$ such that
$$
v_{ij}\mapsto \begin{cases}
               \mathbf{e}_i\tensor \mathbf{e}_i & i=j\\
               \mathbf{e}_i\tensor \mathbf{e}_j + \mathbf{e}_j\tensor \mathbf{e}_i & i< j
              \end{cases}
$$
under the isomorphism (\ref{eq-isomrg}).
\end{defi}

Alternatively, let
\begin{align*}
\langle \  , \ \rangle : \pi_g^*\mathcal{H}_g \times \pi_g^*\mathcal{H}_g \to \mathcal{O}_{\mathcal{B}_{g,\et}}
\end{align*}
be the symplectic $\mathcal{O}_{\mathcal{B}_{g,\et}}$-bilinear form given, for each étale scheme $(U,u)$ over $\mathcal{B}_g$ corresponding to the object $(X,\lambda,b)_{/U}$ of $\mathcal{B}_g(U)$, by
\begin{align*}
u^*\langle \ , \ \rangle \defeq  \langle \ , \ \rangle_{\lambda} : H^1_{\dR}(X/U) \times H^1_{\dR}(X/U) \to \mathcal{O}_{U}\text{.}
\end{align*}
This is well-defined by Remark \ref{remarkbasechange}. Then the higher Ramanujan vector fields satisfy
\begin{align*}
\kappa_{\pi_g}(v_{ij}) =  \begin{cases}
               \langle \ \ , \eta_i \rangle \tensor \langle \ \ , \eta_i \rangle & i=j\\
               \langle \ \ , \eta_i \rangle\tensor \langle \ \ , \eta_j \rangle + \langle \ \ , \eta_j \rangle\tensor \langle \ \ , \eta_i \rangle & i <j
              \end{cases}
\end{align*}

The next proposition characterizes the higher Ramanujan vector fields in terms of the ``universal" Gauss-Manin connection over $\mathcal{B}_g$ (cf. Paragraph \ref{ksiso}):
\begin{align*}
\nabla : \pi^*_g\mathcal{H}_g \longrightarrow \pi^*_{g}\mathcal{H}_g \tensor_{\mathcal{O}_{\mathcal{B}_{g,\et}}} \Omega^1_{\mathcal{B}_g/\ZZ}\text{.}
\end{align*}

\begin{prop} \label{caracchamps}
Let us regard $b_g$ as a row vector of order $2g$. Then, the higher Ramanujan vector fields are the unique global sections $v_{ij}$ of $T_{\mathcal{B}_g/\ZZ}$ such that
\begin{align*}
 \nabla_{v_{ij}}b_g = b_g\left(\begin{array}{cc}
                              0 & 0 \\
                              \mathbf{E}^{ij} & 0
                             \end{array}\right)
\end{align*}
for every $1\le i\le j\le g$.
\end{prop}

\begin{obs}\label{rem-rephrase}
 The matricial equation above is equivalent to conditions (1) and (2) below
\begin{enumerate}
   \item $\nabla_{v_{ij}}\omega_i = \eta_j$, $\nabla_{v_{ij}}\omega_j=\eta_i$, and $\nabla_{v_{ij}}\omega_k=0$ for $k \not\in\{i,j\}$.
   \item $\nabla_{v_{ij}}\eta_k =0$, for every $1\le k \le g$.
\end{enumerate}
\end{obs}

\begin{proof}[Proof of Proposition \ref{caracchamps}]
The vector fields $v_{ij}$ satisfy (2) in the above remark by definition of $\mathcal{R}_g$. Moreover, using the explicit expression of the Kodaira-Spencer morphism in Remark \ref{explicite}, we see that 
\begin{align} \label{equiv1}
\sum_{k=1}^g\langle \ \ , \eta_k \rangle \tensor \langle \ \ ,\nabla_{v_{ij}}\omega_k \rangle  =  \begin{cases}
               \langle \ \ , \eta_i\rangle  \tensor \langle \ \ , \eta_i \rangle & i=j\\
               \langle \ \ , \eta_i\rangle  \tensor \langle \ \ , \eta_j \rangle +  \langle \ \ , \eta_j\rangle  \tensor \langle \ \ , \eta_i \rangle & i <j
              \end{cases}
\end{align}
in $\Gamma^2(\pi_g^*\mathcal{F}_g^{\vee})$ for every $1\le i \le j \le g$. As $b_g$ is symplectic with respect to $\langle \ , \ \rangle$, by evaluating the second factors at $\eta_l$ for every $1\le l \le g$ in the above equation, we see that $\nabla_{v_{ij}}\omega_k$ lies in the subbundle of $\pi^*_g\mathcal{H}_g$ generated by $\eta_1,\ldots,\eta_g$, for every $1\le i\le j \le g$ and $1\le k \le g$.

Thus, to prove that the vector fields $v_{ij}$ satisfy (1), it is sufficient to prove that
\begin{align}
 \langle \omega_l, \nabla_{v_{ij}}\omega_i\rangle = \delta_{lj}\text{, }\langle \omega_l, \nabla_{v_{ij}}\omega_j\rangle = \delta_{li}\text{, and } \langle \omega_l, \nabla_{v_{ij}}\omega_k\rangle = 0\text{ for }k\not\in\{i,j\}
\end{align}  
for every $1\le l \le g$. This in turn follows immediately from (\ref{equiv1}) by evaluating the second factors at $\omega_l$.

To prove unicity, let $(w_{ij})_{1\le i \le j \le g}$ be a family of vector fields on $\mathcal{B}_g$ satisfying (1) and (2). It follows immediately from (2) that each $w_{ij}$ is a section of $\mathcal{R}_g$. Moreover, by the explicit expression of the Kodaira-Spencer morphism in Remark \ref{explicite}, the equations in (1) imply that
\begin{align*}
\kappa_{\pi_g}(w_{ij}) = \begin{cases}
               \langle \ \ , \eta_i \rangle \tensor \langle \ \ , \eta_i \rangle & i=j\\
               \langle \ \ , \eta_i \rangle\tensor \langle \ \ , \eta_j \rangle + \langle \ \ , \eta_j \rangle\tensor \langle \ \ , \eta_i \rangle & i <j 
              \end{cases}
\end{align*}
Since $\kappa_{\pi_g}: \mathcal{R}_g \to \Gamma^2(\pi_g^*\mathcal{F}_g^{\vee})$ is an isomorphism, we must have $w_{ij}=v_{ij}$.
\end{proof}

\begin{lemma}\label{corocommute1}
Let $S$ be a scheme, and $\theta$ be a section of $T_{\mathcal{B}_{g,S}/S}$ such that $\nabla_{\theta}\omega_i=\nabla_{\theta}\eta_i=0$ for every $1\le i \le g$. Then $\theta=0$.
\end{lemma}

\begin{proof}
Let $\theta$ be as in the statement. Note that $\theta$ is in the subbundle $\mathcal{R}_{g,S}$ of $T_{\mathcal{B}_{g,S}/S}$; thus, there exist sections $(f_{ij})_{1\le i \le j \le g}$ of $\mathcal{O}_{\mathcal{B}_{g,S,\et}}$ such that
\begin{align*}
\theta = \sum_{1\le i \le j \le g}f_{ij}v_{ij}\text{.}
\end{align*}
We prove that each $f_{ij}=0$ by induction on $i$. For $i=1$, we have by Proposition \ref{caracchamps}
\begin{align*}
0 = \nabla_{\theta}\omega_1 = \sum_{1 \le i \le j \le g}f_{ij}\nabla_{v_{ij}}\omega_1 = \sum_{j=1}^gf_{1j}\eta_{j}\text{,} 
\end{align*}
thus $f_{1j}=0$ for every $1\le j \le g$. Let $2\le i_0 \le g$, and assume that $f_{ij}=0$ for every $i< i_0$ and $i\le j \le g$. From
\begin{align*}
0 = \nabla_{\theta}\omega_{i_0} = \sum_{i_0 \le i \le j \le g}f_{ij}\nabla_{v_{ij}}\omega_{i_0}= \sum_{j=i_0}^gf_{i_0j}\eta_j 
\end{align*}
we conclude that $f_{i_0j}=0$ for every $i_0\le j \le g$.
\end{proof}

Let $[ \ , \ ]$ denote the Lie bracket in $T_{\mathcal{B}_g/\ZZ}$.

\begin{coro} \label{corocommute}
 The higher Ramanujan vector fields commute. That is,
\begin{align*}
[v_{ij},v_{i'j'}] = 0
\end{align*}
for any $1\le i\le j \le g$ and $1\le i'\le j' \le g$.
\end{coro} 

\begin{proof}
We already remarked that $\mathcal{R}_g$ is integrable (Remark \ref{rem-rgint}). In particular, for any $1\le i\le j \le g$ and any $1\le i'\le j'\le g$, the vector field $\theta\defeq [v_{ij},v_{i'j'}]$ is a section of $\mathcal{R}_g$. By Lemma \ref{corocommute1}, to prove that $\theta=0$, it is sufficient to prove that $\nabla_{\theta}\omega_k=0$ for every $1\le k \le g$.

We have
\begin{align*}
\nabla_{\theta} \omega_k  = \nabla_{v_{ij}}(\nabla_{v_{i'j'}}\omega_k) - \nabla_{v_{i'j'}}(\nabla_{v_{ij}}\omega_k)\text{.}
\end{align*}
It follows from Proposition \ref{caracchamps} that $\nabla_{v_{i'j'}}\omega_k$ (resp. $\nabla_{v_{ij}}\omega_k$) is an element of $\{0,\eta_1,\ldots,\eta_g\}$; hence $\nabla_{v_{ij}}(\nabla_{v_{i'j'}}\omega_k) = 0$ (resp. $\nabla_{v_{i'j'}}(\nabla_{v_{ij}}\omega_k) = 0$).
\end{proof}

% \begin{obs}
% For $1\le i\le j \le g$, let $v_{ij}^{\vee}$ be the global section of $\Omega^1_{\mathcal{B}_g/\ZZ}$ dual to $v_{ij}$. By defining $v_{ji}^{\vee} = v_{ij}^{\vee}$, we obtain a global section $v^\vee = (v_{ij}^{\vee})_{1\le i,j\le g}$ of $M_{g\times g}(\Omega^1_{\mathcal{B}_g/\ZZ})$. Then, in matricial notation (see \ref{matricialnotation}), if $\omega = (\omega_1 \  \cdots \ \omega_g)$ and $\eta = (\eta_1 \ \cdots \ \eta_g)$ are considered as row vectors of order $g$, the conditions in the corollary above are equivalent to $\nabla \omega\transp = v^{\vee}\tensor \eta\transp$ and $\nabla \eta\transp = 0$.
% \end{obs}

\subsection{The higher Ramanujan vector fields on $\mathcal{B}_F$} \label{subsec-hrvfhb}

We argue as in the Siegel case: since the Ramanujan subbundle $\mathcal{R}_F\subset T_{\mathcal{B}_F/\ZZ}$ is horizontal with respect to $\pi_F: \mathcal{B}_F\to \mathcal{A}_F$, the tangent map
$$
T\pi_F : \mathcal{R}_{F} \to\pi_F^*T_{\mathcal{A}_F/\ZZ}
$$
is an isomorphism. By composing it with the Kodaira-Spencer isomorphism for $\mathcal{A}_F$, we obtain an isomorphism
$$
\kappa_{\pi_F}:\mathcal{R}_{F} \stackrel{\sim}{\to} \Gamma^2_{\mathcal{O}_{\mathcal{A}_{F,\et}}\tensor R}(\pi_F^*\mathcal{F}^*_F)\tensor_R D^{-1}\text{.}
$$

\label{symb:universalbF}

Let $b_F\defeq (\omega_F,\eta_F)$ be the ``universal" symplectic-Hodge basis over $\mathcal{B}_F$. By duality, the trivialization of $\pi_F^*\mathcal{F}_F$ as a (rank 1) $\mathcal{O}_{\mathcal{B}_{F,\et}}\tensor R$-module given by $\omega_F$ induces a trivialization of $\pi_F^*\mathcal{F}_F^*$. As the $\ZZ$-module $\Gamma^2(\ZZ)$ may be canonically identified with $\ZZ$, we then obtain an isomorphism (of $\mathcal{O}_{\mathcal{B}_{F,\et}}$-modules)
\begin{align}\label{eq-trivrk}
 \mathcal{R}_{F} \stackrel{\sim}{\to} \mathcal{O}_{\mathcal{B}_{F,\et}}\tensor D^{-1}\text{.}
\end{align}
% Observe that any choice of $\ZZ$-basis of $D^{-1}$ determines a trivialization of $\mathcal{R}_F$. We shall not fix such basis.

\label{symb:vF}
\begin{defi}
 The \emph{higher Ramanujan vector field} over $\mathcal{B}_F$ is the $\mathcal{O}_{\mathcal{B}_{F,\et}}$-isomorphism
 $$
 v_F : \mathcal{O}_{\mathcal{B}_{F,\et}}\tensor D^{-1} \stackrel{\sim}{\to} \mathcal{R}_{F}
 $$
 given by the inverse of (\ref{eq-trivrk}).
\end{defi}

Strictly speaking, $v_F$ is not a vector field on $\mathcal{B}_F$, but for any fixed choice of $\ZZ$-basis of $D^{-1}$ it determines $g$ \emph{bona fide} vector fields trivializing $\mathcal{R}_F$.

If we endow the tangent bundle $T_{\mathcal{B}_F/\ZZ}= T_{\mathcal{B}_F/\mathcal{A}_F}\oplus \mathcal{R}_F$ with the $R$-multiplication induced by the isomorphisms (\ref{eq-isomvertbk}) and (\ref{eq-trivrk}), then $v_F$ is $\mathcal{O}_{\mathcal{B}_{F,\et}}\tensor R$-linear, and can be thought as a global section of $T_{\mathcal{B}_F/\ZZ}\tensor_R D$. 

As the Gauss-Manin connection on $\pi_F^*\mathcal{H}_F$ is $R$-linear, it induces, for any fractional ideal $I\subset F$, an $\mathcal{O}_{\mathcal{B}_{F,\et}}\tensor R$-morphism
$$
T_{\mathcal{B}_F/\ZZ}\tensor_RI \to \mathcal{H}om_R(\pi_F^*\mathcal{H}_F,\pi_F^*\mathcal{H}_F\tensor_R I)\text{.}
$$

We omit the proof of the analogous of Proposition \ref{caracchamps}.

\begin{prop}\label{caracchampshb}
 The higher Ramanujan vector field $v_F$ is the unique global section of $T_{\mathcal{B}_F/\ZZ}\tensor_RD$ such that $\nabla_{v_F}\omega = \eta$ and $\nabla_{v_F}\eta = 0$. \hfill $\blacksquare$
\end{prop}

\begin{obs}\label{rem-compatsieglhbhre}
  As an application of Propositions \ref{caracchamps} and \ref{caracchampshb}, we can compute the effect of the morphism $f_t: \mathcal{B}_F \to \mathcal{B}_g$ of Remark \ref{rem-relationbfbg} on the higher Ramanujan vector fields. Namely, one may check that the following diagram commutes
  $$
  \begin{tikzcd}
    T_{\mathcal{B}_F/\ZZ} \arrow{r}{Tf_T} & f_t^*T_{\mathcal{B}_g/\ZZ}\\
    \mathcal{O}_{\mathcal{B}_F}\tensor D^{-1} \arrow{u}{v_F} \arrow{r}& f_t^*(\mathcal{O}_{\mathcal{B}_g}\tensor \Sym_g(\ZZ)) \arrow{u}[swap]{f_t^*(v_{ij})_{1\le i \le j \le g}}
  \end{tikzcd}
  $$
  where the bottom arrow is induced by the morphism of abelian groups
  \begin{align*}
    D^{-1} &\to {\Sym}_g(\ZZ)\\
         x &\mapsto (\Tr(r_ir_jx))_{1\le i,j\le g}\text{.}
  \end{align*}
\end{obs}

\section{Integral solution of the higher Ramanujan equations} \label{sec-intsol}

In this section, we define the \emph{higher Ramanujan equations} over $\mathcal{B}_g$ and $\mathcal{B}_F$, and we construct particular solutions of such differential equations defined over $\ZZ$. The definition of these solutions is based on Mumford's construction of degenerating families of abelian varieties, which we shall not recall in detail. Besides Mumford's original paper \cite{mumford72}, the reader may consult \cite{chai89} 2.3 and \cite{FC90} III as general references.

Our main theorems here, whose statement are purely algebraic, are immediate corollaries of their analytic counterparts to be proved in Section \ref{sec-analhre}. 

\subsection{Higher Ramanujan equations over $\mathcal{B}_g$}

Let $1\le i\le j\le g$, and $q_{ij}$ be formal variables. For any commutative ring $\Lambda$, we denote the ring of formal power series in the variables $q_{ij}$ with coefficients in $\Lambda$ by 
$$
\Lambda[\![q_{ij}]\!] \defeq \Lambda [\![q_{ij}\text{ ; } 1\le i\le j\le g ]\!]\text{.}
$$
We set \label{symb:Zqij}
$$
\Lambda (\!( q_{ij} )\!) \defeq \Lambda [\![ q_{ij}]\!][(\prod_{1\le i\le j\le g}q_{ij})^{-1}]\text{.}
$$

Recall that every $\Lambda$-derivation of $\Lambda[\![q_{ij}]\!]$ is continuous for the linear topology given by the ideal generated by the $q_{ij}$, and that $\Der_{\Lambda}(\Lambda[\![q_{ij}]\!])$ is freely generated by $\frac{\partial}{\partial q_{ij}}$. In particular, as each $q_{ij}$ is invertible in $\Lambda (\!(q_{ij})\!)$, the derivations \label{symb:thetaijalg}
$$
\theta_{ij}\defeq q_{ij}\frac{\partial}{\partial q_{ij}}\text{, }\ \ \ 1\le i\le j\le g
$$
of $\Lambda (\!(q_{ij})\!)$ form a basis of the $\Lambda (\!(q_{ij})\!)$-module $\Der_{\Lambda}(\Lambda (\!(q_{ij})\!))$.

\begin{defi}
 A \emph{solution of the higher Ramanujan equations over $\mathcal{B}_g$ defined over $\Lambda$} is a $\Lambda$-morphism (of Deligne-Mumford stacks over $\Lambda$)
 $$
 \hat{\varphi} : \Spec \Lambda (\!( q_{ij})\!) \to \mathcal{B}_{g,\Lambda}
 $$
 such that
 $$
 T\hat{\varphi}(\theta_{ij}) = \hat{\varphi}^*v_{ij}\text{, } \ \ \  1\le i\le j \le g\text{.}
 $$
\end{defi}

A morphism $\hat{\varphi}:\Spec \Lambda (\!( q_{ij})\!) \to \mathcal{B}_{g,\Lambda}$ as above corresponds to a principally polarized abelian scheme $(X,\lambda)$ over $\Lambda (\!( q_{ij})\!)$ endowed with a symplectic-Hodge basis $b$. Let $\nabla$ be the Gauss-Manin connection over $H^1_{\dR}(X/\Lambda(\!(q_{ij})\!))$.

\begin{prop}\label{prop-equivrameq}
 With the above notation, $\hat{\varphi}:\Spec \Lambda (\!( q_{ij})\!) \to \mathcal{B}_{g,\Lambda}$ is a solution of the higher Ramanujan equations over $\mathcal{B}_g$ defined over $\Lambda$ if and only if
 $$
 \nabla_{\theta_{ij}}b = b\left(\begin{array}{cc}
                              0 & 0 \\
                              \mathbf{E}^{ij} & 0
                             \end{array}\right)
 $$
 for every $1\le i\le j\le g$.
\end{prop}

\begin{proof}
 For any (formally) smooth scheme $U$ over $\Lambda$ and any object $(X,\lambda,b)_{/U}$ of $\mathcal{B}_g(U)$, with $b=(\omega_1,\ldots,\omega_g,\eta_1,\ldots,\eta_g)$, we may consider the $\mathcal{O}_U$-morphism
 \begin{align*}
 \rho:T_{U/\Lambda}&\to \Gamma^2(F^1(X/U)^{\vee})\oplus H^1_{\dR}(X/U)^{\oplus g}\\
   \theta &\mapsto (\kappa(\theta), \nabla_{\theta}\eta_1,\ldots,\nabla_{\theta}\eta_g)\text{.}
 \end{align*}
 This construction is compatible with base change in $U$; in particular, if $u: U \to \mathcal{B}_{g,\Lambda}$ is the morphism associated to $(X,\lambda,b)_{/U}$, we get a commutative diagram
 $$
 \begin{tikzcd}
  T_{U/\Lambda} \arrow{r}{Tu}\arrow{d}[swap]{\rho} & u^*T_{\mathcal{B}_{g,\Lambda}/\Lambda}\arrow{d}\\
  \Gamma^2(F^1(X/U)^{\vee})\oplus H^1_{\dR}(X/U)^{\oplus g} \arrow{r}{\sim}& u^*(\Gamma^2(\pi_{g,\Lambda}^*\mathcal{F}_{g,\Lambda}^{\vee})\oplus \pi_{g,\Lambda}^*\mathcal{H}_{g,\Lambda}^{\oplus g})
 \end{tikzcd}
 $$
 where the arrow on the right is the pullback by $u$ of the morphism
 $$
 (\kappa_{\pi_g},P): T_{\mathcal{B}_{g,\Lambda}/\Lambda}\to \Gamma^2(\pi_{g,\Lambda}^*\mathcal{F}_{g,\Lambda}^{\vee})\oplus \pi_{g,\Lambda}^*\mathcal{H}_{g,\Lambda}^{\oplus g}
 $$
 which identifies $T_{\mathcal{B}_{g,\Lambda}/\Lambda}$ with the subbundle $\Gamma^2(\pi_{g,\Lambda}^*\mathcal{F}_{g,\Lambda}^{\vee})\oplus \mathcal{S}_{g,\Lambda}$ of $\Gamma^2(\pi_{g,\Lambda}^*\mathcal{F}_{g,\Lambda}^{\vee})\oplus \pi_{g,\Lambda}^*\mathcal{H}_{g,\Lambda}^{\oplus g}$ by Theorem \ref{thm-vertbg} (cf. Paragraph \ref{ssec-hrvf}).
 
 By taking $u=\hat{\varphi} : \Spec \Lambda (\!( q_{ij})\!) \to \mathcal{B}_{g,\Lambda}$ in the above construction, we observe that $\hat{\varphi}$ is a solution of the higher Ramanujan equations if and only if
 $$
 \rho(\theta_{ij}) = \hat{\varphi}^*(\kappa_{\pi_g}(v_{ij}),P(v_{ij}))
 $$
 for every $1\le i\le j\le g$. By the definition of $\rho$, our statement now follows from  Proposition \ref{caracchamps}.
\end{proof}

\subsection{Integral solution of the higher Ramanujan equations; Siegel case}\label{subsec-mumfordsiegel}

Let $K\defeq \Frac \ZZ[\![q_{ij}]\!]$. Consider the ``period subgroup"
$$
Y\defeq \langle (q_{1j},\ldots,q_{gj}) \mid 1\le j \le g\rangle \subset \GG_m^g(K)\text{,}
$$
and let
$$
\phi: Y \to  \ZZ^g(\cong{\Hom}_{\textsf{GpSch}}(\GG_m^g,\GG_m))
$$
be the unique group isomorphism such that
$$
\phi(q_{1j},\ldots,q_{gj}) = \mathbf{e}_j\text{, }\ \ \ 1\le j \le g\text{.}
$$
Then, Mumford's construction \cite{mumford72} (cf. \cite{chai89} 2.3, \cite{FC90} V.1) canonically attaches to $(\GG_m^g,Y,\phi)$ a principally polarized semi-abelian scheme $(G,\lambda)$ over $\ZZ[\![q_{ij}]\!]$ of relative dimension $g$. The restriction of $(G,\lambda)$ to $\ZZ(\!(q_{ij})\!)$ is a principally polarized abelian scheme that we denote by $(\hat{X}_g,\hat{\lambda}_g)$. \label{symb:hatXg}

If we denote $\GG_m^g = \Spec \ZZ[t_1^{\pm 1},\ldots,t_g^{\pm 1}]$, then the Hodge subbundle $F^1(\hat{X}_g/\ZZ(\!(q_{ij})\!))$ is canonically trivialized by
$$
\hat{\omega}_k \defeq \frac{dt_k}{t_k}\text{, }\ \ \ 1\le k\le g\text{.}
$$

\begin{obs}
 For $g=1$, $\hat{X}_1$ is known as the \emph{Tate elliptic curve} over $\ZZ(\!(q)\!)$, and $\hat{\omega} = dt/t$ is its ``canonical differential form". See \cite{deligne75}, Paragraph 8, for an explicit algebraic equation of $\hat{X}_1$.
\end{obs}

\label{symb:hatphig}
\begin{theorem}\label{thm-intsolsiegel}
  Let $\nabla$ be the Gauss-Manin connection on $H^1_{\dR}(\hat{X}_g/\ZZ(\!(q_{ij})\!))$ and, for $1\le k \le g$, define
$$
\hat{\eta}_k \defeq \nabla_{\theta_{kk}}\hat{\omega}_k \in H^1_{\dR}(\hat{X}_g/\ZZ(\!(q_{ij})\!))\text{.}
$$
Then:
\begin{enumerate}
 \item The $2g$-uple $\hat{b}_g \defeq (\hat{\omega}_1,\ldots,\hat{\omega}_g,\hat{\eta}_1,\ldots,\hat{\eta}_g)$ is a symplectic-Hodge basis of $(\hat{X}_g,\hat{\lambda}_g)_{/\ZZ(\!(q_{ij})\!)}$.
 \item The morphism of Deligne-Mumford stacks
 $$
 \hat{\varphi}_g: \Spec \ZZ(\!(q_{ij})\!)\to \mathcal{B}_g
 $$
 given by $(\hat{X}_g,\hat{\lambda}_g, \hat{b}_g)_{/\ZZ(\!(q_{ij})\!)}$ is a solution of the higher Ramanujan equations over $\mathcal{B}_g$ defined over $\ZZ$.
\end{enumerate}\hfill $\blacksquare$
\end{theorem}

This result follows directly from its complex analytic counterpart (Theorem \ref{theoremsolution}); see Paragraph \ref{subsec-compatsiegel}.

\begin{obs}\label{rem-explicitcoords}
For concreteness, we have chosen to work with the ``coordinates'' $q_{ij}$ as above. We refer to \cite{FC90}, p. 138-139, for a discussion on how to generalize some of the above constructions to more general coordinate rings $\ZZ[\![S^2(\ZZ^g)\cap \sigma^{\vee}]\!]$ associated to a rational polyhedral cone $\sigma$ in the cone of positive definite symmetric bilinear forms on $\RR^g$.
\end{obs}

\begin{obs}
Note that Mumford's construction yields a semi-abelian scheme over $\Spec \ZZ[\![q_{ij}]\!]$ which only becomes an abelian scheme (so that it fits into our framework) after inverting $q_{ij}$. This explains why our solution $\hat{\varphi}_g$ of the higher Ramanujan equations is only defined over $\Spec \ZZ(\!(q_{ij})\!)$. See also Remark \ref{rem-conditionatinfinity}. 
\end{obs}

\subsection{Higher Ramanujan equations over $\mathcal{B}_F$}\label{subsec-hrebf}

From now on, for simplicity, we fix a $\ZZ$-basis $(x_1,\ldots,x_g)$ of $D^{-1}$ with each $x_i$ totally positive --- that is, $\sigma_j(x_i)> 0$ for every $1\le i ,j \le g$ ---, and we let $(r_1,\ldots,r_g)$ be its dual $\ZZ$-basis of $R$ with respect to the trace form.\label{symb:Zqri}

Let $q^{r_1},\ldots,q^{r_g}$ be formal variables. For any commutative ring $\Lambda$, we set
$$
\Lambda[\![q^{r_i}]\!] \defeq \Lambda[\![q^{r_1},\ldots,q^{r_g}]\!]
$$
and 
$$
\Lambda (\!(q^{r_i})\!) \defeq \Lambda[\![q^{r_1},\ldots,q^{r_g}]\!][(\prod_{i=1}^g q^{r_i})^{-1}]\text{.}
$$
For every $r \in R$, we denote
$$
q^r \defeq \prod_{i=1}^g(q^{r_i})^{\Tr(rx_i)} \in \Lambda (\!(q^{r_i})\!)\text{.} 
$$

As in the Siegel case, note that\label{symb:thetarialg}
$$
\theta^{r_i} \defeq q^{r_i}\frac{\partial}{\partial q^{r_i}}\text{, }\ \ \ 1\le i \le g
$$
form a basis of the $\Lambda(\!(q^{r_i})\!)$-module $\Der_{\Lambda}(\Lambda(\!(q^{r_i})\!))$. We consider the following isomorphism of $\Lambda(\!(q^{r_i})\!)$-modules:
\begin{align*}
  {\theta}_F: \Lambda(\!(q^{r_i})\!)\tensor D^{-1} &\to {\Der}_{\Lambda}(\Lambda(\!(q^{r_i})\!))\\
  1\tensor x &\mapsto \sum_{i=1}^g \Tr(r_ix)\theta^{r_i}\text{.}
\end{align*}

\begin{defi}
  A \emph{solution of the higher Ramanujan equations over $\mathcal{B}_F$ defined over $\Lambda$} is a $\Lambda$-morphism of (Deligne-Mumford stacks over $\Lambda$)
  \begin{align*}
   \hat{\varphi}: \Spec \Lambda(\!(q^{r_i})\!) \to \mathcal{B}_{F,\Lambda}
  \end{align*}
  such that
  \begin{align*}
   T\hat{\varphi}\circ \theta_F = \hat{\varphi}^*v_F\text{,}
  \end{align*}
  that is, such that the diagram
  $$
  \begin{tikzcd}
    \Lambda(\!(q^{r_i})\!)\tensor D^{-1} \arrow{r}{\theta_F}\arrow{d}[swap]{\cong} & \Der_{\Lambda}(\Lambda(\!(q^{r_i})\!)) \arrow{d}{T\hat{\varphi}}\\
    \hat{\varphi}^*(\mathcal{O}_{\mathcal{B}_{F,\Lambda}}\tensor D^{-1})\arrow{r}[swap]{v_F} & \hat{\varphi}^*T_{\mathcal{B}_{F,\Lambda}/\Lambda}
  \end{tikzcd}
  $$
  commutes.
\end{defi}

More concretely, if we denote $v_F(1\tensor x_i)\eqdef v^{r_i} \in \Gamma (\mathcal{B}_F, T_{\mathcal{B}_F/\ZZ})$ for every $1\le i \le g$, then $\hat{\varphi}$ is a solution of the higher Ramanujan equations over $\mathcal{B}_F$ if and only if it satisfies
$$
T\hat{\varphi}\left(\theta^{r_i} \right) = \hat{\varphi}^*v^{r_i}\text{, } \ \ \ 1\le i \le g\text{.}
$$

A morphism $\hat{\varphi}:\Spec \Lambda (\!( q^{r_i})\!) \to \mathcal{B}_{F,\Lambda}$ as above corresponds to a principally polarized abelian scheme with $R$-multiplication $(X,\lambda,m)$ over $\Lambda (\!( q^{r_i})\!)$ endowed with a symplectic-Hodge basis $b$. Let $\nabla$ be the Gauss-Manin connection over $H^1_{\dR}(X/\Lambda(\!(q^{r_i})\!))$.

\begin{prop}[cf. Proposition \ref{prop-equivrameq}]\label{prop-equivrameqrm}
 With the above notation, $\hat{\varphi}:\Spec \Lambda (\!( q^{r_i})\!) \to \mathcal{B}_{F,\Lambda}$ is a solution of the higher Ramanujan equations over $\mathcal{B}_F$ defined over $\Lambda$ if and only if
 $$
 \nabla_{\theta_F}b = b\left(\begin{array}{cc}
                              0 & 0 \\
                              1 & 0
                             \end{array}\right)\text{.}
                           $$
                           \hfill $\blacksquare$
\end{prop}

By considering the $\mathcal{O}_{\mathcal{B}_{F,\Lambda,\et}}$-\emph{isomorphism} (cf. Theorem \ref{thm-isomvertbk} and Paragraph \ref{ksiso})
$$
(\kappa_{\pi_F},P): T_{\mathcal{B}_{F,\Lambda}}\stackrel{\sim}{\to} (\Gamma_{\mathcal{O}_{\mathcal{B}_{F,\Lambda,\et}}\tensor R}^2(\pi_{F,\Lambda}^*\mathcal{F}_{F,\Lambda}^*)\tensor_R D^{-1}) \oplus (\pi_{F,\Lambda}^*\mathcal{H}_{F,\Lambda}\tensor_R D)
$$
 the proof of the above proposition is analogous to that of Proposition \ref{prop-equivrameq}.

\subsection{Integral solution of the higher Ramanujan equations; Hilbert-Blumenthal case}\label{subsec-integralsolhb}

Let $K\defeq \Frac \ZZ[\![q^{r_i}]\!]$. Consider the split torus $\mathbf{G}_m\tensor D^{-1}$ over $\Spec \ZZ$ defined by
$$
(\mathbf{G}_m\tensor D^{-1})(\Lambda) = \Lambda^{\times}\tensor_{\ZZ}D^{-1}
$$
for any commutative ring $\Lambda$. Note that the $\ZZ$-basis $(x_1,\ldots,x_g)$ of $D^{-1}$ induces an isomorphism of group schemes
\begin{align}\label{eq-isomsplitorus}
\mathbf{G}_m\tensor D^{-1} \stackrel{\sim}{\to} \GG_m^g
\end{align}
  given on points by
  $$
  t\tensor x \mapsto (t^{\Tr(r_1x)},\ldots,t^{\Tr(r_gx)})\text{.}
  $$

To define the period subgroup, consider the morphism of abelian groups
  \begin{align*}
    R &\to K^{\times}\\
    r &\mapsto q^r\text{.}
  \end{align*}
  By Remark \ref{rem-dualityrelation}, there exists a unique $R$-linear morphism
  $$
   \varpi: R \to K^{\times}\tensor D^{-1}
   $$
   such that $\Tr(\varpi(r))=q^r$ for every $r \in R$. Set
   $$
Y\defeq \varpi(R) \subset (\mathbf{G}_m\tensor D^{-1})(K)\text{.}
   $$
   Observe that, since $\varpi$ is injective, it induces an isomorphism of $R$ onto $Y$. We let
   $$
\phi \defeq \varpi^{-1}: Y \stackrel{\sim}{\to} R \left(\cong {\Hom}_{\textsf{GpSch}} (\GG_m\tensor D^{-1},\GG_m) \right)\text{.}
$$
Then Mumford's construction \cite{mumford72} canonically attaches to $(\mathbf{G}_m\tensor D^{-1},Y,\phi)$ a principally polarized semi-abelian scheme $(G,\lambda)$ over $\ZZ[\![q^{r_i}]\!]$ of relative dimension $g$. The restriction of $(G,\lambda)$ to $\ZZ(\!(q^{r_i})\!)$ is a principally polarized abelian scheme that we denote by $(\hat{X}_F,\hat{\lambda}_F)$. Moreover, the canonical action of $R$ on $\mathbf{G}_m\tensor D^{-1}$, which preserves the period subgroup $Y$ and is compatible with the polarization $\phi$, induces an $R$-multiplication $\hat{m}_F: R \to \End_{\ZZ(\!(q^{r_i})\!)}(\hat{X}_F)^{\hat{\lambda}_F}$; we thus obtain a principally polarized abelian scheme with $R$-multiplication $(\hat{X}_F,\hat{\lambda}_F,\hat{m}_F)$ over $\ZZ(\!(q^{r_i})\!)$. \label{symb:hatXF}

Since $\Lie \hat{X}_F$ is canonically isomorphic to $\Lie (\mathbf{G}_{m,\ZZ(\!(q^{r_i})\!)}\tensor D^{-1}) \cong \ZZ(\!(q^{r_i})\!) \tensor D^{-1}$, we obtain by duality a canonical isomorphism of $\ZZ(\!(q^{r_i})\!)\tensor R$-modules
$$
\ZZ(\!(q^{r_i})\!) \tensor R \cong F^1(\hat{X}_F/\ZZ(\!(q^{r_i})\!))\text{;}
$$
we let $\hat{\omega}_F$ be the $\ZZ(\!(q^{r_i})\!)\tensor R$-generator of $F^1(\hat{X}_F/\ZZ(\!(q^{r_i})\!))$ corresponding to the above trivialization.

\begin{obs}\label{rem-formulaomegaf}
  If we identify $\GG_m\tensor D^{-1} \stackrel{\sim}{\to} \Spec \ZZ[(t^{r_1})^{\pm 1},\ldots, (t^{r_g})^{\pm 1}]$ via (\ref{eq-isomsplitorus}), then the canonical $\ZZ(\!(q^{r_i})\!) \tensor R$-trivialization of $F^1(\hat{X}_F/\ZZ(\!(q^{r_i})\!))$ is given by
  \begin{align*}
    \ZZ(\!(q^{r_i})\!) \tensor R &\stackrel{\sim}{\to} F^1(\hat{X}_F/\ZZ(\!(q^{r_i})\!)) \\
    1\tensor r &\mapsto \sum_{i=1}^g\Tr(rx_i)\frac{dt^{r_i}}{t^{r_i}}\text{,} 
  \end{align*}
  so that 
  $$
\hat{\omega}_F = \sum_{i=1}^g\Tr(x_i)\frac{dt^{r_i}}{t^{r_i}}\text{.}
  $$
\end{obs}

\label{symb:hatphiF}
\begin{theorem}\label{thm-intsolrm}
Let $\nabla$ be the Gauss-Manin connection on $H^1_{\dR}(\hat{X}_F/\ZZ(\!(q^{r_i})\!))$ and denote
$$
\hat{\eta}_F \defeq \nabla_{\theta_F}\hat{\omega}_F \in H^1_{\dR}(\hat{X}_F/\ZZ(\!(q^{r_i})\!))\tensor D\text{.}
$$
Then:
\begin{enumerate}
 \item The couple $\hat{b}_F \defeq (\hat{\omega}_F,\hat{\eta}_F)$ is a symplectic-Hodge basis of $(\hat{X}_F,\hat{\lambda}_F,\hat{m}_F)_{/\ZZ(\!(q^{r_i})\!)}$.
 \item The morphism of Deligne-Mumford stacks
 $$
 \hat{\varphi}_F: \Spec \ZZ(\!(q^{r_i})\!)\to \mathcal{B}_F
 $$
 given by $(\hat{X}_F,\hat{\lambda}_F,\hat{m}_F, \hat{b}_F)_{/\ZZ(\!(q^{r_i})\!)}$ is a solution of the higher Ramanujan equations over $\mathcal{B}_F$ defined over $\ZZ$.
\end{enumerate}\hfill $\blacksquare$
\end{theorem}

As in the Siegel case, this result follows directly from its complex analytic counterpart (Theorem \ref{thm-defiphik})); see Paragraph \ref{subsec-compatibilityhb}.

\begin{obs}
Here again, we have chosen to work with explicit ``coordinates'' $q^{r_i}$ induced by a fixed $\ZZ$-basis of $D^{-1}$ (cf. Remark \ref{rem-explicitcoords}). We refer to \cite{goren02} 5.2 for an exposition on how to work with more general coordinate rings.
\end{obs}

\begin{obs}
  We have the following compatibility between $\hat{\varphi}_F$ and $\hat{\varphi}_g$. Let $(x_1,\ldots,x_g)$ be as above, and $f_t: \mathcal{B}_F \to \mathcal{B}_g$ be the corresponding morphism as defined in Remark \ref{rem-relationbfbg}. Define a morphism
  $$
  \hat{h}_t: \Spec \ZZ(\!(q^{r_i})\!) \to \Spec \ZZ(\!(q_{ij})\!)
  $$
  by
  $$
\hat{h}_t^*(q_{ij}) = q^{r_ir_j}\text{.}
$$
Then, the diagram
$$
\begin{tikzcd}
  \Spec \ZZ(\!(q^{r_i})\!) \arrow{r}{\hat{\varphi}_F}\arrow{d}[swap]{\hat{h}_t}& \mathcal{B}_F \arrow{d}{f_t}\\
  \Spec \ZZ(\!(q_{ij})\!) \arrow{r}[swap]{\hat{\varphi}_g}& \mathcal{B}_g
\end{tikzcd}
$$
commutes. This can be checked directly using the above constructions; it also follows from the corresponding complex analytic statement (see Remark \ref{rem-relationphianal}).
\end{obs}

\section{Representability of $\mathcal{B}_g$ and $\mathcal{B}_F$ by a scheme} \label{representability}

It is easy to see that if $S$ is a scheme over $\mathbf{F}_2$, then $\mathcal{B}_g\times_{\ZZ} S\to S$ is not representable. Indeed, if $(X,\lambda,b)_{/U}$ is an object of $\mathcal{B}_g$ lying over a scheme $U$ over $\mathbf{F}_2$, then the involution $[-1]: P \mapsto -P$ on $X$ defines a non-trivial automorphism $[-1]_{/\id_U}: (X,\lambda)_{/U} \to (X,\lambda)_{/U}$ in $\mathcal{A}_g(U)$ such that
\begin{align*}
[-1]^*b = -b = b\text{,}
\end{align*}
thus a non-trivial automorphism of $(X,\lambda,b)_{/U}$ in $\mathcal{B}_g(U)$. This same argument applies to $\mathcal{B}_F$.

For any commutative ring $\Lambda$, let us denote $\mathcal{B}_{g,\Lambda} \defeq \mathcal{B}_g\tensor_{\ZZ}\Lambda$ (resp. $\mathcal{B}_{F,\Lambda} \defeq \mathcal{B}_F\tensor_{\ZZ}\Lambda$). In this section we prove the following theorem.\label{symb:straightBgBF}

\begin{theorem} \label{repr}
The stack $\mathcal{B}_{g,\ZZ[1/2]}\to \Spec \ZZ[1/2]$ (resp. $\mathcal{B}_{F,\ZZ[1/2]}\to \Spec \ZZ[1/2]$) is representable by a smooth quasi-affine scheme $B_g$ (resp. $B_F$) over $\ZZ[1/2]$ of relative dimension $2g^2+g$ (resp. $3g$).
\end{theorem}

For the sake of concision, we shall only treat in detail the case of $\mathcal{B}_g$; there should be no difficulty in translating our arguments to obtain the analogous statement for $\mathcal{B}_F$ (see Remark \ref{rem-reprbk}).

\subsection{Representability by an algebraic space}

Let $\Lambda$ be a commutative ring. The following terminology has been borrowed from \cite{KM85} 4.4.

\begin{defi} \label{defrigid}
We say that the functor $\underline{B}_g$ (cf. Paragraph \ref{defbg}) is \emph{rigid} over $\Lambda$ if, for every $\Lambda$-scheme $U$, and every object $(X,\lambda)$ of $\mathcal{A}_g$ lying over $U$, the action of $\Aut_U(X,\lambda)$ on $\underline{B}_g((X,\lambda)_{/U})$ is free. 
\end{defi}

Note that $\underline{B}_g$ is rigid over $\Lambda$ if and only if the fiber categories of $\mathcal{B}_{g,\Lambda} \to \Spec \Lambda$ are discrete. As $\mathcal{B}_g$ is a Deligne-Mumford stack over $\Spec \ZZ$, this amounts to saying that $\mathcal{B}_{g,\Lambda} \to \Spec \Lambda$ is an algebraic space over $\Spec \Lambda$ (cf. \ref{stacks}).

\begin{lemma} \label{rig0}
Let $k$ be a field of characteristic 0. Then $\underline{B}_g$ is rigid over $k$.
\end{lemma}

\begin{proof}
Let $(X,\lambda,b)$ be an object of $\mathcal{B}_g$ lying over $k$ and $\varphi :X \to X$ be a $k$-automorphism of $(X,\lambda)$ such that $\varphi^*b = b$; we must show that $\varphi = \id_X$.

We claim that it is sufficient to treat the case $k=\CC$. In fact, as $X$ is of finite type over $k$, by ``elimination of Noetherian hypothesis'' (cf. \cite{EGAIV3} 8.8, 8.9, 8.10, 12.2.1, and \cite{EGAIV4} 17.7.9), there exists a subfield $k_0$ of $k$, of finite type over $\QQ$, and a principally polarized abelian variety $(X_0,\lambda_0)$ over $k_0$ endowed with a symplectic-Hodge basis $b_0$ and  a $k_0$-automorphism $\varphi_0$ of $(X_0,\lambda_0)$ satisfying $\varphi_0^*b_0=b_0$, such that $(X,\lambda,b)$ (resp. $\varphi$) is obtained from $(X_0,\lambda_0,b_0)$ (resp. $\varphi_0$) by the base change $\Spec k \to \Spec k_0$. After fixing an embedding of $k_0$ in $\CC$, we finally remark that if $\varphi_{0,\CC}$ is the identity over $X_0 \tensor_{k_0}\CC$, then the same holds for $\varphi_0$, and thus also for $\varphi$.

Let then $k=\CC$. It is sufficient to prove that the induced automorphism of complex Lie groups $\varphi^{\an} : X^{\an} \to X^{\an}$ is the identity. As $X^{\an}$ is a complex torus, the exponential $\exp : \Lie X \to X^{\an}$ is a surjective morphism of complex Lie groups. Therefore, it follows from the commutative diagram
%$$
%\raisebox{-0.5\height}{\includegraphics{rameq1-d4.pdf}}
%$$
 $$
 \begin{tikzcd}[column sep=large]
 \Lie X \arrow{r}{\Lie \varphi} \arrow{d}[swap]{\exp} & \Lie X\arrow{d}{\exp}\\
 X^{\an} \arrow{r}[swap]{\varphi^{\an}} & X^{\an}
 \end{tikzcd}
 $$
that it sufficient to prove that $\Lie \varphi = \id_{\Lie X}$. Now, if $\varphi$ preserves symplectic-Hodge basis of $(X,\lambda)$, then in particular the $\CC$-linear map $\varphi^* : H^0(X,\Omega^1_{X/\CC}) \to H^0(X,\Omega^1_{X/\CC})$ is the identity, and thus its dual $\Lie \varphi : \Lie X \to \Lie X$ is also the identity. 
\end{proof}

We now treat the case of positive characteristic. Let us briefly recall some notions in Dieudonné theory and its relation with abelian varieties.

Let $k$ be a perfect field of characteristic $p>0$. We denote by $W(k)$ the ring of Witt vectors over $k$, and by $\sigma$ the unique ring automorphism of $W(k)$ lifting the absolute Frobenius $x \mapsto x^p$ of $k$. We can then define a $W(k)$-algebra $D(k)$ generated by elements $F$ and $V$ subject to the relations
\begin{align*}
FV = VF = p\text{, }\ \ \ Fx = \sigma(x)F\text{,  }\ \ \ 
xV=V\sigma(x)
\end{align*} 
for any $x\in W(k)$.

The theory of Dieudonné (cf. \cite{oda69} Definition 3.12) provides an additive contravariant functor
\begin{align} \label{functordieudonne}
G \mapsto M(G)
\end{align} 
from the category of commutative finite $k$-group schemes of $p$-power order to the category of left $D(k)$-modules. This functor is shown to be faithful and its essential image is given by the category of left $D(k)$-modules of finite $W(k)$-length: $M(G)$ is of $W(k)$-length $r$ if and only if $G$ is of order $p^r$ (\cite{oda69} Corollary 3.16).

Let $X$ be an abelian variety over $k$ and consider the $k$-vector space $H^1_{\dR}(X/k)$ as a $W(k)$-module via the canonical map $W(k) \to k$. Then one can endow $H^1_{\dR}(X/k)$ with the structure of a $D(k)$-module, the action of $F$ (resp. $V$) being induced by the relative Frobenius on $X$ (resp. the Cartier operator in degree 1); we refer to \cite{oda69} Definition 5.3 and Definition 5.6 for further details. This construction is functorial in the sense that for any morphism $\varphi:X \to Y$ of abelian varieties over $k$, if we endow $H^1_{\dR}(X/k)$ and $H^1_{\dR}(Y/k)$ with the preceding $D(k)$-module structure, then the induced morphism on de Rham cohomology $\varphi^*: H^1_{\dR}(Y/k) \to H^1_{\dR}(X/k)$ is $D(k)$-linear.

In the next statement, for any abelian variety $X$ over $k$, we regard $H^1_{\dR}(X/k)$ with the above $D(k)$-module structure, and we denote its $p$-torsion subscheme by $X[p]$. Note that $X[p]$ is a commutative finite $k$-group scheme of order $p^{2\dim X}$.

\begin{theorem}[Oda, \cite{oda69} Corollary 5.11\footnote{Oda's theorem can be seen nowadays as part of the much more general Grothendieck-Messing theory; see for instance the introduction of B. Mazur, W. Messing, \emph{Universal Extensions and One Dimensional Crystalline Cohomology}, Lecture Notes in Mathematics 370, Springer-Verlag.}] \label{odathm}
The contravariant functors $X \mapsto M(X[p])$ and $X\mapsto H^1_{\dR}(X/k)$ from the category of abelian varieties over $k$ to the category of ($p$-torsion) $D(k)$-modules of finite $W(k)$-length are naturally equivalent.
\end{theorem} 

\begin{lemma} \label{rigp}
Let $k$ be a perfect field of characteristic $p>2$. Then $\underline{B}_g$ is rigid over $k$.
\end{lemma}

\begin{proof}
Let $(X,\lambda)$ be a principally polarized abelian variety over $k$ of dimension $g$ and $\varphi :X \to X$ be a $k$-automorphism of $(X,\lambda)$.

If $\varphi$ preserves a symplectic-Hodge basis of $(X,\lambda)_{/k}$, then in particular $\varphi^* : H^1_{\dR}(X/k) \to H^1_{\dR}(X/k)$ is the identity; a fortiori, $\varphi$ induces the identity on $H^1_{\dR}(X/k)$ regarded as a $D(k)$-module. Then, by Theorem \ref{odathm}, $\varphi$ induces the identity on the $D(k)$-module $M(X[p])$. As the functor $G\mapsto M(G)$ in (\ref{functordieudonne}) is faithful, $\varphi$ restricts to the identity on the $p$-torsion subscheme $X[p]$ of $X$. As $\varphi$ preserves, in addition, the polarization $\lambda$ on $X$, and since $p\ge 3$, then necessarily $\varphi=\id_X$ by a lemma of Serre (cf. \cite{mumford70} IV.21, Theorem 5).   
\end{proof}

Recall the following version of the classical ``rigidity lemma'' for abelian schemes which follows from the arguments in the proof of Proposition 6.1 in \cite{GIT94}.

\begin{lemma} \label{rigidite}
Let $A$ be a local Artinian ring, and $X$ be an abelian scheme over $\Spec A$. If an abelian scheme endomorphism $\varphi \in \End_A(X)$ restricts to the identity on the closed fiber of $X \to \Spec A$, then $\varphi=\id_X$.\hfill $\blacksquare$
\end{lemma}

% \begin{proof}
% Let $p:X\to \Spec A$ denote the structural morphism and $e: \Spec A \to X$ the identity section. If $\varphi$ is as in the statement, then $f \defeq \varphi - \id$ restricts to the ``zero morphism'' on the closed fiber. As $\Spec A$ consisting in only one point, we obtain the set-theoretical identity $f=e \circ p$. To finish the proof, we must show that $f=e\circ p$ as morphisms of schemes.

% Let us momentarily denote the set-theoretical identity section $\Spec A \to X$ by $s$. Now, since $p$ is proper, smooth and with geometrically connected fibers, the induced morphism of sheaves $p^{\#}:\mathcal{O}_{\Spec A} \to p_*\mathcal{O}_X$ is an isomorphism, and we conclude that we can endow $s$ with a unique morphism of locally ringed spaces $s^{\#} : \mathcal{O}_X \to s_* \mathcal{O}_{\Spec A}$ such that $f=s\circ p$ as morphisms of schemes. Finally, we remark that
% \begin{align*}
% s = s \circ p \circ e = f\circ e = e
% \end{align*}  
% as morphisms of schemes.
% \end{proof}

\begin{prop} \label{rigid}
The functor $\underline{B}_g$ is rigid over $\ZZ[1/2]$.
\end{prop}

\begin{proof}
Let $U$ be a $\ZZ[1/2]$-scheme, $(X,\lambda)$ be an object of $\mathcal{A}_{g}$ lying over $U$, and $\varphi$ be an automorphism of $(X,\lambda)$ in the fiber category $\mathcal{A}_g(U)$ preserving an element $b$ of $\underline{B}_g(X,\lambda)$. We must show that $\varphi=\id_X$. This being a local property over $U$, we can assume that $U$ is affine.

Suppose that $U$ is Noetherian. By Lemmas \ref{rig0} and \ref{rigp}, for every geometric point $\overline{u}$ of $U$, we have $\varphi_{X_{\overline{u}}}=\id_{X_{\overline{u}}}$. Let $Z$ be the closed subscheme of $U$ where $\varphi = \id$. Then $Z$ contains every closed point of $U$. By  Lemma \ref{rigidite}, and Krull's intersection theorem, $Z$ is also an open subscheme of $U$; hence $Z=U$, which amounts to saying that $\varphi = \id_X$.

In general, by ``elimination of Noetherian hypothesis'' (cf. \cite{EGAIV3}, 8.8, 8.9, 8.10, 12.2.1, and \cite{EGAIV4}, 17.7.9), there exists an affine Noetherian scheme $U_0$ under $U$, and a principally polarized abelian scheme $(X_0,\lambda_0)$ over $U_0$ endowed with a symplectic-Hodge basis $b_0$, and with an $U_0$-automorphism $\varphi_0$, such that $\varphi_0^*b_0=b_0$, and $(X,\lambda)$ (resp. $b$, resp. $\varphi$) is deduced from $(X_0,\lambda_0)$ (resp. $b_0$, resp. $\varphi_0$) by the base change $U\to U_0$. The preceding paragraph shows that $\varphi_0=\id_{X_0}$, hence $\varphi = \id_X$.
\end{proof}

\subsection{Representability of $\mathcal{B}_{g,\ZZ[1/2]}$ by a quasi-projective scheme $B_g$}

We briefly recollect some facts on quotients of schemes by actions of finite groups. 

Let $S$ be a scheme and $\Gamma$ be a finite constant group scheme over $S$, that is, an $S$-group scheme associated to a finite abstract group $|\Gamma|$. 

For any $S$-scheme $X$, an $S$-action of $\Gamma$ on $X$ is equivalent to a morphism of groups $|\Gamma| \to \Aut_S(X)$. If $X$ is an $S$-scheme, we say that an action of $\Gamma$ on $X$ is \emph{free} if the action of $\Gamma(U)$ on $X(U)$ is free for any $S$-scheme $U$.

The next lemma easily follows from \cite{SGA103} V and \cite{knutson71} IV.1.

\begin{lemma} \label{finiteaction}
Let $S$ be an affine Noetherian scheme and $X$ be a quasi-projective $S$-scheme equipped with an $S$-action of a finite constant group scheme $\Gamma$ over $S$. Then
\begin{enumerate}
     \item Up to isomorphism, there exists a unique quasi-projective $S$-scheme $Y$ together with a $\Gamma$-invariant finite surjective morphism $p:X \to Y$ such that the natural morphism of sheaves of rings over $Y$
\begin{align*}
\mathcal{O}_Y\to (p_*\mathcal{O}_X)^{|\Gamma|}
\end{align*}
is an isomorphism.  If $X$ is affine (resp. quasi-affine) over $S$, so is $Y$. We denote $Y\eqdef X/\Gamma$.%In particular, $p$ is a \emph{quotient} of $X$ by $\Gamma$ in the category of $S$-schemes, that is, for any $S$-scheme $Z$, every $\Gamma$-invariant morphism $X \to Z$ factors through $p$.
     %\item The morphism $p$ is finite and surjective. 
     % \item $Y$ is quasi-projective over $S$.
     \item If moreover the action of $\Gamma$ on $X$ is free, then $p$ is étale and
\begin{align*}
\Gamma \times_S X &\to X \times_Y X\\
  (\gamma,x) &\mapsto (x,\gamma\cdot x)
\end{align*}
is an isomorphism.
\end{enumerate}
\end{lemma}

\begin{obs} \label{remfiniteaction}
Part (2) in the above lemma implies that,  when the action of $\Gamma$ on $X$ is free, then the stacky quotient $[X/\Gamma]$ (cf. \cite{olsson16} Example 8.1.12) is representable by the scheme $X/\Gamma$.
\end{obs}
 
% \begin{proof}
% Since $X$ is quasi-projective over $S$, by a homogeneous version of the prime-avoidance lemma, any finite set of points in $X$ is contained in an affine open subscheme of $X$. Therefore, 1 is a direct consequence of \cite{SGA103}, V.1.8.

%  It follows from \cite{SGA103} V.1.3 that $p$ is surjective. Moreover, as any quasi-projective morphism is of finite type, $p$ is finite by \cite{SGA103}, V.1.5. 

%  Quasi-projectiveness of $Y$ is a consequence of the following general fact: if $R$ is a ring, $n \ge 1$ is an integer, and $G$ is a finite abstract group acting linearly on $A=R[x_0,\ldots,x_n]$, then $\Proj A^G$ is projective over $R$. We refer to \cite{knutson71} IV.1 for a proof of 3.

% Finally, to prove 4, we observe that if $\Gamma$ acts freely on $X$, then for any $x\in X$, the inertia group at $x$ (that is, the subgroup of $|\Gamma|$ consisting of those elements that fix $x$ and induces the identity on the residue field $k(x)$) is trivial. Then, by \cite{SGA103} V.2.3, $p$ is étale. The last assertion is obtained by \cite{SGA103} V.2.6.  
% \end{proof}

For clarity, we split the proof of Theorem \ref{repr} for $\mathcal{B}_g$ in two parts; see Remark \ref{rem-reprbk} for $\mathcal{B}_F$.

\begin{proof}[Proof of Theorem \ref{repr}, part 1]
Recall from \cite{GIT94} Theorem 7.9 (cf. \cite{olsson12} proof of Theorem 2.1.11) that there exists a quasi-projective scheme $A$ over $\ZZ[1/2]$ endowed with an action by the constant finite group scheme $\Gamma$ over $\ZZ[1/2]$ given by $|\Gamma|= \GL_{g}(\ZZ/4\ZZ)$, and with a surjective étale morphism $A \to \mathcal{A}_{g,\ZZ[1/2]}$ inducing an isomorphism of the stacky quotient $[A/\Gamma]$ with $\mathcal{A}_{g,\ZZ[1/2]}$; namely, $A$ is the fine moduli scheme classifying of principally polarized abelian schemes with a full level 4 structure.

As the morphism of Deligne-Mumford stacks over $\Spec \ZZ$
\begin{align*}
\pi_g : \mathcal{B}_g \to \mathcal{A}_g
\end{align*}
is representable by smooth affine schemes (Remark \ref{relrepr}), the fiber product
\begin{align*}
A\times_{\mathcal{A}_{g,\ZZ[1/2]}}\mathcal{B}_{g,\ZZ[1/2]}\to A
\end{align*}
is representable by a smooth affine scheme $B$ over $A$. In particular, $B$ is affine and of finite type over $A$. Since $A$ is quasi-projective over $\ZZ[1/2]$, it follows that $B$ is a quasi-projective $\ZZ[1/2]$-scheme.

The action of $\Gamma$ on $A$ naturally induces an action of $\Gamma$ on the fiber product $B$; as $\mathcal{B}_{g,\ZZ[1/2]}$ is an algebraic space by Proposition \ref{rigid} (cf. remark following Definition \ref{defrigid}), this action is free. Moreover, by the compatibility of quotients of stacks by group actions with base change (cf. \cite{romagny05} Proposition 2.6), the second projection $B \to \mathcal{B}_{g,\ZZ[1/2]}$ induces an isomorphism of the stacky quotient $[B/\Gamma]$ with $\mathcal{B}_{g,\ZZ[1/2]}$.  Finally, by Lemma \ref{finiteaction} and Remark \ref{remfiniteaction}, we conclude that $\mathcal{B}_{g,\ZZ[1/2]}$ is representable by the quasi-projective $\ZZ[1/2]$-scheme $B/\Gamma$.
\end{proof}

\subsection{$B_g$ is quasi-affine over $\ZZ[1/2]$}

Our proof that $B_g$ is quasi-affine over $\ZZ[1/2]$ is based on the following elementary fact from algebraic geometry.

\begin{lemma}
 Let $S$ be an affine Noetherian scheme, $X$ be a separated $S$-scheme of finite type, and $\mathcal{L}$ be an \emph{ample} or \emph{anti-ample} (i.e., the dual $\mathcal{L}^{\vee}$ is ample) line bundle over $X$. Let  $T(\mathcal{L})\to X$ be the $\GG_{m,S}$-torsor associated to $\mathcal{L}$. Then $T(\mathcal{L})$ is a quasi-affine $S$-scheme.
\end{lemma}

\begin{proof}
 Assume first that $\mathcal{L}^{\vee}$ is very ample over $S$. Then there exists $n\in \NN$ and an $S$-immersion $i: X \to \PP^n_S=\Proj \mathcal{O}_S(S)[X_0,\ldots,X_n]$ such that $\mathcal{L}^{\vee}=i^*\mathcal{O}_{\PP^n_S}(1)$. Let $\varphi^j = i^*X_j\in \Gamma(X,\mathcal{L}^{\vee})$ for $0\le j\le n$, and denote by $p:T(\mathcal{L})\to X$ the canonical projection. The morphism of $S$-schemes
 \begin{align*}
 i_{\mathcal{L}}:  T(\mathcal{L}) &\to \AA_S^{n+1}\setminus\{0\} = T(\mathcal{O}_{\PP^n_S}(-1))\\
  \ell &\mapsto (\varphi^0_{p(\ell)}(\ell),\ldots ,\varphi^n_{p(\ell)}(\ell))
 \end{align*}
is an immersion, since it fits into the Cartesian square
 $$
 \begin{tikzcd}
 T(\mathcal{L}) \arrow{r}{i_{\mathcal{L}}}\arrow{d}[swap]{p} & \AA_S^{n+1}\setminus\{0\}\arrow{d}\\
  X \arrow{r}[swap]{i} &\PP^n_S \arrow[draw=none]{ul}[description]{\square}
 \end{tikzcd}
 $$
 Thus $T(\mathcal{L})$ is a quasi-affine $S$-scheme. 
 
 If $\mathcal{L}^{\vee}$ is only ample, then we consider some very ample tensor power $(\mathcal{L}^{\vee})^{\tensor k} = (\mathcal{L}^{\tensor k})^{\vee}$ of $\mathcal{L}^{\vee}$. Since the $k$-th power map $T(\mathcal{L})\to T(\mathcal{L}^{\tensor k})$ is a finite morphism of $S$-schemes, and $T(\mathcal{L}^{\tensor k})$ is quasi-affine over $S$ by the above reasoning, $T(\mathcal{L})$ is also quasi-affine over $S$.
 
 If $\mathcal{L}$ is ample, then $T(\mathcal{L}^{\vee})$ is a quasi-affine $S$-scheme.  By duality, $T(\mathcal{L})$ is isomorphic to $T(\mathcal{L}^{\vee})$ as an $S$-scheme, thus $T(\mathcal{L})$ is quasi-affine over $S$.
\end{proof}

To conclude, we apply the above lemma and the fact that the determinant of the Hodge bundle on a fine moduli space of principally polarized abelian varieties with level structure is ample (see \cite{FC90} or \cite{lan13}):

\begin{proof}[Proof of Theorem \ref{repr}, part 2] 
 Let $T(\det \mathcal{F}_g)$ be the category fibered in groupoids over $\Spec \ZZ$ whose objects over a scheme $U$ are triples $(X,\lambda,t)$, where $(X,\lambda)$ is a principally polarized abelian scheme over $U$ of relative dimension $g$, and $t$ is a trivialization of the line bundle $\det F^1(X/U)$ over $U$ --- in other words, $T(\det \mathcal{F}_g)$ is the $\GG_m$-torsor associated to the determinant of the universal Hodge bundle $\mathcal{F}_g$ over $\mathcal{A}_g$. Then $\pi_g:\mathcal{B}_g\to \mathcal{A}_g$ factors through the forgetful functor $T(\det\mathcal{F}_g)\to \mathcal{A}_g$ via
  $$
  f:\mathcal{B}_g \to T(\det\mathcal{F}_g)\text{,}
  $$
  given by $(X,\lambda,(\omega_1,\ldots,\omega_g,\eta_1,\ldots,\eta_g))_{/U}\mapsto (X,\lambda,\omega_1\wedge \cdots \wedge \omega_g)_{/U}$.
  
  We keep the notation of the first part of this proof. Let $(X,\lambda)$ the principally polarized abelian scheme over $A$ corresponding to the finite étale covering $A\to \mathcal{A}_{g,\ZZ[1,2]} \cong [A/\Gamma]$, then it follows from \cite{FC90} Theorem V.2.5 (cf. \cite{lan13} Theorem 7.2.4.1 (2)) that  $\det F^1(X/A)$ is an ample line bundle over $A$. By the above lemma, $T(\det F^1(X/A))$ is a quasi-affine $\ZZ[1/2]$-scheme.
  
  Consider now the following commutative diagram
  $$
  \begin{tikzcd}
   B \arrow{rr}\arrow{d} & & B_g\arrow{d}{f} \\
   T(\det F^1(X/A))\arrow{rr}\arrow{d} & & T(\det \mathcal{F}_g)_{\ZZ[1/2]}\arrow{d}\\
   A\arrow{rr} &&\mathcal{A}_{g,\ZZ[1/2]}
  \end{tikzcd}
  $$
  in which every square is Cartesian. As $f$ is relatively representable by affine schemes, $B$ is affine over $T(\det F^1(X/A))$, thus quasi-affine over $\ZZ[1/2]$. Since $B_g\cong B/\Gamma$, we conclude by the part (1) of Lemma \ref{finiteaction}.
\end{proof}

\begin{obs} \label{rem-reprbk}
 By considering level structures on principally polarized abelian schemes with $R$-multiplication and the ampleness of the determinant of the Hodge bundle (\cite{lan13} Theorem 7.2.4.1 (2)), virtually the same proof can be applied to the case of $\mathcal{B}_F$. 
\end{obs}

\section{The case of elliptic curves: explicit equations}

When $g=1$ (or, equivalently, $F=\QQ$), we can compute explicit equations for $B_1 = B_{\QQ}$, for the Ramanujan vector field, and for the integral solution $\hat{\varphi}_1$ of the Ramanujan equation.

\subsection{Explicit equation for the universal elliptic curve $X_1$ over $B_1$ and its universal symplectic-Hodge basis}

Fix a scheme $U$. Let us recall that every \emph{elliptic curve} $E$ over $U$ (namely, an abelian scheme of relative dimension 1) has a canonical unique principal polarization $\lambda_{E}:E \to E^t$ given, for any $U$-scheme $V$ and any point $P\in E(V)$, by 
\begin{align*}
\lambda_E(P)= \mathcal{O}_E([P]-[O])
\end{align*}
where $O\in E(V)$ denotes the identity section and $\mathcal{O}_E([P]-[O])$ denotes the class in $E^t(V)$ of the inverse of the ideal sheaf defined by the relative Cartier divisor $[P]-[O]$.

Therefore, the functor 
\begin{align*}
E \mapsto (E,\lambda_{E})
\end{align*}
defines an equivalence between the category of elliptic curves over $U$ and that of principally polarized elliptic curves over $U$. We can thus ``forget'' the principal polarization: an elliptic curve $E$ will always be assumed to be endowed with its canonical principal polarization $\lambda_{E}$. In particular, an object of $\mathcal{B}_1$ will be denoted simply by a ``couple'' $(E,b)_{/U}$.

\begin{obs}
The symplectic form induced by $\lambda_E$ coincides with the composition of the cup product in de Rham cohomology $H^1_{\dR}(E/U) \times H^1_{\dR}(E/U) \to H^2_{\dR}(E/U)$ with the trace map $H^2_{\dR}(E/U) \to \mathcal{O}_U$.  
\end{obs}

\begin{theorem} \label{casg=1}
Let 
\begin{align*}
B_1 \defeq \Spec \ZZ[1/2,b_2,b_4,b_6,\Delta^{-1}]
\end{align*}
where
\begin{align*}
\Delta \defeq \frac{b_2^2(b_4^2 - b_2b_6)}{4} - 8 b_4^3 - 27b_6^2 + 9b_2b_4b_6 = 16\, {\rm disc}\left(x^3 + \frac{b_2}{4}x^2 + \frac{b_4}{2}x + \frac{b_6}{4}\right)\text{,}
\end{align*}
and let $X_1$ be the elliptic curve over $B_1$ given by the equation
\begin{align*}
y^2 = x^3 + \frac{b_2}{4}x^2 + \frac{b_4}{2}x + \frac{b_6}{4}\text{.}
\end{align*}
Then $b_1 = (\omega_1,\eta_1)$ defined by 
\begin{align*}
\omega_1 \defeq \frac{dx}{2y}\text{, } \ \ \ \eta_1 \defeq x\frac{dx}{2y}
\end{align*}
is a symplectic-Hodge basis of ${X_1}_{/B_1}$ and the morphism $B_1 \to \mathcal{B}_1$ corresponding to $(X_1,b_1)_{/B_1}$ induces an isomorphism of $B_1$ with the $\ZZ[1/2]$-stack $\mathcal{B}_{1,\ZZ[1/2]}$.
\end{theorem}

In other words, if $(X_1,b_1)_{/B_1}$ is defined as above, then for any $\ZZ[1/2]$-scheme $U$, and any elliptic curve $E$ over $U$ endowed with a symplectic-Hodge basis $b$, there exists a unique morphism $F_{/f}:E_{/U} \to {X_1}_{/B_1}$  in $\mathcal{A}_{1,\ZZ[1/2]}$ such that $F^*b_1=b$.

\begin{proof}
It is classical that $\omega_1$ so defined is in $F^1(X_1/B_1)$. To prove that $\langle \omega_1,\eta_1\rangle_{\lambda_E} = 1$ one can, for instance,  use the compatibility with base change to reduce this statement to an analogous statement concerning an elliptic curve over $\CC$, and then apply the classical residue formula (cf. \cite{DMOS82} pp. 23-25). 

Let $U$ be a $\ZZ[1/2]$-scheme and $(E,b)_{/U}$ be an object of $\mathcal{B}_{1}(U)$, with $b=(\omega,\eta)$. It is sufficient to prove that, locally for the Zariski topology over $U$, there exists a unique morphism $(E,b)_{/U} \to (X_1,b_1)_{/B_1}$ in $\mathcal{B}_{1,\ZZ[1/2]}$.

 We follow essentially the same steps in \cite{KM85} 2.2 to find a Weierstrass equation for an elliptic curve. Let us denote by $O: U \to E$ the identity section of the elliptic curve $E$ over $U$ and by $p:E \to U$ its structural morphism. Locally for the Zariski topology on $U$ we can find a formal parameter $t$ in the neighborhood of $O$ such that $\omega$ has a formal expansion in $t$ of the form
\begin{align*}
\omega = (1+ O(t))dt\text{,}
\end{align*} 
where $O(t)$ stands for a formal power series in $t$ of order $\ge 1$. Up to replacing $U$ by an open subscheme, we can and shall assume from now on that $t$ exists globally over $U$. 

There exist bases $(1,x)$ of $p_*\mathcal{O}_E(2 [O])$, and $(1,x,y)$ of $p_*\mathcal{O}_E(3[ O])$, such that
\begin{align} %\tag{$*$}
x = \frac{1}{t^2}(1+ O(t))\ \  \text{ and }\ \ y=\frac{1}{t^3}(1+O(t))\text{.}
\end{align}
Then the rational functions $x$ and $y$ necessarily satisfy an equation of the form
\begin{align*}
y^2 + a_1xy +a_3y= x^3 + a_2x^2 + a_4x+a_6\text{,}
\end{align*}
where $a_i$ are uniquely defined global sections of $\mathcal{O}_U$. Since 2 is invertible in $U$, the above equation is equivalent to
\begin{align*}
\left(y+\frac{a_1}{2}x + \frac{a_3}{2}\right)^2 = x^3 + \left(\frac{a_1^2 + 4 a_2}{4}\right)x^2 + \left(\frac{a_1a_3+2a_4}{2} \right)x + \frac{a_3^2 + 4 a_6}{4}\text{.}
\end{align*}
Therefore, after the change of coordinates $(x,y)\mapsto (x,y + \frac{a_1}{2}x + \frac{a_3}{2})$, we can assume that $x$ and $y$ satisfy
\begin{align*}
y^2 = x^3 + \frac{b_2}{4}x^2 + \frac{b_4}{2}x + \frac{b_6}{4}\text{,}
\end{align*}
where $b_i$ are global sections of $\mathcal{O}_U$. Put differently, we obtain a morphism $F_{/f} : E_{/U} \to {X_1}_{/B_1}$ in $\mathcal{A}_{1,\ZZ[1/2]}$.

By considering formal expansions in $t$, we see that $F^*\omega_1 = \omega$. In particular, 
\begin{align*}
(\omega, F^*\eta_1) = F^*b_1
\end{align*}
is a symplectic-Hodge basis of $E_{/U}$, and there exists a section $s$ of $\mathcal{O}_U$ such that $\eta = F^*\eta_1 + s\omega$. Thus, after the change of coordinates $(x,y)\mapsto (x+s,y)$, we have $F^*b_1 = b$. Therefore, we have constructed a morphism $F_{/f}:(E,b)_{/U} \to (X_1,b_1)_{/B_1}$ in $\mathcal{B}_{1,\ZZ[1/2]}$. 

We now prove that the morphism $F_{/f}$ is unique. Let $F'_{/f'} : (E,b)_{/U} \to (X_1,b_1)_{/B_1}$ be any morphism in $\mathcal{B}_{1,\ZZ[1/2]}$. If $f'=(b_2',b_4',b_6')$ are the coordinates of $f'$, then $F'$ is given by a basis $(1,x',y')$ of $p_*\mathcal{O}_E(3[ O])$ satisfying
\begin{align*} \tag{$*$}
(y')^2 = (x')^3 + \frac{b_2'}{4}(x')^2 + \frac{b_4'}{2}x' + \frac{b_6'}{4}\text{.}
\end{align*} 
As both $(1,x,y)$ and $(1,x',y')$ (resp. $(1,x)$ and $(1,x')$) are a basis of $p_*\mathcal{O}_E(3[ O])$ (resp. $p_*\mathcal{O}_E(2[O])$), then there exists global sections $c_1,c_2,c_3$ of $\mathcal{O}_{U}$ (resp. $u,v$ of $\mathcal{O}_U^{\times}$) such that
\begin{align*}
x' &= u(x+c_1)\\
y' &= v(y + c_2x+c_3)\text{.} 
\end{align*}
Note that equation $(*)$ implies that $u^3 = v^2$.

Now, as $(F')^*\omega_1 = F^*\omega_1$, we obtain
\begin{align*}
\frac{dx'}{2y'} = \frac{dx}{2y} \iff \frac{u}{v}\frac{dx}{2(y+c_2x+c_3)} = \frac{dx}{2y}\text{,}
\end{align*}
thus $c_2x+c_3=0$ and $u = v$. Since $u^3=v^2$, we obtain $u=v=1$ and $(x',y')=(x+c_1,y)$. Finally, as $(F')^*\eta_1=F^*\eta_1$, we have
\begin{align*}
x'\frac{dx'}{2y'} = \frac{dx}{2y} \iff  x\frac{dx}{2y} + c_1\frac{dx}{2y}= x\frac{dx}{2y}\text{,}
\end{align*}
hence $c_1=0$. Thus $(x',y')=(x,y)$ and this also implies that $f=f'$.
\end{proof}

\begin{obs} \label{classique1}
By considering the change of variables
\begin{align*}
\left\{\begin{array}{l}
         b_2 = e_2 \\
         b_4 = (e_2^2-e_4)/24\\
         b_6 = (4e_2^3 - 12 e_2e_4 + 8e_6)/1728
         \end{array}\right.\iff
\left\{\begin{array}{l}
         e_2 = b_2 \\
         e_4= b_2^2-24b_4\\
         e_6 = b_2^3 -36b_2b_4 + 216b_6
         \end{array}\right.
\end{align*}
we see that $B_1 \tensor_{\ZZ[1/2]}\ZZ[1/6]$ is isomorphic to
\begin{align*}
\Spec \ZZ[1/6, e_2,e_4,e_6,(e_4^3-e_6^2)^{-1}]\text{.}
\end{align*}
Under this identification, the universal elliptic curve $X_1$ is given by the equation
\begin{align*}
y^2 = 4\left(x+\frac{e_2}{12} \right)^3 -\frac{e_4}{12}\left(x+\frac{e_2}{12} \right) + \frac{e_6}{216}\text{,}
\end{align*}
and the universal symplectic-Hodge basis $b_1$ by $(dx/y,xdx/y)$. 
\end{obs}

%\warn{remark action of $P_1$}

\subsection{Explicit formulas for the Ramanujan vector field} 

It is also possible to give an explicit formula for the Ramanujan vector field $v_{11}$ over $B_1$. Indeed, consider the global section of $T_{B_1/\ZZ[1/2]}$ given by\begin{align*}
v \defeq 2b_4\frac{\partial}{\partial b_2} + 3b_6\frac{\partial}{\partial b_4} + (b_2b_6 - b_4^2)\frac{\partial}{\partial b_6}\text{.}
\end{align*}
One may easily verify using the expression for the Gauss-Manin connection on $H^1_{\dR}(X_1/B_1)$ given in \ref{GMunivell} that
\begin{align*}
\nabla_v\left(\begin{array}{cc}
       \omega_1 & \eta_1
       \end{array}\right) =
       \left(\begin{array}{cc}
       \omega_1 & \eta_1
       \end{array}\right)\left(\begin{array}{cc}
       0 & 0\\
       1 & 0
       \end{array}\right)
\end{align*}
By Proposition \ref{caracchamps}, $v$ is the Ramanujan vector field $v_{11}$ over $B_1$.

\begin{obs} \label{explanation}
Under the isomorphism $B_1\tensor_{\ZZ[1/2]} \ZZ[1/6] \cong \ZZ[1/6,e_2,e_4,e_6,(e_4^3 - e_6^2)^{-1}]$ of Remark \ref{classique1}, $v$ gets identified with the vector field associated to the classical Ramanujan equations:
\begin{align*}
v = \frac{e_2^2-e_4}{12} \frac{\partial}{\partial e_2} + \frac{e_2e_4-e_6}{3}\frac{\partial}{\partial e_4} + \frac{e_2e_6-e_4^2}{2}\frac{\partial}{\partial e_6}\text{.}
\end{align*}
\end{obs}

\subsection{Explicit formulas for $\hat{\varphi}_1$}

We now explicitly describe the integral solution
$$
\hat{\varphi}_1 : \Spec \ZZ(\!(q)\!) \to \mathcal{B}_1
$$
constructed in Section \ref{sec-intsol}.

Recall that we denote $\theta \defeq \theta_{11} = q\frac{d}{dq}$, and 
\begin{align*}
E_{2}(q) = 1 - 24 \sum_{n=1}^{\infty}\frac{n q^n}{1-q^n}\text{, }\ E_{4}(q) = 1+240 \sum_{n=1}^{\infty}\frac{n^3 q^n}{1-q^n}\text{, } \ 
E_6(q) = 1-504 \sum_{n=1}^{\infty}\frac{n^5 q^n}{1-q^n} \in \ZZ[\![q]\!]\text{.}
\end{align*}

\begin{prop} \label{prop-solg=1}
 We have:
 \begin{enumerate}
    \item under the identification $B_1 \cong \Spec \ZZ[1/2,b_2,b_4,b_6,\Delta^{-1}]$ of Theorem \ref{casg=1},
 \begin{align*}
 \hat{\varphi}^*_{1,\ZZ[1/2]}(b_2,b_4,b_6) = \left(E_2(q),\frac{1}{2}\theta E_2(q),\frac{1}{6}\theta^2 E_2(q) \right) \in (\ZZ(\!(q)\!)\tensor \ZZ[1/2])^3\text{;}
 \end{align*}
    \item under the identification  $B_{1,\ZZ[1/6]} \cong \Spec \ZZ[1/6, e_2,e_4,e_6,(e_4^3-e_6^2)^{-1}]$ of Remark \ref{classique1}, we have
 \begin{align*}
 \hat{\varphi}^*_{1,\ZZ[1/6]}(e_2,e_4,e_6) = (E_2(q),E_4(q),E_6(q))\in (\ZZ(\!(q)\!) \tensor \ZZ[1/6])^3\text{.}
 \end{align*} 
 \end{enumerate}
 \end{prop}

 \begin{proof}
   By the change-of-coordinates formulas in Remark \ref{classique1}, it is sufficient to prove (2).

   It is classical that the Tate curve $\hat{X}_{1,\ZZ[1/6]}$ over $\ZZ(\!(q)\!)\tensor \ZZ[1/6]$ is given by the equation
 $$
y^2 =4x^3 -\frac{E_4(q)}{12}x + \frac{E_6(q)}{216}\text{,}
$$
with canonical differential $\hat{\omega}_1 = \frac{dx}{y}$. This can be deduced from its analytic counterpart (see Paragraph \ref{subsec-compatsiegel}), which implies moreover that
$$
\hat{\eta_1}\defeq \nabla_{\theta}\omega_1 = \hat{\omega}_1 - \frac{E_2(q)}{12}\hat{\omega}_1\text{;}
$$
cf. equation (\ref{eq-gmtate}) in Appendix \ref{app-gm}.

Let $\varphi : \ZZ(\!(q)\!) \tensor \ZZ[1/6] \to B_{1,\ZZ[1/6]}$ be defined by
$$
\varphi^*(e_2,e_4,e_6) = (E_2(q),E_4(q),E_6(q))\text{.}
$$
Observe that we have a morphism in $\mathcal{A}_{1,\ZZ[1/6]}$
$$
\begin{tikzcd}
  \hat{X}_{1,\ZZ[1/6]} \arrow{r}{\Phi}\arrow{d} & X_{1,\ZZ[1/6]}\arrow{d}\\
  \Spec \ZZ(\!(q)\!)\tensor \ZZ[1/6] \arrow{r}[swap]{\varphi} & B_{1,\ZZ[1/6]}
\end{tikzcd}
$$
where the top arrow is defined by
$$
\Phi^*(x,y) = \left(x-\frac{E_2(q)}{12},y \right)\text{.}
$$
By the universal property of $B_{1,\ZZ[1/6]}$, to prove that $\hat{\varphi}_{1,\ZZ[1/6]} = \varphi$, it is sufficient to prove that $\Phi^*b_1 = \hat{b}_1$, i.e., that
$$
\Phi^*\omega_1 = \hat{\omega}_1\ \ \ \text{ and } \Phi^*\eta_1 = \hat{\eta}_1\text{.}
$$
This, in turn, is a simple computation using the explicit formulas for $\Phi$ and $\hat{\omega}_1$, $\hat{\eta}_1$ above, and the formulas for $\omega_1$ and $\eta_1$ in Remark \ref{classique1}.
 \end{proof}
 
% \begin{obs}
% It follows from this proposition that, under the identification $B_{1}\tensor_{\ZZ[1/2]} \ZZ[1/6] =\linebreak \Spec  \ZZ[1/6,e_2,e_4,e_6,(e_2^3-e_3^2)^{-1}]$ from Remark \ref{classique1}, the composition of the isomorphism $\mathbf{H}\times H_1(\CC) \stackrel{\sim}{\to} \bfb_1$ of Remark \ref{unif2} with the canonical projection $\bfb_1 \to B_1(\CC)$ is given by
% \begin{align*}
% \left(\tau, \left(\begin{array}{cc}
%              x & y \\
%              0 & x^{-1}
%              \end{array}\right)\right) \mapsto (x^{-2}E_2(\tau)-12x^{-1}y,x^{-4}E_4(\tau),x^{-6}E_6(\tau))\text{.}
% \end{align*} 
% \end{obs}

  Note that, by the explicit formulas given at the beginning of this paragraph, we know beforehand that the coefficients of $E_2$, $E_4$, $E_6$ (and of $\frac{1}{2}\theta E_2$ and $\frac{1}{6}\theta^2E_2$ as well) are integral, but our explicit expression for $\hat{\varphi}$ in terms of Eisenstein series relies on a base change to $\ZZ[1/2]$ or $\ZZ[1/6]$ (so that $\mathcal{B}_1$ becomes representable).

  As hinted in Paragraph \ref{subsec-introhresiegel} of our introductory section, in order to remain in a purely integral situation, we should consider the ring of global sections $\Gamma(\mathcal{B}_1,\mathcal{O}_{\mathcal{B}_{1,\et}})$. Let $E_{/U}$ be an elliptic curve endowed with a symplectic-Hodge basis $b=(\omega,\eta)$. Arguing as in the proof of Theorem \ref{casg=1}, we see that locally over $U$ the elliptic curve $E$ admits a Weierstrass equation
  $$
y^2 + a_1xy + a_3y = x^3 + a_2x^2 + a_4x + a_6
$$
with $\omega = \frac{dx}{y}$ and $\eta = x\frac{dx}{y}$. If we set, as in Tate's classical formulas (cf. \cite{deligne75} 1.4),
$$
b_2\defeq a_1^2 + 4a_2\text{, }\  b_4\defeq a_1a_3 + 2a_4\text{, }\  b_6\defeq a_3^2 + 4a_6\text{, }\ b_8\defeq -a_1a_3a_4 - a_4^2 +a_1^2a_6 + a_2a_3^2 +4a_2a_6\text{,}
$$
then we check that $b_4^2-b_2b_6+4b_8=0$, and that $b_2$, $b_4$, $b_6$, and $b_8$ do not depend on the choice of the particular Weierstrass equation for which $\omega = \frac{dx}{y}$ and $\eta=x\frac{dx}{y}$. In particular, they define global sections of $\mathcal{O}_{\mathcal{B}_{1,\et}}$. In this sense, Theorem \ref{casg=1} simply says that the morphism
$$
(b_2,b_4,b_6):\mathcal{B}_1 \to \AA_{\ZZ}^3
$$
induces, after base change to $\ZZ[1/2]$, an isomorphism of $\mathcal{B}_{1,\ZZ[1/2]}$ with the open affine subscheme of $\AA^3_{\ZZ[1/2]}$ defined by $\Delta \neq 0$, and it follows from Proposition \ref{prop-solg=1} (1) that
$$
\hat{\varphi}_1^*(b_2,b_4,b_6) = \left(E_2(q),\frac{1}{2}\theta E_2(q),\frac{1}{6}\theta^2 E_2(q) \right)\in \ZZ(\!(q)\!)^3\text{.}
$$
Analogously, the formulas for $e_2$, $e_4$, and $e_6$ in Remark \ref{classique1} also define global sections of $\mathcal{O}_{\mathcal{B}_{1,\et}}$, so that the components of $\hat{\varphi}_1^*$ in the ``coordinates'' $(e_2,e_4,e_6)$, namely $E_2(q)$, $E_4(q)$, and $E_6(q)$, are in $\ZZ(\!(q)\!)$. 

\begin{obs}
  The ring $\Gamma(\mathcal{B}_1,\mathcal{O}_{\mathcal{B}_{1,\et}})$, which can be shown to be isomorphic to
  $$
  \ZZ[b_2,b_4,b_6,b_8,\Delta^{-1}]/(b_4^2-b_2b_6+4b_8)
  $$
  by arguments similar to \cite{deligne75}, Paragraph 6, can be thought as the ring of ``integral weakly holomorphic quasimodular forms'', i.e., integral quasimodular forms which are only meromorphic at infinity (cf. \ref{subsubsec-intronearlyhol}). 
\end{obs}

\begin{appendices}

  \section{Gauss-Manin connection on some elliptic curves}\label{app-gm}

\subsection{The Weierstrass elliptic curve}

Let
\begin{align*}
W \defeq \Spec \CC[g_2,g_3,\Delta^{-1}]
\end{align*}
where
\begin{align*}
\Delta \defeq g_2^3 - 27g_3^2\text{.}
\end{align*}
Then we can define an elliptic curve $E$ over $W$ by the classical Weierstrass equation
\begin{align*}
y^2 = 4x^3 -g_2x-g_3\text{.}
\end{align*}
Further, we define a symplectic-Hodge basis $(\omega,\eta)$ of $E_{/W}$ by the formulas
\begin{align*}
\omega \defeq \frac{dx}{y}\text{, } \ \ \ \eta \defeq x\frac{dx}{y}\text{.} 
\end{align*}

\begin{lemma} \label{GMconW}
With the above notations, the Gauss-Manin connection $\nabla$ on $H^1_{\dR}(E/W)$ is given by
\begin{align*}
\nabla \left(\begin{array}{cc}
       \omega & \eta
       \end{array}\right) = \left(\begin{array}{cc}
       \omega & \eta\end{array}\right)  \tensor \frac{1}{\Delta}
       \left(\begin{array}{cc}
       \Omega_{11} & \Omega_{12}\\
       \Omega_{21} & \Omega_{22}
       \end{array}\right)
\end{align*}
where
\begin{align*}
\Omega_{11} &= -\frac{1}{4}g_2^2\, dg_2 + \frac{9}{2}g_3\,dg_3\\
\Omega_{12} &= \frac{3}{8}g_2g_3\,dg_2 -\frac{1}{4}g_2^2\, dg_3\\
\Omega_{21} &= -\frac{9}{2}g_3 \, dg_2 + 3g_2\, dg_3\\
\Omega_{22} &= -\Omega_{11}\text{.}
\end{align*}
\end{lemma}

Let us briefly explain how these expressions follow from the description given in \cite{katz73} A1.3 of the Gauss-Manin connection on the relative first de Rham cohomology of the universal elliptic curve $\mathbf{E}$ over the Poincaré half-plane $\mathbf{H}$ (whose fiber at each $\tau\in \mathbf{H}$ is given by the complex torus $\mathbf{E}_{\tau}=\CC/(\ZZ + \ZZ\tau)$; in the notation of Example \ref{torus}, we have $\mathbf{E}=\mathbf{X}_1$).\footnote{A direct algebraic approach is also possible. See for instance \cite{KO68} 3, \cite{kedlaya07} 3.4, and \cite{movasati12} 3.4.}

 We first remark that for any $u\in \CC^{\times}$ we can define an automorphism ${M_u}_{/\mu_u}: E_{/W} \to E_{/W}$ in the category $\mathcal{A}_{1,\CC}$ by
\begin{align*}
\mu_u(g_2,g_3) = (u^{-4}g_2,u^{-6}g_3)\text{, } \ \ \ M_u(x,y)=(u^{-2}x,u^{-3}y)\text{.}
\end{align*}
Using that the Gauss-Manin connection commutes with base change and admits regular singularities, we deduce by homogeneity that there exists constants $c_1,\ldots,c_8$ in $\CC$ such that
\begin{align*}
\Omega_{11} = c_1g_2^2\, dg_2 + c_2g_3\, dg_3\text{, }\ \ &\Omega_{12} = c_3g_2g_3\, dg_2 +c_4g_2^2\, dg_3\text{, }\\
 \Omega_{21} = c_5g_3\, dg_2 + c_6g_2\, dg_3\text{, } \ \ & \Omega_{22} = c_7g_2^2\, dg_2 + c_8g_3\, dg_3 \text{.}
\end{align*}
To determine these constants, we consider the Cartesian diagram in the category of complex analytic spaces
%$$
%\raisebox{-0.5\height}{\includegraphics{rameq1-d7.pdf}}
%$$
 $$
 \begin{tikzcd}
 \mathbf{E} \arrow{d} \arrow{r}{\Psi} & E(\CC)\arrow{d} \\
 \mathbf{H} \arrow{r}[swap]{\psi} & W(\CC)\arrow[draw=none]{ul}[description]{\square}
 \end{tikzcd}
 $$
given by the classical Weierstrass theory:
\begin{align*}
\psi(\tau)=(g_2(\tau),g_3(\tau))\text{, }\ \ \ \Psi_{\tau}(z)=(\wp_{\tau}(z),\wp_{\tau}'(z)) 
\end{align*}
Finally, we apply once again that that the formation of the Gauss-Manin connection (now in the complex analytic category) commutes with base change, and we use the formulas in \cite{katz73} A1.3:
\begin{align}\label{eq-gmtate}
\nabla \left(\begin{array}{cc}
       dz & \wp_{\tau}(z)dz
       \end{array}\right) = \left(\begin{array}{cc}
       dz & \wp_{\tau}(z)dz\end{array}\right)  \tensor \frac{1}{2\pi i}
       \left(\begin{array}{cc}
       -(2\pi i)^2E_{2}(\tau)/12 & -(2\pi i)^4E_4(\tau)/144\\
       1 & (2\pi i)^2E_2(\tau)/12
       \end{array}\right)d\tau
\end{align}

\subsection{The elliptic curve ${X}_{/B}$ over $\ZZ[1/6]$} \label{GMconB}

Let
\begin{align*}
B \defeq \Spec \ZZ[1/6, e_2,e_4,e_6, \Delta^{-1}]
\end{align*}
where
\begin{align*}
\Delta \defeq e_4^3 - e_6^2\text{.}
\end{align*}
We define an elliptic curve $X$ over $B$ by
\begin{align*}
y^2 = 4\left(x+ \frac{e_2}{12}\right)^3 - \frac{e_4}{12}\left(x+ \frac{e_2}{12}\right) + \frac{e_6}{216}\text{.}
\end{align*}
We define a symplectic-Hodge basis $(\omega, \eta)$ of $X_{/B}$ by the formulas
\begin{align*}
\omega \defeq \frac{dx}{y}\text{, } \ \ \ \eta \defeq x\frac{dx}{y}\text{.} 
\end{align*}

Note that there is a morphism $F_{/f} : (X_{\CC})_{/B_{\CC}} \to E_{/W}$ in $\mathcal{A}_{1,\CC}$ given by
\begin{align*}
f(e_2,e_4,e_6) = \left(\frac{e_4}{12},-\frac{e_6}{216}\right)\text{, } \ \ \ F(x,y) = \left(x + \frac{e_2}{12},y \right)\text{.}
\end{align*}
By pulling back the Gauss-Manin connection on $H^1_{\dR}(E/W)$ described in Lemma \ref{GMconW} by the morphism $F_{/f}$, we obtain that the Gauss-Manin connection $\nabla$ on $H^1_{\dR}(X/B)$ over $\ZZ[1/6]$ is given by
\begin{align*}
\nabla \left(\begin{array}{cc}
       \omega & \eta
       \end{array}\right) = \left(\begin{array}{cc}
       \omega & \eta\end{array}\right)  \tensor \frac{1}{\Delta}
       \left(\begin{array}{cc}
       \Omega_{11} & \Omega_{12}\\
       \Omega_{21} & \Omega_{22}
       \end{array}\right)
\end{align*}
where
\begin{align*}
\Omega_{11} &= \left(\frac{e_2e_6 - e_4^2}{4} \right)de_4 + \left(\frac{e_6-e_2e_4}{6} \right)de_6\\
\Omega_{12} &=-\frac{\Delta}{12}de_2 - \left(\frac{e_4e_6-2e_2e_4^2+e_2^2e_6}{48} \right) de_4 + \left(\frac{e_4^2-2e_2e_6 + e_2^2e_4}{72}\right)de_6\\
\Omega_{21} &= 3e_6de_4 -2e_4de_6\\
\Omega_{22} &= -\Omega_{11}\text{.}
\end{align*}

\subsection{The universal elliptic curve ${X_1}_{/B_1}$ over $\ZZ[1/2]$} \label{GMunivell}

Consider the elliptic curve $X_{1}$ over $B_1$ defined in Theorem \ref{casg=1} and let $\Phi_{/\varphi} : (X_{1,\ZZ[1/6]})_{/B_{1,\ZZ[1/6]}} \to X_{/B}$ be the isomorphism in $\mathcal{A}_{1,\ZZ[1/6]}$ given by
\begin{align*}
\varphi(b_2,b_4,b_6) = (b_2,b_2^2-24b_4,b_2^3 -36b_2b_4 + 216b_6)\text{, } \ \ \ \Phi(x,y)=(x,2y)\text{.} 
\end{align*}

If $(\omega_1,\eta_1)$ denotes de symplectic-Hodge basis of ${X_{1}}_{/B_1}$ defined in Theorem \ref{casg=1}, then by pulling back the Gauss-Manin connection on $H^1_{\dR}(X/B)$ described in \ref{GMconB} by the isomorphism $\Phi_{/\varphi}$, we obtain that the Gauss-Manin connection $\nabla$ on $H^1_{\dR}(X_{1}/B_1)$ over $\ZZ[1/2]$ is given by

\begin{align*}
\nabla \left(\begin{array}{cc}
       \omega_1 & \eta_1
       \end{array}\right) = \left(\begin{array}{cc}
       \omega_1 & \eta_1\end{array}\right)  \tensor \frac{1}{\Delta}
       \left(\begin{array}{cc}
       \Omega_{11} & \Omega_{12}\\
       \Omega_{21} & \Omega_{22}
       \end{array}\right)
\end{align*}
where
\begin{align*}
\Omega_{11} &= \frac{b_2^2b_6-6b_4b_6-b_2b_4^2}{8} db_2 + \frac{4b_4^2-3b_2b_6}{2}db_6 + \frac{18b_6 - b_2b_4}{4}db_6\\
\Omega_{12} &= \frac{2b_4^3 +9b_6^2 - 2b_2b_4b_6}{4}db_2 + \frac{b_2^2b_6-b_2b_4^2 -6b_4b_6}{4}db_6 + \frac{4b_4^2-3b_2b_6}{4}db_6\\
\Omega_{21} &= \frac{3b_2b_6 - 4b_4^2}{4}db_2 + \frac{b_2b_4-18b_6}{2}db_4 + \frac{24b_4-b_2^2}{4}db_6\\
\Omega_{22} &= -\Omega_{11}\text{.}
\end{align*}
  
\end{appendices}

\newpage

\part{The analytic higher Ramanujan equations and periods of abelian varieties}

\section{Analytic families of complex tori, abelian varieties, and their uniformization}

In this section we briefly transpose some of the standard theory of complex tori to a relative situation, that is, we shall consider analytic families of complex tori. To both simplify and shorten our exposition, we shall assume that the parameter space is smooth (i.e., a complex manifold); this largely suffices for our needs.

Most of the material included in here, and in the following section, is well known to experts --- and may be even considered as ``classical'' --- but we could not find a convenient reference in the literature.

\subsection{Relative complex tori} \label{rct}

Let $M$ be a complex manifold.

\begin{defi}
A \emph{(relative) complex torus over $M$} is a relative complex Lie group $\pi:X\to M$ over $M$ such that $\pi$ is proper with connected fibers. A morphism of complex tori over $M$ is a morphism of relative complex Lie groups over $M$.
\end{defi}

As any compact connected complex Lie group is a complex torus, every fiber of $\pi$ in the above definition is a complex torus.

In general, for any relative complex Lie group $\pi: X \to M$ over $M$, we may consider its \emph{relative Lie algebra} $\Lie_M X$; this is a holomorphic vector bundle over $M$ whose fiber at each $p\in M$ is the Lie algebra $\Lie X_p$ of the Lie group $X_p \defeq \pi^{-1}(p)$. Moreover, there exists a canonical morphism of complex manifolds over $M$
\begin{align*}
\exp : {\Lie}_MX \to X
\end{align*}
restricting to the usual exponential map of complex Lie groups at each fiber. \label{symb:R1ZX}

\begin{lemma} \label{uniftorus}
Let $\pi: X \to M$ be a complex torus over $M$. Then $\exp : \Lie_MX \to X$ is a surjective and submersive morphism of relative complex Lie groups over $M$. Moreover, the sheaf of sections of the relative complex Lie group $\ker (\exp)$ over $M$ is canonically isomorphic to 
\begin{align*}
R_1\pi_*\ZZ_X  \defeq (R^1\pi_* \ZZ_X)^{\vee}\text{.}
\end{align*}
\hfill $\blacksquare$
\end{lemma}

This follows from the classical case where $M$ is a point via a fiber-by-fiber consideration (cf. \cite{mumford70} I.1). Note that $R_1\pi_*\ZZ_X$ is a local system of free abelian groups over $M$ whose fiber at $p\in M$ is given by the first singular homology group $H_1(X_p,\ZZ)$.

\begin{defi}
Let $V$ be a holomorphic vector bundle of rank $g$ over $M$. By a \emph{lattice} in $V$, we mean a subsheaf of abelian groups $L$ of $\mathcal{O}_M(V)$ such that
\begin{enumerate}
   \item $L$ is a local system of free abelian groups of rank $2g$,
   \item for each $p\in M$, the quotient $V_p/L_p$ is compact.
\end{enumerate}
\end{defi}

It follows from Lemma \ref{uniftorus} that, for any complex torus $\pi: X \to M$ of relative dimension $g$, $R_1\pi_*\ZZ_X$ may be canonically identified to a lattice in $\Lie_M X$.

Conversely, if $V$ is a holomorphic vector bundle of rank $g$ over $M$ and $L$ is a lattice in $V$, then the étalé space $E(L)$ of $L$ is a relative complex Lie subgroup of $V$ over $M$ and $X\defeq V/E(L)$ is a complex torus over $M$ of relative dimension $g$. Furthermore, the relative Lie algebra $\Lie_MX$ gets canonically identified with $V$ and, under this identification, $E(L)$ is the kernel of the exponential map $\exp: \Lie_MX \to X$.

\begin{obs} \label{remarktorus}
The above reasoning actually proves that the category of complex tori over $M$ of relative dimension $g$ is equivalent to the category of couples $(V,L)$ where $V$ is a holomorphic vector bundle of rank $g$ over $M$ and $L$ is a lattice in $V$; a morphism $(V,L) \to (V',L')$ in this category is given by a morphism of holomorphic vector bundles $\varphi:V \to V'$ such that $\varphi(E(L))\subset E(L')$.
\end{obs}

In what follows, we shall drop the notation $E(L)$ and identify a local system with its étalé space.

\subsection{Riemann forms and principally polarized complex tori}

Let $M$ be a complex manifold and $\pi:X \to M$ be a complex torus over $M$.

\begin{defi} \label{defiriemannform}
A \emph{Riemann form} over $X$ is a $C^{\infty}$ Hermitian metric\footnote{Our convention is that Hermitian forms are anti-linear on the first coordinate and linear on the second.} $H$ on the vector bundle  $\Lie_M X$ over $M$ such that
\begin{align*}
E \defeq \Im H
\end{align*}
takes integral values on $R_1\pi_*\ZZ_X$.
\end{defi}

Observe that $E$ is an alternating $\RR$-bilinear form. We also remark that the Hermitian metric $H$ is completely determined by $E$: for any sections $v$ and $w$ of $\Lie_M X$ we have $H(v,w)= E(v,iw) + iE(v,w)$. In particular, by abuse, we may also say that $E$ is Riemann form over $X$.

\begin{defi}
With the above notation, we say that the Riemann form $E$ is \emph{principal} if the induced morphism of local systems
\begin{align*}
   R_1\pi_*\ZZ_X &\longrightarrow (R_1\pi_*\ZZ_X)^{\vee} \cong R^1\pi_*\ZZ_X\\
           \gamma &\mapsto E( \gamma , \ )
\end{align*}
is an isomorphism.
\end{defi}

\begin{defi}
Let $M$ be a complex manifold. A \emph{principally polarized complex torus} over $M$ of relative dimension $g$ is a couple $(X,E)$, where $X$ is a complex torus over $M$ of relative dimension $g$ and $E$ is a principal Riemann form over $X$. 
\end{defi}

\label{symb:siegelspace}
\begin{ex} \label{torus}
Let $g\ge 1$ and consider the Siegel upper half-space
\begin{align*}
\mathbf{H}_{g} \defeq \{\tau \in M_{g\times g}(\CC) \mid \tau = \tau^{\textsf{T}}\text{, } \Im \tau >0\}\text{.}
\end{align*}
If $g=1$, we denote $\mathbf{H} \defeq \mathbf{H}_1$; this is the Poincaré upper half-plane. Let us consider the trivial vector bundle $V\defeq \CC^g \times \mathbf{H}_g$ over $\mathbf{H}_g$ and let $L$ be the subsheaf of $\mathcal{O}_{\mathbf{H}_g}(V)$ given by the image of the morphism of sheaves of abelian groups
\begin{align*}
(\ZZ^g \oplus \ZZ^g)_{\mathbf{H}_g} &\to \mathcal{O}_{\mathbf{H}_g}(V) = \mathcal{O}_{\mathbf{H}_g}^{\oplus g}\\
            (m,n) &\mapsto m+\tau n
\end{align*}
where $m$ and $n$ are considered as column vectors of order $g$. Then $L$ is a lattice in $V$ and we denote by
\begin{align*}
\bfp_g: \mathbf{X}_g  \to \mathbf{H}_g
\end{align*}
the corresponding  complex torus over $\mathbf{H}_g$ of relative dimension $g$ (cf. Remark \ref{remarktorus}). Let $E_g$ be imaginary part of the Hermitian metric over $V$ given by
\begin{align*}
(v,w) \mapsto \overline{v}\transp (\Im \tau)^{-1}w\text{.}
\end{align*}
One may easily verify that $E_g$ takes integral values on $L$ and that $\gamma \mapsto E_g( \gamma ,  \ )$ induces an isomorphism $L \stackrel{\sim}{\to} L^{\vee}$. We thus obtain a principally polarized complex torus $(\mathbf{X}_g,E_g)$ over $\mathbf{H}_g$ of relative dimension $g$.\label{symb:Xg}
\end{ex}

\subsection{The category $\mathcal{A}^{\an}_g$ of principally polarized complex tori of relative dimension $g$} \label{defitg}

 Let $\textsf{Man}_{/\CC}$ denote the category of complex manifolds. We define a category  $\mathcal{A}^{\an}_g$ fibered in groupoids over $\Man_{/\CC}$ as follows.

\begin{enumerate}
     \item An object of the category $\mathcal{A}^{\an}_g$ consists in a complex manifold $M$ and a principally polarized complex torus $(X,E)$ over $M$ of relative dimension $g$; we denote such an object by $(X,E)_{/M}$.
     \item Let $(X,E)_{/M}$ and $(X',E')_{/M'}$ be objects of $\mathcal{A}^{\an}_g$. A morphism
\begin{align*}
\varphi_{/f}:(X',E')_{/M'} \to (X,E)_{/M}
\end{align*}
in $\mathcal{A}^{\an}_g$ is a Cartesian diagram of complex manifolds
%$$
%  \raisebox{-0.5\height}{\includegraphics{rameq2-d1.pdf}}
%$$
  $$
 \begin{tikzcd}
   X' \arrow{r}{\varphi}\arrow{d} & X \arrow{d}\\
   M' \arrow{r}[swap]{f} & M \arrow[draw=none]{ul}[description]{\square}
 \end{tikzcd}
 $$
preserving the identity sections of the complex tori and such that $E' = f^*E$ under the isomorphism of holomorphic vector bundles $\Lie_{M'}X' \stackrel{\sim}{\to} f^*\Lie_{M}X$ induced by $\varphi$. We may also denote $(X',E') = (X,E)\times_{M}M'$.
   \item The structural functor $\mathcal{A}^{\an}_g \to \Man_{/\CC}$ sends an object $(X,E)_{/M}$ of $\mathcal{A}^{\an}_g$ to the complex manifold $M$, and a morphism $\varphi_{/f}$ as above to $f$.
\end{enumerate}

\begin{ex}\label{exactionsp}
We define an action of $\Sp_{2g}(\ZZ)$ on the object $(\mathbf{X}_g,E_g)_{/\mathbf{H}_g}$ of $\mathcal{A}_g^{\an}$
\begin{align*}
{\Sp}_{2g}(\ZZ) &\longrightarrow {\Aut}_{\mathcal{A}_g^{\an}}\left((\mathbf{X}_g,E_g)_{/\mathbf{H}_g}\right)\\
        \gamma &\mapsto {\varphi_\gamma}_{/f_{\gamma}}
\end{align*} 
as follows. Recall that an element $\gamma = (A \ B \ ; \ C \ D) \in \Sp_{2g}(\RR)$ acts on $\mathbf{H}_g$ by
\begin{align*}
f_{\gamma} : \mathbf{H}_g &\longrightarrow \mathbf{H}_g\\
          \tau &\mapsto \gamma\cdot \tau \defeq (A\tau + B)(C\tau+D)^{-1}\text{.}
\end{align*}
For $\gamma$ as above, consider the holomorphic map
\begin{align*}
\widetilde{\varphi}_{\gamma}:\CC^g\times \mathbf{H}_g &\to  \CC^g\times \mathbf{H}_g\\
(z,\tau) &\mapsto  ((j(\gamma,\tau)\transp)^{-1}z,\gamma \cdot \tau)           
\end{align*}   
where 
$$
j(\gamma,\tau)\defeq C\tau+D \in {\GL}_g(\CC)\text{.}
$$
 If $\gamma \in \Sp_{2g}(\ZZ)$, then for every $\tau \in \mathbf{H}_g$ we have
\begin{align*}
\widetilde{\varphi}_{\gamma,\tau}(\ZZ^g + \tau \ZZ^g) = \ZZ^g + (\gamma \cdot\tau) \ZZ^g\text{,}
\end{align*}
so that $\widetilde{\varphi}_{\gamma}$ induces a holomorphic map $\varphi_{\gamma} : \mathbf{X}_g \to \mathbf{X}_g$. One easily verifies that
%$$
%  \raisebox{-0.5\height}{\includegraphics{rameq2-d2.pdf}}
%$$
 $$
 \begin{tikzcd}
   \mathbf{X}_g \arrow{r}{\varphi_{\gamma}}\arrow{d}[swap]{\bfp_g} & \mathbf{X}_g \arrow{d}{\bfp_g}\\
   \mathbf{H}_g \arrow{r}[swap]{f_{\gamma}} & \mathbf{H}_g 
 \end{tikzcd}
 $$
is a Cartesian diagram of complex manifolds preserving the identity sections and the Riemann forms $E_g$, i.e., it defines a morphism ${\varphi_{\gamma}}_{/f_{\gamma}} : (\mathbf{X}_g,E_g)_{/\mathbf{H}_g} \to (\mathbf{X}_g,E_g)_{/\mathbf{H}_g}$ in $\mathcal{A}_g^{\an}$. Finally, the formula 
$$
j(\gamma_1\gamma_2,\tau) = j(\gamma_1,\gamma_2\cdot \tau)j(\gamma_2,\tau)
$$
implies that ${\varphi_{\gamma}}_{/f_{\gamma}}$ is in fact an automorphism of $(\mathbf{X}_g,E_g)_{/\mathbf{H}_g}$ in $\mathcal{A}_g^{\an}$ and that $\gamma \mapsto {\varphi_{\gamma}}_{/f_{\gamma}}$ is a morphism of groups.\footnote{Actually, it follows from Proposition \ref{reprsintsympl} below (see also Remark \ref{actionsp}) that $\gamma \mapsto {\varphi_{\gamma}}_{/f_{\gamma}}$ is an \emph{isomorphism} of groups.}% c'est un iso!
\end{ex}

\subsection{De Rham cohomology of complex tori} 

Let $M$ be a complex manifold and $\pi : X \to M$ be a complex torus over $M$ of relative dimension $g$.

\subsubsection{}  For any integer $i\ge 0$, we define the $i$th \emph{analytic de Rham cohomology} sheaf of $\mathcal{O}_M$-modules by \label{symb:analyticderham}
\begin{align*}
\mathcal{H}^i_{\dR}(X/M) \defeq \mathbf{R}^i\pi_* \Omega^{\bullet}_{X/M}\text{,}
\end{align*}
where $\Omega^{\bullet}_{X/M}$ is the complex of relative holomorphic differential forms. If $M$ is a point, we denote $\mathcal{H}^i_{\dR}(X) \defeq \mathcal{H}^i_{\dR}(X/M)$.

\begin{obs} \label{cinftyderham}
If $M$ is a point, then the analytic de Rham cohomology $\mathcal{H}^i_{\dR}(X)$ is canonically isomorphic to the quotient of the complex vector space of $C^{\infty}$ closed $i$-forms over $X$ with values in $\CC$  by the subspace of exact $i$-forms (cf. \cite{DMOS82} I.1 p. 16). 
\end{obs}

The arguments in \cite{BBM82} 2.5 prove, \emph{mutatis mutandis}, that there is a canonical isomorphism of $\mathcal{O}_M$-modules given by cup product
\begin{align*}
\bigwedge^i \mathcal{H}^1_{\dR}(X/M) \stackrel{\sim}{\to} \mathcal{H}^i_{\dR}(X/M)\text{,}
\end{align*}
and that $\mathcal{H}^1_{\dR}(X/M)$ is (the sheaf of sections of) a holomorphic vector bundle over $M$ of rank $2g$. Moreover, the canonical $\mathcal{O}_M$-morphism $\pi_*\Omega^1_{X/M} \to \mathcal{H}^1_{\dR}(X/M)$ induces an isomorphism of $\pi_*\Omega^1_{X/M}$ onto a rank $g$ subbundle of $\mathcal{H}^1_{\dR}(X/M)$ that we denote by $\mathcal{F}^1(X/M)$.

Analogously, it follows from the arguments of \cite{KO68} that $\mathcal{H}^1_{\dR}(X/M)$ is equipped with a canonical integrable holomorphic connection
\begin{align*}
\nabla : \mathcal{H}^1_{\dR}(X/M) \to \mathcal{H}^1_{\dR}(X/M) \tensor_{\mathcal{O}_{M}}\Omega^1_M \text{,} 
\end{align*}
the \emph{Gauss-Manin connection}.

Furthermore, the formation of $\mathcal{H}^1_{\dR}(X/M)$ (resp. $\mathcal{F}^1(X/M)$, resp. $\nabla$) is compatible with every base change in $M$.

% The arguments in \cite{BBM82} 2.5 and \cite{KO68} concerning abelian schemes prove, \emph{mutatis mutandis},  that $\mathcal{H}^1_{\dR}(X/M)$ is (the sheaf of holomorphic sections of) a holomorphic vector bundle over $M$ of rank $2g$, endowed with a subbundle $\mathcal{F}^1(X/M)$ of rank $g$ canonically isomorphic to $\pi_*\Omega^1_{X/M}$, and a canonical integrable holomorphic connection
% \begin{align*}
% \nabla : \mathcal{H}^1_{\dR}(X/M) \to \mathcal{H}^1_{\dR}(X/M) \tensor_{\mathcal{O}_{M}}\Omega^1_M \text{,} 
% \end{align*} 
% the \emph{Gauss-Manin connection}.

%  Furthermore, the formation of $\mathcal{H}^1_{\dR}(X/M)$, $\mathcal{F}^1(X/M)$, and $\nabla$, are compatible with every base change in $M$.

 \subsubsection{}\label{sectioncompisom}

 There is a canonical \emph{comparison isomorphism} of holomorphic vector bundles\label{symb:comp}
\begin{align} \label{compisom}
\comp: \mathcal{H}^1_{\dR}(X/M)   \stackrel{\sim}{\longrightarrow}    \mathcal{H}om_{\ZZ}(R_1\pi_*\ZZ_X,\mathcal{O}_M) \cong \mathcal{O}_M\tensor_{\ZZ} R^1\pi_*\ZZ_X  
\end{align}
identifying the subsheaf of $\mathcal{H}^1_{\dR}(X/M)$ consisting of horizontal sections for the Gauss-Manin connection with the local system of $\CC$-vector spaces $\mathcal{H}om_{\ZZ}(R_1\pi_*\ZZ_X,\CC_M) \cong R^1\pi_*\CC_X$   (\cite{deligne70} I Proposition 2.28 and II 7.6-7.7). The induced pairing
\begin{align*}
\mathcal{H}^1_{\dR}(X/M) \otimes_{\ZZ}R_1\pi_*\ZZ_X  &\longrightarrow \mathcal{O}_M\\
       \alpha\tensor \gamma &\mapsto \comp(\alpha)(\gamma)  \eqdef \int_{\gamma}\alpha
\end{align*}
is given at each fiber by ``integration of differential forms'' (cf. Remark \ref{cinftyderham}).
 
\begin{obs}\label{derivint}
In particular, for any section $\gamma$ of $R_1\pi_*\ZZ_X$, any $C^{\infty}$ section $\alpha$ of the vector bundle $\mathcal{H}^1_{\dR}(X/M)$, and any holomorphic vector field $\theta$ on $M$, we have
\begin{align*}
\theta \left( \int_{\gamma}\alpha \right) = \int_{\gamma}\nabla_{\theta}\alpha\text{.}
\end{align*}
\end{obs}

\begin{obs}\label{obs:comp}
  In the absolute case (where $M$ is a point), the comparison isomorphism can be written
  $$
\comp: \mathcal{H}^1_{\dR}(X) \stackrel{\sim}{\longrightarrow} \CC\tensor_{\ZZ}H^1(X,\ZZ)
$$
where $X$ is a complex torus. If $X$ is now an abelian variety over a subfield $k$ of $\CC$, then the associated analytic space $X_{\CC}^{\an}$ is a complex torus and we have a canonical isomorphism $\CC\tensor_k H^1_{\dR}(X/k) \cong \mathcal{H}^1_{\dR}(X_{\CC}^{\an})$. In this case, we also write 
$$
\comp: \CC\tensor_kH^1_{\dR}(X/k) \stackrel{\sim}{\longrightarrow} \CC\tensor_{\ZZ}H^1(X,\ZZ)
$$
for the composition of $\comp$ with the above canonical identification.
\end{obs}

Recall that $R_1\pi_*\ZZ_X$ may be naturally identified with a lattice in the holomorphic vector bundle $\Lie_{M}X$. Accordingly, the dual bundle $(\Lie_MX)^{\vee}$ gets naturally identified with a holomorphic subbundle of $\mathcal{H}om_{\ZZ}(R_1\pi_*\ZZ_X,\mathcal{O}_M)$.

\begin{lemma} \label{identcomp}
With notation as above, the comparison isomorphism (\ref{compisom}) induces an isomorphism of the holomorphic vector bundle $\mathcal{F}^1(X/M)$ onto $(\Lie_MX)^{\vee}$. \hfill $\blacksquare$
\end{lemma} 

This also follows from a fiber-by-fiber argument: if $M$ is a point, by identifying $\mathcal{H}^1_{\dR}(X)$ with the $C^{\infty}$ de Rham cohomology with values in $\CC$ (Remark \ref{cinftyderham}), the subspace $\mathcal{F}^1(X)$ gets identified with the space of $(1,0)$-forms in $\mathcal{H}^1_{\dR}(X)$, and these correspond to $\Hom_{\CC}(\Lie X, \CC)$ under the de Rham isomorphism (cf. \cite{BL04} Theorem 1.4.1). 

\subsubsection{} \label{defi-sfrf}

If $X$ admits a principal Riemann form  $E$, then, by linearity, we may define a holomorphic symplectic form $\langle \ , \ \rangle_E$ on the holomorphic vector bundle $\mathcal{H}^1_{\dR}(X/M)$ over $M$  by
\begin{align*}
\langle E(\gamma, \ ) , E(\delta, \ ) \rangle_E \defeq \frac{1}{2\pi i}E(\gamma,\delta) 
\end{align*}
for any sections $\gamma$ and $\delta$ of $R_1\pi_*\ZZ_X$, where $E(\gamma , \ )$ and $E(\delta , \ )$ are regarded as sections of $\mathcal{H}^1_{\dR}(X/M)$ via the comparison isomorphism (\ref{compisom}).

Since every section of $R^1\pi_*\ZZ_X$ is horizontal for the Gauss-Manin connection $\nabla$ on $\mathcal{H}^1_{\dR}(X/M)$ under the comparison isomorphism (\ref{compisom}), the symplectic form $\langle \ , \ \rangle_{E}$ is compatible with $\nabla$: for every sections $\alpha,\beta$ of $\mathcal{H}^1_{\dR}(X/M)$, and every holomorphic vector field $\theta$ on $M$, we have
\begin{align} \label{compconn}
\theta \langle \alpha,\beta\rangle_E = \langle \nabla_{\theta}\alpha,\beta \rangle_{E} + \langle \alpha,\nabla_{\theta}\beta\rangle_E\text{.}
\end{align}
  
\subsection{Relative uniformization of complex abelian schemes} \label{unifabsch}

Let $U$ be a smooth separated $\CC$-scheme of finite type and $(X,\lambda)$ be a principally polarized abelian scheme over $U$ of relative dimension $g$. Denote by $p: X \to U$ its structural morphism. Then the associated analytic space $U^{\an}$ is a complex manifold, and the analytification $p^{\an}:X^{\an} \to U^{\an}$ of $p$ is a complex torus over $U^{\an}$ of relative dimension $g$.

Since the analytification of the coherent $\mathcal{O}_U$-module $H^1_{\dR}(X/U)$ is canonically isomorphic to $\mathcal{H}^1_{\dR}(X^{\an}/U^{\an})$, the symplectic form $\langle \ , \ \rangle_{\lambda}$ on $H^1_{\dR}(X/U)$  induces a symplectic form $\langle \ , \ \rangle_{\lambda}^{\an}$ on the holomorphic vector bundle $\mathcal{H}^1_{\dR}(X^{\an}/U^{\an})$ over $U^{\an}$.

\begin{lemma} \label{compsympl}
Let $\gamma$ and $\delta$ be sections of $R_1p^{\an}_*\ZZ_{X^{\an}}$, and let $\alpha$ and $\beta$ be sections of $\mathcal{H}^1_{\dR}(X^{\an}/U^{\an})$ such that $\gamma = \langle \ \ , \alpha\rangle^{\an}_{\lambda}$ and $\delta = \langle \ \ , \beta \rangle^{\an}_{\lambda}$ under (the dual of) the comparison isomorphism (\ref{compisom}). Then
\begin{enumerate}
   \item The formula
\begin{align*}
E_{\lambda}(\gamma,\delta) \defeq \frac{1}{2\pi i}\langle \alpha,\beta \rangle^{\an}_{\lambda}
\end{align*}
defines a Riemann form over $X^{\an}$.
    \item The holomorphic symplectic forms $\langle \ , \ \rangle_{E_{\lambda}}$ and $\langle \ , \ \rangle_{\lambda}^{\an}$ over $\mathcal{H}^1_{\dR}(X^{\an}/U^{\an})$ coincide.
\end{enumerate} 
\end{lemma}

\begin{proof}
We can assume $U=\Spec \CC$, so that $(X,\lambda)$ is a principally polarized complex abelian variety.  

Recall from Paragraph \ref{shb2} that we have constructed  an alternating bilinear form $E_{\lambda}^{\dR}$ on $H^1_{\dR}(X/\CC)^{\vee}$, and that the bilinear form $\langle \ , \ \rangle_{\lambda}$ over $H^1_{\dR}(X/\CC)$ is obtained from $E_{\lambda}^{\dR}$ by duality. Therefore, to prove (1), it is sufficient to prove that, under the identification of $H_1(X^{\an},\ZZ)$ with an abelian subgroup of $H^1_{\dR}(X/\CC)^{\vee}$ via (the dual of) the comparison isomorphism (\ref{compisom}), for any elements $\gamma$ and $\delta$ of $H_1(X^{\an},\ZZ)$, 
\begin{align*}
E_{\lambda}(\gamma,\delta) \defeq \frac{1}{2\pi i}  E_{\lambda}^{\dR}(\gamma,\delta)
\end{align*}
is in $\ZZ$, and that the induced morphism
\begin{align*}\tag{$*$}
H_1(X^{\an},\ZZ) &\to  \Hom(H_1(X^{\an},\ZZ),\ZZ)\\
\gamma &\mapsto E_{\lambda}( \gamma , \ )
\end{align*}
is an isomorphism of abelian groups.

Note that, with this definition, (2) is automatic, since for any $\gamma,\delta \in H_1(X^{\an},\ZZ)$ we have
\begin{align*}
\langle E_{\lambda}(\gamma, \ ), E_{\lambda}(\delta, \ )\rangle_{E_{\lambda}} &= \frac{1}{2\pi i}E_{\lambda}(\gamma,\delta) = \frac{1}{(2\pi i)^2}E_{\lambda}^{\dR}(\gamma,\delta) = \frac{1}{(2\pi i)^2}\langle E_{\lambda}^{\dR}(\gamma, \ ),E_{\lambda}^{\dR}(\delta, \ ) \rangle^{\an}_{\lambda}\\
             &= \langle \frac{1}{2\pi i}E_{\lambda}^{\dR}(\gamma, \ ), \frac{1}{2\pi i}E_{\lambda}^{\dR}(\delta, \ )\rangle^{\an}_{\lambda} = \langle E_{\lambda}(\gamma, \ ),E_{\lambda}(\delta, \ )\rangle^{\an}_{\lambda}\text{.}
\end{align*}
where we identified the vector space $H^1_{\dR}(X/\CC)$ with $\mathcal{H}^1_{\dR}(X^{\an})$ via the canonical analytification isomorphism. 

Now, the topological Chern class $c_{1,\text{top}}:\Pic(X) \longrightarrow H^2(X^{\an},\ZZ)$, defined via the exponential sequence
\begin{align*}
0 \longrightarrow \ZZ_{X^{\an}} \longrightarrow \mathcal{O}_{X^{\an}} &\longrightarrow \mathcal{O}_{X^{\an}}^{\times} \longrightarrow 0\\
 f&\mapsto \exp(2\pi i f)
\end{align*}
and the de Rham Chern class $c_{1,\text{dR}} : \Pic(X) \to H^2_{\dR}(X/\CC)$ (cf. Paragraph \ref{shb2}) are related by the following commutative diagram (cf. \cite{deligne71} 2.2.5.2)
%$$
%  \raisebox{-0.5\height}{\includegraphics{rameq2-d3.pdf}}
%$$
 $$
 \begin{tikzcd}
 \Pic(X) \arrow{d}[swap]{c_{1,\text{top}}} \arrow{r}{c_{1,\text{dR}}} & H^2_{\dR}(X/\CC)\arrow{d} \\
 H^2(X^{\an},\ZZ) \arrow{r}[swap]{-2\pi i} & H^2(X^{\an},\CC)
 \end{tikzcd}
 $$
where the arrow $H^2_{\dR}(X/\CC) \longrightarrow H^2(X^{\an},\CC) \cong \Hom(H_2(X^{\an},\ZZ),\CC)$ is given by the comparison isomorphism.

If $\mathcal{L}$ is an ample line bundle on $X$ inducing $\lambda$, then $E_{\lambda}^{\dR} = c_{1,\dR}(\mathcal{L})$ under the identification $H^2_{\dR}(X/\CC)$ with the vector space of alternating bilinear forms on $H^1_{\dR}(X/\CC)^{\vee}$ (cf. proof of Lemma \ref{alternating}). By the commutativity of the above diagram, we see that $E_{\lambda} = -c_{1,\text{top}}(\mathcal{L})$ under the identification of $H^2(X^{\an},\ZZ)$ with the module of alternating (integral) bilinear forms on $H_1(X^{\an},\ZZ)$. This proves that $E_{\lambda}$ takes integral values.

To prove that  $(*)$ is an isomorphism, we simply use the fact that $\lambda^{\an}$ is an isomorphism of $X^{\an}$ onto its dual torus, hence the determinant of the bilinear form on $H_1(X^{\an},\ZZ)$ induced by $c_{1,\text{top}}(\mathcal{L})$ is 1 (cf. \cite{BL04} 2.4.9).
\end{proof}

Thus, for any smooth separated $\CC$-scheme of finite type $U$ and any principally polarized abelian scheme $(X,\lambda)$ over $U$ of relative dimension $g$, the above construction gives a principally polarized complex torus $(X^{\an},E_{\lambda})$ over $U^{\an}$ of relative dimension $g$.

Let $\SmVar_{/\CC}$ be the full subcategory of $\Sch_{/\CC}$ consisting of smooth separated $\CC$-schemes of finite type, and $\mathcal{A}_{g,\CC}^{\text{sm}}$ be the full subcategory of $\mathcal{A}_{g,\CC}$ consisting of objects $(X,\lambda)_{/U}$ of $\mathcal{A}_{g,\CC}$ such that $U$ is an object of $\SmVar_{/\CC}$. 

We can summarize this paragraph by remarking that we have constructed a ``relative uniformization functor'' $\mathcal{A}_{g,\CC}^{\text{sm}} \to \mathcal{A}_g^{\an}$ making the diagram
%$$
%  \raisebox{-0.5\height}{\includegraphics{rameq2-d4.pdf}}
%$$
 $$
 \begin{tikzcd}
 \mathcal{A}_{g,\CC}^{\text{sm}}\arrow{r}\arrow{d} & \mathcal{A}_g^{\an}\arrow{d}\\
 \SmVar_{/\CC} \arrow{r} & \Man_{/\CC}
 \end{tikzcd}
 $$
(strictly) commutative, where  $\SmVar_{/\CC}  \to \Man_{/\CC}$ is the classical analytification functor $U\mapsto U^{\an}$.

\begin{obs} \label{algebraization}
One can prove that the above diagram is ``Cartesian'' in the sense that it induces an equivalence of categories between $\mathcal{A}_{g,\CC}^{\text{sm}}$ and the full subcategory of $\mathcal{A}_g^{\an}$ formed by the objects lying above the essential image of the analytification functor $\SmVar_{/\CC} \to \Man_{/\CC}$  (cf. \cite{deligne71} Rappel 4.4.3 and \cite{borel72} Theorem 3.10). In particular, for any object $U$ of $\SmVar_{/\CC}$ and any principally polarized complex torus $(X',E)$ over $U^{\an}$ of relative dimension $g$, there exists up to isomorphism a unique principally polarized abelian scheme $(X,\lambda)$ over $U$ of relative dimension $g$ such that $(X',E)_{/U^{\an}}$ is isomorphic to $(X^{\an},E_{\lambda})_{/U^{\an}}$ in $\mathcal{A}_g^{\an}(U^{\an})$. In this paper, we shall only need this algebraization result when  $U=\Spec \CC$, which is classical (cf. \cite{mumford70} Corollary p. 35).
\end{obs}

\subsection{Principally polarized complex tori with real multiplication}\label{ssec-ppctrm}

Recall that $F$ denotes a totally real number field of degree $g$ with ring of integers $R$ and inverse different ideal $D^{-1}$.

For a complex manifold $M$, we may also consider \emph{principally polarized complex tori with $R$-multiplication over $M$}. By this we mean a triple $(X,E,m)_{/M}$, where $(X,E)$ is a principally polarized complex torus of relative dimension $g$ over $M$, and $m: R \to \End_M(X)$ is a ring morphism such that, for every $r\in R$, and every sections $v,w$ of $\Lie_M X$,
$$
E(\Lie m(r) (v),w) = E(v,\Lie m(r) (w))\text{.}
$$

\label{symb:powerH}
\begin{ex}\label{ex-ppctrm}
 Consider the complex manifold
 $$
 \mathbf{H}^g = \{\tau=(\tau_1,\ldots,\tau_g) \in \CC^n \mid \Im \tau_j>0 \text{, }1\le j\le g\}\text{.}
 $$
 Let $V \defeq \CC^g\times \mathbf{H}^g$ be the trivial vector bundle over $\mathbf{H}^g$, and $L$ be the subsheaf of $\mathcal{O}_{\mathbf{H}^g}(V)$ given by the image of the morphism of sheaves of abelian groups
 \begin{align*}
  (D^{-1}\oplus R)_{\mathbf{H}^g}&\to \mathcal{O}_{\mathbf{H}^g}(V) =\mathcal{O}_{\mathbf{H}^g}^{\oplus g}\\
                           (x,y)&\mapsto x + \tau y  \defeq (\sigma_j(x) + \tau_j\sigma_j(y))_{1\le j \le g}
 \end{align*}
 where $\sigma_1,\ldots,\sigma_g$ are the field embeddings of $F$ into $\CC$. Then $L$ is a lattice in $V$ and we denote by
 $$
 \bfp_F: \mathbf{X}_F \to \mathbf{H}^g
 $$
 the corresponding complex torus over $\mathbf{H}^g$ of relative dimension $g$. Let $E_F$ be the imaginary part of the Hermitian metric over $V$ given by
 $$
 (v,w)\mapsto \sum_{j=1}^g\frac{\overline{v}_jw_j}{\Im \tau_j}\text{.}
 $$
 Then $E_F$ defines a principal Riemann form on $\mathbf{X}_F$. The action of $R$ on $L$ given by its natural action on $D^{-1}\oplus R$ via the above isomorphism induces an $R$-multiplication $m_F$ on the principally polarized complex torus $(\mathbf{X}_F,E_F)$. We thus obtain a principally polarized complex torus with $R$-multiplication $(\mathbf{X}_F,E_F,m_F)_{/\mathbf{H}^g}$.
\end{ex}\label{symb:XFEFmF}

Let $(X,E,m)_{/M}$ be a principally polarized complex torus with $R$-multiplication, with structural morphism $\pi:X \to M$. Then $m$ induces an action of $R$ on the holomorphic vector bundle $\Lie_M X$ making its sheaf of holomorphic sections a locally free $\mathcal{O}_M\tensor R$-module of rank 1 (that is, Rapoport's condition is automatically satisfied; see the remark following Definition \ref{defi-ppasrm}).  We denote by
$$
\Phi_E : {\Lie}_M X \times {\Lie}_M X \to \RR_M\tensor D^{-1}
$$
the unique $R_M$-bilinear form such that $\Tr \Phi_E = E$ (cf. Remark \ref{rem-dualityrelation}). We also have a compatible action of $R$ on the lattice $R_1\pi_*\ZZ_X$, making it a locally free $R_M$-module of rank $2g$; the restriction of $\Phi_E$ to $R_1\pi_*\ZZ_X$ is a $D^{-1}_M$-valued integral $R_M$-bilinear symplectic form.

Let $\Psi_E$ be the unique $D^{-1}$-valued $\mathcal{O}_M\tensor R$-bilinear symplectic form on $\mathcal{H}^1_{\dR}(X/M)$ satisfying $\Tr \Psi_E = \langle \ , \ \rangle_E$. By unicity, $\Psi_E$ satisfies
$$
\Psi_E(\Phi_E(\gamma, \ ),\Phi_E(\delta,\ )) = \frac{1}{2\pi i}\Phi_E(\gamma,\delta)
$$
for every sections $\gamma,\delta$ of $R_1\pi_*\ZZ_X$; here, we use that the comparison isomorphism (\ref{compisom}) is $R$-linear, and we regard $\Phi_E(\gamma, \ ), \Phi_E(\delta, \ )$ as sections of $\mathcal{H}^1_{\dR}(X/M)$.

The category fibered in groupoids over $\Man_{/\CC}$ of principally polarized complex tori with $R$-multiplication, defined in an obvious way, is denoted by $\mathcal{A}_F^{\an}$.

\begin{ex}\label{ex-actionsl}
 Let
 $$
 \SL(D^{-1}\oplus R) \defeq \left.\left\{\left(\begin{array}{cc}
                               a & b \\
                               c & d
                              \end{array}\right)\in {\SL}_2(F)\right| a,d\in R\text{, }b\in D^{-1}\text{, }c\in D \right\}\text{.}
 $$
 Alternatively, $\SL(D^{-1}\oplus R)$ can be defined as $\Res_{R/\ZZ}\Aut_{(D^{-1}\oplus R, \Phi)}(\ZZ)$, where $\Phi$ denotes the standard $D^{-1}$-valued $R$-bilinear symplectic form on $D^{-1}\oplus R$. As in Example \ref{exactionsp}, we may define a group action
 \begin{align*}
  \SL(D^{-1}\oplus R) &\to {\Aut}_{\mathcal{A}_F^{\an}}((\mathbf{X}_F,E_F,m_F)_{/\mathbf{H}^g})\\
                \gamma &\mapsto {\varphi_{\gamma}}_{/\gamma}
 \end{align*}
by the following explicit formulas: the left action of $\SL(D^{-1}\oplus R)$ on $\mathbf{H}^g$ is given by
$$
\left(\begin{array}{cc}
                               a & b \\
                               c & d
                              \end{array}\right)\cdot \tau  = \left(\frac{\sigma_1(a)\tau_1 + \sigma_1(b)}{\sigma_1(c)\tau_1 + \sigma_1(d)},\ldots,\frac{\sigma_g(a)\tau_g + \sigma_g(b)}{\sigma_g(c)\tau_g + \sigma_g(d)} \right)
$$
where $\sigma_1,\ldots,\sigma_g$ denote the field embeddings of $F$ into $\CC$, and, for $\tau\in \mathbf{H}^g$, the isomorphism
$$
\varphi_{\gamma,\tau} : \mathbf{X}_{F,\tau}\stackrel{\sim}{\to}\mathbf{X}_{F,\gamma\cdot\tau}
$$
is induced by
\begin{align*}
\tilde{\varphi}_{\gamma,\tau} : \CC^g &\to \CC^g\\
                      z&\mapsto \frac{z}{c\tau+d} \defeq \left(\frac{z_1}{\sigma_1(c)\tau_1+\sigma_1(d)},\ldots,\frac{z_g}{\sigma_g(c)\tau_g+\sigma_g(d)} \right)\text{.}
\end{align*}
\end{ex}

\vspace{10pt}
 
Finally, to a principally polarized abelian scheme with $R$-multiplication $(X,\lambda,m)_{/U}$, with $U$ a smooth separated $\CC$-scheme of finite type, we may functorially associate the object $(X^{\an},E_{\lambda},m^{\an})_{/U^{\an}}$ of $\mathcal{A}_F^{\an}$. We remark that by Lemma \ref{compsympl}, and by the unicity of the $D^{-1}$-valued $\mathcal{O}_{U^{\an}}\tensor R$-bilinear symplectic forms, we have
$$
\Psi_{E_{\lambda}} = \Psi_{\lambda}^{\an}\text{.}
$$

\section{Analytic moduli spaces of complex abelian varieties with a symplectic-Hodge basis}

In this section we consider some moduli problems of principally polarized complex tori, regarded as functors
\begin{align*}
(\mathcal{A}_g^{\an})^{\opp} \to \Sets \text{ (resp. }(\mathcal{A}_F^{\an})^{\opp} \to \Sets\text{)}
\end{align*}
where $\mathcal{A}_g^{\an}$ (resp. $\mathcal{A}_F^{\an}$) is the category fibered in groupoids over the category of complex manifolds $\Man_{/\CC}$ defined in Paragraph \ref{defitg} (resp. \ref{ssec-ppctrm}). As usual, we provide a detailed account for the Siegel case $\mathcal{A}_g^{\an}$, and merely indicate the necessary modifications to treat the Hilbert-Blumenthal case $\mathcal{A}_F^{\an}$.

%Recall that we denote by $B_g$ the smooth quasi-projective scheme over $\ZZ[1/2]$ representing the stack $\mathcal{B}_g\tensor_{\ZZ} \ZZ[1/2]$ (see \cite{fonseca16} Theorem 4.1). We shall prove in particular that the complex manifold $B_g(\CC) = B_{g,\CC}^{\an}$ is a fine moduli space in the analytic category for principally polarized complex abelian varieties of dimension $g$ endowed with a symplectic-Hodge basis.

\subsection{Descent of principally polarized complex tori} 

Let $M$ be a complex manifold and $(X,E)$ be a principally polarized complex torus over $M$ of relative dimension $g$. 

If $M_0$ is another complex manifold and $M \to M_0$ is a holomorphic map, we say that $(X,E)$ \emph{descends} to $M_0$ if there exists a principally polarized complex torus $(X_0,E_0)$ over $M_0$ and a morphism $(X,E)_{/M} \to (X_0,E_0)_{/M_0}$ in $\mathcal{A}_g^{\an}$.

% We shall be mainly concerned with the case 
% Let $M$ be a complex manifold and $\Gamma$ be an abstract group acting on $M$. Let us recall that an action of $\Gamma$ on $M$ is said to be \emph{properly discontinuous} if, for every $p \in M$, there exists an open neighborhood $U$ of $p$ in $M$ such that $\{\gamma \in \Gamma \mid \gamma \cdot U \cap U \neq \emptyset\}$ is a finite set. 

% If a group $\Gamma$ acts freely and properly discontinuously on a complex manifold $M$, then the topological quotient $M/\Gamma$ has a unique structure of a complex manifold such that the canonical map $M \to M/\Gamma$ is biholomorphic. Let us also observe that $M \to M/\Gamma$ is a topological covering with Galois group $\Gamma$.

\begin{lemma} \label{descentlemma}
With the above notation, suppose that there exists a proper and free left action of a discrete group $\Gamma$ on $M$. If the action of $\Gamma$ on $M$ lifts to an action of $\Gamma$ on $(X,E)_{/M}$ in the category $\mathcal{A}_g^{\an}$, then $(X,E)_{/M}$ descends to a principally polarized complex torus over the quotient $\Gamma \backslash M$. 
\end{lemma} 

\begin{proof}[Sketch of the proof]
  Consider $X$ as a pair $(V,L)$, where $V$ is a holomorphic vector bundle over $M$ of rank $g$, and $L$ is a lattice in $V$ (cf. Remark \ref{remarktorus}). Then, to every $\gamma \in \Gamma$ there is associated a holomorphic map $\varphi_{\gamma} :V \to V$ making the diagram
%$$
%  \raisebox{-0.5\height}{\includegraphics{rameq2-d5.pdf}}
%$$
 $$
 \begin{tikzcd}
 V \arrow{r}{\varphi_{\gamma}} \arrow{d} & V \arrow{d} \\
 M \arrow{r}[swap]{\gamma} & M
 \end{tikzcd}
 $$
commute, and compatible with the vector bundle structures. It follows from the commutativity of this diagram that the action of $\Gamma$ on $V$ is also proper and free. Thus, there exists a unique holomorphic vector bundle structure on the complex manifold $\Gamma \backslash V$ over $\Gamma \backslash M$ such that the canonical holomorphic map $V \to \Gamma \backslash V$ induces a vector bundle isomorphism of $V$ onto the pullback to $M$ of the vector bundle $\Gamma \backslash V$ over $\Gamma \backslash M$.

Analogously, one descends the lattice $L$ to a lattice in $\Gamma \backslash V$ (consider the étalé space, for instance), and the bilinear form $E$ on $V$ to a bilinear form on $\Gamma \backslash V$, which is seen to be a principal polarization \emph{a posteriori}.
\end{proof}

\begin{obs}\label{rem-descentlemma}
It is not difficult to check that an analogous statement holds for principally polarized complex tori with $R$-multiplication: if a proper and free action of a discrete group $\Gamma$ on a complex manifold $M$ lifts to an action on a principally polarized complex torus with $R$-multiplication $(X,E,m)$ over $M$, then $(X,E,m)_{/M}$ descends to a principally polarized complex torus with $R$-multiplication over $\Gamma\backslash M$.
\end{obs}

\subsection{Integral symplectic bases over principally polarized complex tori} 

Let $M$ be a complex manifold and $(X,E)$ be a principally polarized complex torus over $M$ of relative dimension $g$. We denote by $\pi : X \to M$ its structural morphism. 

\begin{defi}
An \emph{integral symplectic basis} of $(X,E)_{/M}$ is a trivializing $2g$-uple $(\gamma_1,\ldots,\gamma_g,\delta_1,\ldots,\delta_g)$ of global sections of $R_1\pi_*\ZZ_X$ which is symplectic with respect to the Riemann form $E$, that is,
\begin{align*}
E(\gamma_i,\gamma_j) = E(\delta_i,\delta_j)=0\ \ \text{  and } \ \ E(\gamma_i,\delta_j)=\delta_{ij}
\end{align*}
for any $1\le i, j\le g$.
\end{defi}

\begin{ex} \label{intsymplbasis}
Consider the principally polarized complex torus $(\mathbf{X}_g,E_g)$ over $\mathbf{H}_g$ of Example \ref{torus} and recall that a section of $R_1{\bfp_g}_*\ZZ_{\mathbf{X}_g}$ is given by a column vector of holomorphic functions on $\mathbf{H}_g$ of the form $\tau \mapsto m + \tau n$, for some sections $(m,n)$ of $(\ZZ^{g}\oplus \ZZ^g)_{\mathbf{H}_g}$.  We can thus define an integral symplectic basis 
\begin{align*}
\beta_g = (\gamma_1,\ldots,\gamma_g,\delta_1,\ldots,\delta_g)
\end{align*}
of $(\mathbf{X}_g,E_g)_{/\mathbf{H}_g}$ by 
\begin{align*}
\gamma_i(\tau) \defeq \textbf{e}_i\ \ \text{ and }\ \  \delta_i(\tau)\defeq \tau \textbf{e}_i
\end{align*}
for any $\tau \in \mathbf{H}_g$.
\end{ex}

Let $(X',E')_{/M'}$ and $(X,E)_{/M}$ be objects of $\mathcal{A}_g^{\an}$ with structural morphisms $\pi': X' \to M'$ and $\pi : X \to M$. If $\varphi_{/f} : (X',E')_{/M'} \to (X,E)_{/M}$ is a morphism in $\mathcal{A}_g^{\an}$, then the isomorphism of vector bundles 
\begin{align} \label{isomorphism}
{\Lie}_{M'}X' \stackrel{\sim}{\to} f^*{\Lie}_MX
\end{align}
induced by $\varphi$ identifies the lattice $R_1\pi'_*\ZZ_{X'}$ with $f^*R_1\pi_*\ZZ_X$. If $\gamma$ is a section of $R_1\pi_*\ZZ_X$, we denote by $\varphi^*\gamma$ the section of $R_1\pi'_*\ZZ_{X'}$ mapping to $f^*\gamma$ under (\ref{isomorphism}). As the isomorphism (\ref{isomorphism}) also preserves the corresponding Riemann forms, for any integral symplectic basis $(\gamma_1,\ldots,\gamma_g,\delta_1,\ldots,\delta_g)$ of $(X,E)_{/M}$, the $2g$-uple of global sections of $R_1\pi'_*\ZZ_{X'}$ given by
\begin{align*}
\varphi^*\beta \defeq (\varphi^{*}\gamma_1,\ldots,\varphi^*\gamma_g,\varphi^*\delta_1,\ldots,\varphi^*\delta_g)
\end{align*}
is an integral symplectic basis of $(X',E')_{/M'}$. 

\begin{prop}[cf. \cite{BL04} Proposition 8.1.2] \label{reprsintsympl}
  The functor $(\mathcal{A}_g^{\an})^{\opp} \longrightarrow \mathsf{Set}$ sending an object $(X,E)_{/M}$ of $\mathcal{A}_g^{\an}$ to the set of integral symplectic bases of $(X,E)_{/M}$ is representable by $(\mathbf{X}_g,E_g)_{/\mathbf{H}_g}$, with universal integral symplectic basis $\beta_g$ defined in Example \ref{intsymplbasis}.
\end{prop}

\begin{proof}
Let $(X,E)_{/M}$ be an object of $\mathcal{A}_g^{\an}$ with structural morphism $\pi:X \to M$, and let $\beta = (\gamma_1,\ldots,\gamma_g,\delta_1,\ldots,\delta_g)$ be an integral symplectic basis of $(X,E)_{/M}$. Let $W$ be the real subbundle of $\Lie_MX$ generated by $\gamma_1,\ldots,\gamma_g$. Since $E$ is the imaginary part of a Hermitian metric, for any nontrivial section $\gamma$ of $W$, we have $E(\gamma,i\gamma)\neq 0$. As $W$ is isotropic with respect to $E$, it follows that $\Lie_M X = W \oplus iW$ as a real vector bundle. In particular, $\gamma \defeq (\gamma_1,\ldots,\gamma_g)$ trivializes $\Lie_M X$ as a holomorphic vector bundle. Hence, if $\delta \defeq  (\delta_1,\ldots,\delta_g)$, then there exists a unique holomorphic map $\tau : M \to \GL_g(\CC)$ such that $\delta = \gamma \tau$, where $\gamma$ and $\delta$ are regarded as row vectors of global holomorphic sections of $\Lie_M X$. 

Let $A \defeq (E(\gamma_k,i\gamma_l))_{1\le k ,l\le g} \in M_{g\times g}(\CC)$. Since 
$$
\delta = \gamma \Re \tau + i \gamma \Im \tau\text{,}
$$
the matrix of $E$ in the basis $\beta$ is given by
\begin{align*}
\left(\begin{array}{cc}
      0 & A\Im \tau\\
      -(A \Im\tau)\transp & (\Re \tau)\transp A \Im \tau - (\Im \tau)\transp A\transp \Re\tau
      \end{array} \right)\text{.}
\end{align*}
Using that $\beta$ is symplectic with respect to $E$, and that $A$ is symmetric and positive-definite (recall that $E$ is the imaginary part of a Hermitian metric), we conclude that $\tau$ factors through $\mathbf{H}_g\subset \GL_g(\CC)$.

Finally, writing $X$ as the quotient of $\Lie_M X$ by $R_1\pi_*\ZZ_X$, we see that $\tau$ lifts to a unique morphism in $\mathcal{A}_g^{\an}$
\begin{align*}
\varphi_{/\tau} : (X,E)_{/M} \to (\mathbf{X}_g,E_g)_{/\mathbf{H}_g}
\end{align*}
satisfying $\varphi^*\beta_g = \beta$. %cf debarre VI 1.3
\end{proof}

\begin{obs}\label{actionsp}
We may define a \emph{left} action of the group $\Sp_{2g}(\ZZ)$ on the functor $(\mathcal{A}_g^{\an})^{\opp} \to \Sets$ of integral symplectic bases, considered in the above proposition, as follows. Let $(X,E)_{/U}$ be an object of $\mathcal{A}_g^{\an}$ and $\beta$ be an integral symplectic basis of $(X,E)_{/U}$. Let $\gamma =( A \ \ B \ ; \ C \ D) \in \Sp_{2g}(\ZZ)$, and consider $\beta = (\gamma_1, \ldots,\gamma_g,\delta_1,\ldots,\delta_g)$ as a row vector of order $2g$; then we define
\begin{align*}
\gamma \cdot \beta \defeq (\begin{array}{cccccc}\gamma_1 & \cdots & \gamma_g & \delta_1 & \cdots & \delta_g\end{array})
\left(\begin{array}{cc}
D\transp & B\transp \\
C\transp & A\transp
\end{array}\right)
\end{align*}
The morphism
\begin{align*}
{\varphi_{\gamma}}_{/f_{\gamma}}: (\mathbf{X}_g,E_g)_{/\mathbf{H}_g} \to (\mathbf{X}_g,E_g)_{/\mathbf{H}_g}
\end{align*}
defined in Example \ref{exactionsp} is the unique morphism in $\mathcal{A}_g^{\an}$ satisfying
\begin{align*}
\varphi_{\gamma}^*\beta_g = \gamma \cdot \beta_g\text{.}
\end{align*}
\end{obs}

\subsection{Principal (symplectic) level structures} \label{level}

\subsubsection{} \label{pls}

Let $U$ be a scheme, and $X$ be an abelian scheme over $U$. Recall that, for any integer $n\ge 1$, we may define a natural pairing, the so-called \emph{Weil pairing},
\begin{align*}
X[n] \times X^t[n] \to \mu_{n,U}\text{,}
\end{align*}
where $\mu_{n,U}$ denotes the $U$-group scheme of $n$th roots of unity (cf. \cite{mumford70} IV.20).

Fix an integer $n\ge 1$, and let $\zeta_n\in \CC$ be the $n$th root of unity $e^{\frac{2\pi i}{n}}$. For any scheme $U$ over $\ZZ[1/n,\zeta_n]$, and any principally polarized abelian scheme $(X,\lambda)$ over $U$ of relative dimension $g$, by identifying $X^t[n]$ with $X[n]$ via $\lambda$, and $\mu_{n,U}$ with $(\ZZ/n\ZZ)_U$ via $\zeta_n$, we obtain a pairing
\begin{align*}
e^{\lambda}_n: X[n] \times X[n] \to (\ZZ/n\ZZ)_U\text{.}
\end{align*} 
The formation of $e^{\lambda}_n$ is compatible with every base change in $U$. Moreover, $e^{\lambda}_n$ is skew-symmetric and non-degenerate (cf. \cite{mumford70} IV.23). %By a \emph{symplectic basis} of $X[n]$ we mean a $2g$-uple \linebreak $(P_1,\ldots,P_g,Q_g,\ldots,Q_g)$ of global sections of the $U$-scheme $X[n]$ satisfying $e_n^{\lambda}(P_i,P_j)=e_{n}^{\lambda}(Q_i,Q_j)=0$ and $e_n^{\lambda}(P_i,Q_j)=\delta_{ij}$ for every $1\le i \le j \le g$.

Since, for any integer $n\ge 3$, there exists a fine moduli space $A_{g,1,n}$ over $\ZZ[1/n]$ for principally polarized abelian varieties of dimension $g$ endowed with a full level $n$-structure (see \cite{GIT94} Theorem 7.9, and the following remark; see also \cite{moret-bailly85} Théorème VII.3.2), there also exists a fine moduli space $A_{g,n}$ over $\ZZ[1/n,\zeta_n]$ for principally polarized abelian varieties $(X,\lambda)$ of dimension $g$ endowed with a symplectic basis of $X[n]$ for the pairing $e^{\lambda}_n$ (cf. \cite{FC90} IV.6).  The scheme $A_{g,n}$ is quasi-projective and smooth over $\ZZ[1/n,\zeta_n]$, with connected fibers. 

In the sequel, we denote the universal principally polarized abelian scheme over $A_{g,n}$ by $(X_{g,n},\lambda_{g,n})$, and the universal symplectic basis of $X_{g,n}[n]$ by $\alpha_{g,n}$.

\subsubsection{}

Let $(X,E)_{/M}$ be an object of $\mathcal{A}_g^{\an}$ with structural morphism $\pi:X \to M$. For any integer $n\ge 1$, by an \emph{integral symplectic basis modulo $n$} of $(X,E)_{/M}$, we mean a $2g$-uple of global sections of the local system of $\ZZ/n\ZZ$-modules 
\begin{align*}
R_1\pi_*(\ZZ/n\ZZ)_X = R_1\pi_*\ZZ_{X}/nR_1\pi_*\ZZ_{X}
\end{align*}
which is symplectic with respect to the alternating $\ZZ/n\ZZ$-linear form on $R_1\pi_*(\ZZ/n\ZZ)_X$ induced by $E$.

\begin{obs} \label{localift}
Every integral symplectic basis of $(X,E)_{/M}$ induces an integral symplectic basis modulo $n$ of $(X,E)_{/M}$. Conversely, since the natural map ${\Sp}_{2g}(\ZZ) \to {\Sp}_{2g}(\ZZ/n\ZZ)$ is surjective, locally on $M$, every integral symplectic basis modulo $n$ of $(X,E)_{/M}$ can be lifted to an integral symplectic basis of $(X,E)_{/M}$. 
\end{obs}

The notion of integral symplectic bases modulo $n$ is compatible with the notion of principal level $n$ structures of \ref{pls} in the following sense. Let $(X,\lambda)_{/U}$ be an object of $\mathcal{A}_{g,\CC}^{\text{sm}}$ (see Paragraph \ref{unifabsch}) with structural morphism $p:X \to U$. The étalé space of the local system $R_1 p^{\an}_*(\ZZ/n\ZZ)_{X^{\an}}$ is canonically isomorphic to the $n$-torsion Lie subgroup $X^{\an}[n]$ of $X^{\an}$. Under this identification, the pairing $e_n^{\lambda}$ on $X[n]$ coincides, up to a sign, with the reduction modulo $n$ of the Riemann form $E_{\lambda}$ (cf. \cite{mumford70} IV.23 and IV.24), and thus an integral symplectic basis modulo $n$ of $(X^{\an},E_{\lambda})_{/U^{\an}}$ canonically corresponds  to a symplectic trivialization of $X^{\an}[n]$ with respect to $e_n^{\lambda}$.

\subsubsection{} Let $\Gamma(n)$ the kernel of the natural map ${\Sp}_{2g}(\ZZ) \to {\Sp}_{2g}(\ZZ/n\ZZ)$. Recall that for any $n\ge 3$ the induced action of $\Gamma(n)$ on $\mathbf{H}_g$ is free (\cite{mumford70} IV.21 Theorem 5) and proper.

\begin{prop}[cf. \cite{BL04} Theorem 8.3.2] \label{levelstruc}
For any integer $n\ge 3$, the complex manifold $A_{g,n}(\CC)=A_{g,n,\CC}^{\an}$  is canonically biholomorphic to the quotient of $\mathbf{H}_g$ by $\Gamma(n)$, and the functor $(\mathcal{A}_g^{\an})^{\opp} \longrightarrow \mathsf{Set}$ sending an object $(X,E)_{/M}$ of $\mathcal{A}_g^{\an}$ to the set of integral symplectic bases modulo $n$ of $(X,E)_{/M}$ is representable by $(X_{g,n,\CC}^{\an},E_{\lambda_{g,n}})_{/A_{g,n,\CC}^{\an}}$. 
\end{prop}

\begin{proof}
As the action of $\Gamma(n)$ on $\mathbf{H}_g$ is proper and free, the quotient
\begin{align*}
\mathbf{A}_{g,n} \defeq \Gamma(n)\backslash \mathbf{H}_g
\end{align*}
is a complex manifold, and the canonical holomorphic map $\mathbf{H}_g \to \mathbf{A}_{g,n}$ is a covering map with Galois group $\Gamma(n)$. Moreover, since the action of $\Gamma(n)$ on $\mathbf{H}_g$ lifts to an action of $\Gamma(n)$ on $(\mathbf{X}_g,E_g)_{/\mathbf{H}_g}$ in the category $\mathcal{A}_g^{\an}$, the principally polarized complex torus $(\mathbf{X}_g,E_g)$ over $\mathbf{H}_g$ descends to a principally polarized complex torus $(\mathbf{X}_{g,n},E_{g,n})$ over $\mathbf{A}_{g,n}$ (Lemma \ref{descentlemma}).

 Let $\overline{\beta}_{g}$ be the integral symplectic basis modulo $n$ of $(\mathbf{X}_{g},E_g)_{/\mathbf{H}_g}$ obtained from $\beta_g$ by reduction modulo $n$. Then $\overline{\beta}_g$ is invariant under the action of $\Gamma(n)$, and thus it descends to an integral symplectic basis modulo $n$ of $(\mathbf{X}_{g,n},E_{g,n})_{/\mathbf{A}_{g,n}}$, say $\beta_{g,n}$. 

The object $(\mathbf{X}_{g,n},E_{g,n})_{/\mathbf{A}_{g,n}}$ of $\mathcal{A}_g^{\an}$ so constructed represents the functor in the statement with $\beta_{g,n}$ serving as universal symplectic basis modulo $n$. Indeed, let $(X,E)_{/M}$ be an object of $\mathcal{A}_g^{\an}$, and $\beta$ be an integral symplectic basis modulo $n$ of $(X,E)_{/M}$. By Remark \ref{localift}, there exists an open covering $M= \bigcup_{i\in I}U^i$ and, for each $i\in I$, an integral symplectic basis $\beta^i$ of $(X,E)_{/U^i}$ lifting $\beta$. By Proposition \ref{reprsintsympl}, we obtain for each $i\in I$ a morphism $\varphi^i_{/f^i} : (X,E)_{/U^i} \to (\mathbf{X}_g,E_g)_{/\mathbf{H}_g}$ in $\mathcal{A}_g^{\an}$ satisfying $(\varphi^i)^*\beta_g=\beta^i$. Finally, by construction, for any $i,j \in I$, the compositions of $\varphi^i_{/f^i}$ and $\varphi^j_{/f^j}$ with the projection $(\mathbf{X}_g,E_g)_{/\mathbf{H}_g} \to (\mathbf{X}_{g,n},E_{g,n})_{/\mathbf{A}_{g,n}}$ agree over the intersection $U^i\cap U^j$; hence they glue to a morphism 
$$
\varphi_{/f} : (X,E)_{/M} \to (\mathbf{X}_{g,n},E_{g,n})_{/\mathbf{A}_{g,n}}
$$ 
satisfying $\varphi^*\beta_{g,n}=\beta$, and uniquely determined by this property.

To finish the proof, it is sufficient to show that $(X_{g,n,\CC}^{\an},E_{\lambda_{g,n}})_{/A_{g,n,\CC}^{\an}}$ is isomorphic to $(\mathbf{X}_{g,n},E_{g,n})_{/\mathbf{A}_{g,n}}$ in the category $\mathcal{A}_g^{\an}$. By the compatibility of principal level $n$ structures with integral symplectic bases modulo $n$, there exists a unique morphism in $\mathcal{A}_g^{\an}$
\begin{align*}
\varphi_{/f} : (X_{g,n,\CC}^{\an},E_{\lambda_{g,n}})_{/A_{g,n,\CC}^{\an}} \to (\mathbf{X}_{g,n},E_{g,n})_{/\mathbf{A}_{g,n}}
\end{align*}
such that $\varphi^*\beta_{g,n}$ is the integral symplectic basis modulo $n$ of $(X_{g,n,\CC}^{\an},E_{\lambda_{g,n}})_{/A_{g,n,\CC}^{\an}}$ associated to $\alpha_{g,n}$ (the universal principal level $n$ structure of $(X_{g,n},\lambda_{g,n})_{/A_{g,n}}$). Since complex tori (over a point) endowed with a principal Riemann form are algebraizable (cf. Remark \ref{algebraization}), the holomorphic map
\begin{align*}
f: A_{g,n}(\CC) = A_{g,n,\CC}^{\an} \to \mathbf{A}_{g,n}
\end{align*}
is bijective. As the complex manifolds $\mathbf{A}_{g,n}$ and $A_{g,n}(\CC)$ have same dimension, $f$ is necessarily a biholomorphism (\cite{GH78} p. 19).
\end{proof}

\subsection{Symplectic-Hodge bases over complex tori} 

\subsubsection{} \label{defi-shbppct}

Let $M$ be a complex manifold and $(X,E)$ be a principally polarized complex torus over $M$ of relative dimension $g$. As in Definition \ref{defi-shb}, by a \emph{symplectic-Hodge basis} of $(X,E)_{/M}$, we mean a $2g$-uple $b=(\omega_1,\ldots,\omega_g,\eta_1,\ldots,\eta_g)$ of global sections of the holomorphic vector bundle $\mathcal{H}^1_{\dR}(X/M)$ such that $\omega_1,\ldots,\omega_g$ are sections of the subbundle $\mathcal{F}^1(X/M)$, and $b$ is symplectic with respect to the holomorphic symplectic form $\langle \ , \ \rangle_E$.

It follows from Lemma \ref{compsympl} that this notion of symplectic-Hodge basis is compatible with its algebraic counterpart via the ``relative uniformization functor'' in Paragraph \ref{unifabsch}. 

\subsubsection{} \label{principalbundle}

Consider Siegel parabolic subgroup of $\Sp_{2g}(\CC)$
\begin{align*}
P_g(\CC) = 
 \left.\left\{\left(\begin{array}{cc}
                           A & B \\
                           0 & (A\transp)^{-1}
                          \end{array} \right) \in M_{2g\times 2g}(\CC) \ \right|\  A \in {\GL}_g(\CC)\text{ and } B\in M_{g\times g}(\CC)\text{ satisfy }AB\transp=BA\transp \right\}\text{.}
\end{align*}
Note that $P_g(\CC)$ is a complex Lie group of dimension $g(3g+1)/2$.

Let $(X,E)$ be a principally polarized complex torus of dimension $g$. If $b = (  \omega \  \eta )$ is a symplectic-Hodge basis of $(X,E)$, seen as a row vector of order $2g$ with coefficients in $\mathcal{H}^1_{\dR}(X)$, and $p = (A \ B \ ; \ 0 \ (A\transp)^{-1}) \in P_g(\CC)$, then we put
\begin{align*}
  b \cdot p \defeq (\begin{array}{cc}\omega A & \omega B + \eta (A\transp)^{-1}
                   \end{array})\text{.}
\end{align*}
It is easy to check that $b\cdot p$ is a symplectic-Hodge basis of $(X,E)$, and that the above formula defines a free and transitive action of $P_g(\CC)$ on the set of symplectic-Hodge bases of $(X,E)$ (cf. Lemma \ref{torsor}).

\subsubsection{}
For a complex manifold $M$, let us denote by $\Man_{/M}$ the category of complex manifolds endowed with a holomorphic map to $M$.

\begin{lemma}[cf. Corollary \ref{relrepr0}] \label{relrepr}
Let $M$ be a complex manifold and $(X,E)$ be a principally polarized complex torus over $M$ of relative dimension $g$. The functor
\begin{align*}
\mathsf{Man}_{/M}^{\opp} &\to \mathsf{Set}\\
           M' & \mapsto \{\text{symplectic-Hodge bases of }(X,E)\times_M M'\}
\end{align*}
is representable by a principal $P_g(\CC)$-bundle $B(X,E)$ over $M$.
\end{lemma}

\begin{proof}
Let us denote by $\pi: V \to M$ the holomorphic vector bundle $\mathcal{H}^1_{\dR}(X/M)^{\oplus g}$ over $M$. For any $p\in M$, the fiber $\pi^{-1}(p)=V_p$ is the vector space of $g$-uples $(\alpha_1,\ldots,\alpha_g)$, with each $\alpha_i \in \mathcal{H}^1_{\dR}(X_p)$. Let $B$ be the locally closed analytic subspace of $V$ consisting of points $v=(\alpha_1,\ldots,\alpha_g)$ of $V$ such that 
\begin{align*}
L \defeq \CC \alpha_1 + \cdots + \CC \alpha_g
\end{align*}
is a Lagrangian subspace of $\mathcal{H}^1_{\dR}(X_{\pi(v)})$ with respect to $\langle \ , \ \rangle_{E_{\pi(v)}}$ satisfying
\begin{align*}
\mathcal{F}^1(X_{\pi(v)}) \oplus L = \mathcal{H}^1_{\dR}(X_{\pi(v)})\text{.}
\end{align*}

By Proposition \ref{exisunic} (2), a symplectic-Hodge basis $(\omega_1,\ldots,\omega_g,\eta_1,\ldots,\eta_g)$ of a principally polarized complex torus is uniquely determined by $(\eta_1,\ldots,\eta_g)$. In particular, for each $p\in M$, the fiber $B_p=B\cap V_p$ may be naturally identified with the set of symplectic-Hodge bases of $(X_p,E_p)$.

Thus, it follows from \ref{principalbundle} that $B$ is a principal $P_g(\CC)$-bundle over $M$; in particular, it is a complex manifold. We also conclude from the above paragraph that $B$ represents the functor in the statement.
\end{proof}

\begin{obs}\label{relreprcomp}
The above construction is compatible, under analytification, with its algebraic counterpart. Namely, let $U$ be a smooth separated $\CC$-scheme of finite type, and $(X,\lambda)$ be a principally polarized abelian scheme over $U$. The complex manifold $B(X^{\an},E_{\lambda})$ over $U^{\an}$ constructed in Lemma \ref{relrepr} is canonically isomorphic to the analytification of the scheme $B(X,\lambda)$ over $U$ constructed in Corollary \ref{relrepr0}.
\end{obs}

Recall that we denote by $(X_g,\lambda_g)$ the universal principally polarized abelian scheme over $B_g$, and by $b_g$ the universal symplectic-Hodge basis of $(X_g,\lambda_g)_{/B_g}$.

\begin{prop} \label{reprsympl}
 The functor $(\mathcal{A}_g^{\an})^{\opp} \longrightarrow \mathsf{Set}$ sending an object $(X,E)_{/M}$ of $\mathcal{A}_g^{\an}$ to the set of symplectic-Hodge bases of $(X,E)_{/M}$ is representable by $(X_{g,\CC}^{\an},E_{\lambda_g})_{/B_{g,\CC}^{\an}}$, with universal symplectic-Hodge basis $b_g$.
 \end{prop}

\begin{proof}
By Lemma \ref{relrepr}, there exists a complex manifold $\mathbf{B}_g\defeq B(\mathbf{X}_g,E_g)$ over $\mathbf{H}_g$ representing the functor
\begin{align*}
\Man_{/\mathbf{H}_g}^{\opp} &\to \Sets\\
                    M &\mapsto \{\text{symplectic-Hodge bases of }(\mathbf{X}_g,E_{g})\times_{\mathbf{H}_g}M\}
\end{align*}
Let $(\mathbf{X}_{\mathbf{B}_g}, E_{\mathbf{B}_g}) = (\mathbf{X}_g,E_g)\times_{\mathbf{H}_g}\mathbf{B}_g$. Note that the principally polarized complex torus $(\mathbf{X}_{\mathbf{B}_g}, E_{\mathbf{B}_g})$ over $\mathbf{B}_g$ is equipped with a universal symplectic-Hodge basis $b_{\mathbf{B}_g}$, and with an integral symplectic basis $\beta_{\mathbf{B}_g}$ obtained by pullback from $\beta_g$ via the canonical morphism $(\mathbf{X}_{\mathbf{B}_g}, E_{\mathbf{B}_g})_{/\mathbf{B}_g} \to (\mathbf{X}_g,E_g)_{/\mathbf{H}_g}$ in $\mathcal{A}_g^{\an}$.

We now remark that $(\mathbf{X}_{\mathbf{B}_g}, E_{\mathbf{B}_g})_{/\mathbf{B}_g}$ represents the functor $(\mathcal{A}_g^{\an})^{\opp} \to \Sets$ sending an object $(X,E)_{/M}$ of $\mathcal{A}_g^{\an}$ to the Cartesian product of the set of symplectic-Hodge bases of $(X,E)_{/M}$ with the set of integral symplectic bases of $(X,E)_{/M}$, with $(b_{\mathbf{B}_g},\beta_{\mathbf{B}_g})$ serving as a universal object. Thus, for any element $\gamma\in \Sp_{2g}(\ZZ)$, there exists a unique automorphism ${\Psi_{\gamma}}_{/\psi_{\gamma}}$ of $(\mathbf{X}_{\mathbf{B}_g}, E_{\mathbf{B}_g})_{/\mathbf{B}_g}$ in $\mathcal{A}_g^{\an}$ such that $\Psi_{\gamma}^*b_{\mathbf{B}_g} =b_{\mathbf{B}_g}$ and $\Psi_{\gamma}^*\beta_{\mathbf{B}_g}=\gamma\cdot \beta_{\mathbf{B}_g}$ (where the left action of $\Sp_{2g}(\ZZ)$ on integral symplectic bases is defined as in Remark \ref{actionsp}). 

As the functor $\underline{B}_g:\mathcal{A}_g^{\opp} \to \Sets$ is rigid over $\CC$ (Lemma \ref{rig0}), we see that
\begin{enumerate}
  \item $\gamma \mapsto {\Psi_{\gamma}}_{/\psi_{\gamma}}$ is in fact an action of $\Sp_{2g}(\ZZ)$ on $(\mathbf{X}_{\mathbf{B}_g}, E_{\mathbf{B}_g})_{/\mathbf{B}_g}$ in the category $\mathcal{A}_g^{\an}$, and
  \item the action $\gamma \mapsto \psi_{\gamma}$ of $\Sp_{2g}(\ZZ)$ on the complex manifold $\mathbf{B}_g$ is free; it is also proper since it lifts the action on $\mathbf{H}_g$.
\end{enumerate} 

Let $M$ be the quotient manifold $\Sp_{2g}(\ZZ)\backslash \mathbf{B}_g$ and descend $(\mathbf{X}_{\mathbf{B}_g},E_{\mathbf{B}_g})$ to a principally polarized complex torus $(X,E)$ over $M$. Since $b_{\mathbf{B}_g}$ is invariant under the action of $\Sp_{2g}(\ZZ)$, we can descend it to a symplectic-Hodge basis $b$ of $(X,E)_{/M}$. As in the proof of Proposition \ref{levelstruc}, we may check that $(X,E)_{/M}$ represents the functor in the statement, with $b$ serving as universal symplectic-Hodge basis. 

To finish the proof, we must prove that $(X,E)_{/M}$ is isomorphic to $(X_{g,\CC}^{\an},E_{\lambda_g})_{/B_{g,\CC}^{\an}}$ in $\mathcal{A}_g^{\an}$. For this, it is sufficient to remark that, by the universal property of $(X,E)_{/M}$, there exists a unique morphism in $\mathcal{A}_g^{\an}$
$$
\varphi_{/f}:(X_{g,\CC}^{\an},E_{\lambda_g})_{/B_{g,\CC}^{\an}} \to (X,E)_{/M}
$$
satisfying $\varphi^*b = b_g$, and that the holomorphic map
\begin{align*}
f: B_g(\CC)= B_{g,\CC}^{\an} \to M
\end{align*}
is bijective since principally polarized complex tori (over a point) are algebraizable (cf. Remark \ref{algebraization}); then $f$ is necessarily a biholomorphism (\cite{GH78} p. 19).
\end{proof}

% \begin{obs} \label{remaction}
% Let $\mathbf{B}_g$ be as in the above proof. For later reference, let us make more explicit the left action of $\Sp_{2g}(\ZZ)$ on $\mathbf{B}_g$. A point $p$ of $\mathbf{B}_g$ lying over $\tau \in \mathbf{H}_g$ is given by a symplectic-Hodge basis $b$ of $(\mathbf{X}_{g,\tau},E_{g,\tau})$. Then, for any $\gamma\in \Sp_{2g}(\ZZ)$, $\gamma\cdot p$ is the point in $\mathbf{B}_g$ lying over $\gamma \cdot \tau$ given by the symplectic-Hodge basis $(F_{\gamma,\tau}^{-1})^*b$ of $(\mathbf{X}_{g,\gamma\cdot \tau},E_{g,\gamma\cdot \tau})$, where $F_{\gamma,\tau}: \mathbf{X}_{g,\tau} \to \mathbf{X}_{g,\gamma\cdot \tau}$ is the isomorphism defined in Example \ref{exactionsp}.
% \end{obs}

\subsection{The Hilbert-Blumenthal case}

In this paragraph we state without proof the $R$-multiplication counterparts of the above results.

\subsubsection{}

Let $M$ be a complex manifold, $(X,E,m)_{/M}$ be a principally polarized complex torus with $R$-multiplication over $M$, and denote by $\pi:X \to M$ the structural morphism.

Consider the local system of abelian groups $(D^{-1}\oplus R)_M \defeq \ZZ_M \tensor (D^{-1}\oplus R)$ over $M$, endowed with its natural $R$-multiplication, and with the standard $D^{-1}$-valued $R$-bilinear symplectic form $\Phi$.

\begin{defi}
 An \emph{integral symplectic basis} of $(X,E,m)_{/M}$ is an $R$-linear isomorphism 
 $$
 \beta : ((D^{-1}\oplus R)_M,\Phi) \stackrel{\sim}{\to} (R_1\pi_*\ZZ_X,\Phi_E)\text{.}
 $$ 
\end{defi}

Equivalently, we may think of an integral symplectic basis as a couple $\beta=(\gamma,\delta)$, where $\gamma$ (resp. $\delta$) is a global section of $R_1\pi_*\ZZ_X\tensor D$ (resp. $R_1\pi_*\ZZ_X$), satisfying $\Phi_E(\gamma,\delta)=1$. Here, we see $\Phi_E$ as an $R$-bilinear map
$$
\Phi_E : (R_1\pi_*\ZZ_X\tensor D) \times R_1\pi_*\ZZ_X \to R_M\text{.}
$$

\begin{ex}
 The principally polarized complex torus with $R$-multiplication $(\mathbf{X}_F,E_F,m_F)_{/\mathbf{H}^g}$ constructed in Example \ref{ex-ppctrm} is equipped with a canonical integral symplectic basis $\beta_F$ given by the defining isomorphism $(D^{-1}\oplus R)_{\mathbf{H}^g} \stackrel{\sim}{\to} L$ and the natural identification $L \cong R_1{\bfp_F}_{*}\ZZ_{\mathbf{X}_F}$. 
\end{ex}

We then have the analogous of Proposition \ref{reprsintsympl}.

\begin{prop}\label{reprsintsymplhb}
  The functor $(\mathcal{A}_F^{\an})^{\opp} \longrightarrow \mathsf{Set}$ sending an object $(X,E,m)_{/M}$ of $\mathcal{A}_F^{\an}$ to the set of integral symplectic bases of $(X,E,m)_{/M}$ is representable by $(\mathbf{X}_F,E_F,m_F)_{/\mathbf{H}^g}$, with universal integral symplectic basis $\beta_F$. \hfill $\blacksquare$
\end{prop}

\begin{obs}\label{rem-leftactrm}
  As in Remark \ref{actionsp}, we define a \emph{left} action of $\SL(D^{-1}\oplus R)$ (cf. Example \ref{ex-actionsl}) on the functor $(\mathcal{A}_F^{\an})^{\opp} \longrightarrow \mathsf{Set}$ considered in the above proposition: if $\gamma = (a \ b \ ; \ c\ d)\in \SL(D^{-1}\oplus R)$, and $\beta = (\ \gamma_1 \ \ \delta_1 \ )$ is an integral symplectic basis, then
 $$
 \gamma \cdot \beta \defeq \left(\begin{array}{c c}
  \gamma_1 & \delta_1
 \end{array}\right) \cdot \left(\begin{array}{cc}
                                 d & b \\
                                 c & a
                                \end{array}\right) = \left(\begin{array}{c c}
  d\gamma + c\delta & b\gamma + a \delta
 \end{array}\right)\text{.}
 $$
 The morphism
 $$
 {\varphi_{\gamma}}_{/\gamma} : (\mathbf{X}_{F},E_F,m_F)_{/\mathbf{H}^g}\to (\mathbf{X}_{F},E_F,m_F)_{/\mathbf{H}^g}
 $$
 defined in Example \ref{ex-actionsl} is the unique morphism in $\mathcal{A}_F^{\an}$ satisfying
 $$
 \varphi_{\gamma}^*\beta_F = \gamma\cdot \beta_F\text{.}
 $$
\end{obs}

\begin{obs}\label{rem-relhalfspace}
  Let $t: (\ZZ^{2g}, \langle \ , \ \rangle_{\text{std}}) \stackrel{\sim}{\to} (R\oplus D^{-1}, \Tr \Psi)$ be the trivialization of the symplectic $\ZZ$-module $(R\oplus D^{-1}, \Tr \Psi)$ as defined in Remark \ref{rem-relationbfbg}, so that $(t^{\vee})^{-1}$ is a trivialization of $(D^{-1}\oplus R, \Tr \Phi)$. Then we can use Propositions \ref{reprsintsympl} and \ref{reprsintsymplhb} to see that $t$ induces a holomorphic map
  $$
h_t: \mathbf{H}^g \to \mathbf{H}_g
$$
given, under the moduli theoretic interpretation, by
$$
(X,E,m,\beta)\mapsto (X,E,\beta \circ (t^{\vee})^{-1})\text{.}
$$
It follows from the construction in the proof of Proposition \ref{reprsintsympl} that $h_t$ is given in coordinates by
$$
(\tau_1,\ldots,\tau_g)\mapsto (\sigma_i(x_j))_{1\le i,j\le g}^{-1}\text{diag}(\tau_1,\ldots,\tau_g)(\sigma_i(r_j))_{1\le i,j\le g}
$$
Note that $h_t$ actually lifts to a morphism in $\mathcal{A}_g^{\an}$
$$
 (\mathbf{X}_F,E_F)_{/\mathbf{H}^g} \to (\mathbf{X}_g,E_g)_{/\mathbf{H}_g}
 $$
 given on the fiber of $\tau \in \mathbf{H}^g$ by
 \begin{align*}
   \mathbf{X}_{F,\tau}&\to \mathbf{X}_{g,h_t(\tau)}\\\
                 z & \mapsto (\sigma_i(x_j))_{1\le i,j\le g}^{-1}\cdot z\text{.}
 \end{align*}
 Finally, let us remark that, by definition of $r_i$ and $x_i$, we have
 $$
(\sigma_i(x_j))_{1\le i,j\le g}^{-1} = (\sigma_j(r_i))_{1\le i,j\le g} = (\sigma_i(r_j))\transp_{1\le i,j\le g}\text{.}
 $$
\end{obs}

\subsubsection{}

Let $n\ge 1$ be an integer, and $(X,\lambda,m)_{/U}$ be a principally polarized abelian scheme with $R$-multiplication. Clearly, the action of $R$ on $X$ preserves the $n$-torsion subscheme $X[n]$. If $U$ is a $\ZZ[1/n,\zeta_n]$-scheme,  then  there exists a perfect alternating $R$-bilinear pairing (see Remark \ref{rem-dualityrelation}) 
$$
\epsilon^{\lambda}_n : X[n]\times X[n] \to (D^{-1}/nD^{-1})_U
$$
such that
$$
\Tr \epsilon^{\lambda}_n = e^{\lambda}_n\text{.}
$$

If $n\ge 3$, then there exists a fine moduli scheme $A_{F,n}$ over $\ZZ[1/n,\zeta_n]$ classifying principally polarized abelian schemes with $R$-multiplication $(X,\lambda,m)_{/U}$ equipped with an $R$-trivialization of  $(X[n],\epsilon^{\lambda}_n)$, i.e., an $R$-isomorphism
$$
(((D^{-1}/nD^{-1})\oplus (R/nR))_U,\Phi_n) \stackrel{\sim}{\to} (X[n],\epsilon^{\lambda}_n)\text{,}
$$
where $\Phi_n$ denotes the standard symplectic form modulo $n$. We denote the universal principally polarized abelian scheme with $R$-multiplication over $A_{F,n}$ by $(X_{F,n},\lambda_{F,n}, m_{F,n})$, and its universal symplectic $R$-trivialization by $\alpha_{F,n}$.

In the analytic category $\mathcal{A}_F^{\an}$, we may consider the notion of an ``integral symplectic basis modulo $n$" of a principally polarized complex torus with $R$-multiplication $(X,E,m)_{/M}$; namely, an $R$-linear isomorphism
$$
(((D^{-1}/nD^{-1})\oplus (R/nR))_U,\Phi_n) \stackrel{\sim}{\to} (R_1\pi_*\ZZ_X/nR_1\pi_*\ZZ_X,\Phi_{E,n})\text{,}
$$
where $\Phi_{E,n}$ denotes the reduction modulo $n$ of $R$-bilinear symplectic form $\Phi_E$. This notion coincides with its algebraic counterpart, since for a principally polarized abelian scheme with $R$-multiplication $(X,\lambda,m)_{/U}$, with $U$ a smooth separated $\CC$-scheme of finite type, the $R$-symplectic modules $(R_1p^{\an}_*\ZZ_{X^{\an}}/nR_1p^{\an}_*\ZZ_{X^{\an}},\Phi_{E_{\lambda},n})$ and $(X^{\an}[n], \epsilon_n^{\lambda})$ are naturally isomorphic.

For any integer $n\ge 1$, let $\Gamma_F(n)$ be the kernel of the ``reduction modulo $n$" map $\SL(D^{-1}\oplus R) \to \SL((D^{-1}/nD^{-1})\oplus (R/nR))$. If $n\ge 3$, then $\Gamma_F(n)$ acts properly and freely on $\mathbf{H}^g$.

\begin{prop}[cf. Proposition \ref{levelstruc}]\label{prop-modulilevelrm}
 For any integer $n\ge 3$, the complex manifold $A_{F,n}(\CC)=A_{F,n,\CC}^{\an}$ is canonically biholomorphic to the quotient of $\mathbf{H}^g$ by $\Gamma_F(n)$, and the functor $(\mathcal{A}_F^{\an})^{\opp} \to \Sets$ sending an object $(X,E,m)_{/M}$ of $\mathcal{A}_F^{\an}$ to the set of integral symplectic bases modulo $n$ of $(X,E,m)_{/M}$ is representable by $(X_{F,n,\CC}^{\an}, E_{F,n},m_{F,n})_{/A_{F,n,\CC}^{\an}}$. \hfill $\blacksquare$
\end{prop}

\subsubsection{}

Finally, we define \emph{symplectic-Hodge bases} of principally polarized complex tori as in Paragraph \ref{subsec-shbrm} (cf. \ref{defi-shbppct}).

Let $(X_F,\lambda_F,m_F)$ be the universal principally polarized abelian scheme with $R$-multiplication over $B_F$, and let $b_F$ be its universal symplectic-Hodge basis.

\begin{prop}
 The functor $(\mathcal{A}_F^{\an})^{\opp}\to \Sets$ sending an object $(X,E,m)_{/M}$ of $\mathcal{A}_F^{\an}$ to the set of symplectic-Hodge bases of $(X,E,m)_{/M}$ is representable by $(X_{F,\CC}^{\an}, E_{\lambda_F},m^{\an}_{F})_{/B_{F,\CC}^{\an}}$, with universal symplectic-Hodge basis $b_F$. \hfill $\blacksquare$
\end{prop}

\section{The analytic higher Ramanujan equations}\label{sec-analhre}

In this section we consider the complex analytic avatars of the higher Ramanujan equations introduced in Section \ref{sec-intsol}.

We shall then construct particular solutions $\varphi_g$ and $\varphi_F$ of these differential equations, defined on $\mathbf{H}_g$ in the Siegel case, and on $\mathbf{H}^g$ in the Hilbert-Blumenthal case. The ``$q$-expansions'' of these solutions coincide with the previously defined integral solutions $\hat{\varphi}_g$ and $\hat{\varphi}_F$.

\subsection{Definition of $\varphi_g$ and statement of our main theorem in the Siegel case} \label{defphig}

Let us first define the analytic higher Ramanujan equations. Consider the holomorphic coordinate system $(\tau_{kl})_{1\le k \le l \le g}$ on the complex manifold $\mathbf{H}_g$, where $\tau_{kl}: \mathbf{H}_g \to \CC$ associates to any $\tau \in \mathbf{H}_g$ its entry in the $k$th row and $l$th column. To this system of coordinates is attached a family $(\theta_{kl})_{1\le k \le l \le g}$ of holomorphic vector fields on $\mathbf{H}_g$, defined by \label{symb:thetaijanal}
\begin{align*}
\theta_{kl} \defeq \frac{1}{2\pi i}\frac{\partial}{\partial\tau_{kl}}\text{.} 
\end{align*}

Let $(v_{kl})_{1\le k \le l \le g}$ be the family of holomorphic vector fields on $B_g(\CC)$ induced by the higher Ramanujan vector fields on $\mathcal{B}_g$ defined in Section \ref{ramvecfields}.

\begin{defi} \label{defihreq}
Let $U$ be an open subset of $\mathbf{H}_g$. We say that a holomorphic map $u: U \to B_g(\CC)$ is an \emph{analytic solution of the higher Ramanujan equations} over $\mathcal{B}_g$ if
\begin{align*}
Tu(\theta_{kl}) = u^*v_{kl}
\end{align*}
for every $1\le k \le l \le g$.
\end{defi}

We now construct a global holomorphic solution
\begin{align*}
\varphi_g: \mathbf{H}_g \to B_g(\CC)
\end{align*}
of the higher Ramanujan equations. In view of the universal property of the moduli space $B_g(\CC)$ (Proposition \ref{reprsympl}), the holomorphic map $\varphi_g$ will be induced by a certain symplectic-Hodge basis of the principally polarized complex torus $(\mathbf{X}_g,E_g)$ over $\mathbf{H}_g$.

Recall that the comparison isomorphism (\ref{compisom}) identifies the holomorphic vector bundle $(\Lie_{\mathbf{H}_g} \mathbf{X}_g)^{\vee}$ over $\mathbf{H}_g$ with $\mathcal{F}^1(\mathbf{X}_g/\mathbf{H}_g)$ (Lemma \ref{identcomp}). Moreover, it follows from the construction of $\mathbf{X}_g$ in Example \ref{torus} that $\Lie_{\mathbf{H}_g}\mathbf{X}_g$ is canonically isomorphic to the trivial vector bundle $\CC^g\times \mathbf{H}_g$ over $\mathbf{H}_g$. Under this isomorphism, we define the holomorphic frame
\begin{align*}
(dz_1,\ldots,dz_g)
\end{align*}
of $\mathcal{F}^1(\mathbf{X}_g/\mathbf{H}_g)$ as the dual of the canonical holomorphic frame of $\CC^g\times \mathbf{H}_g$.
\label{symb:phig}
\begin{theorem}\label{theoremsolution}
For each $1\le k \le g$, consider the global sections of $\mathcal{H}^1_{\dR}(\mathbf{X}_g/\mathbf{H}_g)$
\begin{align*}
\bfomega_k \defeq 2\pi i\, dz_k\text{, } \ \ \ \bfeta_k \defeq \nabla_{\theta_{kk}}\bfomega_k\text{,}
\end{align*}
where $\nabla$ denotes the Gauss-Manin connection on $\mathcal{H}^1_{\dR}(\mathbf{X}_g/\mathbf{H}_g)$. Then,
\begin{enumerate}
   \item The $2g$-uple 
\begin{align*}
\bfb_g \defeq (\bfomega_1,\ldots,\bfomega_g,\bfeta_1,\ldots,\bfeta_g)
\end{align*}
of holomorphic global sections of $\mathcal{H}^1_{\dR}(\mathbf{X}_g/\mathbf{H}_g)$ is a symplectic-Hodge basis of the principally polarized complex torus $(\mathbf{X}_g,E_g)_{/\mathbf{H}_g}$.
   \item The holomorphic map  
\begin{align*}
\varphi_g: \mathbf{H}_g\to B_g(\CC)
\end{align*}
   corresponding to $\bfb_g$ by the universal property of $B_g(\CC)$ is a solution of the higher Ramanujan equations (Definition \ref{defihreq}).
\end{enumerate}
\end{theorem}

The main idea in our proof is to compute with a $C^{\infty}$ trivialization of the vector bundle $\mathcal{H}^1_{\dR}(\mathbf{X}_g/\mathbf{H}_g)$; in the next subsection we develop some preliminary background.

\subsection{Preliminary results}\label{ssec-prelres}

Consider the \emph{complex conjugation}, seen as a $C^{\infty}$ morphism of real vector bundles over $\mathbf{H}_g$,
\begin{align*}
\mathcal{H}^1_{\dR}(\mathbf{X}_g/\mathbf{H}_g) &\to \mathcal{H}^1_{\dR}(\mathbf{X}_g/\mathbf{H}_g)\\
                               \alpha &\mapsto \overline{\alpha}
\end{align*}
induced by the comparison isomorphism (\ref{compisom}), and denote $d\bar{z}_k \defeq \overline{dz_k}$ for every $1\le k \le g$. We may check fiber by fiber that the $2g$-uple of $C^{\infty}$ global sections of $\mathcal{H}^1_{\dR}(\mathbf{X}_g/\mathbf{H}_g)$
\begin{align*}
(dz_1,\ldots,dz_g,d\bar{z}_1,\ldots,d\bar{z}_g)
\end{align*}
trivializes $\mathcal{H}^1_{\dR}(\mathbf{X}_g/\mathbf{H}_g)$ as a $C^{\infty}$ complex vector bundle over $\mathbf{H}_g$.

For $1\le i \le j \le g$ and $1\le k \le g$, let us define
\begin{align*}
\bfeta^{ij}_k \defeq \nabla_{\theta_{ij}} \bfomega_k\text{,}
\end{align*}
so that
\begin{align*}
\bfeta_k = \bfeta_k^{kk}\text{.}
\end{align*}

% In the next proposition, we express $\bfeta_k^{ij}$ in the frame $(dz_1,\ldots,dz_g,d\bar{z}_1,\ldots,d\bar{z}_g)$.

\begin{prop} \label{lemme1}
Consider the notations in \ref{notationmatrices}. For every $1\le i \le j \le g$ and $1\le k \le g$, we have
\begin{align*}
\bfeta^{ij}_k = \sum_{l=1}^g \mathbf{e}^{\mathsf{T}}_k  \mathbf{E}^{ij}(\Im \tau)^{-1}\mathbf{e}_l \Im dz_l
\end{align*} 
as a $C^{\infty}$ section of $\mathcal{H}^1_{\dR}(\mathbf{X}_g/\mathbf{H}_g)$, where $\Im dz_l \defeq (dz_l -d\bar{z}_l)/2i$. %In particular, $\bfeta^{ij}_k=0$ for any $k\notin\{i,j\}$.
\end{prop}

\begin{proof}
For $1\le i \le j \le g$ and $1\le k,l \le g$, let $\lambda_{kl}^{ij}$ and $\mu_{kl}^{ij}$ be the $C^{\infty}$ functions on $\mathbf{H}_g$ with values in $\CC$  defined by the equation 
\begin{align*}
\bfeta_{k}^{ij} = \sum_{l=1}^g (\lambda_{kl}^{ij} dz_l + \mu_{kl}^{ij}d\bar{z}_l)\text{.}
\end{align*}
 We must prove that $\lambda_{kl}^{ij} + \mu_{kl}^{ij}=0$ and that $\lambda^{ij}_{kl} = \frac{1}{2i}\textbf{e}\transp_k\textbf{E}^{ij}(\Im \tau)^{-1}\textbf{e}_l$.

Let us consider the integral symplectic basis $\beta_g=(\gamma_1,\ldots,\gamma_g,\delta_1,\ldots,\delta_g)$ of $R_1{\bfp_g}_*\ZZ_{\mathbf{X}_g}$ defined in Example \ref{intsymplbasis}. For every $1\le i \le j \le g$ and $1\le k,l\le g$, we have (cf. Remark \ref{derivint})
\begin{align*}
\int_{\gamma_l}\bfeta^{ij}_k=\int_{\gamma_l}\nabla_{\frac{\partial}{\partial \tau_{ij}}}dz_k = \frac{\partial}{\partial \tau_{ij}}\int_{\gamma_l}dz_k = \frac{\partial}{\partial \tau_{ij}}\delta_{kl} = 0
\end{align*}
and
\begin{align*}
\int_{\delta_l} \bfeta_{k}^{ij} = \int_{\delta_l}\nabla_{\frac{\partial}{\partial \tau_{ij}}}dz_k =\frac{\partial}{\partial \tau_{ij}}\int_{\delta_l}dz_k = \frac{\partial}{\partial \tau_{ij}}\tau_{kl} = \textbf{E}^{ij}_{kl}\text{.}
\end{align*}
Thus, by definition of $\lambda^{ij}_{kl}$ and $\mu^{ij}_{kl}$, we obtain
\begin{align*}
0=\int_{\gamma_l}\bfeta^{ij}_k =\sum_{m=1}^g\left( \lambda^{ij}_{km}\int_{\gamma_l}dz_m + \mu^{ij}_{km}\int_{\gamma_l}d\bar{z}_m\right) = \lambda^{ij}_{kl} + \mu^{ij}_{kl}
\end{align*}
and
\begin{align*}
\textbf{E}^{ij}_{kl} = \int_{\delta_l}\bfeta^{ij}_k = \sum_{m=1}^g\left( \lambda^{ij}_{km}\int_{\delta_l}dz_m + \mu^{ij}_{km}\int_{\delta_l}d\bar{z}_m\right) = \sum_{m=1}^g \lambda^{ij}_{km}(\tau_{ml}-\overline{\tau_{ml}}) = 2i\sum_{m=1}^g \lambda^{ij}_{km}(\Im\tau)_{ml}\text{.}
\end{align*}
In matricial notation, if we put $\lambda^{ij} \defeq (\lambda^{ij}_{kl})_{1\le k,l\le g} \in M_{g\times g}(\CC)$, then we have shown that
\begin{align*}
2i\lambda^{ij}\Im \tau = \textbf{E}^{ij}
\end{align*}
The assertion follows.
\end{proof}

Specializing to the case $i=j=k$ in the above proposition, we obtain the following formulas.

\begin{coro} \label{caraceta}
For any $1\le k \le g$, we have
\begin{align*}
\bfeta_k = \sum_{l=1}^g ((\Im \tau)^{-1})_{kl} \Im dz_l\text{.}
\end{align*}
In particular, $\bfeta_k$ is the unique global section of $\mathcal{H}^1_{\dR}(\mathbf{X}_g/\mathbf{H}_g)$ satisfying
\begin{align*}
\int_{\gamma_l}\bfeta_k =0\ \ \ \text{ and } \ \ \ \int_{\delta_l}\bfeta_k = \delta_{kl}
\end{align*}
for every $1\le l \le g$. In other words, $\bfeta_k$ may be identified with $E_g(\gamma_k, \ )$ under the comparison isomorphism (\ref{compisom}).
\end{coro}

Since every section of $R^1{\bfp_g}_*\ZZ_{\mathbf{X}_g} = (R_1{\bfp_g}_*\ZZ_{\mathbf{X}_g})^{\vee}$, seen as a section of $\mathcal{H}^1_{\dR}(\mathbf{X}_g/\mathbf{H}_g)$ via the comparison isomorphism (\ref{compisom}), is horizontal for the Gauss-Manin connection, we obtain the next corollary.

% Since $\int_{\gamma_l}\bfeta_k$ and $\int_{\delta_l}\bfeta_k$ are constant functions on $\mathbf{H}_g$, and since  $\beta_g =(\gamma_1,\ldots,\gamma_g,\delta_1,\ldots,\delta_g)$ trivializes $R_1{\bfp_g}_*\ZZ_{\mathbf{X}_g}$, we obtain
% \begin{align*}
% \int_{\gamma}\nabla_{\theta}\bfeta_k = \theta\left( \int_{\gamma}\bfeta_k\right) = 0\text{.}
% \end{align*}
% for any holomorphic vector field $\theta$ on $\mathbf{H}_g$, and any section $\gamma$ of $R_1{\bfp_g}_*\ZZ_{\mathbf{X}_g}$.

\begin{coro}\label{corohor}
For any $1\le k \le g$, the global section $\bfeta_k$ of $\mathcal{H}^1_{\dR}(\mathbf{X}_g/\mathbf{H}_g)$ is horizontal for the Gauss-Manin connection:
\begin{align*}
\nabla \bfeta_k = 0\text{.}
\end{align*}
\end{coro}

Our next goal is to use the duality given by the Riemann form $E_g$ to express $dz_l$ in terms of $C^{\infty}$ sections of $\Lie_{\mathbf{H}_g}\mathbf{X}_g$.

\begin{lemma} \label{lemme2}
 Let $1\le k \le g$, and denote by $\tau_k$ the $k$-th column of $\tau \in \mathbf{H}_g$. Then
\begin{align*}
dz_k = -E_g(i\Im\tau_k, \ ) + iE_g(\Im \tau_k, \ )
\end{align*}
as a $C^{\infty}$ section of $\mathcal{H}^1_{\dR}(\mathbf{X}_g/\mathbf{H}_g)$ under the comparison isomorphism (\ref{compisom}).
\end{lemma}

%i.e. dz_k = H(\Im \tau_k, \ )

\begin{proof}
Note that $\Im\tau_k = (\Im\tau) \textbf{e}_k$. Let $\gamma$ be a section of $R_1{\bfp_g}_*\ZZ_{\mathbf{X}_g}$. As $\Im \tau$ is symmetric and $\gamma = \Re \gamma + i \Im \gamma$, we have
\begin{align*}
-E_g(i\Im\tau_k, \gamma ) + iE_g(\Im\tau_k ,\gamma) &= -\Im (\overline{i\Im \tau_k}\transp (\Im \tau)^{-1} \gamma) + i\Im(\overline{\Im \tau_k}\transp (\Im \tau)^{-1}\gamma)\\
&= \Im (i\textbf{e}_k\transp (\Im \tau) (\Im\tau)^{-1}\gamma) + i\Im (\textbf{e}_k\transp (\Im \tau) (\Im\tau)^{-1}\gamma) \\ 
             &=\Re (\textbf{e}_k\transp \gamma) + i\Im (\textbf{e}\transp_k \gamma) \\
             &= \textbf{e}_k\transp \gamma = dz_k(\gamma) \text{.}
\end{align*}
\end{proof}

\subsection{Proof of Theorem \ref{theoremsolution}}

We prove parts (1) and (2) separately.

\begin{proof}[Proof of Theorem \ref{theoremsolution} (1)]
As each $\bfomega_k$ is by definition a section of $\mathcal{F}^1(\mathbf{X}_g/\mathbf{H}_g)$, to prove that $\bfb_g$ is a symplectic-Hodge basis of $(\mathbf{X}_g,E_g)_{/\mathbf{H}_g}$ it is sufficient to show that it is a symplectic trivialization of $\mathcal{H}^1_{\dR}(\mathbf{X}_g/\mathbf{H}_g)$ with respect to the holomorphic symplectic form $\langle \ , \ \rangle_{E_g}$. For this, we claim that it is enough to prove that
\begin{align} \label{equation1} \tag{$*$}
\langle \bfomega_i,\bfeta_j \rangle_{E_g} = \delta_{ij}
\end{align}
for every $1\le i \le j\le g$. Indeed, by Corollary \ref{corohor} and by the compatibility (\ref{compconn}), equation (\ref{equation1}) implies that $\langle \bfeta_i,\bfeta_j\rangle_{E_g}=0$ (apply $\nabla_{\theta_{ii}}$). Since we already know that $\mathcal{F}^1(\mathbf{X}_g/\mathbf{H}_g)$ is Lagrangian, this proves indeed that $\bfb_g$ is a symplectic trivialization of $\mathcal{H}^1_{\dR}(\mathbf{X}_g/\mathbf{H}_g)$.

Fix $1\le i \le j \le g$. By Corollary \ref{caraceta}, we have 
\begin{align*}
 \bfeta_j = \sum_{l=1}^g ((\Im \tau)^{-1})_{jl}\Im dz_l \text{,}
 \end{align*}
thus
 \begin{align*}
 \langle \bfomega_i,\bfeta_j \rangle_{E_g} = 2\pi i\sum_{l=1}^g ((\Im \tau)^{-1})_{jl} \langle dz_i,\Im dz_l \rangle_{E_g}\text{.} 
 \end{align*}
Now, using Lemma \ref{lemme2}, we obtain
 \begin{align*}
 \langle dz_i,\Im dz_l \rangle_{E_g} &= \langle -E_g(i\Im\tau_i, \ ) + iE_g(\Im \tau_i, \ ), E_g( \Im\tau_l, \ )\rangle_{E_g}\\
         &=-\langle E_g(i\Im \tau_i, \ ), E_g( \Im \tau_l,\ )\rangle_{E_g} + i\langle E_g(\Im \tau_i , \ ),E_g( \Im\tau_l , \ )\rangle_{E_g}\\
&=\frac{1}{2\pi i}\left(-E_g(i\Im \tau_i,\Im\tau_l) + iE_g(\Im \tau_i,\Im \tau_l) \right)\\
&=\frac{1}{2\pi i}\Im (i\Im \tau_i\transp(\Im\tau)^{-1} \Im\tau_l)\\
&=\frac{1}{2\pi i}\textbf{e}_i\transp (\Im \tau)\textbf{e}_l = \frac{1}{2\pi i}(\Im \tau)_{il}\text{.}
\end{align*}
Therefore, since $\Im \tau$ is symmetric,
\begin{align*}
 \langle \bfomega_i,\bfeta_j \rangle_{E_g} = \sum_{l=1}^g((\Im \tau)^{-1})_{jl}(\Im \tau)_{li}  = \delta_{ij}\text{.}
\end{align*}
\end{proof}

Part (2) in Theorem \ref{theoremsolution} will be an easy consequence of the following analytic analog of Proposition \ref{prop-equivrameq}.

\begin{prop} \label{equivalences}
Let $U\subset \mathbf{H}_g$ be an open subset and $u: U \to B_g(\CC)$ be the holomorphic map corresponding to a principally polarized complex torus $(X,E)$ over $U$ endowed with some symplectic-Hodge basis $b=(\omega_1,\ldots,\omega_g,\eta_1,\ldots,\eta_g)$. Then the following are equivalent:
\begin{enumerate}
    \item $u$ is a solution of the higher Ramanujan equations.
%     \item For every $1\le i \le j \le g$, we have
% \begin{align*}
% c(\theta_{ij}) = u^*c_{g,\CC}^{\an}(v_{ij})
% \end{align*}
% where $c : T_U \longrightarrow \Gamma^2(\mathcal{F}^1(X/U)^{\vee})\oplus \mathcal{H}^1_{\dR}(X/U)^{\oplus g}$ is the morphism defined above for the symplectic-Hodge basis $b$ of $(X,E)_{/U}$, and $v_{ij}$ are the higher Ramanujan vector fields on $B_g(\CC)$.
    \item For every $1\le i \le j \le g$, we have
    $$
    \nabla_{\theta_{ij}}b = b\left(\begin{array}{cc}
                              0 & 0 \\
                              \mathbf{E}^{ij} & 0
                             \end{array}\right)
    $$
    that is,
      \begin{itemize}
    \item[($i$)] $\nabla_{\theta_{ij}} \omega_i = \eta_j$, $\nabla_{\theta_{ij}} \omega_j = \eta_i$, and $\nabla_{\theta_{ij}}\omega_k =0$, for $k\notin \{i,j\}$
    \item[($ii$)] $\nabla_{\theta_{ij}}\eta_k =0$, for $1 \le k \le g$.
 \end{itemize}
\end{enumerate} 
\hfill $\blacksquare$
\end{prop} 

\begin{proof}[Proof of Theorem \ref{theoremsolution} (2)]
By Proposition \ref{equivalences}, it is sufficient to prove that, for every $1\le i \le j \le g$, we have
\begin{itemize}
     \item[($i$)] $\nabla_{\theta_{ij}} \bfomega_i = \bfeta_j$, $\nabla_{\theta_{ij}} \bfomega_j = \bfeta_i$, and $\nabla_{\theta_{ij}}\bfomega_k =0$, for $k\notin \{i,j\}$
     \item[($ii$)] $\nabla_{\theta_{ij}}\bfeta_k =0$, for $1 \le k \le g$.
\end{itemize}
  Now, ($i$) follows directly from Proposition \ref{lemme1}, and ($ii$) is the content of Corollary \ref{corohor}.
\end{proof}

% \begin{obs} \label{nablaomega}
% In matricial notation (see \ref{matricialnotation}), equations ($i$) and ($ii$) in the above proof read as:
% \begin{align*}
% \nabla \bfomega  = \bfeta \tensor 2\pi i\, d\tau \ \ \ \text{ and }\ \ \  \nabla \bfeta = 0\text{,}
% \end{align*}
% where $\bfomega = ( \bfomega_1, \ldots, \bfomega_g)$ and $\bfeta = (\bfeta_1,\ldots,\bfeta_g)$ are regarded as row vectors of global sections of $\mathcal{H}^1_{\dR}(\mathbf{X}_g/\mathbf{H}_g)$. 
% \end{obs}

\subsection{Compatibility of $\varphi_g$ with $\hat{\varphi}_g$}\label{subsec-compatsiegel}

Recall that we have constructed in Section \ref{sec-intsol} a morphism of stacks $\hat{\varphi}_g: \Spec \ZZ(\!(q_{ij})\!) \to \mathcal{B}_g$. Let us briefly explain how Theorem \ref{thm-intsolsiegel}, which claims that $\hat{\varphi}_g$ is an integral solution of the higher Ramanujan equations on $\mathcal{B}_g$, follows from Theorem \ref{theoremsolution} above.

  Recall that the group of $g\times g$ integral symmetric matrices $\Sym_g(\ZZ)$ is isomorphic to the subgroup
  $$
   \left.\left\{\left(\begin{array}{cc}
                           \mathbf{1}_g & N \\
                           0  &\mathbf{1}_g
                           \end{array}\right)  \in {M}_{g\times g}(\ZZ) \, \right| N \in {\Sym}_g(\ZZ)\right\}
  $$
                     of $\Sp_{2g}(\ZZ)$, so that it acts on the object $(\mathbf{X}_g,E_g)_{/\mathbf{H}_g}$ of $\mathcal{A}_g^{\an}$ by Example \ref{exactionsp}; its action on the base manifold $\mathbf{H}_g$ is given by translations:
  $$
N \cdot \tau = \tau + N\text{,}
  $$
  hence it is proper and free.

  By Lemma \ref{descentlemma}, the principally polarized complex torus $(\mathbf{X}_g,E_g)_{/\mathbf{H}_g}$ descends to a principally polarized complex torus $(X,E)$ over the quotient $\Sym_g(\ZZ)\backslash \mathbf{H}_g$. Moreover, since the symplectic-Hodge basis $\bfb_g$ is easily checked to be invariant under the action of $\Sym_g(\ZZ)$, it also descends to a symplectic-Hodge basis $b$ on $(X,E)_{/\Sym_g(\ZZ)\backslash \mathbf{H}_g}$.   It follows that the holomorphic map $\varphi_g:\mathbf{H}_g\to B_g(\CC)$ defined in Theorem \ref{theoremsolution} factors through a  map
  $$
\psi: {\Sym}_g(\ZZ)\backslash\mathbf{H}_g \to B_g(\CC)
  $$
  associated to the principally polarized complex torus with symplectic-Hodge basis $(X,E,b)$ over $\Sym_g(\ZZ)\backslash \mathbf{H}_g$.
  
 Observe that
  \begin{align*}
    \mathbf{H}_g &\to {\Sym}_{g}(\CC)\\
            \tau &\mapsto q(\tau)=(q_{kl}(\tau))_{1\le k,l\le g}\defeq (e^{2\pi i \tau_{kl}})_{1\le k,l\le g}
  \end{align*}
  induces a biholomorphism of the quotient $\Sym_g(\ZZ)\backslash \mathbf{H}_g$ onto an open submanifold $\mathbf{D}_g$ of ${\Sym}_{g}(\CC)$. Under this identification, we have $ \frac{1}{2\pi i}\frac{\partial}{\partial \tau_{kl}}= q_{kl}\frac{\partial}{\partial q_{kl}}$, and one may check that $(X,E,b)$ corresponds formally to $(\hat{X}_{g},\hat{\lambda}_{g},\hat{b}_g)$ defined in Paragraph \ref{subsec-mumfordsiegel} (that is, $(X,E)$ is obtained by the Mumford construction performed in the analytic category, and $b$ is defined as $b_g$). For instance, for $q= (q_{ij})_{1\le i,j\le g}\in \mathbf{D}_g$, we have
  $$
  X_q = (\CC^{\times})^{g}/\langle (q_{1j},\ldots,q_{g,j})\mid 1\le j \le g \rangle \text{,}
  $$
  and the isomorphism $\mathbf{X}_{g,\tau} \stackrel{\sim}{\to} X_{q(\tau)}$ is induced by
  $$
 z=(z_1,\ldots,z_g)\mapsto (t_1(z),\ldots,t_g(z))\defeq (e^{2\pi i z_1},\ldots,e^{2\pi i z_g})
 $$
 so that
 $$
\bfomega_k = \frac{dt_k}{t_k}\text{.} 
 $$
  
It follows that $\hat{\varphi}_g$ is the ``Taylor expansion'' of $\psi$ in the variables $q_{kl}$. In particular, Theorem \ref{thm-intsolsiegel} is an immediate corollary of Theorem \ref{theoremsolution}.

\begin{obs}\label{rem-rigorous}
A rigorous construction of such correspondence requires the theory of toroidal compactification and completion at components at infinity; we refer to \cite{FC90} p. 141-142 for further details.
\end{obs}

\subsection{Analytic Higher Ramanujan equations over $\mathcal{B}_F$}\label{subsec-ahrerm}

Let $(\mathbf{X}_F,E_{F},m_F)$ be the principally polarized complex torus with $R$-multiplication over $\mathbf{H}^g$ constructed in Example \ref{ex-ppctrm}. As $\Lie_{\mathbf{H}^g}\mathbf{X}_F$ is canonically isomorphic to the trivial vector bundle $\CC^g\times \mathbf{H}^g$ over $\mathbf{H}^g$, we may define a global section of $\mathcal{F}^1(\mathbf{X}_F/\mathbf{H}^g) = (\Lie_{\mathbf{H}^g}\mathbf{X}_F)^{\vee} $ by the formula
$$
\bfomega_F \defeq 2\pi i\sum_{j=1}^gdz_j\text{.}
$$
It is easy to check that $\bfomega_F$ trivializes $\mathcal{F}^1(\mathbf{X}_F/\mathbf{H}^g)$ as a $\mathcal{O}_{\mathbf{H}^g}\tensor R$-module.

\begin{prop}\label{prop-dualksrm}
 The dual of the Kodaira-Spencer morphism
 $$
 \kappa^{\vee} : S^{2}_{\mathcal{O}_{\mathbf{H}^g}\tensor R}(\mathcal{F}^1(\mathbf{X}_F/\mathbf{H}^g)) \to \Omega^1_{\mathbf{H}^g}
 $$
 is an isomorphism of $\mathcal{O}_{\mathbf{H}^g}$-modules satisfying
 \begin{align}\label{eq-dualksomega}
  \kappa^{\vee}(\bfomega_F) = 2\pi i \sum_{j=1}^gd\tau_j\text{.}
 \end{align}
\end{prop}

\begin{proof}
 Recall from Remark \ref{rem-formulaksrmdual} that
 $$
 \kappa^{\vee}(\bfomega_F) = \langle \bfomega_F, \nabla \bfomega_F\rangle_{E_{F}}
 $$
 where $\nabla$ denotes the Gauss-Manin connection on $\mathcal{H}^1_{\dR}(\mathbf{X}_F/\mathbf{H}^g)$. Thus, (\ref{eq-dualksomega}) is equivalent to
 \begin{align}\label{eq-dualksomega2}
\langle \bfomega_F,\nabla_{\frac{1}{2\pi i}\frac{\partial}{\partial \tau_j}}\bfomega_F\rangle_{E_F} = 1\text{,  }\ \ \ 1\le j \le g\text{.}
 \end{align}
 
 To prove this, we may argue as in Paragraph \ref{ssec-prelres}, to which we refer for further details on the computations:
 \begin{enumerate}
    \item we have, for any $1\le i,j\le g$,
 $$
 \nabla_{\frac{\partial}{\partial \tau_j}}dz_i = \begin{cases}
                                                  0 & i\neq j\\
                                                  \frac{\Im dz_j}{\Im \tau_j} & i=j
                                                 \end{cases}
 $$
 as $C^{\infty}$ global sections of $\mathcal{H}^1_{\dR}(\mathbf{X}_F/\mathbf{H}^g)$;
    \item under the comparison isomorphism (\ref{compisom}), we may write
 $$
 dz_j = - E_F(i\Im \tau_j \mathbf{e}_j, \ ) + i E_F(\Im \tau_j \mathbf{e}_j, \ )\text{,}
 $$
 and we deduce from the definition of $\langle \  , \ \rangle_{E_F}$ (\ref{defi-sfrf}) that
 $$
 \langle dz_j, \Im dz_j \rangle_{E_F} = \frac{\Im \tau_j}{2\pi i}\text{.}
 $$
 \end{enumerate}
 The equation (\ref{eq-dualksomega2}) now easily follows from (1) and (2) above.
 
 If we endow $\Omega^1_{\mathbf{H}^g}$ with the unique $R$-multiplication satisfying $r\cdot d\tau_j = \sigma_j(r)d\tau_j$ for every $1\le j \le g$, then $\kappa^{\vee}$ becomes $\mathcal{O}_{\mathbf{H}^g}\tensor R$-linear. Since $2\pi i \sum_{j=1}^gd\tau_{j}$ trivializes $\Omega^1_{\mathbf{H}^g}$ as an $\mathcal{O}_{\mathbf{H}^g}\tensor R$-module, we conclude from (\ref{eq-dualksomega}) that $\kappa^{\vee}$ is an isomorphism.\footnote{Alternatively, we might deduce that $\kappa^{\vee}$ is an isomorphism from the corresponding fact on the universal Kodaira-Spencer morphism over $\mathcal{A}_F$ (cf. Paragraph \ref{ksiso} and Proposition \ref{prop-modulilevelrm}).}
\end{proof}

By composing the Kodaira-Spencer isomorphism
$$
\kappa : T_{\mathbf{H}^g}\stackrel{\sim}{\to} \Gamma^2_{\mathcal{O}_{\mathbf{H}^g}\tensor R}(\mathcal{F}^1(\mathbf{X}_F/\mathbf{H}^g))\tensor_R D^{-1}
$$
with the trivialization of $\Gamma^2_{\mathcal{O}_{\mathbf{H}^g}\tensor R}(\mathcal{F}^1(\mathbf{X}_F/\mathbf{H}^g))$ induced by $\bfomega_F$, we obtain an isomorphism
$$
 T_{\mathbf{H}^g}\stackrel{\sim}{\to} \mathcal{O}_{\mathbf{H}^g}\tensor D^{-1}\text{.}
$$
We denote the inverse of this isomorphism by
$$
\theta_F : \mathcal{O}_{\mathbf{H}^g}\tensor D^{-1} \stackrel{\sim}{\to}T_{\mathbf{H}^g}\text{.}
$$

\begin{obs}\label{rem-explicitthetak}
 Explicitly, we deduce from Proposition \ref{prop-dualksrm} that, for any $x\in D^{-1}$,
 $$
 \theta_F(1\tensor x) =\frac{1}{2\pi i}\sum_{j=1}^g\sigma_j(x)\frac{\partial}{\partial \tau_j}\text{.}
 $$
\end{obs}

\begin{defi}\label{defi-ahrerm}
 Let $U\subset \mathbf{H}^g$ be an open subset, and $u : U \to B_F(\CC)$ be a holomorphic map. We say that $u$ is an \emph{analytic solution of the higher Ramanujan equations over $\mathcal{B}_F$} if
 \begin{align}
  Tu \circ \theta_F = u^*v_F\text{,}
 \end{align}
 that is, if the diagram
 $$
 \begin{tikzcd}
  \mathcal{O}_{\mathbf{H}^g}\tensor D^{-1} \arrow{r}{\theta_F}\arrow{d}[swap]{\cong} & T_{\mathbf{H}^g}\arrow{d}{Tu}\\\
  u^*(\mathcal{O}_{B_F(\CC)}\tensor D^{-1})\arrow{r}[swap]{v_F} & u^*T_{B_F(\CC)}
 \end{tikzcd}
 $$
 commutes.
\end{defi}

Let $(x_1,\ldots,x_g)$ be a $\ZZ$-basis of $D^{-1}$, and let $(r_1,\ldots,r_g)$ be the dual $\ZZ$-basis of $R$. If we denote $\theta^{r_j} = \theta_{F}(1\tensor x_j)$ (resp. $v^{r_j} = v_F(1\tensor x_j)$), then the higher Ramanujan equations acquire the more concrete form \label{symb:thetarianal} \label{symb:vri}
$$
Tu(\theta^{r_j}) = u^*v^{r_j}\text{, }\ \ \ 1\le j \le g\text{.} 
$$

To construct an analytic solution of the higher Ramanujan equations over $\mathcal{B}_F$ defined on $\mathbf{H}^g$, we proceed as in the Siegel case. The proof of the next result is analogous to its Siegel counterpart.

\begin{prop}[cf. Proposition \ref{equivalences}]\label{prop-characsolhrerm}
Let $U\subset \mathbf{H}^g$ be an open subset and $u: U \to B_F(\CC)$ be the holomorphic map corresponding to a principally polarized complex torus with $R$-multiplication $(X,E,m)$ over $U$ endowed with some symplectic-Hodge basis $b=(\omega,\eta)$. Then the following are equivalent:
\begin{enumerate}
    \item $u$ is an analytic solution of the higher Ramanujan equations over $\mathcal{B}_F$.
%     \item For every $1\le i \le j \le g$, we have
% \begin{align*}
% c(\theta_{ij}) = u^*c_{g,\CC}^{\an}(v_{ij})
% \end{align*}
% where $c : T_U \longrightarrow \Gamma^2(\mathcal{F}^1(X/U)^{\vee})\oplus \mathcal{H}^1_{\dR}(X/U)^{\oplus g}$ is the morphism defined above for the symplectic-Hodge basis $b$ of $(X,E)_{/U}$, and $v_{ij}$ are the higher Ramanujan vector fields on $B_g(\CC)$.
    \item We have
    $$
    \nabla_{\theta_F}b = b\left(\begin{array}{cc}
                              0 & 0 \\
                              1 & 0
                             \end{array}\right)\text{.}
    $$
    
\end{enumerate} 
\hfill $\blacksquare$
\end{prop}

\label{symb:phiF}
\begin{theorem}\label{thm-defiphik}
 Let
 $$
 \bfeta_F \defeq \nabla_{\theta_F}\bfomega_F \in \Gamma(\mathbf{H}^g,\mathcal{H}^1_{\dR}(\mathbf{X}_F/\mathbf{H}^g)\tensor_RD)\text{.}
 $$
 Then:
 \begin{enumerate}
  \item The couple $\bfb_F\defeq (\bfomega_F,\bfeta_F)$ is a symplectic-Hodge basis of $(\mathbf{X}_F,E_F,m_F)_{/\mathbf{H}^g}$.
  \item The holomorphic map
  $$
  \varphi_F: \mathbf{H}^g\to B_F(\CC)
  $$
  induced by $\bfb_F$ is an analytic solution of the higher Ramanujan equations over $\mathcal{B}_F$.
 \end{enumerate}
\end{theorem}

\begin{proof}
 In view of Remark \ref{remshb}, to prove (1) it suffices to prove that $\Psi_{E_F}(\bfomega_F,\bfeta_F) =1$, i.e., that the $\mathcal{O}_{\mathbf{H}^g}\tensor R$-linear morphism
 \begin{align*}
 \mathcal{O}_{\mathbf{H}^g}\tensor D^{-1} &\to  \mathcal{O}_{\mathbf{H}^g}\tensor D^{-1} \\
 1\tensor x &\mapsto \Psi_{E_F}(\bfomega_F,\nabla_{\theta_F(1\tensor x)}\bfomega_F)
 \end{align*}
 is the identity. By Remark \ref{rem-dualityrelation}, this is yet equivalent to proving that, for every $x\in D^{-1}$,
 $$
  \langle \bfomega_F,\nabla_{\theta_F(1\tensor x)}\bfomega_F\rangle_{E_F} = \Tr(x)\text{.}
 $$
 This follows immediately from Remark \ref{rem-explicitthetak} and from formula (\ref{eq-dualksomega2}) in the proof of Proposition \ref{prop-dualksrm}:
 \begin{align*}
  \langle \bfomega_F,\nabla_{\theta_F(1\tensor x)}\bfomega_F\rangle_{E_F} = \sum_{j=1}^g\sigma_j(x)\langle \bfomega_F,\nabla_{\frac{1}{2\pi i}\frac{\partial}{\partial \tau_j}}\bfomega_F \rangle_{E_F} = \sum_{j=1}^g\sigma_j(x)= \Tr(x)\text{.}
 \end{align*}
 
 To prove (2), we apply Proposition \ref{prop-characsolhrerm}: the equation $\nabla_{\theta_F}\bfomega_F = \bfeta_F$ holds by definition, whereas $\nabla_{\theta_F}\bfeta_F = 0$ is equivalent to asserting that
 $$
 \theta_F(1\tensor x)\theta_F(1\tensor y)\int_{\gamma}\bfomega_F = 0
 $$
 for every $x,y\in D^{-1}$ and $\gamma$ local section of $R_1{\bfp_F}_*\ZZ_{\mathbf{X}_F}$; this, in turn, is an easy consequence of Remark \ref{rem-explicitthetak} and of the explicit definition of $\bfomega_F$.
\end{proof}

\begin{obs}\label{rem-relationphianal}
  Consider the morphism of stacks $f_t: \mathcal{B}_F \to \mathcal{B}_g$ of Remark \ref{rem-relationbfbg}, and the holomorphic map $h_t: \mathbf{H}^g \to \mathbf{H}_g$ of Remark \ref{rem-relhalfspace}. One may check using the characterization in Corollary \ref{caraceta} that the following diagram is commutative:
  $$
  \begin{tikzcd}
    \mathbf{H}^g \arrow{r}{\varphi_F}\arrow{d}[swap]{h_t}& B_F(\CC)\arrow{d}{f_t}\\
    \mathbf{H}_g \arrow{r}[swap]{\varphi_g} & B_g(\CC)
  \end{tikzcd}
  $$
\end{obs}

\subsection{Compatibility of $\varphi_F$ and $\hat{\varphi}_F$}\label{subsec-compatibilityhb}

Analogously to the Siegel case, $\varphi_F$ and $\hat{\varphi}_F$ are compatible.

To see this, we first recall that the abelian group $D^{-1}$ can be seen as a subgroup of $\SL(D^{-1}\oplus R)$ via $x\mapsto ( 1 \ x \ ; \ 0 \ 1 )$, so that it acts on the object $(\mathbf{X}_F,E_F,m_F)_{/\mathbf{H}^g}$ of $\mathcal{A}_F^{\an}$ by Example \ref{ex-actionsl}. The action of $D^{-1}$ on the base manifold $\mathbf{H}^g$ is given by translations:
  \begin{align*}
   x\cdot \tau = \tau + (\sigma_{j}(x))_{1\le j\le g}\text{,}
  \end{align*}
  so that it is proper and free.

  Therefore, by Lemma \ref{descentlemma} and Remark \ref{rem-descentlemma}, $(\mathbf{X}_F,E_F,m_F)_{/\mathbf{H}^g}$ descends to a principally polarized complex torus with $R$-multiplication $(X,E,m)$ over the quotient $D^{-1}\backslash\mathbf{H}^g$. Since $\bfb_F$ is invariant under the action of $D^{-1}$, it also descends to a symplectic-Hodge basis $b$ of $(X,E,m)_{/D^{-1}\backslash\mathbf{H}^g}$, so that $\varphi_F : \mathbf{H}^g \to B_F(\CC)$ factors through an analytic map
  $$
\psi : D^{-1}\backslash\mathbf{H}^g \to B_F(\CC)\text{.}
$$

To check that $\hat{\varphi}_F$ is the formal version of $\psi$, we let $(x_1,\ldots,x_g)$ be the same $\ZZ$-basis of $D^{-1}$ considered in Paragraph \ref{subsec-hrebf}, and we observe that $D^{-1}\backslash \mathbf{H}^g$ can be identified with an open submanifold $\mathbf{D}_F$ of $\CC^g$ via
$$
\tau \mapsto q(\tau) = (q^{r_1}(\tau),\ldots,q^{r_g}(\tau))\defeq (e^{2\pi i \Tr(r_1\tau)},\ldots,e^{2\pi i \Tr(r_g\tau)}) \in \CC^g\text{,}
$$
where, for $r \in R$, we denote $\Tr(r\tau) \defeq \sum_{j=1}^g\sigma_j(r)\tau_j$.

If we identify $\CC^g$ with $\CC\tensor D^{-1}$ via the field embeddings $\sigma_i: F \to \CC$, then
$$
\mathbf{X}_{F,\tau} = \CC\tensor D^{-1}/(D^{-1}+ \tau R)
$$
and the natural isomorphism
$$
\mathbf{X}_{F,\tau} \stackrel{\sim}{\to} X_{q(\tau)}
$$
is induced by $z\tensor x \mapsto e^{2\pi i z}\tensor x$. We deduce from this that, for $q=(q^{r_1},\ldots,q^{r_g}) \in \mathbf{D}_F$, we have
$$
X_q = \CC^{\times}\tensor D^{-1}/ Y_q\text{,}
$$
where $Y_q$ is the image of the unique $R$-linear map $R \to \CC^{\times}\tensor D^{-1}$ whose trace $R \to \CC^{\times}$ is given by $r\mapsto q^{r}(\tau)\defeq e^{2\pi i \Tr(r\tau)}$ (cf. Remark \ref{rem-dualityrelation}). This shows that $\hat{X}_F$ is the formal analog of $X$, and we may argue similarly for the principal polarization and the $R$-multiplication.

To see that $\hat{b}_F$ coincides with $b$, we consider the identification of $\CC^{\times} \tensor D^{-1}$ with $\CC^g$ given by $(x_1,\ldots,x_g)$, so that $\mathbf{X}_{F,\tau} \stackrel{\sim}{\to} Y_{q(\tau)}$ is induced by
$$
\tau \mapsto (t^{r_1}(z),\ldots,t^{r_g}(z))\defeq  (e^{2\pi i \Tr(r_1z)},\ldots,e^{2\pi i \Tr(r_gz)})
$$
where, for $r \in R$, we define $\Tr(rz)\defeq \sum_{j=1}^gr_jz_j$. Thus (cf. Remark \ref{rem-formulaomegaf})
$$
\bfomega_F = 2\pi i \sum_{j=1}^gdz_j = \sum_{i=1}^g\Tr(x_i)\frac{dt^{r_i}}{t^{r_i}}\text{.}
$$
Also, if $q^{r_i} : \mathbf{H}^g \to \CC$ is defined as above, a computation shows that, for $x \in D^{-1}$, the vector field $\theta_F(1\tensor x)$ defined in Paragraph \ref{subsec-ahrerm} (cf. Remark \ref{rem-explicitthetak}) is given by
$$
\theta_F(1\tensor x) = \sum^g_{i=1} \Tr(r_ix)q^{r_i}\frac{\partial}{\partial q^{r_i}}\text{.}
$$

It follows from these formulas that Theorem \ref{thm-intsolrm} is an immediate corollary of Theorem \ref{thm-defiphik} (see also Remark \ref{rem-rigorous}).

\section{Values of $\varphi_g$ and $\varphi_F$;  periods of abelian varieties} \label{sectionvaluesphi}

In this section we show that the values of the analytic maps $\varphi_g : \mathbf{H}_g \to B_g(\CC)$ (resp. $\varphi_F: \mathbf{H}^g \to B_F(\CC)$) defined in Theorem \ref{theoremsolution} (resp. Theorem \ref{thm-defiphik}) can be used to ``compute'', up to a finite extension, the fields generated by the periods of principally polarized abelian varieties (resp. principally polarized abelian varieties with real multiplication). In particular, the transcendence degree of such fields of periods can be read from the analytic maps $\varphi_g$ and $\varphi_F$.

\subsection{Fields of periods of abelian varieties and statement of our main theorems}

Let $X$ be a complex abelian variety (resp. a complex torus). A \emph{field of definition} of $X$ is a subfield $k$ of $\CC$ for which there exists an abelian variety $X_0$ over $k$ such that $X$ is isomorphic to $X_0\tensor_k \CC$ as a complex abelian variety (resp. isomorphic to $X_0(\CC)$ as a complex torus); we say that $X_0$ is a $k$-model of $X$. 

\label{symb:PXk}
\begin{defi}
Let $X$ be a complex abelian variety, $k$ be a field of definition of $X$, and fix a $k$-model $X_0$ of $X$. The \emph{field of periods} $\mathcal{P}(X/k)$ of $X$ over $k$ is defined as the smallest subfield of $\CC$ containing $k$ and the image of pairing
\begin{align*}
H^1_{\dR}(X_0/k)\tensor H_1(X_0(\CC),\ZZ)   &\to \CC\\
           \alpha \tensor \gamma &\mapsto \int_{\gamma}\alpha
\end{align*} 
given by ``integration of differential forms'' (cf. \ref{sectioncompisom}).
\end{defi}

The field $\mathcal{P}(X/k)$ does not depend on the choice of $X_0$.

\begin{obs}\label{rem-fieldofrationality}
Alternatively, the field of periods $\mathcal{P}(X/k)$ can be regarded as the ``field of rationality" of the comparison isomorphism (see Remark \ref{obs:comp})
$$
\comp: \CC\tensor_{k} H^1_{\dR}(X/k) \stackrel{\sim}{\to} \CC\tensor_{\QQ}H^1(X(\CC),\QQ)\text{,}
$$
that is, the field of definition (cf. \ref{notationresidue}) of the complex point $\comp$ of the $k$-variety
$$
\text{Isom}(H^1_{\dR}(X/k),k\tensor_{\QQ}H^1(X(\CC),\QQ))\text{.}
$$
\end{obs}

Let $A_g$ be the coarse moduli space associated to the Deligne-Mumford stack $\mathcal{A}_g \to \Spec \ZZ$ (which exists as an algebraic space by the Keel-Mori theorem, cf. \cite{olsson16} Theorem 11.1.2). \label{symb:coarseAg} We recall that $A_g$ is a quasi-projective scheme over $\Spec \ZZ$ (cf. \cite{moret-bailly85} VII Théorème 4.2) endowed with a canonical morphism $\mathcal{A}_g \to A_g$ inducing, for every algebraically closed field $k$, a bijection of $A_g(k)$ with the set of isomorphism classes of principally polarized abelian varieties over $k$.

Since any principally polarized complex torus $(X,E)$ of dimension $g$ is algebraizable, $(X,E)$ defines an isomorphism class in the category $\mathcal{A}_g(\CC)$ that we shall denote $[(X,E)]$. Let \label{symb:jg} 
\begin{align*}
j_g : \mathbf{H}_g &\to A_g(\CC)\\
               \tau &\mapsto [(\mathbf{X}_{g,\tau},E_{g,\tau})]\text{.}
\end{align*}
Observe that, for any $\tau\in \mathbf{H}_g$, the field $\QQ(j_g(\tau))\subset \CC$ (see \ref{notationresidue}) is a field of definition of $\mathbf{X}_{g,\tau}$.

This section is devoted to the proof of the following theorem.

\begin{theorem} \label{trdeg}
With notation as in Example \ref{torus} and Theorem \ref{theoremsolution}, for every $\tau \in \mathbf{H}_g$ the field of periods $\mathcal{P}(\mathbf{X}_{g,\tau}/\QQ(j_g(\tau)))$ is a finite field extension of $\QQ(2\pi i,\tau,\varphi_g(\tau))$. In particular,
\begin{align*}
\trdeg_{\QQ}\QQ(2\pi i,\tau,\varphi_g(\tau)) = \trdeg_{\QQ} \mathcal{P}(\mathbf{X}_{g,\tau}/\QQ(j_g(\tau)))\text{.}
\end{align*}
\end{theorem}

Here, we see $(2\pi i,\tau,\varphi_g(\tau))$ as a complex point of the $\QQ$-variety $\AA^1_{\QQ}\times_{\QQ}\Sym_{g,\QQ}\times B_{g,\QQ}$, and $\QQ(2\pi i,\tau,\varphi_g(\tau))$ denotes its field of definition; see \ref{notationresidue}.

The above result also admits a Hilbert-Blumenthal analog, and we indicate at the end of this section, without proofs, how to obtain it. As above, we denote by $A_F$ the coarse moduli space associated to $\mathcal{A}_F$,\label{symb:coarseAF} and we consider a map \label{symb:jF}
\begin{align*}
j_F : \mathbf{H}^g &\to A_F(\CC)\\
               \tau &\mapsto [(\mathbf{X}_{F,\tau},E_{F,\tau},m_{F,\tau})]\text{.}
\end{align*}

\begin{theorem}\label{trdegrm}
 With notation as in Example \ref{ex-ppctrm} and Theorem \ref{thm-defiphik}, for every $\tau \in \mathbf{H}^g$ the field of periods $\mathcal{P}(\mathbf{X}_{F,\tau}/\QQ(j_{F}(\tau)))$ is a finite field extension of $\QQ(2\pi i, \tau, \varphi_F(\tau))$. In particular,
 \begin{align*}
\trdeg_{\QQ}\QQ(2\pi i,\tau,\varphi_F(\tau)) = \trdeg_{\QQ} \mathcal{P}(\mathbf{X}_{F,\tau}/\QQ(j_{F}(\tau)))\text{.}
\end{align*}
\end{theorem}

\subsection{Period matrices}

Let us consider the \emph{general symplectic group} (or the group of ``symplectic similitudes"); namely, the subgroup scheme  $\GSp_{2g}$ of $\GL_{2g}$ over $\Spec \ZZ$ such that, for every affine scheme $V=\Spec \Lambda$, we have
\begin{align*}
{\GSp}_{2g}(V) = \left.\left\{\left(\begin{array}{cc}
                           A & B \\
                           C & D
                          \end{array} \right) \in M_{2g\times 2g}(\Lambda) \ \right| \begin{aligned} &\ \ \ \ \ \ \ \ \ \ \ \   A,B,C,D \in M_{g\times g}(\Lambda) \text{ satisfy  } \\ &AB\transp = BA\transp\text{, }  CD\transp= DC\transp\text{, and } AD\transp -BC\transp \in \Lambda^{\times} \mathbf{1}_g \end{aligned}\right\}\text{.}
\end{align*}
Then we have the canonical character
$$
\nu: {\GSp}_{2g} \to \mathbf{G}_m
$$ 
defined as follows: if $s=( A \ B \ ; \ C \ D  ) \in \GSp_{2g}(V)$, then $\nu(s) \in R^{\times}$ satisfies $AD\transp - BC\transp = \nu(s) \textbf{1}_g$. Note that $\Sp_{2g}$ is the kernel of $\nu$. 

We denote by $\GSp_{2g}^*$ the open subscheme of $\GSp_{2g}$ defined by the condition $A \in \GL_g(\Lambda)$ in the above notation.

Let $(X,E)$ be a principally polarized complex torus of dimension $g$, and $b=(\omega_1,\ldots,\omega_g,\eta_1,\ldots,\eta_g)$ (resp. $\beta =  (\gamma_1,\ldots,\gamma_g,\delta_1,\ldots,\delta_g)$) be a symplectic-Hodge basis (resp. an integral symplectic basis) of $(X,E)$.

\begin{defi} \label{defiperiodmatrix}
The \emph{period matrix} of $(X,E)$ with respect to $b$ and $\beta$ is defined by
\begin{align*}
P(X,E,b,\beta) \defeq \left(\begin{array}{cc}
                               \Omega_1 & N_1 \\
                               \Omega_2 & N_2
                               \end{array}\right) \in M_{2g\times 2g}(\CC)\text{,}
\end{align*}
where
\begin{align*}
(\Omega_1)_{ij} \defeq \int_{\gamma_i}\omega_j \ \ & \ \ (N_1)_{ij} \defeq \int_{\gamma_i}\eta_j\\
(\Omega_2)_{ij} \defeq \int_{\delta_i}\omega_j \ \ & \ \ (N_2)_{ij} \defeq \int_{\delta_i}\eta_j\text{.}
\end{align*}
\end{defi} 

 Note that $P(X,E,b,\beta)$ is simply the matrix of the comparison isomorphism (\ref{compisom}) with respect to the bases $b$ of $\mathcal{H}^1_{\dR}(X)$ and $(E( \ \ ,\delta_1),\ldots,E(\ \ ,\delta_g),E( \gamma_1 , \ \ ),\ldots, E(\gamma_g, \ \ ))$ of $\Hom(H_1(X,\ZZ),\CC)$. 

\begin{obs}\label{remarkperiodmatrix}
In particular, let $(X,\lambda)$ be a principally polarized complex abelian variety, $k$ be a field of definition of $X$, and $X_0$ be a $k$-model of $X$. Assume moreover that $\lambda$ descends to a principal polarization $\lambda_0$ on $X_0$. Then, if $b$ is any symplectic-Hodge basis of $(X_0,\lambda_0)$, and $\beta$ is any integral symplectic basis of $(X^{\an},E_{\lambda})$, the field of periods $\mathcal{P}(X/k)$ of $X$ is generated over $k$ by the coefficients of the period matrix $P(X^{\an},E_{\lambda},b,\beta)$ (cf. Remark \ref{rem-fieldofrationality}). 
\end{obs} 

\begin{lemma}\label{lemmaperiodmatrix}
For any $(X,E,b,\beta)$ as above, we have
\begin{enumerate}
    \item $P(X,E,b,\beta) \in \GSp_{2g}(\CC)$ and $\nu(P(X,E,b,\beta)) =2\pi i$, 
    \item $\Omega^1 \in \GL_g(\CC)$ (i.e., $P(X,E,b,\beta) \in \GSp_{2g}^*(\CC)$) and $\Omega_2\Omega_1^{-1} \in \mathbf{H}_g$.
\end{enumerate}
\end{lemma}

Observe that $\Omega_2\Omega_1^{-1}$ is the point of $\mathbf{H}_g$ corresponding to $(X,E,\beta)$ via Proposition \ref{reprsintsympl}.

\begin{proof}
Knowing that $P(X,E,b,\beta)$ is a base change matrix with respect to symplectic bases, (1) is simply a reformulation of Lemma \ref{compsympl}; (2) is a particular case of the classical \emph{Riemann relations} (cf. proof of Proposition \ref{reprsintsympl}). 
\end{proof}

% Note that, if $X/\CC$ is an abelian variety and $(X_0/k_X,\lambda,b,\beta)$ is as above, then
% \begin{align*}
% \mathcal{P}_X = k_X(P(X_0/k_X,\lambda,b,\beta)) = k_X\left( \frac{1}{2\pi \sqrt{-1}} P(X_0/k_X,\lambda,b,\beta)\right)\text{.}
% \end{align*}

\subsection{Auxiliary lemmas}

We shall need the following auxiliary results. 

\begin{lemma}\label{isomorphismgsp}
The morphism of schemes
\begin{align*}
{\GSp}_{2g}^* &\to \mathbf{G}_m \times_{\ZZ} {\Sym}_g \times_{\ZZ} P_g\\
         s &\mapsto (\nu(s),\tau(s),p(s))
\end{align*}
where
\begin{align*}
\tau\left(\begin{array}{cc}
          A & B \\
          C & D \end{array}\right) \defeq CA^{-1}\ \ \ \text{ and }\ \ \  p\left(\begin{array}{cc}
          A & B \\
          C & D \end{array}\right) \defeq \left(\begin{array}{cc}
          A^{-1} & -B\transp \\
          0 & A\transp \end{array}\right)
\end{align*}
is an isomorphism.
\end{lemma}

\begin{proof}
We simply remark that
\begin{align*}
\left(\lambda, Z,\left(\begin{array}{cc}
          X & Y \\
          0 & (X\transp)^{-1} \end{array}\right)  \right)\mapsto \left(\begin{array}{cc}
          X^{-1} & -Y\transp \\
          ZX^{-1} & (\lambda \mathbf{1}_g  - ZX^{-1}Y)X\transp \end{array}\right) 
\end{align*}
is an inverse to the morphism defined in the statement.
\end{proof}

A straightforward computation yields the following result. 

\begin{lemma} \label{computationallemma}
Let $\varphi: (X,E) \to (X',E')$ be an isomorphism of principally polarized complex tori of dimension $g$, $\beta= (\gamma_1,\ldots,\gamma_g,\delta_1,\ldots,\delta_g)$  be an integral symplectic basis of $(X,E)$ and $b'$ be a symplectic-Hodge basis of $(X',E')$. We denote by $\varphi_*\beta$ the integral symplectic basis of $(X',E')$ given by pushforward in singular homology. Then the symplectic-Hodge basis
\begin{align*}
b = (\omega_1,\ldots,\omega_g,\eta_1,\ldots,\eta_g) \defeq \varphi^*b' \cdot p\left( \frac{1}{2\pi i}P(X',E',b',\varphi_*\beta)\right)
\end{align*} 
of $(X,E)$ satisfies
\begin{align*}
\int_{\gamma_i}\eta_j = 0 \text{, }\int_{\delta_i}\eta_j = \delta_{ij}
\end{align*}
for every $1\le i,j\le g$.\hfill $\blacksquare$
\end{lemma}

\subsection{Proof of Theorem \ref{trdeg}} \label{proofthmtrdeg}

% Let $A_g$ be the coarse moduli space associated to the Deligne-Mumford stack $\mathcal{A}_g \to \Spec \ZZ$ (which exists as an algebraic space by the Keel-Mori theorem, cf. \cite{olsson16} Theorem 11.1.2). We recall that $A_g$ is a quasi-projective scheme over $\Spec \ZZ$ (cf. \cite{moret-bailly85} VII Théorème 4.2) endowed with a canonical morphism $\mathcal{A}_g \to A_g$ inducing, for every algebraically closed field $k$, a bijection of $A_g(k)$ with the set of isomorphism classes of principally polarized abelian varieties over $k$.
% 
% Since any principally polarized complex torus $(X,E)$ of dimension $g$ is algebraizable, $(X,E)$ defines an isomorphism class in the category $\mathcal{A}_g(\CC)$ that we shall denote $[(X,E)]$. Let  
% \begin{align*}
% j_g : \mathbf{H}_g &\to A_g(\CC)\\
%                \tau &\mapsto [(\mathbf{X}_{g,\tau},E_{g,\tau})]\text{.}
% \end{align*}

% The next result follows immediately from our proof of the Lemma \ref{fielddef}.
% 
% \begin{lemma}\label{remarkfielddef}
% For any $\tau \in \mathbf{H}_g$, the smallest algebraically closed subfield of $\CC$ over which $\mathbf{X}_{g,\tau}$ is definable is given by the algebraic closure in $\CC$ of the residue field $\mathbf{Q}(j_g(\tau))$. \hfill$\blacksquare$
% \end{lemma}

Let $\varpi_g: B_{g,\QQ}\to A_{g,\QQ}$ be the map obtained by composition of $\pi_g: B_{g,\QQ}\cong \mathcal{B}_{g,\QQ}\to\mathcal{A}_{g,\QQ}$ with the natural map $\mathcal{A}_{g,\QQ}\to A_{g,\QQ}$; for a field $k\supset \QQ$, it acts on $k$-points by sending the isomorphism class $[(X,\lambda,b)]$ of a principally polarized abelian variety with symplectic-Hodge basis $(X,\lambda,b)_{/k}$ to the isomorphism class $[(X,\lambda)]$.

Note that $\varpi_g$ is invariant under the right action of $P_{g,\QQ}$ on $B_{g,\QQ}$ and that each fiber of $\varpi_g$ is a $P_{g,\QQ}$-homogeneous space.

\begin{lemma}\label{lemma-finsurj}
 Let $k\supset \QQ$ be a field, $y\in B_{g,\QQ}(k)$, and denote $x=\varpi_g(y) \in A_{g,\QQ}(k)$. Then the orbit map $P_{g,k}\to \varpi_g^{-1}(x) = B_{g,\QQ}\times_{\QQ}x$ associated to $y$ is a finite and surjective morphism of $k$-schemes.
\end{lemma}

\begin{proof}
 Let $G$ be the stabilizer of $y$, seen as a $k$-subgroup scheme of $P_{g,k}$; it is sufficient to prove that $G$ is a finite $k$-group scheme.
 
 Let $(X,\lambda,b)$ be a principally polarized abelian variety with symplectic-Hodge basis over $k$ for which $y=[(X,\lambda,b)]$.  For any $k$-algebra $\Lambda$, we may define a antihomomorphism of groups
 $$
 h:\Aut ((X,\lambda)\tensor_k\Lambda) \to P_{g,k}(\Lambda)
 $$
 by sending $\sigma$ to the unique element $p\in  P_{g,k}(\Lambda)$ such that $\sigma^*b=b\cdot p$. By definition of $G$, the image of $h$ is precisely $G(\Lambda)$.
 
 Now, if $\Lambda$ is a field, then $\Aut ((X,\lambda)\tensor_k\Lambda)$ is finite (\cite{mumford70} IV.21 Theorem 5). Since $G$ is an (affine) algebraic group over $k$, this implies that $G$ is finite.
\end{proof}

\begin{proof}[Proof of Theorem \ref{trdeg}]
Fix $\tau\in \mathbf{H}_g$, let $k=\QQ(j_g(\tau))$, and let $(X,\lambda)_{/k}$ be a $k$-model of $(\mathbf{X}_{g,\tau},E_{g,\tau})$. Fix an isomorphism
$$
F: (\mathbf{X}_{g,\tau},E_{g,\tau})\stackrel{\sim}{\to} (X(\CC),E_{\lambda})\text{,}
$$
and a symplectic-Hodge basis $b$ of $(X,\lambda)_{/k}$.

We set
$$
s \defeq\frac{1}{2\pi i}P(X(\CC),E_{\lambda},b,F_*\beta_{g,\tau})\in {\GSp}^*_{2g}(\CC)\text{.}
$$
If $f: P_{g,k}\to \varpi_g^{-1}([(X,\lambda)])$ denotes the orbit map associated to $[(X,\lambda,b)]\in B_{g,\QQ}(k)$, then it follows from Lemma \ref{computationallemma} and Corollary \ref{caraceta} that
\begin{align*}
f(p(s))= [(X_{\CC}, \lambda_{\CC},b\cdot p(s))]=[(\mathbf{X}_{g,\tau},E_{g,\tau}, F^*b\cdot p(s))]= [(\mathbf{X}_{g,\tau},E_{g,\tau}, \bfb_{g,\tau})] = \varphi_g(\tau)\text{.}
\end{align*}
Thus, by Lemma \ref{lemma-finsurj}, $k(p(s))$ is a finite field extension of $k(\varphi_g(\tau))$. But $k(\varphi_g(\tau)) = \QQ(\varphi_g(\tau))$, since $\QQ(\varphi_g(\tau))$ is the field of definition of $\varphi_g(\tau)$ in $B_{g,\QQ}$, which maps to $j_{g}(\tau)$ via $\varpi_g$.

By Lemma \ref{lemmaperiodmatrix}, we have $\nu(s) = \frac{1}{2\pi i}$, and $\tau(s)=\tau$. Thus, it follows from Remark \ref{remarkperiodmatrix} and Lemma \ref{isomorphismgsp} that 
$$
\mathcal{P}(\mathbf{X}_{g,\tau}/k) = k(s) = k(2\pi i,\tau,p(s))\text{.}
$$ 
Finally, we conclude from the last paragraph that $\mathcal{P}(\mathbf{X}_{g,\tau}/k)$ is a finite field extension of 
$$
k(2\pi i,\tau,\varphi_g(\tau)) = \QQ(2\pi i,\tau,\varphi_g(\tau))\text{.}
$$
\end{proof}

\begin{obs}\label{rem-valuesphig}
 For latter use, let us remark that with notation as in the above proof, if we denote
 $$
 s= \left(\begin{array}{cc}
                                                                              \Omega_1 & N_1\\
                                                                              \Omega_2 & N_2
                                                                             \end{array}\right)\text{,}
 $$
 then we have actually showed that
 $$
 \QQ(j_g(\tau),\Omega_1, N_1)\supset \QQ(\varphi_g(\tau))
 $$
 is a finite field extension.
\end{obs}

\subsection{Periods of abelian varieties with real multiplication}\label{subsec-periodsrm}

As in Paragraph \ref{subsec-shbrm}, consider the $R$-module $M\defeq R\oplus D^{-1}$ endowed with its standard $D^{-1}$-valued $R$-bilinear symplectic form $\Psi$. The $\ZZ$-dual of $M$ is given by $M^{\vee} = D^{-1}\oplus R$, and we denote by $\Phi$ its standard $D^{-1}$-valued $R$-bilinear symplectic form (cf. Example \ref{ex-actionsl}).

Let $(X,E,m)$ be a principally polarized complex torus with $R$-multiplication (over a point). In order to define period matrices for $(X,E,m)$, it is convenient to adopt the following slightly more abstract approach.

Recall that a symplectic-Hodge basis $b$ of $(X,E,m)$ is a $\CC\tensor R$-linear isomorphism
$$
b: \CC\tensor M \stackrel{\sim}{\to}\mathcal{H}^1_{\dR}(X)
$$
such that $b^*\Psi_E = 1\tensor \Psi$ and $b(\CC\tensor (R\oplus 0)) = \mathcal{F}^1(X)$; an integral symplectic basis of $(X,E,m)$ is an $R$-linear isomorphism
$$
\beta : M^{\vee} \stackrel{\sim}{\to}H_1(X,\ZZ)
$$
satisfying $\beta^*\Phi_E = \Phi$, so that $\beta$ induces a $\CC\tensor R$-linear isomorphism
$$
(\beta_{\CC}^{\vee})^{-1} : \CC\tensor M \stackrel{\sim}{\to} {\Hom}_{\ZZ}(H_1(X,\ZZ),\CC)\text{.}
$$
Since the comparison isomorphism
$$
\comp : \mathcal{H}^1_{\dR}(X) \stackrel{\sim}{\to} {\Hom}_{\ZZ}(H_1(X,\ZZ),\CC)
$$
is $\CC\tensor R$-bilinear,  we obtain a $\CC\tensor R$-linear isomorphism
$$
\comp^{-1}\circ (\beta_{\CC}^{\vee})^{-1} : \CC\tensor M \stackrel{\sim}{\to} \mathcal{H}^1_{\dR}(X)\text{.}
$$

\begin{defi}\label{def-periodmatrixrm}
 The \emph{period matrix} of $(X,E,m)$ with respect to $b$ and $\beta$ is defined as the unique element $P(X,E,m,b,\beta)$ of $\Aut_{\CC\tensor R}(\CC\tensor M) = (\Res_{R/\ZZ}\Aut_M) (\CC)$ such that
 $$
 \comp^{-1}\circ (\beta_{\CC}^{\vee})^{-1} \circ P(X,E,m,b,\beta) = b\text{.}
 $$
\end{defi}

\begin{obs}\label{rem-coeffpmrm0}
  It follows from Remark \ref{rem-fieldofrationality} that, if $k\subset \CC$ is a subfield, $(X,\lambda,m)_{/k}$ is a principally polarized abelian variety with $R$-multiplication over $k$, $b$ is a symplectic-Hodge basis of $(X,\lambda,m)_{/k}$, and $\beta$ is an integral symplectic basis of $(X(\CC),E_{\lambda},m^{\an})$, then
  $$
  \mathcal{P}(X/k) = k(P(X(\CC),E_{\lambda},m^{\an},b,\beta))\text{,}
  $$
  where $k(P(X(\CC),E_{\lambda},m^{\an},b,\beta))$ is the field of definition of the complex point $P(X(\CC),E_{\lambda},m_{\CC},b,\beta)$ of the $k$-variety $k \tensor \Res_{R/\ZZ}\Aut_M$ (cf. \ref{notationresidue}).
\end{obs}

In order to realize $P(X,E,m,b,\beta)$ as an actual matrix we remark that, for every commutative ring $\Lambda$, if $V= \Spec \Lambda$, then we have the natural identification
$$
({\Res}_{R/\ZZ}{\Aut}_M) (V) = \left.\left\{\left(\begin{array}{cc}
                                                   a & b \\
                                                   c & d
                                                  \end{array}
\right)\in {\GL}_2(\Lambda \tensor R)\right| a,d\in \Lambda\tensor R\text{, }b\in \Lambda\tensor D\text{, }c\in \Lambda \tensor D^{-1} \right\}\text{,}
$$
so that we can write
$$
P(X,E,m,b,\beta) = \left(\begin{array}{cc}
                           \omega_1 & \eta_1 \\
                           \omega_2 & \eta_2
                   \end{array}\right) \in ({\Res}_{R/\ZZ}{\Aut}_M) (\CC)\text{.}
$$

\begin{obs}\label{rem-coeffpmrm}
 The coefficients of $P(X,E,m,b,\beta)$ in the above presentation can be understood as follows. With the above notation, since the comparison isomorphism is $\CC\tensor R$-linear, and since the trace form induces a natural identification $(\CC\tensor H_1(X,\ZZ))^*\tensor_R D^{-1}\cong \Hom_{\ZZ}(H_1(X,\ZZ),\CC)$, we obtain an $R$-bilinear pairing
\begin{align*}
\mathcal{H}^1_{\dR}(X)\times  H_1(X,\ZZ) &\to \CC\tensor D^{-1}\\
(\alpha,\gamma)&\mapsto I_{\gamma}\alpha
\end{align*}
satisfying
$$
\Tr I_{\gamma}\alpha = \int_{\gamma}\alpha\text{.}
$$
Then, if we write $b = (\omega,\eta)$, and $\beta = (\gamma,\delta)$, we have
$$
P(X,E,m,b,\beta) = \left(\begin{array}{cc}
                                                   I_{\gamma}\omega & I_{\gamma}\eta \\
                                                   I_{\delta}\omega & I_{\delta}\eta
                                                  \end{array}
\right)\in ({\Res}_{R/\ZZ}{\Aut}_M) (\CC)\text{.}
$$
\end{obs}

\vspace{10pt}

Consider the subgroup scheme $G_F$ of ${\Res}_{R/\ZZ}{\Aut}_M$ defined, for every affine scheme $V=\Spec \Lambda$, by
$$
G_F(V) = \left.\left\{s = \left(\begin{array}{cc}
                                                   a & b \\
                                                   c & d
                                                  \end{array}
\right)\in ({\Res}_{R/\ZZ}{\Aut}_M)(V)\right| \det(s) = ad-bc \in \Lambda^{\times}\subset (\Lambda \tensor R)^{\times} \right\}\text{.}
$$
We denote by $G^*_F$ the open subscheme of $G_F$ given by the condition $a\in (\Lambda\tensor R)^{\times}$.

In the next lemma we see $\mathbf{H}^g$ inside the $\CC$-vector space $\CC\tensor D^{-1}$ via the identification $\CC\tensor D^{-1}\stackrel{\sim}{\to}\CC^g$ given by $1\tensor x \mapsto (\sigma_1(x),\ldots,\sigma_g(x))$. 

\begin{lemma}[cf. Lemma \ref{lemmaperiodmatrix}]\label{lemma-periodmatrixrm}
For any $(X,E,m,b,\beta)$ as above, we have
\begin{enumerate}
 \item $P(X,E,m,b,\beta)\in G_F(\CC)$ and $\det P(X,E,m,b,\beta) = 2\pi i$,
 \item $I_{\gamma}\omega \in (\CC\tensor R)^{\times}$ (i.e., $P(X,E,m,b,\beta)\in G^*_F(\CC)$) and $(I_{\delta}\omega)(I_{\gamma}\omega)^{-1}\in \mathbf{H}^g$.\hfill $\blacksquare$
\end{enumerate}
\end{lemma}

% Observe that
% $$
% S_K\defeq \ker (\det: G_K \to \GG_m) = {\Res}_{R/\ZZ}{\Aut}_{(M,\Psi)}\text{,}
% $$
% and that $P_K$ defined in Paragraph \ref{} is a parabolic subgroup of $S_K$.

Next, we state the analogous auxiliary lemmas.

\begin{lemma}[cf. Lemma \ref{isomorphismgsp}]
 The morphism of schemes
 \begin{align*}
  G_F^* &\to \GG_m \times_{\ZZ}{\Res}_{R/\ZZ}\AA_R^1 \times_{\ZZ}P_F\\
    s&\mapsto (\det(s),\tau(s),p(s))
 \end{align*}
where
\begin{align*}
 \tau \left(\begin{array}{cc}
       a & b \\
       c & d
      \end{array}\right) \defeq ca^{-1} \ \ \ \text{ and }\ \ \ p\left(\begin{array}{cc}
       a & b \\
       c & d
      \end{array}\right) \defeq \left(\begin{array}{cc}
       a^{-1} & -b \\
       0 & a
      \end{array}\right)
\end{align*}
is an isomorphism. \hfill $\blacksquare$
\end{lemma}

\begin{lemma}[cf. Lemma \ref{computationallemma}]
Let $\varphi : (X,E,m) \to (X',E',m')$ be an isomorphism of principally polarized complex tori with $R$-multiplication, $\beta= (\gamma, \delta)$  be an integral symplectic basis of $(X,E,m)$ and $b'$ be a symplectic-Hodge basis of $(X',E',m')$. We denote by $\varphi_*\beta$ the integral symplectic basis of $(X',E',m')$ given by pushforward in singular homology. Then the symplectic-Hodge basis
\begin{align*}
b = (\omega,\eta) \defeq \varphi^*b' \cdot p\left( \frac{1}{2\pi i}P(X',E',b',m',\varphi_*\beta)\right)
\end{align*} 
of $(X,E,m)$ satisfies (cf. Remark \ref{rem-coeffpmrm})
\begin{align*}
I_{\gamma}\eta = 0 \text{, }I_{\delta}\eta = 1\text{.}
\end{align*}
\hfill $\blacksquare$
\end{lemma}
% 
% The last ingredient necessary to elaborate a proof of Theorem \ref{trdeg} analogous to that of Theorem \ref{trdegrm} is the following lemma concerning ``uniformization" of the coarse moduli scheme of $\mathcal{A}_K$:
% \begin{align*}
%  j_K: \mathbf{H}^g&\to A_K(\CC)\\
%                 \tau &\mapsto [(\mathbf{X}_{K,\tau},E_{K,\tau},m_{K,\tau})]
% \end{align*}
% 
% \begin{lemma}[cf. Lemma \ref{remarkfielddef}]
%  For any $\tau\in \mathbf{H}^g$, the smallest algebraically closed subfield of $\CC$ over which $\mathbf{X}_{K,\tau}$ is definable is given by the algebraic closure in $\CC$ of the residue field $\QQ(j_K(\tau))$.\hfill $\blacksquare$
% \end{lemma}

Using the above preliminary results, the proof of Theorem \ref{trdegrm} is completely analogous to that of Theorem \ref{trdeg}.

\section{An algebraic independence conjecture on the values of $\varphi_F$}

In this paragraph, we use the analytic maps $\varphi_F$, for $F$ real quadratic, to formulate a transcendence conjecture containing Grothendieck's Period Conjecture (GPC) for abelian surfaces with complex multiplication, much like Nesterenko's theorem on $ \varphi_{\QQ} = (E_2,E_4,E_6)$ allows to recover GPC for complex multiplication elliptic curves.

% In this paragraph, we state a conjecture on the transcendence of the values of $\varphi_F$, for $F$ a quadratic totally real number field field, close in spirit to the theorem of Nesterenko on the values of $\varphi_1 = \varphi_{\QQ} = (E_2,E_4,E_6)$, and we explain the relation of our statement with Grothendieck's Period Conjecture.

In such higher dimensional versions of Nesterenko-type statements, it is necessary to take into account the presence of ``special subvarieties" of positive dimension of the corresponding moduli problem of abelian varieties. In the case of $A_F$, for $F$ quadratic, these are given by the \emph{Hirzebruch-Zagier} divisors.

\subsection{Hirzebruch-Zagier divisors and statement of the conjecture}\label{subsec-hirzebruchzagier}

Let $F$ be a real quadratic number field, and let $\sigma$ the non-trivial element of $\Gal(F/\QQ)$. The next definition is due to Kudla and Rapoport \cite{KR99} (cf. \cite{HY12} Chapter 3).

\begin{defi}
A \emph{special endomorphism} of a principally polarized abelian scheme with $R$-multiplication $(X,\lambda,m)_{/U}$ is an element $j\in \End_U(X)^{\lambda}$ such that
\begin{align}\label{eq-condspecialendo}
j\circ m(r) = m(r^{\sigma})\circ j
\end{align}
for every $r\in R$.
\end{defi}

For every integer $N\ge 1$, let $\mathcal{T}_F(N)$ be the moduli stack classifying principally polarized abelian schemes with $R$-multiplication endowed with a special endomorphism $j$ satisfying $j^2=N$. These are Deligne-Mumford stacks over $\Spec \ZZ$; moreover, as shown in \cite{HY12} Paragraph 3.3, the forgetful functor $\mathcal{T}_F(N)\to \mathcal{A}_F$ is finite and unramified, and its image defines an effective Cartier divisor in the stack $\mathcal{A}_F$.

For every $N\ge 1$, we denote by $T_F(N)$ the divisor on the $\CC$-scheme $A_{F,\CC}$ induced by $\mathcal{T}_F(N)_{\CC} \to \mathcal{A}_{F,\CC}$. These are known as \emph{Hirzebruch-Zagier divisors}, or ``modular curves'' (cf. \cite{geer88} Chapter V), on the Hilbert modular surface $A_{F,\CC}$. 

% \begin{obs}
%   The holomorphic map $j_F: \mathbf{H}^2 \to A_{F}(\CC)$ identifies $A_{F}(\CC)$ with the left quotient $ \SL(D^{-1}\oplus R)\setminus \mathbf{H}^2$. \warn{linear equation}
% \end{obs}

Recall that Nesterenko's theorem \cite{nesterenko96} states that, for every $\tau\in \mathbf{H}$, we have
$$
\trdeg_{\QQ}\QQ(e^{2\pi i \tau},E_2(\tau),E_4(\tau),E_6(\tau))\ge 3\text{.}
$$
As a corollary, we get
$$
\trdeg_{\QQ}\QQ(\varphi_{\QQ}(\tau))\ge 2\text{.}
$$
We next state the conjectural analog of the above lower bound for a real quadratic number field $F$.

\begin{conj}\label{conj-Fquadratic}
 Let $F$ be a real quadratic number field. Then, for every $\tau\in \mathbf{H}^2\minus \bigcup_{N=1}^{\infty}j_F^{-1}(T_{F}(N))$, we have
 $$
 \trdeg_{\QQ}\QQ(\varphi_F(\tau))\stackrel{?}{\ge} 3\text{.}
 $$
\end{conj}

In the following paragraphs, we explain the precise relation between the above conjecture and Grothendieck's Period Conjecture for abelian surfaces.

\subsection{Periods in the presence of complex multiplication}

In this paragraph, we let $F$ be a totally real number field of any degree $g\ge 1$. Recall that we denote by $\sigma_1,\ldots,\sigma_g$ the field embeddings of $F$ into $\CC$.

Let $k$ be an algebraically closed subfield of $\CC$, and $(X,\lambda,m)$ be a principally polarized abelian variety with $R$-multiplication over $k$. We have already remarked that $m: R\to \End(X)^{\lambda}$ is injective, and that each element in its image is an isogeny (Remark \ref{rem-isogeny}); we thus obtain an embedding of $\QQ$-algebras $m:F \to \End^0(X)\defeq \QQ\tensor_{\ZZ}\End(X)$.

\begin{defi}
 We say that $(X,\lambda,m)$ has \emph{complex multiplication}, or that it is \emph{CM}, if there exists a totally imaginary quadratic extension $E$ of $F$, and an embedding of $\QQ$-algebras $E \to \End^0(X)$  extending $m$.
\end{defi}

If $X$ is a \emph{simple} abelian variety, then $\End^0(X)$ is a division algebra acting faithfully on the $\QQ$-vector space $H_1(X(\CC),\QQ)$, so that $\dim_{\QQ}\End^0(X)$ divides $2g$; in particular, the map $E\to \End^0(X)$ in the above definition is necessarily an isomorphism of $\QQ$-algebras.

We say that a point $\tau\in \mathbf{H}^g$ is CM if $(\mathbf{X}_{F,\tau},E_{F,\tau},m_{F,\tau})$ is CM. We shall need the following well known fact.

\begin{lemma}\label{lemma-bialg}
 If $\tau\in\mathbf{H}^g$ is CM, then $\tau \in (\overline{\mathbf{Q}}\cap \mathbf{H})^g$ and $j_F(\tau)\in A_F(\overline{\mathbf{Q}})$. \hfill $\blacksquare$
\end{lemma}

The classical proof for the case $F=\QQ$ (see, for instance, \cite{shimura71} 4.4-4.6) generalizes to any totally real $F$. Here, as in the case of elliptic curves, if $F$ is seen as a subring of $\CC^g$ via $(\sigma_1,\ldots,\sigma_g)$, then $\tau\in \mathbf{H}^g\subset \CC^g$ satisfies a \emph{quadratic} equation with coefficients in $F$. 

Although not necessary for the sequel, let us mention that the converse of the above result is also true, thus providing a characterization of CM points by a ``bi-algebraicity" property. This characterization actually holds in a much broader framework (see \cite{SW95} and \cite{cohen96}).

\begin{prop}\label{prop-periodscm}
 Let $(X,\lambda,m)$ be a simple CM principally polarized abelian variety with $R$-multiplication over $\overline{\mathbf{Q}}$, $b=(\omega,\eta)$ be a symplectic-Hodge basis of $(X,\lambda,m)_{/\overline{\QQ}}$, and $\beta=(\gamma,\delta)$ be an integral symplectic basis of $(X(\CC), E_{\lambda}, m^{\an})$. Then
 $$
 \mathcal{P}(X/\overline{\QQ}) = \overline{\mathbf{Q}}(I_{\gamma}\omega,I_{\gamma}\eta)\text{.}
 $$
\end{prop}

The notation $I_{\gamma}(\cdot)$ was introduced in Remark \ref{rem-coeffpmrm}. Concretely, by identifying $\CC\tensor F$ with $\CC^g$ via $(\sigma_1,\ldots,\sigma_g)$, the element $I_{\gamma}\omega \in \CC\tensor R$ (resp. $I_{\gamma}\eta\in \CC\tensor D$) defines $g$ complex numbers; the field $\overline{\mathbf{Q}}(I_{\gamma}\omega,I_{\gamma}\eta)$ is obtained from $\overline{\mathbf{Q}}$ by adjoining these $2g$ numbers.

\begin{proof}
 Let $\varphi$ be any element of $\End^0(X)\setminus m(F)$. Since the right $R$-module of symmetric morphisms $\mu:X\to X^t$ satisfying $m(r)^t\circ \mu =\mu \circ m(r)$ for every $r\in R$ is projective of rank 1 (see, for instance, \cite{rapoport78} Proposition 1.17), and since $\varphi$ commutes with every element of $m(F)$, there exists $u\in F^{\times}$ such that
 $$
 \varphi^t\circ \lambda \circ \varphi = \lambda \circ m(u)\text{.}
 $$
 It follows that $\varphi$ induces an automorphism of $\overline{\QQ}\tensor F$-modules
 $$
 \varphi^*: H^1_{\dR}(X/\overline{\QQ}) \to H^1_{\dR}(X/\overline{\QQ})
 $$
 preserving $F^1(X/\overline{\QQ})$, and satisfying
 $$
 \Psi_{\lambda}(\varphi^*\alpha,\varphi^*\beta) = u\Psi_{\lambda}(\alpha,\beta)
 $$
 for every $\alpha,\beta\in H^1_{\dR}(X/\overline{\QQ})$. In particular, there exists a $\overline{\QQ}\tensor F$-automorphism of $\overline{\QQ}\tensor (R\oplus D^{-1}) = (\overline{\QQ}\tensor F)^{\oplus 2}$ of the form
 $$
 A =\left(\begin{array}{cc}
     r & s \\
     0 & t
    \end{array}\right)\in M_{2\times 2}(\overline{\QQ}\tensor F)
 $$
 with $rt = u$ such that
 $$
 \varphi^*b = b\cdot A\text{.}
 $$
 
 Analogously, $\varphi$ induces an automorphism of $F$-vector spaces
 $$
 \varphi_*:H_1(X(\CC),\QQ) \to H_1(X(\CC),\QQ)
 $$
 such that
 $$
 \Phi_{E_{\lambda}}(\varphi_*\gamma,\varphi_*\delta) = u\Phi_{E_{\lambda}}(\gamma,\delta)
 $$
 for every $\gamma,\delta \in H_1(X(\CC),\QQ)$. Thus, there exists a $F$-automorphism of $\QQ\tensor (D^{-1}\oplus R) = F^{\oplus 2}$ of the form
 $$
 B=\left(\begin{array}{cc}
     a & b \\
     c & d
    \end{array}\right)\in M_{2\times 2}(F)
 $$
 with $ad-bc = u$ such that
 $$
 \varphi_*\beta = \beta\cdot B\text{.}
 $$
 
 It follows from the commutativity of the diagram of $\CC\tensor F$-isomorphisms (given by the naturality of the comparison isomorphism)
 $$
 \begin{tikzcd}
  \CC\tensor H^1_{\dR}(X/\overline{\QQ})  \arrow{r}{\comp}\arrow{d}[swap]{\varphi^*} & \Hom_{\QQ}(H_1(X(\CC),\QQ),\CC)\arrow{d}{\varphi_*^{\vee}}\\
  \CC\tensor H^1_{\dR}(X/\overline{\QQ})  \arrow{r}[swap]{\comp} & \Hom_{\QQ}(H_1(X(\CC),\QQ),\CC)
 \end{tikzcd}
 $$
 and from the definition of the period matrix $P = P(X(\CC),E_{\lambda},m^{\an},b,\beta)$ (Definition \ref{def-periodmatrixrm}) that
 $$
 B\transp P = P A\text{,}
 $$
 that is,
 $$
 \left(\begin{array}{cc}
     aI_{\gamma}\omega +c I_{\delta}\omega & aI_{\gamma}\eta +c I_{\delta}\eta \\
     bI_{\gamma}\omega +d I_{\delta}\omega & bI_{\gamma}\eta +d I_{\delta}\eta
    \end{array}\right) = \left(\begin{array}{cc}
     rI_{\gamma}\omega  & sI_{\gamma}\omega +t I_{\gamma}\eta \\
     rI_{\delta}\omega  & sI_{\delta}\omega +t I_{\delta}\eta
    \end{array}\right) \text{.} 
 $$
 We claim that $c\neq 0$. By contradiction, if $c=0$, then by comparing the $(1,1)$ entries, we obtain $a=r$ (recall that $I_{\gamma}\omega\in (\CC\tensor R)^{\times}$ by Lemma \ref{lemma-periodmatrixrm}). This also implies that $d=t$, since $ad=u=rt$. Now, by comparing $(2,1)$ entries, we obtain $bI_{\gamma}\omega = (a-d)I_{\delta}\omega$; since $(I_{\delta}\omega)(I_{\gamma}\omega)^{-1}\in \mathbf{H}^g\subset \CC\tensor D^{-1} = \CC\tensor F$ (cf. Lemma \ref{lemma-periodmatrixrm}), this is only possible if $a-d=b=0$. In particular, $\varphi_* = m(a)_*$, but since $X$ is simple, the action of $\End^0(X)$ on $H_1(X(\CC),\QQ)$ is faithful, so that $\varphi=m(a)$. This contradicts the fact that $\varphi \notin m(F)$.
 
 By comparing $(1,2)$ entries, we obtain the linear equation in $\CC\tensor F$
 $$
 I_{\delta}\eta = sc^{-1}I_{\gamma}\omega + (t-a)c^{-1}I_{\gamma}\eta\text{,}
 $$
 so that $I_{\delta}\eta \in \overline{\QQ}(I_{\gamma}\omega,I_{\gamma}\eta)$. As $(X,\lambda,m)$ is CM, we have $\tau\defeq(I_{\delta}\omega)(I_{\gamma}\omega)^{-1}\in \overline{\QQ}\tensor F$ by Lemma \ref{lemma-bialg}\footnote{Actually, the quadratic equation satisfied by $\tau$ (see the remark following Lemma \ref{lemma-bialg}) is obtained by dividing the $(2,1)$ entries by the $(1,1)$ entries.}; thus $I_{\delta}\omega \in \overline{\QQ}(I_{\gamma}\omega,I_{\gamma}\eta)$. To conclude, it is enough to recall that $\mathcal{P}(X/\overline{\QQ}) = \overline{\QQ}(I_{\gamma}\omega,I_{\delta}\omega,I_{\gamma}\eta,I_{\delta}\eta)$ (Remarks \ref{rem-coeffpmrm0} and \ref{rem-coeffpmrm}). 
\end{proof}

\begin{coro}\label{coro-periodcm}
 Let $\tau\in\mathbf{H}^g$ be a CM point, and assume that $\mathbf{X}_{F,\tau}$ is simple. Then
 $$
 \trdeg_{\QQ}\QQ(2\pi i,\tau,\varphi_F(\tau)) = \trdeg_{\QQ}\QQ(\varphi_F(\tau))\text{.}
 $$
\end{coro}

\begin{proof}
  Let $(X,\lambda,m)$ be a model of $(\mathbf{X}_{F,\tau},E_{F,\tau},m_{F,\tau})$ over $\overline{\QQ}$, $b$ be a symplectic-Hodge basis of $(X,\lambda,m)_{/\overline{\QQ}}$, and $\beta$ be an integral symplectic basis  of $(\mathbf{X}_{F,\tau},E_{F,\tau},m_{F,\tau})$. Our statement follows immediately from the diagram of field extensions
  $$
\begin{tikzcd}[row sep = 0.3cm, column sep = 1cm]
  & \overline{\QQ}(2\pi i,\tau,I_{\gamma}\omega,I_{\gamma}\eta)=\mathcal{P}(\mathbf{X}_{F,\tau}/\overline{\QQ})\\  
  \overline{\QQ}(I_{\gamma}\omega,I_{\gamma}\eta)\arrow[-,swap]{dd}{<\infty}\arrow[equal]{ru}{\text{Prop. \ref{prop-periodscm}}}  &\\
  & \overline{\QQ}(2\pi i,\tau,\varphi_F(\tau))\arrow[-]{uu}\\
  \overline{\QQ}(\varphi_F(\tau))\arrow[-]{ru} & 
\end{tikzcd}
$$
where the finiteness of $\overline{\QQ}(I_{\gamma}\omega,I_{\gamma}\eta)\supset \overline{\QQ}(\varphi_F(\tau))$ derives from a Hilbert-Blumenthal analog of Remark \ref{rem-valuesphig}.
\end{proof}

% \begin{proof}
%   Let $(X,\lambda,m)$ be a model of $(\mathbf{X}_{F,\tau},E_{F,\tau},m_{F,\tau})$ over $\QQ(j_F(\tau))$, $b$ be a symplectic-Hodge basis of $(X,\lambda,m)$, and $\beta$ be an integral symplectic basis  of $(\mathbf{X}_{F,\tau},E_{F,\tau},m_{F,\tau})$. By an analog of Remark \ref{rem-valuesphig} for the Hilbert modular case, we have a finite field extension
%   $$
% \QQ(j_F(\tau),I_{\gamma}\omega,I_{\gamma}\eta)\supset \QQ(\varphi_F(\tau))\text{.}
% $$
% Since $\tau = (I_{\delta}\omega)(I_{\gamma}\omega)^{-1}$, $(I_{\gamma}\omega)(I_{\delta}\eta) - (I_{\delta}\omega)(I_{\gamma}\eta) = 2\pi i$, and
% $$
% \mathcal{P}(\mathbf{X}_{F,\tau}/\QQ(j_F(\tau))) = \QQ(j_F(\tau), I_{\gamma}\omega, I_{\delta}\omega,I_{\gamma}\eta,I_{\delta}\eta)\text{,}
% $$
% the statement follows directly from Proposition \ref{prop-periodscm}.
% \end{proof}

\subsection{Grothendieck's Period Conjecture for abelian surfaces with real multiplication}\label{subsec-gpccm}

In this paragraph we assume that $F$ is a real \emph{quadratic} number field.

\begin{lemma}[cf. \cite{chai95} Lemma 6, \cite{geer88} Proposition IX.1.2]\label{lemma-endoalgrm}
  Let $k$ be an algebraically closed field of characteristic 0, and $(X,\lambda,m)$ be a principally polarized abelian variety with $R$-multiplication over $k$. If $X$ is simple, then $\End^0(X)$ is a division algebra over $\QQ$ isomorphic to one of the following:
  \begin{enumerate}
    \item[(S1)] $F$,
    \item[(S2)] $E\supset F$ totally imaginary quadratic extension (CM case),
    \item[(S3)] $B\supset F$ indefinite quaternion algebra over $\QQ$.
    \end{enumerate}
    If $X$ is not simple, then $X$ is necessarily isogenous to $Y\times_k Y$ for some elliptic curve $Y$ over $k$, and $\End^0(X) = M_{2\times 2}(\End^0(Y))$; in particular, $\End^0(X)$ is a $\QQ$-algebra isomorphic to
    \begin{enumerate}
    \item[(N1)] $M_{2\times 2}(\QQ)$, if $Y$ is not CM,
    \item[(N2)] $M_{2\times 2}(K)$, where $K=\End^0(Y)$ is an imaginary quadratic field if $Y$ is CM.  \hfill $\blacksquare$
    \end{enumerate}
  \end{lemma} 
  
\begin{prop}
Let $k$ be an algebraically closed field of characteristic 0, and $(X,\lambda,m)$ be a principally polarized abelian variety with $R$-multiplication over $k$. The $\QQ$-algebra $\End^0(X)$ is isomorphic to a (commutative) field if and only if $X$ does not admit a non-trivial special endomorphism.
\end{prop}

\begin{proof}
  If $\End^0(X)$ is commutative, then the condition (\ref{eq-condspecialendo}) in the definition of special endomorphisms is clearly only satisfied by $j=0$.   Conversely, let us prove that, if $\End^0(X)$ is not commutative, i.e., cases (S3), (N1), and (N2) in Lemma \ref{lemma-endoalgrm}, then $X$ admits a non-trivial special endomorphism.

  Let us identify $F$ with a subalgebra of $\End^0(X)$ via $m$, and denote by $\varphi \mapsto \varphi^{\dagger}$ the Rosatti involution on $\End^0(X)$ defined by $\lambda$. By hypothesis, for every $x\in F$, we have $x^{\dagger}=x$.   Up to multiplication by a convenient integer, it is sufficient to prove the existence of $j \in \End^0(X)\setminus\{0\}$ such that $j^{\dagger}=j$, and
  $$
  jx = x^{\sigma}j
  $$
  for every $x \in F$. If we write $F=\QQ(\rho)$, for some $\rho \in F$ satisfying $\rho^2 \in \QQ_{>0}$, then it is enough to check that $j^{\dagger} = j$ and $j\rho = -\rho j$.

  \begin{enumerate}
  \item Let us assume that $\End^0(X)=B$ is a quaternion algebra over $\QQ$ (cases (S3) and (N1)). Since $B$ is indefinite, the positive involution $b\mapsto b^\dagger$ cannot coincide with the canonical involution $b\mapsto \overline{b}$ of $B$. By \cite{involutions} Proposition 2.21, there exists $u\in B^{\times}$ such that $\overline{u}=-u$ and
    $$
      b^{\dagger} = u^{-1}\overline{b}u
      $$
      for every $b\in B$. Note that $b\mapsto \overline{b}$ restricts to $\sigma$ on $F$, so that $\overline{\rho}=-\rho$ and the condition $\rho^{\dagger} = \rho$ means that $\rho u = -u\rho$. Thus, we can take $j=\rho u$.

    \item Suppose that $\End^0(X) = M_{2\times 2}(K)$, where $K=\QQ(\theta)$, with $\theta^2 \in \QQ_{<0}$ (case (N2)). Since $\dagger$ is positive, it must restrict to the unique non-trivial automorphism of $K$ (embedded diagonally in $M_{2\times 2}(K)$), i.e., $\theta^{\dagger}=-\theta$. By \cite{involutions} Proposition 2.22, there exists a unique quaternion $\QQ$-subalgebra $B\subset M_{2\times 2}(K)$ such that $B\tensor_{\QQ}K = M_{2\times 2}(K)$ and
      $$
           (b\tensor (s + t\theta ))^{\dagger}=\overline{b}\tensor (s-t\theta)
           $$
           for every $b\in B$, $s,t\in \QQ$, where $b\mapsto \overline{b}$ denotes the canonical involution of $B$. Write $\rho = b\tensor 1 + c\tensor \theta$. Using that $\rho^{\dagger} = \rho$ and $\rho^2\in \QQ$, we get $b=0$ and $\overline{c}=-c$. By the Skolem-Noether theorem (cf. \cite{involutions} Theorem 1.4), there exists $d\in B^{\times}$ such that $dc=-cd$; in particular, the reduced trace of $d$ is zero, so that $d^2\in \QQ$. Thus, we can take $j=d\tensor \theta$.%consequence of skolem-noether, cf.  vigneras corollaire 2.2
  \end{enumerate}
\end{proof}

Since any special endomorphism $j$ of an abelian surface with $R$-multiplication necessarily satisfies $j^2=N$ for some integer $N>0$ (see \cite{HY12} Corollary 3.1.4), we obtain the following corollary. 

\begin{coro}
  Let $\tau \in \mathbf{H}^2$. If $j_F(\tau)\notin \bigcup_{N=1}^{\infty}T_F(N)$, then $\mathbf{X}_{F,\tau}$ is simple and the $\QQ$-division algebra $\End^0(\mathbf{X}_{F,\tau})$ is isomorphic to
  \begin{enumerate}
  \item $E\supset F$ a totally imaginary quadratic extension of $F$, if $\tau$ is CM;
    \item $F$ otherwise. \hfill $\blacksquare$
  \end{enumerate}
\end{coro}

We are now in position to relate Conjecture \ref{conj-Fquadratic} with Grothendieck's Period Conjecture.

Let $\tau\in \mathbf{H}^2\setminus \bigcup_{N=1}^{\infty}j_F^{-1}(T_F(N))$. Set $d \defeq \dim \MT(\mathbf{X}_{F,\tau})$, and $t \defeq \trdeg_{\QQ}\mathcal{P}(\mathbf{X}_{F,\tau}/\QQ(j_F(\tau)))$. It follows from the above corollary, and from the list of possible Mumford-Tate groups of abelian surfaces (see \cite{MZ99} 2.2), that $d=3$ if $\tau$ is CM and $d=7$ otherwise.

Recall from the introduction that the generalized Grothendieck's Period Conjecture asserts that $t\ge d$, i.e., that $t\ge 3$ if $\tau$ is CM and $t\ge 7$ otherwise. By Theorem \ref{trdegrm}, we have $t= \trdeg_{\QQ}\QQ(2\pi i, \tau, \varphi_F(\tau))$. Thus, when $\tau$ is CM, it follows from Corollary \ref{coro-periodcm} that Conjecture \ref{conj-Fquadratic} is \emph{equivalent} to Grothendieck's Period Conjecture for the abelian variety $\mathbf{X}_{F,\tau}$. If $\tau$ is not CM, then, since
$$
\trdeg_{\QQ(\varphi_F(\tau))}\QQ(2\pi i, \tau, \varphi_F(\tau))\le 3\text{,}
$$
Corollary \ref{coro-periodcm} is simply a weaker (but still non-trivial) statement than Grothendieck's Period conjecture for $\mathbf{X}_{F,\tau}$.

Despite being generally weaker than Grothendieck's Period Conjecture, our statement in Conjecture \ref{conj-Fquadratic} already contains some classical transcendence problems, such as the algebraic independence of $\pi$, $\Gamma(1/5)$, and $\Gamma(2/5)$, if $F=\QQ(\sqrt{5})$. Indeed, it is classical (see \cite{vasilev96} Paragraph 4, and references therein) that $\pi$, $\Gamma(1/5)$, and $\Gamma(2/5)$, are generators of the field of periods over $\overline{\QQ}$ of the Jacobian $J(C)$ of the hyperelliptic curve $C$ over $\QQ$ given by the affine equation
$$
C: y^2 = 1-x^5\text{;}
$$
observe that $\mu_5 = \{\zeta \in \overline{\CC} \mid \zeta^5=1\}$ acts on $C$ via
$$
\zeta \cdot (x,y) = (\zeta x ,y)\text{,}
$$
so that $J(C)$, with its canonical principal polarization, admits a real multiplication by $R = \ZZ[(1+\sqrt{5})/2]$ and is actually CM, with CM field $\QQ(\mu_5)$.

\section{Group-theoretic description of the higher Ramanujan vector fields}\label{gpinterpret}

This section is devoted to an alternative description of the complex manifold $B_g(\CC)$ (resp. $B_F(\CC)$) as a domain in the quotient of some Lie group by a discrete subgroup. Under this analytic description, we also give explicit formulas for the higher Ramanujan vector fields and for the solution $\varphi_g: \mathbf{H}_g \to B_g(\CC)$ (resp. $\varphi_F: \mathbf{H}^g \to B_F(\CC)$) of the higher Ramanujan equations.

These results will be applied in Section \ref{sec-zariskidensity} to obtain explicit parametrization of every analytic leaf of the Ramanujan foliation $\mathcal{R}_g^{\an}$ on $B_g(\CC)$ (resp. $\mathcal{R}_F^{\an}$ on $B_F(\CC)$).

\subsection{Realization of $B_g(\CC)$ as an open submanifold of $\Sp_{2g}(\ZZ)\backslash\Sp_{2g}(\CC)$} \label{parareal}

Let $\mathbf{B}_g = B(\mathbf{X}_g,E_g)$ be the principal $P_g(\CC)$-bundle over $\mathbf{H}_g$ associated to the principally polarized complex torus $(\mathbf{X}_g,E_g)_{/\mathbf{H}_g}$ as defined in Lemma \ref{relrepr}, so that the fiber of $\mathbf{B}_g \to \mathbf{H}_g$ over $\tau \in \mathbf{H}_g$ is given by the set of symplectic-Hodge bases of $(\mathbf{X}_{g,\tau},E_{g,\tau})$. 

We shall first realize $\mathbf{B}_g$ as a ``period domain'' in $\Sp_{2g}(\CC)$. For this, let us introduce the following convenient modification of period matrices (Definition \ref{defiperiodmatrix}).

\begin{defi}
  Let $(X,E)$ be a principally polarized complex torus of dimension $g$, and $b$ (resp. $\beta$) be a symplectic-Hodge basis (resp. an integral symplectic basis) of $(X,E)$. Let
  \begin{align*}
P(X,E,b,\beta) = \left(\begin{array}{cc}
                               \Omega_1 & N_1 \\
                               \Omega_2 & N_2
                               \end{array}\right) \in {\GSp}_{2g}(\CC)
\end{align*}
be the period matrix of $(X,E)$ with respect to $b$ and $\beta$. We define
\begin{align*}
\Pi(X,E,b,\beta) \defeq \left(\begin{array}{cc}
                               N_2 & \frac{1}{2\pi i}\Omega_2 \\[0.5em]
                               N_1 & \frac{1}{2\pi i}\Omega_1
                               \end{array}\right) \in {\Sp}_{2g}(\CC)
\end{align*}
Observe that this matrix is indeed symplectic by Lemma \ref{lemmaperiodmatrix}.
\end{defi}

We define a holomorphic map $\Pi : \mathbf{B}_g \to {\Sp}_{2g}(\CC)$ as follows. Let $q$ be a point in $\mathbf{B}_g$ lying above $\tau \in \mathbf{H}_g$, and corresponding to a symplectic-Hodge basis $b$ of $(\mathbf{X}_{g,\tau},E_{g,\tau})$, then
\begin{align*}
\Pi(q) \defeq \Pi (\mathbf{X}_{g,\tau},E_{g,\tau},b,\beta_{g,\tau})
\end{align*}
where $\beta_g$ is the integral symplectic basis of $(\mathbf{X}_g,E_{g})_{/\mathbf{H}_g}$ defined in Example \ref{intsymplbasis}. 

\begin{obs} \label{moduliinterpret}
Alternatively, recall that $\mathbf{H}_g$ may be regarded as the moduli space for principally polarized complex tori of dimension $g$ endowed with an integral symplectic basis (Proposition \ref{reprsintsympl}). In particular, as already remarked in the proof of Proposition \ref{reprsympl}, points in $\mathbf{B}_g$ correspond to isomorphism classes $[(X,E,b,\beta)]$ of quadruples $(X,E,b,\beta)$, where $(X,E)$ is a principally polarized complex torus of dimension $g$, and $b$ (resp. $\beta$) is a symplectic-Hodge basis (resp. integral symplectic basis) of $(X,E)$. Under this identification, the map $\Pi : \mathbf{B}_g \to \Sp_{2g}(\CC)$ is given by $[(X,E,b,\beta)] \mapsto \Pi (X,E,b,\beta)$.
\end{obs}
 
Let us consider the moduli-theoretic interpretation of $\mathbf{B}_g$ of the above remark, and recall that $\mathbf{B}_g$ is endowed with a natural left action of the discrete group $\Sp_{2g}(\ZZ)$ given by
\begin{align*}
\left(\begin{array}{cc}A & B \\ C & D \end{array}\right) \cdot [(X,E,b,\beta)] = \left[\left(X,E,b, \beta\cdot \left(\begin{array}{cc}
D\transp & B\transp \\
C\transp & A\transp
\end{array}\right)
\right)\right]
\end{align*} (cf. Remark \ref{actionsp}), and a right action of the Siegel parabolic subgroup $P_g(\CC)\le \Sp_{2g}(\CC)$ given by
\begin{align*}
[(X,E,b,\beta)]\cdot p = [(X,E,b\cdot p,\beta)]\text{,}
\end{align*}
where both $\beta$ and $b$ are regarded as row vectors of order $2g$.

 Let us denote by $P'_g$ the subgroup scheme of $\Sp_{2g}$ consisting of matrices $(A \ B \ ; C \ D )$ such that $B=0$. A simple computation proves the following equivariance properties of $\Pi : \mathbf{B}_g \to \Sp_{2g}(\CC)$.

\begin{lemma}\label{equivariancepi}
Consider the isomorphism of groups
\begin{align*}
P_g(\CC) &\stackrel{\sim}{\to} P'_g(\CC) \\
 p=\left( \begin{array}{cc}
    A & B \\
    0 & (A\transp)^{-1}
    \end{array}\right) &\mapsto p' \defeq \left( \begin{array}{cc}
    (A\transp)^{-1} & 0 \\
    2\pi i B & A
    \end{array}\right)\text{.}
\end{align*} 
Then, for any $q \in \mathbf{B}_g$, $\gamma \in \Sp_{2g}(\ZZ)$, and $p \in P_g(\CC)$, we have
\begin{align*}
\Pi(\gamma\cdot q) = \gamma\Pi(q)\ \ \ \text{ and }\ \ \ \Pi(q\cdot p) = \Pi(q)p'
\end{align*}
in $\Sp_{2g}(\CC)$.
\end{lemma}

Let us now consider the \emph{Lagrangian Grassmannian}, namely the smooth and quasi-projective $\CC$-scheme of dimension $g(g+1)/2$ obtained as the quotient of complex affine algebraic groups
\begin{align*}
L_g \defeq {\Sp}_{2g,\CC}/P'_{g,\CC}\text{.}
\end{align*}
The complex manifold $L_g(\CC)=\Sp_{2g}(\CC)/P'_g(\CC)$ may be naturally identified with the quotient of
\begin{align*}
  M\defeq \{(Z_1,Z_2) \in M_{g\times g}(\CC)\times M_{g\times g}(\CC) \mid Z_1\transp Z_2 = Z_2\transp Z_1\text{, } \text{rank} (Z_1 \ Z_2) = g\}
\end{align*}
by the right action of $\GL_g(\CC)$ defined by matrix multiplication:
\begin{align*}
(Z_1,Z_2) \cdot S \defeq (Z_1S,Z_2S)\text{.}
\end{align*}
We denote the class in $L_g(\CC)$ of a point $(Z_1,Z_2)\in M$ by $(Z_1:Z_2)$. The canonical map
\begin{align*}
\pi: {\Sp}_{2g,\CC} \to L_g
\end{align*}
is then given on complex points by
\begin{align*}
\pi \left(\begin{array}{cc}A & B \\ C & D \end{array} \right) = (B:D)\text{.}
\end{align*}

\begin{prop}\label{prop1}
  Let $\iota: \mathbf{H}_g \to L_g(\CC)$ be the open embedding given by $\iota(\tau)=(\tau: \mathbf{1}_g)$. Then the diagram of complex manifolds
%$$
%  \raisebox{-0.5\height}{\includegraphics{rameq2-d9.pdf}}
%$$ 
$$
\begin{tikzcd}
\mathbf{B}_g \arrow{r}{\Pi}\arrow{d} & \Sp_{2g}(\CC) \arrow{d}{\pi} \\
\mathbf{H}_g \arrow{r}[swap]{\iota} & L_g(\CC)
\end{tikzcd}
$$
is Cartesian. That is, $\Pi: \mathbf{B}_g \to \Sp_{2g}(\CC)$ induces a biholomorphism of $\mathbf{B}_g$ onto the open submanifold
\begin{align*}
\pi^{-1}(\iota(\mathbf{H}_g))=\left\{\left.\left(\begin{array}{cc}A & B \\ C & D \end{array} \right) \in {\Sp}_{2g}(\CC) \right| D \in {\GL}_{g}(\CC)\text{, }BD^{-1} \in \mathbf{H}_g \right\}
\end{align*}
of $\Sp_{2g}(\CC)$, and makes the above diagram commute.
\end{prop}

\begin{proof}
The commutativity of the diagram in the statement is easy (cf. proof of Proposition \ref{reprsintsympl}). In particular, if $q,q'\in \mathbf{B}_g$ satisfy $\Pi(q)=\Pi(q')$, then they lie above the same point $\tau \in \mathbf{H}_g$. Let $b$ (resp. $b'$) be the symplectic-Hodge basis of $(\mathbf{X}_{g,\tau},E_{g,\tau})$ corresponding to $q$ (resp. $q'$). Since period matrices are base change matrices for the comparison isomorphism,  and
\begin{align*}
\Pi(\mathbf{X}_{g,\tau}, E_{g,\tau},b, \beta_{g,\tau})=\Pi(\mathbf{X}_{g,\tau}, E_{g,\tau},b', \beta_{g,\tau})\text{,}
\end{align*}
it is clear that $b=b'$. This proves that $\Pi$ is injective.

Observe that $\mathbf{B}_g$ and $\Sp_{2g}(\CC)$ are complex manifolds of same dimension. Thus, to finish our proof, it suffices to check that  $\Pi(\mathbf{B}_g)=\pi^{-1}(\iota(\mathbf{H}_g))$ (\cite{GH78} p. 19). Let $s \in\pi^{-1}(\iota(\mathbf{H}_g))$, and let $\tau \in \mathbf{H}_g$ be such that $\iota(\tau)=\pi(s)$. Fix any $q \in \mathbf{B}_g$ lying above $\tau \in \mathbf{H}_g$. Then, there exists a unique $p' \in P_g'(\CC)$ such that $s=\Pi(q)p'$. Hence, by Lemma \ref{equivariancepi}, $s= \Pi(q\cdot p) \in \Pi(\mathbf{B}_g)$.
\end{proof}

\begin{obs}
In other words, through period matrices, one can realize the moduli space $\mathbf{B}_g$ as an open submanifold of $\Sp_{2g}(\CC)$ given by some positivity condition. For a more direct Hodge-theoretic approach, we refer to \cite{movasati13} Section 4.1.
\end{obs}

Recall from Proposition \ref{reprsympl} that the canonical map 
\begin{align}\label{canmap}
\begin{split}
\mathbf{B}_g &\to B_g(\CC)\\
       [(X,E,b,\beta)] &\mapsto [(X,E,b)]
\end{split}
\end{align}
induces a biholomorphism
\begin{align*}
{\Sp}_{2g}(\ZZ) \backslash \mathbf{B}_g \stackrel{\sim}{\to} B_g(\CC)\text{.}
\end{align*}
Furthermore, note that Lemma \ref{equivariancepi} implies that the action of $\Sp_{2g}(\ZZ)$ on $\Sp_{2g}(\CC)$ by left multiplication preserves the open subset $\Pi(\mathbf{B}_g)$.

\begin{coro} \label{realization}
The map $\Pi : \mathbf{B}_g \to \Sp_{2g}(\CC)$ induces a biholomorphism of $B_g(\CC)$ onto the open submanifold of $\Sp_{2g}(\ZZ) \backslash \Sp_{2g}(\CC)$
\begin{align*}
{\Sp}_{2g}(\ZZ)\setminus \Pi(\mathbf{B}_g) = \{{\Sp}_{2g}(\ZZ) s  \in {\Sp}_{2g}(\ZZ) \backslash {\Sp}_{2g}(\CC) \mid \pi(s) \in \iota(\mathbf{H}_g) \}\text{.}
\end{align*}\hfill $\blacksquare$
\end{coro}

\subsection{Explicit analytic description of the higher Ramanujan vector fields $v_{ij}$ and of $\varphi_g$}

Recall that the Lie algebra of $\Sp_{2g}(\CC)$ is given by
\begin{align*}
\Lie {\Sp}_{2g} (\CC) = \left\{\left.\left(\begin{array}{cc}A & B \\ C & D \end{array}\right) \in M_{2g\times 2g}(\CC)\right|B\transp=B\text{, }C\transp = C\text{, }D = -A\transp\right\}\text{.}
\end{align*}
For $1\le k \le l \le g$, let us consider the left invariant holomorphic vector field $\tilde{V}_{kl}$ on $\Sp_{2g}(\CC)$ corresponding to
\begin{align*}
\frac{1}{2\pi i}\left(\begin{array}{cc}0 & \mathbf{E}^{kl} \\ 0 & 0 \end{array}\right) \in \Lie {\Sp}_{2g}(\CC)\text{;}
\end{align*}
it descends to a holomorphic vector field $V_{kl}$ on the quotient $\Sp_{2g}(\ZZ)\backslash \Sp_{2g}(\CC)$.

\begin{theorem}\label{unifhrvf}
  Let $(v_{kl})_{1\le k \le l \le g}$ be the higher Ramanujan vector fields on $B_g(\CC)$. Under the identification of $B_g(\CC)$ with an open submanifold of $\Sp_{2g}(\ZZ)\backslash \Sp_{2g}(\CC)$ of Corollary \ref{realization}, we have:
  \begin{enumerate}
     \item For every $1\le k \le l \le g$,
\begin{align*}
v_{kl} = V_{kl}|_{B_g(\CC)}\text{.}
\end{align*}
\item The analytic solution of the higher Ramanujan equations $\varphi_g: \mathbf{H}_g \to B_g(\CC)$ is given by
  \begin{align*}
    \varphi_g(\tau) = {\Sp}_{2g}(\ZZ)\left(\begin{array}{cc}\mathbf{1}_g & \tau \\ 0 &\mathbf{1}_g\end{array} \right) \in {\Sp}_{2g}(\ZZ)\backslash {\Sp}_{2g}(\CC)\text{.}
    \end{align*}
   \end{enumerate}
 \end{theorem}

 As an example of application, we can prove the following easy consequence of the above theorem.
 
 \begin{coro}\label{ramleafclosed}
The image of $\varphi_g: \mathbf{H}_g \to B_g(\CC)$ is closed for the analytic topology.
\end{coro}

\begin{proof}
  Consider the subgroup
  $$
  U_g(\CC) \defeq \left.\left\{\left(\begin{array}{cc} \mathbf{1}_g & Z \\ 0 & \mathbf{1}_g\end{array}\right) \in M_{2g\times 2g}(\CC) \right| Z\transp =Z\right\} \le {\Sp}_{2g}(\CC)\text{.}
  $$
  The statement is equivalent to asserting that the image of $U_g(\CC)\subset \Sp_{2g}(\CC)$ in the quotient $\Sp_{2g}(\ZZ)\backslash \Sp_{2g}(\CC)$ is closed, or, equivalently, that $\Sp_{2g}(\ZZ)\cdot U_g(\CC) \subset \Sp_{2g}(\CC)$ is closed. Let us consider the (holomorphic) map
  \begin{align*}
    f: {\Sp}_{2g}(\CC) &\to M_{g\times g}(\CC)\times M_{g\times g}(\CC)\\
                \left(\begin{array}{cc}A & B\\ C & D \end{array} \right)&\mapsto (A,C)\text{.}
  \end{align*}
  Now, one simply remarks that
\begin{align*}
{\Sp}_{2g}(\ZZ)\cdot U_g(\CC) = f^{-1}(f({\Sp}_{2g}(\ZZ)))\text{.}
\end{align*}
Since $f(\Sp_{2g}(\ZZ))\subset M_{g\times g}(\ZZ) \times M_{g\times g}(\ZZ)$, and $M_g(\ZZ)\times M_g(\ZZ)$ is a closed discrete subset of $M_{g\times g}(\CC)\times M_{g\times g}(\CC)$ for the analytic topology, we conclude that  ${\Sp}_{2g}(\ZZ)\cdot U_g(\CC)$ is closed in $\Sp_{2g}(\CC)$.
\end{proof}

We prove parts (1) and (2) of Theorem \ref{unifhrvf} separately.

\begin{proof}[Proof of Theorem \ref{unifhrvf} (1)]
  It is sufficient to prove that the solutions of the differential equations defined by $v_{kl}$ and by $V_{kl}$ coincide. More precisely, let $U$ be a simply connected open subset of $\mathbf{H}_g$, and $u:U \to B_g(\CC)$ be a solution of the higher Ramanujan equations (Definition \ref{defihreq}); we shall prove that, for any lifting
% $$
%  \raisebox{-0.5\height}{\includegraphics{rameq2-d10.pdf}}
%$$
 $$
 \begin{tikzcd}
   & \mathbf{B}_g\arrow{d}\\
   U \arrow[bend left]{ur}{\tilde{u}}\arrow{r}[swap]{u} & B_g(\CC)
 \end{tikzcd}
 $$
 of $u$, the holomorphic map $h \defeq \Pi \circ \tilde{u} : U \to \Sp_{2g}(\CC)$ is a solution of the differential equations
\begin{align} \label{rameqdisg}
\theta_{kl}h = \tilde{V}_{kl}\circ h  \text{, }\ \ \ 1\le k \le l \le g \text{.}
\end{align}
where $\theta_{kl} = \frac{1}{2\pi i}\frac{\partial}{\partial \tau_{kl}}$.

By the universal property of $\mathbf{B}_g$, the holomorphic map $\tilde{u}$ corresponds to a principally polarized complex torus $(X,E)$ over $U$, of relative dimension $g$, endowed with a symplectic-Hodge basis $b = (\omega_1,\ldots,\omega_g,\eta_1,\ldots,\eta_g)$ and an integral symplectic basis $\beta = (\gamma_1,\ldots,\gamma_g,\delta_1,\ldots,\delta_g)$. For $\tau\in U$, let us write
\begin{align*}
h(\tau) = \left(\begin{array}{cc} N_2(\tau) & \frac{1}{2\pi i}\Omega_2(\tau)\\[0.5em]
                               N_1(\tau) & \frac{1}{2\pi i}\Omega_1(\tau) \end{array} \right) \in {\Sp}_{2g}(\CC)
\end{align*}
where $\Omega_1,\Omega_2,N_1,N_2: U \to M_{g\times g}(\CC)$ are holomorphic.

Now, since $u$ is a solution of the higher Ramanujan equations, it follows from Proposition \ref{equivalences} (3) that, for every $1\le i \le j \le g$,
\begin{enumerate}
   \item[(i)] $\theta_{ij}\Omega_1 = N_1\mathbf{E}^{ij}$, $\theta_{ij}\Omega_2 = N_2\mathbf{E}^{ij}$
   \item[(ii)] $\theta_{ij}N_1 = 0$, $\theta_{ij}N_2=0$.
\end{enumerate}
As $U$ is connected, (ii) implies that $N_1$ and $N_2$ are constant. Thus, (i) implies that $\frac{1}{2\pi i}\Omega_1 -  N_1\tau$ and $\frac{1}{2\pi i}\Omega_2 -  N_2\tau$ are also constant. In other words, there exists a unique element $s \in \Sp_{2g}(\CC)$ such that
\begin{align*}
h(\tau) = s\left(\begin{array}{cc}\mathbf{1}_g & \tau \\ 0 &\mathbf{1}_g\end{array} \right)
\end{align*}
for every $\tau \in U$. Finally, since each $\tilde{V}_{kl}$ is left invariant, it is easy to see that $h$ is a solution of the differential equations (\ref{rameqdisg}).
\end{proof}

\begin{lemma} \label{psi}
For any $\tau \in \mathbf{H}_g$, we have
  \begin{align*}
    \Pi (\mathbf{X}_{g,\tau},E_{g,\tau}, \bfb_{g,\tau},\beta_{g,\tau}) = \left(\begin{array}{cc} \mathbf{1}_g & \tau \\ 0 & \mathbf{1}_g\end{array} \right)\text{.}
  \end{align*}
\end{lemma}

\begin{proof}
 Let us write
  $$
  \Pi (\mathbf{X}_{g,\tau},E_{g,\tau}, \bfb_{g,\tau},\beta_{g,\tau}) = \left(\begin{array}{cc}
                               N_2(\tau) & \frac{1}{2\pi i}\Omega_2(\tau)  \\[0.5em]
                               N_1(\tau) & \frac{1}{2\pi i}\Omega_1(\tau)
                               \end{array}\right)\text{.}
                             $$
By definition of $\beta_g$ and of $\bfb_g$, it is clear that $\Omega_1(\tau) = 2\pi i\mathbf{1}_g$ and that $\Omega_{2}(\tau) = 2\pi i \tau$. That $N_1(\tau) = 0$ and $N_2(\tau)=\mathbf{1}_g$ is a reformulation of Corollary \ref{caraceta}.
\end{proof}

\begin{proof}[Proof of Theorem \ref{unifhrvf} (2)]
By definition, $\varphi_g$ is given by the composition of
\begin{align*}
\mathbf{H}_g &\to \mathbf{B}_g\\
         \tau &\mapsto [(\mathbf{X}_{g,\tau},E_{g,\tau}, \bfb_{g,\tau},\beta_{g,\tau})]
\end{align*}
with the canonical map $\mathbf{B}_g \to B_g(\CC)$. The result now follows from Lemma \ref{psi}.
\end{proof}

\subsection{Group-theoretic description of $\mathcal{B}_F$, $v_F$, and $\varphi_F$}\label{subsec-gtdescrrm}

In this paragraph, we consider the Hilbert-Blumenthal analogs of the above results. As usual, most proofs here are omitted due to their similarity to those concerning the Siegel case.

Recall that we have defined in Paragraph \ref{subsec-periodsrm} a subgroup scheme $G_F$ of ${\Res}_{R/\ZZ}{\Aut}_M$, where $M = R\oplus D^{-1}$. We set
$$
S_F\defeq \ker (\det: G_F \to \GG_m) = {\Res}_{R/\ZZ}{\Aut}_{(M,\Psi)}\text{,}
$$
where $\Psi$ is the standard $D^{-1}$-valued $R$-bilinear symplectic form on $M$.
Observe that
$$
P_F = {\Res}_{R/\ZZ} {\Aut}_{(M,\Psi,R\oplus 0)}
$$
defined in Paragraph \ref{subsec-proofsmothdmstackrm} is a parabolic subgroup of $S_F$.

We shall also need the dual counterparts of $S_F$ and $P_F$. Namely, consider the $\ZZ$-dual $M^{\vee} = D^{-1}\oplus R$, with its standard $D^{-1}$-valued $R$-bilinear symplectic form $\Phi$, and set
$$
S'_F \defeq {\Res}_{R/\ZZ}{\Aut}_{(M^{\vee},\Phi)}\text{, } \ \ \ P_F' \defeq {\Res}_{R/\ZZ} {\Aut}_{(M,\Phi,0\oplus R)}\text{.}
$$
For a commutative ring $\Lambda$, if $V=\Spec \Lambda$, and $S_F(V),P_F(V),S_F'(V),P_F'(V)$ are regarded as subgroups of $\GL_2(\Lambda \tensor F)$, then $S_F'(V)$ (resp. $P_F'(V)$) is simply the image of $S_F(V)$ (resp. $P_F(V)$) under the operation of matrix transposition $s\mapsto s\transp$.

Also, observe that $S_F'(\ZZ)$ is the group $\SL(D^{-1}\oplus R)$ considered in Example \ref{ex-actionsl} and  Remark \ref{rem-leftactrm}.

\begin{defi}
  Let $(X,E,m)$ be a principally polarized complex torus with $R$-multiplication, and $b$ (resp. $\beta$) be a symplectic-Hodge basis (resp. an integral symplectic basis) of $(X,E,m)$. Let
  \begin{align*}
P(X,E,m,b,\beta) = \left(\begin{array}{cc}
                               \omega_1 & \eta_1 \\
                               \omega_2 & \eta_2
                               \end{array}\right) \in G_F(\CC)
\end{align*}
be the period matrix of $(X,E,m)$ with respect to $b$ and $\beta$, as defined in Paragraph \ref{subsec-periodsrm}. We set
\begin{align*}
\Pi(X,E,m,b,\beta) \defeq \left(\begin{array}{cc}
                               \eta_2 & \frac{1}{2\pi i}\cdot \omega_2 \\[0.5em]
                               \eta_1 & \frac{1}{2\pi i}\cdot \omega_1
                               \end{array}\right) \in S_F'(\CC)
\end{align*}
Observe that $\Pi(X,E,m,b,\beta)$ indeed belongs to $S_F'(\CC)$ by Lemma \ref{lemma-periodmatrixrm}.
\end{defi}

Let $\mathbf{B}_F = B(\mathbf{X}_F,E_F,m_F)$ be the principal $P_F$-bundle over $\mathbf{H}^g$ associated to the principally polarized torus with $R$-multiplication $(\mathbf{X}_F,E_F,m_F)_{/\mathbf{H}^g}$. The manifold $\mathbf{B}_F$ can also be regarded as the moduli space of principally polarized complex tori with $R$-multiplication equipped with a symplectic-Hodge basis and an integral symplectic basis, so that we have a holomorphic map
\begin{align*}
  \Pi : \mathbf{B}_F &\to S_F'(\CC)\\
            [(X,E,m,b,\beta)]&\mapsto \Pi(X,E,m,b,\beta)\text{.}
\end{align*}

The space $\mathbf{B}_F$ is endowed with a left action of $S'_F(\ZZ)$ given by
$$
\left(\begin{array}{cc}
  a & b \\
  c & d
\end{array}\right)\cdot [(X,E,m,b,\beta)] = \left[\left(X,E,m,b,\beta\cdot \left(\begin{array}{cc}
  d & b \\
  c & a
\end{array}\right)\right)\right] 
$$
and a right action of $P_F(\CC)$ given by
$$
[(X,E,m,b,\beta)]\cdot p = [(X,E,m,b\cdot p,\beta)]\text{.}
$$

\begin{lemma}[cf. Lemma \ref{equivariancepi}]
  Consider the isomorphism of groups
  \begin{align*}
    P_F(\CC) &\stackrel{\sim}{\to}P'_F(\CC)\\
    p=\left(\begin{array}{cc}
              a & b \\
              0 & a^{-1}
              \end{array}\right)&\mapsto p' \defeq \left(\begin{array}{cc}
              a^{-1} & 0 \\
              2\pi i\cdot  b & a
              \end{array}\right)\text{.}
  \end{align*}
  Then, for any $q\in \mathbf{B}_F$, $\gamma \in S_F'(\ZZ)$, and $p\in P_F(\CC)$, we have
  $$
\Pi(\gamma\cdot q) = \gamma \Pi(q) \ \ \text{ and }\ \ \Pi(q\cdot p) = \Pi(q)p'
$$
in $S_F'(\CC)$.\hfill $\blacksquare$
\end{lemma}

Consider the smooth quasi-projective $\CC$-scheme of dimension $g$ obtained as the quotient of complex affine algebraic groups
$$
L_F \defeq S_{F,\CC}'/P_{F,\CC}'\text{.}
$$
Observe that for any fractional ideal $I$ of $F$ we have $\CC\tensor I =\CC\tensor R = \CC\tensor_{\QQ}F$. In particular, $S'_{F}(\CC) = \SL_{2}(\CC\tensor R)$, and $L_F(\CC)$ may be identified with $\PP^1(\CC\tensor R)$; the quotient map 
$$
\pi : S_{F,\CC}' \to L_F
$$
is then given at complex points by
$$
\pi\left(\begin{array}{cc}
        a & b \\
        c & d
      \end{array}  
\right)=(b:d)\text{.}
$$

In the next proposition, we identify $\mathbf{H}^g$ with an open submanifold of $\CC\tensor D^{-1}=\CC\tensor R$ via $(\sigma_1,\ldots,\sigma_g): \CC\tensor R \stackrel{\sim}{\to} \CC^g$.

\begin{prop}[cf. Proposition \ref{prop1}]
  Let $\iota: \mathbf{H}^g\to L_F(\CC)$ be the open embedding given by $\iota(\tau) =(\tau:1)$. Then the diagram of complex manifolds
$$
\begin{tikzcd}
\mathbf{B}_F \arrow{r}{\Pi}\arrow{d} & S_F'(\CC) \arrow{d}{\pi} \\
\mathbf{H}^g \arrow{r}[swap]{\iota} & L_F(\CC)
\end{tikzcd}
$$
is Cartesian. That is, $\Pi: \mathbf{B}_F \to S'_F(\CC)$ induces a biholomorphism of $\mathbf{B}_F$ onto the open submanifold
\begin{align*}
\pi^{-1}(\iota(\mathbf{H}^g))=\left\{\left.\left(\begin{array}{cc}a & b \\ c & d \end{array} \right) \in S_F'(\CC) \right| d \in (\CC\tensor R)^{\times}\text{, }bd^{-1} \in \mathbf{H}^g \right\}
\end{align*}
of $S'_F(\CC)$, and makes the above diagram commute. \hfill $\blacksquare$
\end{prop}

Since the canonical map
\begin{align*}
  \mathbf{B}_F &\to B_F(\CC)\\
     [(X,E,m,b,\beta)]&\mapsto [(X,E,m,b)]
\end{align*}
induces a biholomorphism
$$
S_F'(\ZZ)\backslash\mathbf{B}_F \stackrel{\sim}{\to}B_F(\CC)\text{,}
$$
we obtain the next corollary.

\begin{coro}[cf. Corollary \ref{realization}]\label{coro-realrm}
The map $\Pi : \mathbf{B}_F \to S'_F(\CC)$ induces a biholomorphism of $B_F(\CC)$ onto the open submanifold of $S_F'(\ZZ) \backslash S_F'(\CC)$
\begin{align*}
S'_F(\ZZ)\setminus \Pi(\mathbf{B}_F) = \{S_F'(\ZZ) s  \in S_F'(\ZZ) \backslash S_F'(\CC) \mid \pi(s) \in \iota(\mathbf{H}^g) \}\text{.}
\end{align*}\hfill $\blacksquare$
\end{coro}

The Lie algebra of $S'_F(\CC) = \SL_2(\CC\tensor R)$ is given by
\begin{align*}
\Lie S_F'(\CC) = \left\{\left.\left(\begin{array}{cc}a & b \\ c & d \end{array}\right) \in M_{2\times 2}(\CC\tensor R)\right|a+d =0\right\}\text{.}
\end{align*}
Let 
$$
\tilde{V}_F:\mathcal{O}_{S_F'(\CC)}\tensor D^{-1} \to T_{S_F'(\CC)} 
$$
be the unique $\mathcal{O}_{S_F'(\CC)}$-morphism such that, for every $x \in D^{-1}$, $\tilde{V}_F(1\tensor x)$ is the left invariant holomorphic vector field over $T_{S_F'(\CC)}$ corresponding to
\begin{align*}
\frac{1}{2\pi i}\left(\begin{array}{cc}0 & 1\tensor x \\ 0 & 0 \end{array}\right) \in S_F'(\CC)\text{.}
\end{align*}
Note that $\tilde{V}_F$ descends to a $\mathcal{O}_{S_F'(\ZZ)\backslash S'_F(\CC)}$-morphism
$$
V_F : \mathcal{O}_{S_F'(\ZZ)\backslash S'_F(\CC)} \tensor D^{-1}\to T_{S_F'(\ZZ)\backslash S'_F(\CC)}\text{.}
$$ 

\begin{theorem}[cf. Theorem \ref{unifhrvf}]
  Let $v_F: \mathcal{O}_{B_F(\CC)}\tensor D^{-1}\to T_{B_F(\CC)}$ be the higher Ramanujan vector field on $B_F(\CC)$. Under the identification of $B_F(\CC)$ with an open submanifold of $S_F'(\ZZ)\backslash S_F'(\CC)$ of Corollary \ref{coro-realrm}:
  \begin{enumerate}
  \item We have
\begin{align*}
v_F = V_F|_{B_F(\CC)}\text{.}
\end{align*}
\item The analytic solution of the higher Ramanujan equations $\varphi_F: \mathbf{H}^g \to B_F(\CC)$ is given by
  \begin{align*}
    \varphi_F(\tau) = S_F'(\ZZ)\left(\begin{array}{cc}1 & \tau \\ 0 &1\end{array} \right) \in S_F'(\ZZ)\backslash S_F'(\CC)\text{.}
    \end{align*}
   \end{enumerate}\hfill $\blacksquare$
 \end{theorem}

 The proof of this theorem is, as usual, similar to that of the analogous Theorem \ref{unifhrvf}, but it deserves some comments. To prove (1), it is enough to show that, for any analytic solution of the higher Ramanujan equations over $\mathcal{B}_F$ (Definition \ref{defi-ahrerm}) defined on a connected open subset $U\subset \mathbf{H}^g$, say $u:U\to B_F(\CC)$, and any lifting $\tilde{u}: U \to \mathbf{B}_F$ of $u$, the composition $h\defeq \Pi \circ \tilde{u}: U \to S_K'(\CC)$ satisfies
\begin{align}\label{eq-diffequnif}
Th \circ \theta_F|_{U} = h^*\tilde{V}_F\text{,}
\end{align}
where $\theta_F : \mathcal{O}_{\mathbf{H}^g}\tensor D^{-1}\to T_{\mathbf{H}^g}$ is the $\mathcal{O}_{\mathbf{H}^g}$-morphism defined in Paragraph \ref{subsec-ahrerm}.

For any $x\in D^{-1}$, we may extend the derivation $\theta_F(1\tensor x)$ of $\mathcal{O}_{\mathbf{H}^g}$ to a derivation of $\mathcal{O}_{\mathbf{H}^g}\tensor D^{-1} = \mathcal{O}_{\mathbf{H}^g}\tensor R$ by requiring that $\theta_F(1\tensor x)(1\tensor r)=0$ for every $r\in R$.

\begin{lemma}\label{lemma-dertaurm}
  Let us regard the standard coordinate $\tau = (\tau_1,\ldots,\tau_g)$ of $\mathbf{H}^g$ as a global section of $\mathcal{O}_{\mathbf{H}^g}\tensor D^{-1}$ via the identification $(\sigma_1,\ldots,\sigma_g): \mathcal{O}_{\mathbf{H}^g}\tensor D^{-1} \stackrel{\sim}{\to} \mathcal{O}_{\mathbf{H}^g}^{\oplus g}$. Then, for every $x\in D^{-1}$,
  $$
\theta_F(1\tensor x) (\tau) = \frac{1}{2\pi i}\tensor x\text{.}
  $$
\end{lemma}

\begin{proof}
Follows immediately from Remark \ref{rem-explicitthetak}.
\end{proof}

We deduce from the above lemma and from the left invariance of $\tilde{V}_F(1\tensor x)$ that equation (\ref{eq-diffequnif}) is equivalent to asserting the existence of $s \in S_F'(\CC)$ such that
$$
h = s\left(\begin{array}{cc}
             1 & \tau \\
             0 & 0
           \end{array}\right) \in S_F'(\mathcal{O}_U(U))\text{.}
 $$
         
 For this, we write
 $$
h = \left(\begin{array}{cc}
             \eta_2 & \frac{1}{2\pi i}\cdot \omega_2 \\
             \eta_1 & \frac{1}{2\pi i}\cdot \omega_1
           \end{array}\right) \in S_F'(\mathcal{O}_U(U))
         $$
         and we remark that, as in the proof of Theorem \ref{unifhrvf}, it suffices to prove that $\eta_j$ and $\frac{1}{2\pi i}\omega_j - \eta_j\tau$ are constant for $j=1,2$; equivalently, we must prove that, for any $x \in D^{-1}$, $\theta_F(1\tensor x)(\eta_j) = 0$ and (by Lemma \ref{lemma-dertaurm}) $\theta_F(1\tensor x)(\omega_j) = \eta_j1\tensor x$. This, in turn, is a simple consequence of Proposition \ref{prop-characsolhrerm} and of the next lemma.

\begin{lemma}
  Let $M$ be a complex manifold, and $(\pi:X\to M,E,m)$ be a principally polarized complex torus with $R$-multiplication over $M$. Consider the $F$-linear pairing
\begin{align*}
  \mathcal{H}^1_{\dR}(X/M)\times R_1\pi_*\QQ_X &\to \mathcal{O}_M\tensor_{\QQ}F\\
  (\alpha,\gamma) &\mapsto I_\gamma \alpha
\end{align*}
defined as in Remark \ref{rem-coeffpmrm}. Then, for any section $\gamma$ of $R_1\pi_*\QQ_X$, $\alpha$ of $\mathcal{H}^1_{\dR}(X/M)$, and any holomorphic vector field $\theta$ on $M$, we have
  $$
\theta(I_{\gamma}\alpha) = I_\gamma(\nabla_{\theta}\alpha)\text{.}
  $$
\end{lemma} 

\begin{proof}
Use the corresponding result result for $\int$ and apply Remark \ref{rem-dualityrelation}.
\end{proof}

This concludes the proof of (1). The proof of (2) is a simple computation using Lemma \ref{lemma-dertaurm}. The proof of the next corollary is completely analogous to that of Corollary \ref{ramleafclosed}.

\begin{coro}[cf. Corollary \ref{ramleafclosed}]
The image of $\varphi_F: \mathbf{H}^g \to B_F(\CC)$ is closed for the analytic topology.\hfill $\blacksquare$
\end{coro}

\section{Zariski-density of leaves of the higher Ramanujan foliation}
\label{sec-zariskidensity}

 Let $\mathcal{R}^{\an}_g$ be the integrable subbundle of the holomorphic tangent bundle $T_{B_g(\CC)}$ induced by the Ramanujan subbundle $\mathcal{R}_g\subset T_{\mathcal{B}_g/\ZZ}$ introduced in Section \ref{ramvecfields}. By the holomorphic Frobenius Theorem, $\mathcal{R}^{\an}_g$ induces a holomorphic foliation on $B_g(\CC)$; we call it the \emph{higher Ramanujan foliation}.

% \begin{obs} \label{equivalence}
% Let $M$ be a complex manifold and $u: M \to B_g^{\an}$ be a holomorphic map. If $b_g=(\omega_1,\ldots,\omega_g,\eta_1,\ldots,\eta_g)$ denotes the universal symplectic-Hodge basis of $(X_g,\lambda_g)_{/B_g}$, then it follows from \cite{fonseca16} Remark 5.8 that $u$ is \emph{tangent} to $\mathcal{R}_g$ (i.e. $D_pu(T_{M,p})\subset \mathcal{R}_{g,p}$ for every $p\in M$)  if and only if $u^*(\nabla \eta_i)=0$ for every $1\le i \le g$, where $\nabla$ denotes the Gauss-Manin connection on $H^1_{\dR}(X^{\an}_g/B_g^{\an})$.
% \end{obs}

 In this section, we prove that every leaf of the higher Ramanujan foliation is Zariski-dense in $B_{g,\CC}$. In particular, we obtain that the image of the solution of the higher Ramanujan equations $\varphi_g: \mathbf{H}_g \to B_g(\CC)$ defined in Section \ref{sec-analhre} is Zariski-dense in $B_{g,\CC}$. We can actually derive from this the \emph{a priori} stronger result that the graph $\{(\tau,\varphi_g(\tau)) \in {\Sym}_g(\CC)\times B_g(\CC) \mid \tau \in \mathbf{H}_g\}$ is  Zariski-dense in $\Sym_{g,\CC}\times_{\CC} B_{g,\CC}$.

 We apply our Zariski-density results to relate our work to that of Bertrand and Zudilin \cite{BZ03}. Namely, using $\varphi_g$, we prove that the function field of $\mathcal{B}_{g,\QQ}$ is a finite extension of the field generated by derivatives of Siegel modular functions defined over $\QQ$.

Using what has been developed so far in paragraphs \ref{subsec-periodsrm} and \ref{subsec-gtdescrrm}, the above results can be easily carried over to the Hilbert-Blumenthal case. We provide precise statements below, but we omit proofs.

\subsection{Characterization of the leaves of the higher Ramanujan foliation}

\subsubsection{}\label{unipotentfoliation} Let $U_g$ be the unipotent subgroup scheme of $\Sp_{2g}$ defined by
\begin{align*}
U_g(R) = \left\{\left. \left(\begin{array}{cc} \mathbf{1}_g & Z \\ 0 & \mathbf{1}_g \end{array}\right) \in M_{2g\times 2g}(R)\right| Z\transp = Z \right\}
\end{align*}
for any ring $R$.

The Lie algebra of $U_g(\CC)$ is given by
\begin{align*}
\Lie U_g(\CC) = \left\{\left. \left(\begin{array}{cc} 0 & Z \\ 0 & 0 \end{array}\right) \in M_{2g\times 2g}(\CC)\right| Z\transp = Z \right\}\text{,}
\end{align*}
and admit as a basis the vectors
\begin{align*}
\frac{1}{2\pi i}\left(\begin{array}{cc}0 & \mathbf{E}^{kl} \\ 0 & 0 \end{array}\right) \in \Lie {U}_{g}(\CC)\text{, } \ \ \ 1\le k \le l \le g\text{,}
\end{align*}
inducing the higher Ramanujan vector fields on the quotient $\Sp_{2g}(\ZZ)\backslash \Sp_{2g}(\CC)$ (Section \ref{gpinterpret}). In particular, under the realization of $B_g(\CC)$ as an open submanifold of $\Sp_{2g}(\ZZ)\backslash \Sp_{2g}(\CC)$ of Corollary \ref{realization}, the higher Ramanujan foliation on $B_g(\CC)$ is induced by the foliation on $\Sp_{2g}(\CC)$ defined by $U_g(\CC)$, i.e. the foliation whose leaves are left cosets of $U_g(\CC)$ in $\Sp_{2g}(\CC)$.

It follows from the above discussion that, under the identification of $\mathbf{B}_g$ (resp. $B_g(\CC)$) with an open submanifold of $\Sp_{2g}(\CC)$ (resp. $\Sp_{2g}(\ZZ)\backslash \Sp_{2g}(\CC)$) via $\Pi$ (cf. Proposition \ref{prop1} and Corollary \ref{realization}), for any leaf $L$ of the higher Ramanujan foliation on $B_g(\CC)$, there exists $\delta \in \Sp_{2g}(\CC)$ such that $L$ is a connected component of the image of $\delta U_g(\CC)\cap \mathbf{B}_g$ in $B_g(\CC)$ under the quotient map $\Sp_{2g}(\CC)\to \Sp_{2g}(\ZZ)\backslash \Sp_{2g}(\CC)$. We shall provide a more precise result in Proposition \ref{charactleaves}.

\subsubsection{} We may also obtain an explicit \emph{parametrization} of every leaf. For this, let us consider $\Sym_g(\CC) = \{Z \in M_{g\times g}(\CC) \mid Z\transp = Z\}$ as an open subset of the Lagrangian Grassmannian $L_g(\CC)$ (cf. discussion preceding Proposition \ref{prop1}) via
\begin{align*}
  {\Sym}_g(\CC) &\to L_g(\CC)\\
          Z &\mapsto (Z:\mathbf{1}_g)\text{,}
\end{align*}
so that the embedding $\iota : \mathbf{H}_g \to L_g(\CC)$ defined in Proposition \ref{prop1} is given by the restriction of $\Sym_g(\CC) \to L_g(\CC)$ to $\mathbf{H}_g$. Furthermore, let
\begin{align*}
  \psi : {\Sym}_g(\CC) &\to {\Sp}_{2g}(\CC)\\
                 Z &\mapsto \left(\begin{array}{cc}\mathbf{1}_g & Z \\ 0 & \mathbf{1}_g \end{array}\right)\text{.}
 \end{align*}

\begin{obs}
Under the obvious identification of $\Sym_g(\CC)$ with $\Lie U_g(\CC)$, the map $\psi$ is simply the exponential $\exp : \Lie U_g(\CC) \to U_g(\CC)\subset \Sp_{2g}(\CC)$.
\end{obs}

Now, the action of $\Sp_{2g}(\CC)$ on itself by left multiplication descends to a left action of $\Sp_{2g}(\CC)$ on $L_g(\CC)$ given explicitly by
\begin{align*}
  \left( \begin{array}{cc}
           A & B \\
           C & D
 \end{array}\right)\cdot (Z_1:Z_2) = (AZ_1+BZ_2:CZ_1+DZ_2)\text{.}
\end{align*}
For any $\delta \in \Sp_{2g}(\CC)$, let us define
\begin{align*}
  \psi_{\delta} : \delta^{-1}\cdot {\Sym}_g(\CC)\subset L_g(\CC) &\to {\Sp}_{2g}(\CC)\\
                          p &\mapsto \delta^{-1}\psi(\delta\cdot p)\text{.}
\end{align*}
Then $\psi_{\delta}$ induces a biholomorphism of $\delta^{-1}\cdot {\Sym}_g(\CC)$ onto the closed submanifold $\delta^{-1}U_g(\CC)\subset \Sp_{2g}(\CC)$.

We put
\begin{align*}
U_{\delta} \defeq \{\tau \in \mathbf{H}_g \mid \delta\cdot (\tau:1) \in {\Sym}_g(\CC)\subset L_g(\CC)\} = (\delta^{-1}\cdot {\Sym}_g(\CC))\cap \mathbf{H}_g  \text{.}
\end{align*}
Equivalently, if $\delta = (A \ B \ ; \ C \ D)$, then
\begin{align*}
U_{\delta} = \{\tau \in \mathbf{H}_g \mid C\tau +D \in {\GL}_g(\CC)\}\text{.}
\end{align*}

\begin{defi}
  For any $\delta \in \Sp_{2g}(\CC)$, we define a holomorphic map $\varphi_{\delta} : U_{\delta} \to B_g(\CC) \subset \Sp_{2g}(\ZZ)\backslash \Sp_{2g}(\CC)$ by
 \begin{align*}
\varphi_{\delta}(\tau) \defeq {\Sp}_{2g}(\ZZ)\psi_{\delta}(\tau)
\end{align*}
for any $\tau \in U_{\delta}$.
\end{defi}

Note that $\psi_{\delta}(U_{\delta}) = \delta^{-1}U_g(\CC) \cap \mathbf{B}_g \subset \Sp_{2g}(\CC)$ by Proposition \ref{prop1}. In particular, the image of $\varphi_{\delta}$ is indeed in $B_g(\CC)$ . Moreover, if $\delta \in U_g(\CC)$, then $U_{\delta} = \mathbf{H}_g$ and $\varphi_{\delta}=\varphi_g$ (cf. Theorem \ref{unifhrvf} (2)).

\begin{lemma}\label{Uconnected}
For any $\delta \in \Sp_{2g}(\CC)$, $U_{\delta}$ is a dense connected open subset of $\mathbf{H}_g$.
\end{lemma}

\begin{proof}
Let $\delta = (A \ B \ ; \ C \ D) \in \Sp_{2g}(\CC)$. By definition, $U_{\delta}$ is the complement in $\mathbf{H}_g$ of the codimension 1 analytic subset $\{\tau \in \mathbf{H}_g \mid \det(C\tau+D)=0\}$. It is thus a dense open subset of $\mathbf{H}_g$. Since $\mathbf{H}_g$ is a connected open subset of an affine space, it follows from Riemann's extension theorem (cf. \cite{huybrechts05} Proposition 1.1.7)  that $U_{\delta}$ is connected. 
\end{proof}

\begin{prop}\label{charactleaves}
For every $\delta \in \Sp_{2g}(\CC)$, the image of the map $\varphi_{\delta} : U_{\delta} \to B_g(\CC)$ is a leaf of the higher Ramanujan foliation on $B_g(\CC)$, and  coincides with the image of $\delta^{-1}U_g(\CC)\cap \mathbf{B}_g$ in  $B_g(\CC)$ under the quotient map $\Sp_{2g}(\CC) \to \Sp_{2g}(\ZZ)\backslash \Sp_{2g}(\CC)$. Moreover, every leaf is of this form.
\end{prop}

\begin{proof}
  Let $\delta \in \Sp_{2g}(\CC)$. It was already remarked above that $\psi_{\delta}(U_{\delta}) = \delta^{-1}U_g(\CC)\cap \mathbf{B}_g$; by definition,  $\varphi_{\delta}(U_{\delta})$ is the image of $\psi_{\delta}(U_{\delta})$ under the quotient map $\Sp_{2g}(\CC) \to \Sp_{2g}(\ZZ)\backslash \Sp_{2g}(\CC)$. In particular, since the higher Ramanujan foliation on $B_g(\CC)$ is induced by the foliation on $\Sp_{2g}(\CC)$ defined by $U_g(\CC)$ (cf. \ref{unipotentfoliation}), to prove that $\varphi_{\delta}(U_{\delta})$ is a leaf of the higher Ramanujan foliation it is sufficient to prove that it is connected. This is an immediate consequence Lemma \ref{Uconnected}.

Conversely, if $L\subset B_g(\CC)$ is a leaf of the higher Ramanujan foliation, then it follows from \ref{unipotentfoliation} that there exists $\delta \in \Sp_{2g}(\CC)$ such that $L$ is a connected component of the image of $\delta^{-1}U_g(\CC)\cap \mathbf{B}_g$ in $B_g(\CC)$ under the quotient map $\Sp_{2g}(\CC) \to \Sp_{2g}(\ZZ)\backslash \Sp_{2g}(\CC)$. By the last paragraph, $\delta^{-1}U_g(\CC)\cap \mathbf{B}_g = \psi_{\delta}(U_{\delta})$ is connected, and we conclude that $L=\varphi_{\delta}(U_{\delta})$.
\end{proof}

\begin{obs} \label{remarkphi}
The holomorphic maps $\varphi_{\delta} : U_{\delta} \to B_g(\CC)$ are immersive but not injective in general. For instance, if $\delta = \mathbf{1}_{2g}$, then one easily verifies that $\varphi_{g}(\tau)=\varphi_{g}(\tau')$ if and only if $\tau' \in U_g(\ZZ)\cdot \tau$. Thus $\varphi_g$ induces a biholomorphism of the quotient $U_g(\ZZ)\backslash \mathbf{H}_g$ onto the closed submanifold $\varphi_g(\mathbf{H}_g)$ of $B_g(\CC)$. 
\end{obs}

\begin{obs}
 There exist non-closed leaves of the higher Ramanujan foliation on $B_g(\CC)$. Take for instance
  $$
  \delta = \left(\begin{array}{cc}
                   x\mathbf{1}_g & -\mathbf{1}_g\\
                   \mathbf{1}_g & 0
                 \end{array}\right)
  $$
  where $x\in \mathbf{R}\smallsetminus \QQ$. Using the classical fact that the orbit of $(x,1)$ in $\RR^2$ under the obvious left action of $\SL_{2}(\ZZ)$ is dense in $\RR^2$, one may easily deduce that the leaf $L\subset B_g(\CC)$ given by the image of $\delta U_g(\CC)\cap \mathbf{B}_g$ under the quotient map $\Sp_{2g}(\CC) \to \Sp_{2g}(\ZZ)\backslash \Sp_{2g}(\CC)$ has a limit point in $B_g(\CC)\smallsetminus L$. In particular, the ``space of leaves'' of the higher Ramanujan foliation on $B_g(\CC)$, which may be identified with $\Sp_{2g}(\ZZ)\backslash \Sp_{2g}(\CC) /U_g(\CC)$ by Proposition \ref{charactleaves}, is not a Hausdorff topological space.

The dynamics of the higher Ramanujan foliation in the case $g=1$ was thoroughly studied by Movasati in \cite{movasati08}.
\end{obs}

\subsubsection{}

In the sequel, it will be useful to obtain a description of $\varphi_{\delta}$ purely in terms of the universal property of $B_g(\CC)$. Let $\delta  = ( A \ B \ ; \ C \ D ) \in \Sp_{2g}(\CC)$ and define a holomorphic map $p_{\delta} : U_{\delta} \to P_g(\CC)$ by
\begin{align*}
p_{\delta}(\tau) = p_{\delta, \tau}\defeq \left(\begin{array}{cc}
                       (C \tau +D)^{-1} & -\frac{1}{2\pi i}C\transp\\[0.5em]
                         0             & (C\tau + D)\transp
                       \end{array} \right) \in P_g(\CC)\text{.}
\end{align*}

The proof of the next lemma is a straightforward computation using the equations defining the symplectic group (cf. Remark \ref{eqsympl}).

\begin{lemma}\label{psidelta}
  For every $\tau \in U_{\delta}\subset \mathbf{H}_g$, we have
  \begin{align*}
    \psi_{\delta}(\tau) = \psi(\tau)p'_{\delta, \tau}
  \end{align*}
  in $\Sp_{2g}(\CC)$, where $p'_{\delta,\tau}$ denotes the image of $p_{\delta,\tau}$ in $P_g'(\CC)$ under the isomorphism defined in Lemma \ref{equivariancepi}.\hfill $\blacksquare$
\end{lemma}

In particular, by Lemma \ref{equivariancepi} and Lemma \ref{psi}, if  $\mathbf{B}_g$ is regarded as the moduli space of principally polarized complex tori of dimension $g$ equipped with a symplectic-Hodge basis and an integral symplectic basis, we have
  \begin{align}\label{psimoduli}
\psi_{\delta}(\tau) = [(\mathbf{X}_{g,\tau},E_{g,\tau},\bfb_{g,\tau}\cdot p_{\delta,\tau},\beta_{g,\tau})] \in \mathbf{B}_g
    \end{align}
    for every $\tau \in U_{\delta}$. Composing with the canonical map $\mathbf{B}_g \to B_g(\CC)$, we obtain
    \begin{align}\label{phimoduli}
      \varphi_{\delta}(\tau) = [(\mathbf{X}_{g,\tau},E_{g,\tau},\bfb_{g,\tau}\cdot p_{\delta,\tau})] \in B_g(\CC)
    \end{align}
for every $\tau \in U_{\delta}$.

 \subsection{Auxiliary results}

 Our next objective is to prove that the leaves of the higher Ramanujan foliation on $B_g(\CC)$ are Zariski-dense in $B_{g,\CC}$. We collect in this subsection some auxiliary results. In the last analysis, our proof is a reduction to the fact that $\Sp_{2g}(\ZZ)$ is Zariski-dense in $\Sp_{2g,\CC}$ (Lemma \ref{spzdense}).
 
Recall that for every $\tau \in \mathbf{H}_g$ and
\begin{align*}
\delta = \left(\begin{array}{cc}
                A & B \\
                C & D
               \end{array} \right) \in {\Sp}_{2g}(\CC)
\end{align*}
we put
\begin{align*}
j(\delta,\tau) \defeq C\tau+D \in M_{g\times g}(\CC)\text{,}
\end{align*}
so that $U_{\delta} = \{\tau \in \mathbf{H}_g \mid j(\delta,\tau)\in \GL_g(\CC)\}$.

The proof of the next lemma is a simple computation. 

\begin{lemma} \label{lemmej}
For $\delta_1,\delta_2 \in \Sp_{2g}(\CC)$, we have $j(\delta_1\delta_2,\tau) = j(\delta_1,\delta_2\cdot \tau) j(\delta_2,\tau)$. In particular, if $\tau \in U_{\delta_2}$ and $\delta_2\cdot \tau \in U_{\delta_1}$, then $\tau \in U_{\delta_1\delta_2}$.\hfill $\blacksquare$
\end{lemma}

\begin{lemma}\label{lemmeutile}
Let $\delta\in \Sp_{2g}(\CC)$, $\gamma \in \Sp_{2g}(\ZZ)$, and $\tau \in U_{\delta\gamma}\subset \mathbf{H}_g$. Then $\gamma\cdot \tau \in U_{\delta}$ and $\varphi_{\delta\gamma}(\tau) = \varphi_{\delta}(\gamma\cdot \tau)$.  
\end{lemma}

\begin{proof}
  That $\gamma\cdot \tau \in U_{\delta}$ is a direct consequence of Lemma \ref{lemmej} and the fact that $j(\gamma,\tau) \in \GL_g(\CC)$ (this is true for any $\gamma \in \Sp_{2g}(\RR)$ and $\tau \in \mathbf{H}_g$). Under the group-theoretic interpretation, we have
  \begin{align*}
    \varphi_{\delta\gamma}(\tau) &= {\Sp}_{2g}(\ZZ)\psi_{\delta\gamma}(\tau) = {\Sp}_{2g}(\ZZ) (\delta\gamma)^{-1}\psi((\delta\gamma)\cdot \tau) \\                                                                               &= {\Sp}_{2g}(\ZZ) \delta^{-1}\psi(\delta\cdot (\gamma\cdot \tau)) = {\Sp}_{2g}(\ZZ)\psi_{\delta}(\gamma\cdot \tau) = \varphi_{\delta}(\gamma \cdot \tau)\text{.}
  \end{align*}
\end{proof}

\begin{lemma}\label{spzdense}
The set $\Sp_{2g}(\ZZ)\subset \Sp_{2g}(\CC)$ is Zariski-dense in $\Sp_{2g,\CC}$.
\end{lemma}

\begin{proof}
  Let $\Sp_{2g}^*$ be the open subscheme of $\Sp_{2g}$ defined by $\Sp_{2g}^*(R)=\{(A \ B \ ; \ C \ D )\in \Sp_{2g}(R) \mid A\in \GL_g(R)\}$ for any ring $R$. We may define an isomorphism of schemes $\Sp_{2g}^* \stackrel{\sim}{\to} \Sym_g\times_{\ZZ} \Sym_g \times_{\ZZ} \GL_g$ by
\begin{align*}
    \left(\begin{array}{cc}
      A & B \\
      C & D
     \end{array}\right) \mapsto (CA^{-1},AB\transp, A)\text{.}         
\end{align*}
Since $\Sym_g\times_{\ZZ} \Sym_g \times_{\ZZ} \GL_g$ may be identified to an open subscheme of the affine space $\AA_{\ZZ}^{2g^2+g}$, we see that $\Sym_g(\ZZ)\times \Sym_g(\ZZ) \times \GL_g(\ZZ)$ is Zariski-dense in $\Sym_{g,\CC}\times_{\CC} \Sym_{g,\CC}\times_{\CC}\GL_{g,\CC}$. Thus $\Sp_{2g}^*(\ZZ)$ is Zariski-dense in $\Sp_{2g,\CC}^*$. Finally, since $\Sp_{2g,\CC}$ is an irreducible scheme, we conclude that $\Sp_{2g}(\ZZ)$ is Zariski-dense in $\Sp_{2g,\CC}$. 
\end{proof}

%Sp irreducible iff connected (generated by symplectic transvections v -> v + \lambda \langle v, w \rangle w

\begin{lemma} \label{surjective}
Let $\tau \in \mathbf{H}_g$ and $p\in P_g(\CC)$. Then there exists $\delta \in \Sp_{2g}(\CC)$ such that $\tau \in U_{\delta}$ and $p=p_{\delta,\tau}$.
\end{lemma}

\begin{proof}
Let $A \in \GL_g(\CC)$ and $B\in M_{g\times g}(\CC)$ such that
\begin{align*}
p = \left(\begin{array}{cc}
          A & B\\
          0 & (A\transp)^{-1}
\end{array} \right)\text{.}
\end{align*}
One easily verifies, using the equation $A B\transp = BA\transp$, that
\begin{align*}
\delta \defeq \left(\begin{array}{cc}
          A\transp & -A\transp\tau\\
          -2\pi i\, B\transp & A^{-1} + 2\pi i\, B\transp \tau
\end{array} \right) \in M_{2g\times 2g}(\CC)
\end{align*}
is in $\Sp_{2g}(\CC)$ and satisfies the required conditions in the statement.
\end{proof}

\begin{lemma}\label{principallemma}
For every $\delta \in \Sp_{2g}(\CC)$ and $\tau \in \mathbf{H}_g$, the subset
\begin{align*}
S_{\delta,\tau} \defeq \{p_{\delta\gamma,\tau} \in P_g(\CC) \mid \gamma \in {\Sp}_{2g}(\ZZ)\text{ such that }j(\delta\gamma,\tau)\in {\GL}_g(\CC)\}
\end{align*}
of $P_{g}(\CC)$ is Zariski-dense in $P_{g,\CC}$.
\end{lemma}

\begin{proof}
Let $V$ be the unique open subscheme of $\Sp_{2g,\CC}$ such that 
\begin{align*}
V(\CC)= \{\gamma \in {\Sp}_{2g}(\CC) \mid j(\delta \gamma,\tau)\in {\GL}_g(\CC)\}
\end{align*}
and let $h : V \to P_{g,\CC}$ be the morphism of $\CC$-schemes given on complex points by $h(\gamma)= p_{\delta\gamma,\tau}$ (note that $V$ and $P_{g,\CC}$ are reduced separated $\CC$-schemes of finite type). It follows from Lemma \ref{surjective} that $h$ is surjective on complex points, thus a dominant morphism of schemes.

Now, we remark that $S_{\delta,\tau} = h({\Sp}_{2g}(\ZZ)\cap V)$. Since $\Sp_{2g,\CC}$ is irreducible and $\Sp_{2g}(\ZZ)$ is Zariski-dense in $\Sp_{2g,\CC}$ by Lemma \ref{spzdense}, $\Sp_{2g}(\ZZ)\cap V$ is also Zariski-dense in $\Sp_{2g,\CC}$. Hence, as $h$ is dominant and continuous for the Zariski topology, $S_{\delta,\tau}$ is Zariski-dense in $P_{g,\CC}$.
\end{proof}

\subsection{Statement and proof of our Zariski-density results} \label{proofdensite}

Recall  that we denote the coarse moduli scheme of $\mathcal{A}_g$ by $A_g$, and that we have a canonical map $j_g: \mathbf{H}_g \to A_g(\CC)$ associating to each $\tau \in \mathbf{H}_g$ the isomorphism class of the principally polarized complex torus $(\mathbf{X}_{g,\tau},E_{g,\tau})$. 

The proof of our Zariski-density results will rely on the following elementary lemma.

\begin{lemma}[Fibration method] \label{metfib}
Let $p:X \to S$ be a morphism of separated $\CC$-schemes of finite type and let $E\subset X(\CC)$ be a subset. If, for every $s \in p(E)$, the set $E\cap X_s$ is Zariski-dense in $X_s\defeq p^{-1}(s)$, and one of the following conditions is satisfied, 
\begin{enumerate}[(i)]
    \item $p(E) = S(\CC)$,
    \item $p$ is open (in the Zariski topology) and $p(E)$ is Zariski-dense in $S$,
\end{enumerate}
then $E$ is Zariski-dense in $X$.
\end{lemma}

\begin{proof}
Let $U$ be a non-empty Zariski open subset of $X$; we must show that $E\cap U$ is non-empty. In both cases (i) and (ii) above, there exists a closed point $s \in p(E)\cap p(U)$. Since $E\cap X_s$ is Zariski-dense in $ X_s$ and $U\cap X_s$ is a non-empty open subset of $X_s$, there exists a closed point $x\in E\cap U\cap X_s\subset E\cap U$. 
\end{proof}

\begin{theorem} \label{densite}
 Every leaf $L\subset B_g(\CC)$ of the higher Ramanujan foliation is Zariski-dense in $B_{g,\CC}$, that is, for every closed subscheme $Y$ of $B_{g,\CC}$, if $Y(\CC)$ contains $L$, then $Y(\CC)=B_g(\CC)$. 
\end{theorem}

\begin{proof}
By Proposition \ref{charactleaves}, we must prove that, for every $\delta \in \Sp_{2g}(\CC)$, the image of $\varphi_{\delta}: U_{\delta} \to B_g(\CC)$ is Zariski-dense in $B_{g,\CC}$.

Let $\varpi_g: B_{g,\CC} \to A_{g,\CC}$ be as in Paragraph \ref{proofthmtrdeg}. By Lemma \ref{metfib}, we are reduced to proving that, for every $x \in A_g(\CC)$, the set
\begin{align*}
\varphi_{\delta}(U_{\delta})\cap \varpi_g^{-1}(x)
\end{align*}
is Zariski-dense in $\varpi_g^{-1}(x) \subset B_{g,\CC}$. Indeed, by surjectivity of $\varpi_g$ on the level of complex points, this proves in particular that $\varpi_g(\varphi_{\delta}(U_{\delta}))=A_g(\CC)$ (cf. condition (i) in Lemma \ref{metfib}).

% Let 
% \begin{align*}
% \varpi_g: B_{g,\CC} \to A_{g,\CC}
% \end{align*}
% be the composition of the forgetful functor $\pi_g:B_{g,\CC} \cong \mathcal{B}_{g,\CC} \to \mathcal{A}_{g,\CC}$ with the canonical morphism $\mathcal{A}_{g,\CC}\to A_{g,\CC}$. Note that $\varpi_g$ acts on complex points by sending an isomorphism class in $\mathcal{B}_{g}(\CC)$ of a principally polarized complex abelian variety endowed with a symplectic-Hodge basis to the isomorphism class in $\mathcal{A}_{g}(\CC)$ of the same principally polarized complex abelian variety.

Let $(X,\lambda)$ be a representative of the isomorphism class $x$. Recall that the set of complex points of the $\CC$-scheme $\varpi_g^{-1}(x)$ can be identified with the set of isomorphism classes of objects of the category $\mathcal{B}_{g}(\CC)$ lying over $(X,\lambda)$ --- we denote these isomorphism classes by $[(X,\lambda,b)]$ ---, and that the $\CC$-group scheme $P_{g,\CC}$ acts transitively on $\varpi_g^{-1}(x)$ by 
\begin{align*}
[(X,\lambda,b)]\cdot p \defeq [(X,\lambda,b\cdot p)]\text{.}
\end{align*}
 Thus, if $\tau \in \mathbf{H}_g$ satisfies $j_g(\tau)=x$, we can define a surjective morphism of $\CC$-schemes (cf. Lemma \ref{lemma-finsurj})
\begin{align*}
f_{\tau} : P_{g,\CC} &\longrightarrow \varpi_g^{-1}(x)\\
              p &\mapsto \varphi_g(\tau)\cdot p\text{.}
\end{align*}

Now, let $\gamma \in \Sp_{2g}(\ZZ)$ be such that $j(\delta\gamma,\tau) \in \GL_g(\CC)$. By Lemma \ref{lemmeutile}, we have $\gamma\cdot \tau \in U_{\delta}$ and $\varphi_{\delta\gamma}(\tau) = \varphi_{\delta}(\gamma\cdot \tau)$. Thus, by formula (\ref{phimoduli}), we obtain
\begin{align*}
f_{\tau}(p_{\delta\gamma,\tau}) = \varphi_{g}(\tau) \cdot p_{\delta\gamma,\tau} = \varphi_{\delta\gamma}(\tau) = \varphi_{\delta}(\gamma\cdot \tau)\text{.}
\end{align*}
This proves that
\begin{align*}
  S_{\delta,\tau} = \{p_{\delta\gamma,\tau} \in P_g(\CC) \mid \gamma \in {\Sp}_{2g}(\ZZ)\text{ such that }j(\delta\gamma,\tau)\in {\GL}_g(\CC)\} \subset f_{\tau}^{-1}(\varphi_{\delta}(U_{\delta})\cap \varpi_g^{-1}(x))\text{.}
\end{align*}
By Lemma \ref{principallemma}, $S_{\delta,\tau}$ is Zariski-dense in $P_{g,\CC}$. Hence, as $f_{\tau}$ is surjective and continuous for the Zariski topology, we conclude that $\varphi_{\delta}(U_{\delta})\cap \varpi_g^{-1}(x)$ is Zariski-dense in $\varpi_g^{-1}(x)$.
\end{proof}

\begin{coro} \label{graphe}
 The set $ \{(\tau,\varphi_g(\tau)) \in {\Sym}_g(\CC)\times B_g(\CC) \mid \tau \in \mathbf{H}_g\}$ is  Zariski-dense in $\Sym_{g,\CC}\times_{\CC} B_{g,\CC}$.
\end{coro}

\begin{proof}
It is clear that $\Sym_g(\ZZ)$ is Zariski-dense in $\Sym_{g,\CC}$. Thus, by Theorem \ref{densite} and Lemma \ref{metfib} (ii) applied to the projection on the second factor
\begin{align*}
{\Sym}_{g,\CC}\times_{\CC}B_{g,\CC} \to B_{g,\CC}\text{,}
\end{align*} 
it suffices to prove that for every $N \in \Sym_g(\ZZ)$ and $\tau \in \mathbf{H}_g$ we have $\varphi_g(\tau+N)=\varphi_g(\tau)$. This was already observed in Remark \ref{remarkphi}.
\end{proof}

We now state the analogous results for the Hilbert-Blumenthal case, which can be proved \emph{mutatis mutandis} by the same method.

 \begin{theorem}
 Every leaf $L\subset B_F(\CC)$ of the higher Ramanujan foliation (i.e., the holomorphic foliation given by the integrable subbundle $\mathcal{R}_F^{\an}$ of $T_{B_F(\CC)}$ generated by the image of $v_F : \mathcal{O}_{B_F(\CC)}\tensor D^{-1}\to T_{B_{F}(\CC)}$) is Zariski-dense in $B_{F,\CC}$. \hfill $\blacksquare$
 \end{theorem}

 \begin{coro} 
 The set $\{(\tau,\varphi_F(\tau)) \in (\Res_{R/\ZZ}\AA^1_R)(\CC)\times B_F(\CC) \mid \tau \in \mathbf{H}^g\}$ is  Zariski-dense in $(\Res_{R/\ZZ}\AA^1_R)_{\CC}\times_{\CC} B_{F,\CC}$.\hfill $\blacksquare$
 \end{coro}

\subsection{Derivatives of modular functions and $B_g$}\label{subsec-bertrandzudilin}

We next explain how the moduli space $B_g$ and the holomorphic map $\varphi_g: \mathbf{H}_g \to B_g(\CC)$ relate with derivatives of Siegel modular functions and the work of Bertrand-Zudilin \cite{BZ03}.

Recall that a (level 1) \emph{Siegel modular function} of genus $g$ is a meromorphic function on $\mathbf{H}_g$ which is invariant under the action of $\Sp_{2g}(\ZZ)$ on $\mathbf{H}_g$. In particular, a Siegel modular function $f$ is invariant under $U_g(\ZZ)$, so that it admits a Laurent expansion
$$
f(\tau) = \sum_{\alpha}c_{\alpha}\prod_{1\le i \le j \le g}q_{ij}(\tau)^{\alpha_{ij}}\text{,}
$$
where $q_{ij}(\tau) =  e^{2\pi i \tau_{ij}}$ (cf. Paragraph \ref{subsec-compatsiegel}). Here, we denote $\alpha=(\alpha_{ij})_{1\le i \le j \le g}$ with $\alpha_{ij}\in \ZZ$ for every $1\le i \le j \le g$. We say that $f$ is \emph{defined over a subfield $k$ of $\CC$} if each $c_{\alpha}$ is in $k$.

From now on, let us fix a subfield $k$ of $\CC$, and let us denote by $K_g$ the field of modular functions of genus $g$ defined over $k$. It is classical that $j_g: \mathbf{H}_g \to A_g(\CC)$ identifies the $K_g$ with $k(A_{g,k})$, the function field  of $A_{g,k}$ (see, for instance, \cite{shimura98} VI.25). 

Since the image of $\varphi_g: \mathbf{H}_g \to B_g(\CC)$ is Zariski-dense by Theorem \ref{densite}, the function field $k(B_{g,k})$ can be identified with a subfield, say $L_g$, of the field  of meromorphic functions on $\mathbf{H}_g$. From the commutativity of the diagram
$$
\begin{tikzcd}
  & B_g(\CC)\arrow{d}{\pi_g} \\
\mathbf{H}_g \arrow{r}[swap]{j_g} \arrow[bend left]{ru}{\varphi_g}& A_g(\CC)
\end{tikzcd}
$$
it follows that $K_g$ is a subfield of $L_g$.

\begin{lemma}
The field $L_g$ is stable under the derivations $\theta_{ij}=\frac{1}{2\pi i}\frac{\partial}{\partial \tau_{ij}}$, $1\le i \le j \le g$.
\end{lemma}

\begin{proof}
  This follows from the fact that $\varphi_g$ is a solution of the higher Ramanujan equations (Theorem \ref{theoremsolution}): if $f$ is a rational function on $B_{g,k}$, then
  $$
\theta_{ij}(\varphi_g^*f) = \varphi_g^*(v_{ij}(f))\text{.}
  $$
\end{proof}

It follows from the above lemma that, if $M_g$ denotes the differential field generated by $K_g$ and $\theta_{ij}$, $1\le i \le j \le g$, then $L_g$ contains $M_g$.

\begin{theorem}[Bertrand-Zudilin, \cite{BZ03}]
The field $M_g$ has transcendence degree $2g^2+g$ over $k$. 
\end{theorem}

Now, $L_g$ being isomorphic to the function field of the $k$-variety $B_{g,k}$, it is a finitely generated extension of $k$ of transcendence degree $\dim B_{g,k}=2g^2+g$. We conclude that $L_g$ is a \emph{finite} field extension of $M_g$.

\begin{obs}
  When $g=1$, we have $K_1=k(j)$ and $L_1= k(E_2,E_4,E_6)$ (cf. Proposition \ref{prop-solg=1}). The explicit formulas
  $$
E_2 = 6\frac{\theta^2j}{\theta j}-4\frac{\theta j }{j} - 3\frac{\theta j }{j-1728}\text{, } \ \ E_4=\frac{(\theta j)^2}{j(j-1728)}\text{, }\ \ E_6=-\frac{(\theta j)^3}{j^2(j-1728)}
$$
actually show that $M_1=L_1$. We do not know whether $M_g$ should be equal to $L_g$ for $g\ge 2$.
\end{obs}

\begin{obs}
Note that the methods of Bertrand and Zudilin can be adapted to deal with the case of Hilbert-Blumenthal modular functions (see \cite{BZ01} Remark 3; see also \cite{pellarin05} 6.5). Working as above, we can prove that $k(B_{F,k})$ is a finite extension of the differential field generated by the Hilbert-Blumenthal modular functions defined over $k$ for the group $\SL(D^{-1}\oplus R)$. 
\end{obs}

% Recall that a (level 1) \emph{Siegel modular form} of genus $g$ and degree $k\ge 0$ is a holomorphic function
% $$
% f: \mathbf{H}_g \to \CC
% $$
% such that, for every $\gamma \in \Sp_{2g}(\ZZ)$, we have
% \begin{align}\label{eq-defsiegelmf}
% f(\gamma \cdot \tau) = (\det j(\gamma,\tau))^kf(\tau)
% \end{align}
% for every $\tau \in \mathbf{H}_g$. When $g=1$, we also require $f$ to be ``holomorphic at $\infty$'' (cf. \warn{ref}).

% Equation (\ref{eq-defsiegelmf}) implies that $f(\gamma\cdot \tau)=f(\tau)$ for $\gamma \in U_g(\ZZ)$, so that $f$ admits a $q$-expansion: there exist $c_{\alpha}\in \CC$ such that
% $$
% f(\tau) = \sum_{\alpha}c_{\alpha}\prod_{1\le i \le j \le g}q_{ij}(\tau)^{\alpha_{ij}}\text{.}
% $$
% Here, $\alpha=(\alpha_{ij})_{1\le i \le j \le g}$ with $\alpha_{ij}\in \NN$ for every $1\le i \le j \le g$. When $g\ge 2$, the non-negativity of $\alpha_{ij}$ is known as the Koecher principle (see \warn{ref}).

% We say that $f$ is \emph{defined over $\QQ$} if each $c_{\alpha}$ above is a rational number. We denote the $\QQ$-vector space of Siegel modular forms defined over $\QQ$ of genus $g$ and weight $k$ by $R_{g,k}$. We thus obtain the graded $\QQ$-algebra of Siegel modular forms of genus $g$ defined over $\QQ$:
% $$
% R_g \defeq \bigoplus_{k\ge 0}R_{g,k}\text{.}
% $$

\end{document}